\documentclass[10pt]{amsart}
\usepackage{mathrsfs}
\usepackage{amssymb}
\usepackage{graphicx}
\usepackage{mathrsfs}
\usepackage[all]{xy}
\usepackage[colorlinks=true]{hyperref}
\makeindex
\newcommand{\Pf}{\mathrm{Pf}}

\newcommand{\Aut}{\mathrm{Aut}}

\newcommand{\Ad}{\mathrm{Ad}}

\newcommand{\Yok}{Y_{0}^{\mathfrak k}}
\newcommand{\ad}{\mathrm{ad}}
\newcommand{\pa}{\partial}

\newcommand{\ac}{^{\cdot}}

\newcommand{\n}{\nabla}
\newcommand{\N}{\mathbf{N}}
\newcommand{\Z}{\mathbf{Z}}
\newcommand{\R}{\mathbf{R}}
\newcommand{\C}{\mathbf{C}}
\def \Vol{\mathrm {Vol}}

\def \Trs{\mathrm {Tr_{s}}}
\def \Tr{\mathrm{Tr}}

\def \ch{\mathrm {ch}}
\def \End{\mathrm {End}}

\def\even{\mathrm{even}}

\def \reg{\mathrm{reg}}

\theoremstyle{plain}
\newtheorem{lem}{Lemma}[section]
\newtheorem{theorem}[lem]{Theorem}
\newtheorem{proposition}[lem]{Proposition}

\theoremstyle{definition}
\newtheorem{definition}[lem]{Definition}

\theoremstyle{remark}
\newtheorem{remark}[lem]{Remark}

\numberwithin{equation}{section}
\begin{document} 
 
 \bibliographystyle{alpha}
 
\title[Orbital integrals and center of enveloping algebra]{Geometric orbital integrals and the center of the enveloping 
algebra}


\author{Jean-Michel \textsc{Bismut}}
\address{Institut de Mathématique d'Orsay \\ Université Paris-Saclay
\\ Bâtiment 307 \\ 91405 Orsay \\ France}
\curraddr{}
\email{jean-michel.bismut@universite-paris-saclay.fr}
\author{Shu \textsc{Shen}}
\address{Institut de Mathématiques de Jussieu-Paris Rive Gauche\\
Sorbonne Université\\ Case Courrier 247\\4 place Jussieu\\75252 Paris 
Cedex 05\\ France}
\email{shu.shen@imj-prg.fr}
\thanks{The authors are  much indebted to Laurent Clozel for his 
stimulating remarks during the preparation of the paper,  and for 
reading  the preliminary version of this paper very carefully.}
\subjclass[2020]{11F72, 22E30,}
\keywords{Selberg trace 
formula, Analysis on real and complex Lie groups}
\date{\today}


\dedicatory{}

\begin{abstract}
	The purpose of this paper is to extend the explicit geometric 
evaluation of semisimple orbital integrals for smooth kernels for the 
Casimir operator obtained by the first author to the case of kernels for arbitrary elements in the 
center of the enveloping algebra. 
 \end{abstract}
\maketitle
\tableofcontents
\section{Introduction}%
\label{sec:intro}
In \cite[Chapter 6]{Bismut08b}, one of us established a geometric 
formula for the semisimple orbital integrals  of smooth kernels associated with the Casimir.  The purpose of 
this paper is to extend this formula to the smooth kernels where more 
general elements of 
the center of the enveloping algebra also appear. 

Let us briefly describe our main 
result in more detail. 
Let $G$ be a connected real reductive group, and let $\mathfrak g$ be 
its Lie algebra. Let $\theta\in  \mathrm{Aut}\left(G\right)$ be a 
Cartan involution, and let $K \subset G$ be the corresponding maximal 
compact subgroup with Lie algebra $\mathfrak k$. Let $ \mathfrak g= 
\mathfrak p \oplus \mathfrak k$ be the associated Cartan splitting.  
Let $B$ be a symmetric nondegenerate bilinear form  on $\mathfrak g$ 
which is  $G$ and $\theta$ invariant, positive on 
$\mathfrak p$ and negative on $\mathfrak k$.  Let $X=G/K$ be the associated symmetric space, a 
Riemannian manifold with parallel nonpositive curvature. 

Let $\rho^{E}:K\to U\left(E\right)$ be a finite-dimensional unitary 
representation of $K$, and let $F=G\times_{K}E $ be the corresponding 
vector bundle on $X$. Then $G$ acts on the left on $C^{\infty 
}\left(X,F\right)$. Let $U\left(\mathfrak g\right)$ be the enveloping 
algebra of $\mathfrak g$, and let $Z\left(\mathfrak g\right)$ be the 
center of $U\left(\mathfrak g\right)$. Then $Z\left(\mathfrak 
g\right)$ acts on $C^{\infty }\left(X,F\right)$ and its action 
commutes with the left action of $G$. Among the elements of 
$Z\left(\mathfrak g\right)$, there is the Casimir $C^{\mathfrak g}$, 
whose action on $C^{\infty }\left(X,F\right)$ is denoted 
$C^{\mathfrak g,X}$.  

Let 
\index{SR@$\mathcal{S}^{\even}\left(\R\right)$}%
$\mathcal{S}^{\even}\left(\R\right)$ denote the even real 
functions on $\R$ that lie in the Schwartz space 
\index{SR@$\mathcal{S}\left(\R\right)$}%
$\mathcal{S}\left(\R\right)$. Let $\mu\in \mathcal{S}^{\mathrm{even}}\left(\R\right)$ be such that if 
$\widehat{\mu}\in \mathcal{S}^{\mathrm{even}}\left(\R\right)$ is its Fourier transform, 
there is $C>0$, and for any $k\in \N$, there is $c_{k}>0$ such that
\begin{equation}\label{eq:arc1}
\left\vert  \widehat{\mu}^{(k)}\left(y\right)\right\vert\le 
c_{k}\exp\left(-Cy^{2}\right).
\end{equation}
If $A\in \R$, $\mu\left(\sqrt{C^{\mathfrak g,X}+A}\right)$ is a 
well-defined operator with a smooth kernel.

If $\gamma\in G$ is semisimple, as explained in \cite[Section 6.2]{Bismut08b}, 
the orbital integral 
$\Tr^{\left[\gamma\right]}\left[\mu\left(\sqrt{C^{\mathfrak 
g,X}+A}\right)\right]$ is well-defined, and it only depends on the 
conjugacy class of $\gamma$ in $G$. After conjugation, we can write 
$\gamma$ in the form $\gamma=e^{a}k^{-1},a\in \mathfrak p, k\in 
K,\Ad\left(k^{-1}\right)a=a$. If $Z\left(\gamma\right)\subset G$ is the 
centralizer of $\gamma$ with Lie algebra $\mathfrak 
z\left(\gamma\right)$, then $\theta$ acts on $Z\left(\gamma\right)$, 
and $Z\left(\gamma\right)$ is a   possibly nonconnected reductive 
group. Let $\mathfrak z\left(\gamma\right)= \mathfrak 
p\left(\gamma\right) \oplus \mathfrak k\left(\gamma\right)$ be the 
associated Cartan splitting. 

Let 
\index{Ig@$I\ac\left(\mathfrak g\right)$}%
$I\ac\left(\mathfrak g\right)$ be the algebra of invariant 
polynomials on $\mathfrak g^{*}$, and let
\index{tD@$\tau_{\mathrm{D}}$}%
$\tau_{\mathrm{D}}:  
I\ac\left(\mathfrak g\right) \simeq Z\left(\mathfrak g\right)$ denote the Duflo 
isomorphism \cite[Théorème V.2]{Duflo70}. If $\mathfrak h \subset \mathfrak g$ is 
 a Cartan subalgebra, let $I\ac\left(\mathfrak h,\mathfrak 
g\right)$ denote the algebra of polynomials on $\mathfrak h^{*}$ that 
are invariant under the corresponding algebraic Weyl 
group, so that 
we have the canonical identification $I\ac\left(\mathfrak g\right) 
\simeq I\ac\left(\mathfrak h, \mathfrak g\right)$\footnote{This 
isomorphism is usually written in its complex version 
$I\ac\left(\mathfrak g_{\C}\right) \simeq I\ac\left(\mathfrak h_{\C}, 
\mathfrak g_{\C}\right)$. 
In Subsections \ref{subsec:reglie} and \ref{subsec:refo},  the corresponding real version is 
derived. Such considerations  will also apply to other complex 
isomorphisms.}.
There is a Harish-Chandra isomorphism 
\index{fHC@$\phi_{\mathrm{HC }}$}%
$\phi_{\mathrm{HC}}: Z\left(\mathfrak g\right) \simeq I\ac\left(\mathfrak h, \mathfrak g\right)$. By 
\cite[Lemme V.1]{Duflo70}, the Duflo and 
Harish-Chandra isomorphisms are known to be compatible.

There is a canonical projection $\mathfrak g\to \mathfrak 
z\left(\gamma\right)$\footnote{This projection is defined in 
Subsection \ref{subsec:moha}.}, that induces a corresponding projection
  $I\ac\left(\mathfrak g\right)\to 
I\ac\left(\mathfrak z\left(\gamma\right)\right)$. If $L\in Z\left(\mathfrak g\right)$, 
let 
\index{Lzg@$L^{\mathfrak z\left(\gamma\right)}$}%
$L^{\mathfrak z\left(\gamma\right)}$ denote the differential 
operator on $\mathfrak z\left(\gamma\right)$ canonically associated 
with the projection of $\tau_{\mathrm{D}}^{-1}L$ on $I\ac\left(\mathfrak 
z\left(\gamma\right)\right)$. In particular, up to a constant,  $-\left(C^{\mathfrak 
g} \right) ^{\mathfrak z\left(\gamma\right)}$ extends to the   standard 
Laplacian on the Euclidean vector space 
\index{zig@$\mathfrak z_{i}\left(\gamma\right)$}%
$\mathfrak z_{i}\left(\gamma\right)= \mathfrak p\left(\gamma\right) \oplus i 
\mathfrak k\left(\gamma\right)$.

Following \cite[Chapter 5]{Bismut08b}, in Definition \ref{def:Jg}, we define a smooth function 
$\mathcal{J}_{\gamma}: i \mathfrak k\left(\gamma\right)\to \C$. Let 
$\delta_{a}$ be the Dirac mass at $a\in \mathfrak p\left(\gamma\right)$. Then 
$$\mathcal{J}_{\gamma}\left(\Yok\right)\Tr^{E}\left[\rho^{E}\left(k^{-1}e^{-\Yok}\right)\right] \otimes \delta_{a}$$
is a distribution on $\mathfrak z_{i}\left(\gamma\right)$. \footnote{In 
the sequel, $\otimes$ will be omitted.}

Our main result, which is repeated in the text as Theorem 
\ref{thm:Tgenorbis} is as follows.
\begin{theorem}\label{thm:Tgenorbiso}
The following identity holds:
	\begin{multline}\label{eq:diff2biso}
\Tr^{\left[\gamma\right]}\left[L\mu\left(\sqrt{C^{\mathfrak g,X}+A}\right)\right]\\
=L^{\mathfrak z\left(\gamma\right)}
\mu\left(\sqrt{\left( C^{\mathfrak g}\right)^{\mathfrak 
z\left(\gamma\right)}+A}\right)  \left[
\mathcal{J}_{\gamma}\left(\Yok\right)\Tr^{E}\left[\rho^{E}\left(k^{-1}e^{-\Yok}\right)\right]\delta_{a}\right]\left(0\right).
\end{multline}
\end{theorem}

When $L=1$, our theorem was already established in 
\cite[Theorem 6.2.2]{Bismut08b}.

The proofs in \cite{Bismut08b} used a construction of a new object, 
the hypoelliptic Laplacian. Here, we will only need the results of 
\cite{Bismut08b}. 

Our proof is done in two steps. In a first step, using the 
results of \cite{Bismut08b}, we  prove Theorem 
\ref{thm:Tgenorbiso} when $\gamma\in G$ is regular. In this case, 
using the properties of the Harish-Chandra isomorphism 
\cite{Harish66}, the proof is 
relatively easy. 

When $\gamma$ is nonregular, we combine our result for $\gamma$ 
regular with limit arguments due to Harish-Chandra on the behavior of 
orbital integrals when $\gamma'$ regular converges to $\gamma$. 
In both steps, remarkable and nontrivial properties of the function 
$\mathcal{J}_{\gamma}$ are used. One of these  properties is that 
one essential piece of $\mathcal{J}_{\gamma}$ can be calculated
only in terms of  imaginary roots. 

This paper is organized as follows. In Section \ref{sec:geofo}, we 
describe the geometric setting, and we explain the formula for the 
semisimple orbital integrals that was obtained in \cite{Bismut08b}. 

In Section \ref{sec:geofori}, we recall some of the properties of Cartan 
subalgebras, Cartan subgroups, and of the corresponding root systems.

In Section \ref{sec:rofu}, we express the restriction of the function 
$\mathcal{J}_{\gamma}$ to  Cartan subalgebras in terms of a positive root 
system. 

In Section \ref{sec:reg}, we specialize the results of the previous 
section to the case where $\gamma$ is regular. We prove a crucial and unexpected 
smooth dependence of $\mathcal{J}_{\gamma}$ on $\gamma$.

In Section \ref{sec:haricha}, we explain in some detail the 
Harish-Chandra isomorphism.

In Section \ref{sec:cenreg}, we establish Theorem \ref{thm:Tgenorbiso} 
when $\gamma$ is regular. 

In Section \ref{sec:rojg},  when $\gamma$ is non necessarily 
regular, we study the limit of $\mathcal{J}_{\gamma'}$,  and  the limit of 
our formula for regular orbital integrals  as $\gamma'$ 
regular converges to $\gamma$ in a suitable sense. 

In Section \ref{sec:senUg}, using the results of the 
previous section, we establish Theorem \ref{thm:Tgenorbiso} in full 
generality.

Finally, in Section \ref{sec:index}, we prove that our formula is 
compatible to the index theory for Dirac operators, and also with 
known results on Dirac cohomology \cite{HuangPan02}.
 
\section{Geometric formulas for orbital integrals and the Casimir}%
\label{sec:geofo}
In this Section, we explain the geometric formula 
given in \cite[Chapter 6]{Bismut08b} for the 
semisimple orbital integrals associated with the proper smooth 
kernels for the Casimir.

This section is organized as follows. In Subsection 
\ref{subsec:regr}, we introduce the real reductive group $G$, its 
maximal compact subgroup $K$,  the Lie 
algebras $\mathfrak g, \mathfrak k$,   and the 
symmetric space $X=G/K$. 

In Subsection \ref{subsec:sesidis}, we recall the definition of 
semisimple elements in $G$, and  of the corresponding displacement 
function.

In Subsection \ref{subsec:cas}, we introduce the enveloping algebra 
$U\left(\mathfrak g\right)$, and the Casimir element $C^{\mathfrak g}\in 
U\left(\mathfrak g\right)$.

In Subsection \ref{subsec:symsp}, given a unitary representation of 
$K$, we construct the corresponding vector bundle $F$ on $X$, and the elliptic operator 
$C^{\mathfrak g,X}$   which is 
just the action of $C^{\mathfrak g}$ on $C^{\infty }\left(X,F\right)$. 

In Subsection \ref{subsec:wake}, given  $\mu\in 
\mathcal{S}^{\mathrm{even}}\left(\R\right)$ such that its Fourier transform has the 
proper Gaussian decay,  if $A\in \R$, we recall the definition of the  semisimple orbital 
integrals associated with the smooth kernel for 
$\mu\left(\sqrt{C^{\mathfrak g,X}+A}\right)$. Among these kernels, 
there is the heat kernel for $C^{\mathfrak g,X}$. 

In Subsection \ref{subsec:fuJg}, if $\gamma\in G$ is semisimple, if 
$Z\left(\gamma\right)\subset G$ is its centralizer with Lie algebra 
$\mathfrak z\left(\gamma\right)$,  if $\mathfrak k\left(\gamma\right)$ 
is the compact part of $\mathfrak z\left(\gamma\right)$,  we recall the definition of the 
function $\mathcal{J}_{\gamma}$ on $ i \mathfrak 
k\left(\gamma\right)$ given  in \cite[Theorem 5.5.1]{Bismut08b}. 

In Subsection \ref{subsec:propjg}, we study the behavior of 
$\mathcal{J}_{\gamma}$
 when replacing by $\gamma$ by $\gamma^{-1}$, 
and also by complex conjugation. 

Finally, in  Subsection \ref{subsec:geoorb}, we state the  geometric 
formula obtained in \cite{Bismut08b} for the above
orbital integrals, in which the function $\mathcal{J}_{\gamma}$ plays 
a key role.
\subsection{Reductive groups and symmetric spaces}%
\label{subsec:regr}
Let $G$ be a connected reductive real Lie group, and let 
\index{g@$\mathfrak g$}%
$\mathfrak 
g$ be its Lie algebra.  Let 
\index{t@$\theta$}%
$\theta\in \Aut\left(G\right)$ be a 
Cartan involution. Then  $\theta$ acts as an automorphism  of 
$\mathfrak g$.  Let
\index{K@$K$}%
$K \subset G$ be the fixed point set of $\theta$. Then $K$ is a 
compact connected subgroup of $G$, which is a maximal compact 
subgroup.  
If 
\index{k@$\mathfrak k$}%
$\mathfrak k \subset \mathfrak g$ is the Lie algebra of 
$K$, then
$\mathfrak k$ is the  fixed point set of $\theta$ in 
$\mathfrak g$. Let 
\index{p@$\mathfrak p$}%
$\mathfrak p 
\subset \mathfrak g$ be the eigenspace of $\theta$ corresponding to 
the eigenvalue $-1$, so that we have the Cartan decomposition
\begin{equation}\label{eq:cent1}
\mathfrak g= \mathfrak p \oplus \mathfrak k.
\end{equation}
Put
\begin{align}\label{eq:inva1}
&m=\dim \mathfrak p, &n=\dim \mathfrak k,
\end{align}
so that
\begin{equation}\label{eq:inva2}
\dim \mathfrak g=m+n.
\end{equation}

Let 
\index{B@$B$}%
$B$ be a $G$ and $\theta$ invariant 
bilinear symmetric nondegenerate  form on $\mathfrak g$. Then 
(\ref{eq:cent1}) is a $B$-orthogonal splitting. We assume that $B$ is positive on 
$\mathfrak p$ and negative on $\mathfrak k$. Let $\left\langle  \,\right\rangle=-B\left(\cdot,\theta\cdot\right)$ be the corresponding 
scalar product on $\mathfrak g$. Let $B^{*}$ be the  bilinear  
symmetric  form on $\mathfrak g^{*}= \mathfrak p^{*} \oplus \mathfrak 
k^{*}$ which is dual to $B$.

Let $\omega^{\mathfrak g}$ be the canonical left-invariant $1$-form on 
$G$ with values in $\mathfrak g$. By (\ref{eq:cent1}), 
$\omega^{\mathfrak g}$ splits as 
\begin{equation}\label{eq:inva2a1}
\omega^{\mathfrak g}=\omega^{\mathfrak p} + \omega^{\mathfrak k}. 
\end{equation}

Let 
\index{X@$X$}%
$X=G/K$ be the corresponding symmetric space. Then $p:G\to X=G/K$ 
is a $K$-principal bundle, and $\omega^{\mathfrak k}$ is a 
connection form. Also the tangent bundle 
$TX$ is given by 
\begin{equation}\label{eq:inva2a2}
TX =G\times_{K}\mathfrak p.
\end{equation}
Then $TX$ is equipped with the scalar product $\left\langle  
\,\right\rangle$ induced by $B$, so that $X$ is a Riemannian 
manifold. The connection $\n^{TX}$ on $TX$ which is induced by 
$\omega^{\mathfrak k}$ is the Levi-Civita connection of $TX$, and its 
curvature is parallel and nonpositive.  Also $G$  acts isometrically  on the 
left on $X$, and $\theta$ acts as an isometry of 
$X$.

By \cite[Proposition 1.2]{Knapp86}, any element $\gamma\in G$ factorizes uniquely in the form 
\begin{align}\label{eq:fac1}
&\gamma=e^{a}k^{-1}, &a\in  \mathfrak p, \qquad k\in K.
\end{align}

If $\gamma,g\in G$, set
\begin{equation}\label{eq:inva1a1}
C\left(g\right)\gamma=g\gamma g^{-1}.
\end{equation}
Then $C\left(g\right)$ is an automorphism of $G$. Its derivative at 
the identity is 
 the adjoint representation $g\in G\to \Ad\left(g\right)\in 
\Aut\left(\mathfrak g\right)$. The derivative of this last map is 
given by $a\in \mathfrak g\to \mathrm{ad}\left(a\right)\in 
\End\left(\mathfrak g\right)$, with 
$\mathrm{ad}\left(a\right)b=\left[a,b\right]$. If $\gamma\in G$, the fixed point set 
of $C\left(\gamma\right)$ is the centralizer 
\index{Zg@$Z\left(\gamma\right)$}%
$Z\left(\gamma\right) 
\subset G$, whose Lie algebra $\mathfrak z\left(\gamma\right)$ is 
given by
\begin{equation}\label{eq:inva1a2}
\mathfrak z\left(\gamma\right)=\ker 
\left(1-\Ad\left(\gamma\right)\right).
\end{equation}

If $f\in \mathfrak g$, let $Z\left(f\right) \subset G$ be the 
stabilizer of $f$. Its Lie algebra $\mathfrak z\left(f\right) \subset 
\mathfrak g$ is given by
\begin{equation}\label{eq:inva1a3}
\mathfrak z\left(f\right)=\ker \mathrm{ad}\left(f\right).
\end{equation}

In the sequel, if $M$ is a Lie group, we denote by 
\index{M@$M^{0}$}%
$M^{0}$ the 
connected component of the identity.
\subsection{Semisimple elements and their displacement function}%
\label{subsec:sesidis}
Let $d$ be the Riemannian distance on $X$. By \cite[§ 6.1]{BalGroSchro}, $d$ 
is a convex function on $X\times X$. If $\gamma\in G$, let 
$d_{\gamma}$ be the corresponding displacement function on $X$, i.e., 
\begin{equation}\label{eq:inva2a4}
d_{\gamma}\left(x\right)=d\left(x,\gamma x\right). 
\end{equation}
If $g\in G$, then
\begin{equation}\label{eq:inva2a5}
d_{C\left(g\right)\gamma}\left(gx\right)=d_{\gamma}\left(x\right).
\end{equation}
Moreover,
\begin{equation}\label{eq:inva2a6}
d_{\theta\left(\gamma\right)}\left(\theta 
x\right)=d_{\gamma}\left(x\right).
\end{equation}

Set
\begin{equation}\label{eq:inva2a7}
m_{\gamma}=\inf d_{\gamma}.
\end{equation}
Let $X\left(\gamma\right) \subset X$ be the closed subset where 
$d_{\gamma}$ reaches its minimum. By \cite[p. 78 and  §1.2]{BalGroSchro},  $X\left(\gamma\right)$ is a 
closed convex subset,  $d_{\gamma}$ is smooth on $X\setminus X\left(\gamma\right)$ and has no critical 
points on $X\setminus X\left(\gamma\right)$. 
Also by (\ref{eq:inva2a5}), (\ref{eq:inva2a6}),
\begin{align}\label{eq:inva2a8}
&X\left(C\left(g\right)\gamma\right)=gX\left(\gamma\right),
&X\left(\theta\gamma\right)=\theta X\left(\gamma\right),\\
&m_{C\left(g\right)\gamma}=m_{\gamma}, 
&m_{\theta\left(\gamma\right)}=m_{\gamma}.\nonumber 
\end{align}

By  \cite[Definition 2.19.21]{Eberlein96},  $\gamma$ is said to be semisimple if $X\left(\gamma\right)$ is 
nonempty. If $\gamma$ is semisimple, then $C\left(g\right)\gamma$ and 
$\theta\left(\gamma\right)$ are semisimple. Also $\gamma$ is said to 
be elliptic if it is semisimple and $m_{\gamma}=0$. Elliptic elements 
are exactly the group elements that are conjugate to elements of $K$. 
Finally, $\gamma$ is said to be hyperbolic if it is conjugate to 
$e^{a}, a\in \mathfrak p$. 

By \cite[Proposition 2.1]{Kostant73}, \cite[Theorems 2.19.23 and 2.19.24]{BalGroSchro},  $\gamma\in G$ is semisimple if and only if it factorizes as 
$\gamma=he=eh$, with  commuting  hyperbolic $h$ and  elliptic $e$. 
Also  
$e,h$ are uniquely determined by $\gamma$, and  
\begin{equation}\label{eq:inva2a7x1}
Z\left(\gamma\right)=Z\left(h\right)\cap Z\left(e\right).
\end{equation}

Set
\index{x0@$x_{0}$}%
\begin{equation}\label{eq:gon1}
x_{0}=p1.
\end{equation}
\begin{theorem}\label{TGeod1}
Let $\gamma\in G$ be semisimple. If $g\in G,x=p g\in X$, then 
$x\in X\left(\gamma\right)$ if and only if there exist $a\in \mathfrak p,k\in 
K$ such that $\Ad\left(k\right)a=a$, and 
\begin{equation}
    \gamma=C\left(g\right)\left(e^{a}k^{-1}\right).
    \label{eq:dist3}
\end{equation}
Also $C\left(g\right)e^{a}\in G,C\left(g\right)k\in G$ are uniquely 
determined by $\gamma$. If $g_{t}=ge^{ta}$, then $t\in [0,1]\to 
y_{t}=pg_{t}$
is the unique geodesic connecting $x$ and $\gamma x$. Moreover
\begin{equation}
    m_{\gamma}=\left\vert  a\right\vert.
    \label{eq:dis4}
\end{equation}

If $\gamma\in G$ is semisimple, then $x_{0}\in 
	X\left(\gamma\right)$ if and only if there exist $a\in \mathfrak 
	p, k\in K$ such that
	\begin{align}\label{eq:dis4a}
&\gamma=e^{a}k^{-1}, &a\in \mathfrak p, \qquad \Ad\left(k\right)a=a.
\end{align}
Also $a,k$ are uniquely determined by (\ref{eq:dis4a}). 
\end{theorem}
\begin{proof}
	The first part of our theorem was established in \cite[Theorem 
3.1.2]{Bismut08b}.
By taking $g=1$ in the first part, we obtain the second part.
\end{proof}

Let $\gamma\in G$ be a semisimple element written as in (\ref{eq:dis4a}).  
By \cite[Proposition 3.2.8 and eqs. (3.3.4), (3.3.6)]{Bismut08b}, 
\begin{align}\label{eq:inva18a}
	&Z\left(e^{a}\right)=Z\left(a\right),
&Z\left(\gamma\right)=Z\left(a\right)\cap Z\left(k\right),\qquad
\mathfrak z\left(\gamma\right)=\mathfrak z\left(a\right)\cap 
\mathfrak z\left(k\right).
\end{align}

By (\ref{eq:dis4a}), $a\in \mathfrak z\left(\gamma\right)$, and by 
(\ref{eq:inva18a}), $\mathfrak z\left(\gamma\right) \subset \mathfrak 
z\left(a\right)$, so that $a$ is an element of the center of 
$\mathfrak z\left(\gamma\right)$.

Clearly, 
\begin{equation}\label{eq:inva18a2}
\theta\left(\gamma\right)=e^{-a}k^{-1}.
\end{equation}
Therefore $\theta\left(\gamma\right)\in Z\left(\gamma\right)$, so 
that the above centralizers and Lie algebras are preserved by 
$\theta$. Set
\index{Kg@$K\left(\gamma\right)$}%
\begin{equation}\label{eq:inva18a2x1}
K\left(\gamma\right)=Z\left(\gamma\right)\cap K.
\end{equation}
By \cite[Theorem 3.3.1]{Bismut08b}, we have the identity 
\begin{equation}\label{eq:bob-1}
K^{0}\left(\gamma\right)=Z^{0}\left(\gamma\right)\cap K,
\end{equation}
and $K^{0}\left(\gamma\right)$ is a maximal compact subgroup of 
$Z^{0}\left(\gamma\right)$.

Put
\index{pg@$\mathfrak p\left(\gamma\right)$}%
\index{kg@$\mathfrak k\left(\gamma\right)$}%
\begin{align}\label{eq:fina1}
&\mathfrak p\left(\gamma\right)= \mathfrak p\cap \mathfrak 
z\left(\gamma\right), 
&\mathfrak k\left(\gamma\right)=\mathfrak k\cap \mathfrak 
z\left(\gamma\right).
\end{align}
Then 
$\mathfrak k\left(\gamma\right)$ is the Lie algebra of 
$K\left(\gamma\right)$. We use similar notation for the Lie algebras 
$\mathfrak z\left(k\right), \mathfrak z\left(a\right)$.
We have the Cartan decompositions  of Lie algebras, 
\begin{align}\label{eq:inva19}
&\mathfrak z\left(\gamma\right)= \mathfrak p\left(\gamma\right) \oplus 
\mathfrak k\left(\gamma\right),
&\mathfrak z\left(k\right)= \mathfrak p\left(k\right) \oplus 
\mathfrak k\left(k\right),\qquad 
&\mathfrak z\left(a\right)= \mathfrak p\left(a\right) \oplus 
\mathfrak k\left(a\right). 
\end{align}
Then  $B$ restricts to a    nondegenerate form on 
$\mathfrak z\left(\gamma\right), \mathfrak z\left(k\right),  \mathfrak z\left(a\right)$, so that 
$Z\left(\gamma\right), Z\left(k\right), Z\left(a\right)$ are 
possibly nonconnected reductive subgroups of $G$. By \cite[Theorem 
3.3.1]{Bismut08b}, we have the identification of finite groups,
\begin{equation}\label{eq:bonbo1}
Z^{0}\left(\gamma\right)\setminus 
Z\left(\gamma\right)=K^{0}\left(\gamma\right)\setminus 
K\left(\gamma\right).
\end{equation}

Let 
\index{zg@$ \mathfrak 
z^{\perp}\left(\gamma\right)$}%
\index{za@$\mathfrak z^{\perp}\left(a\right)$}%
$ \mathfrak 
z^{\perp}\left(\gamma\right),\mathfrak z^{\perp}\left(a\right)$ be the orthogonal spaces to 
$\mathfrak z\left(\gamma\right),\mathfrak z\left(a\right)$ in 
$\mathfrak g$ with 
respect to $B$. We  have  splittings
\begin{align}\label{eq:inva19z1}
	&\mathfrak z^{\perp}\left(\gamma\right)= \mathfrak p^{\perp}\left(\gamma\right)
\oplus \mathfrak k^{\perp}\left(\gamma\right),
&\mathfrak z^{\perp}\left(a\right)= \mathfrak p^{\perp}\left(a\right) 
\oplus \mathfrak k^{\perp}\left(a\right).
\end{align}
Let 
\index{za@$\mathfrak z^{\perp}_{a}\left(\gamma\right)$}%
$\mathfrak z^{\perp}_{a}\left(\gamma\right)$ denote the 
orthogonal to $\mathfrak z\left(\gamma\right)$ in $\mathfrak 
z\left(a\right)$. We still have a splitting
\begin{equation}\label{eq:inva19z2}
\mathfrak z^{\perp}_{a}\left(\gamma\right)= \mathfrak 
p^{\perp}_{a}\left(\gamma\right) \oplus \mathfrak 
k^{\perp}_{a}\left(\gamma\right).
\end{equation}

Now we recall a result established in \cite[Theorem 3.3.1]{Bismut08b}.
\begin{theorem}\label{Ttotal}
    The set $X\left(\gamma\right)$ is preserved by $\theta$. Moreover, 
\begin{equation}
    X\left(\gamma\right)=X\left(e^{a}\right)\cap X\left(k\right).
    \label{eq:dist14}
\end{equation}
Also $X\left(\gamma\right)$ is a totally geodesic submanifold of $X$. In the 
geodesic coordinate system centered at $x_{0}=p1$, then
\begin{equation}\label{eq:last8x1}
X\left(\gamma\right)= \mathfrak p\left(\gamma\right).
\end{equation}

The actions of $Z^{0}
\left(\gamma\right), Z\left(\gamma\right)$ on $X\left(\gamma\right)$ 
are transitive. More precisely the maps $g\in Z^{0}\left(\gamma\right)
\to p g\in X ,g\in Z\left(\gamma\right)\to pg\in X$ induce the identification of 
$Z^{0}\left(\gamma\right)$-manifolds,
\begin{equation}
   X\left(\gamma\right) = Z^{0}\left(\gamma\right)/K^{0}\left(\gamma\right)=Z\left(\gamma\right)/K\left(\gamma\right).
    \label{eq:dist13}
\end{equation}
\end{theorem}

Now we will establish a simple important consequence of Theorem 
\ref{Ttotal}. 
\begin{theorem}\label{thm:Tincl}
	Let $\gamma$ be a semisimple element of $G$ as in 
	(\ref{eq:dis4a}). Let $\gamma'$ be another semisimple element of $G$ such that
\begin{align}\label{eq:dis4aa}
&\gamma'=e^{a'}k^{\prime -1}, &a'\in \mathfrak p, \qquad 
\Ad\left(k'\right)a'=a'.
\end{align}
Then there exists $g\in G$ such that $\gamma'=C\left(g\right)\gamma$ 
if and only if there exists $k''\in K$ such that 
$C\left(k''\right)\gamma=\gamma'$, in which case
\begin{align}\label{eq:dis4b}
&a'=\Ad\left(k^{\prime \prime }\right)a,&k'=C\left(k^{\prime \prime 
}\right)k.
\end{align}
\end{theorem}
\begin{proof}
Assume that $\gamma'=C\left(g\right)\gamma$. By (\ref{eq:inva2a8}), we 
	get
	\begin{equation}\label{eq:dis4c}
X\left(\gamma'\right)=gX\left(\gamma\right).
\end{equation}
By Theorem \ref{TGeod1},  $x_{0}\in 
X\left(\gamma\right)\cap X\left(\gamma'\right)$. By (\ref{eq:dis4c}), 
$gx_{0}\in X\left(\gamma'\right)$. By Theorem \ref{Ttotal}, there 
exists $h\in Z\left(\gamma'\right)$ such that
\begin{equation}\label{eq:dis4d}
hgx_{0}=x_{0}, 
\end{equation}
which is equivalent to 
\begin{equation}\label{eq:dis4da1}
k''=hg\in K.
\end{equation}
Since $h\in Z\left(\gamma'\right)$, we conclude that 
$C\left(k''\right)\gamma=\gamma'$. Using the uniqueness of  
decomposition in (\ref{eq:dis4aa}), equation (\ref{eq:dis4b}) follows. The proof 
of our theorem is completed. 
\end{proof}
\subsection{Enveloping algebra and the Casimir}%
\label{subsec:cas}
We identify $\mathfrak g$ with the Lie algebra of left-invariant 
vector fields on $G$. Let 
\index{Ug@$U\left(\mathfrak g\right)$}%
$U\left(\mathfrak g\right)$ be the enveloping algebra of 
$\mathfrak g$. Then $U\left(\mathfrak g\right)$ 
can be identified with the algebra of left-invariant differential 
operators on $G$. Let $
\index{Zg@$Z\left(\mathfrak g\right)$}%
Z\left(\mathfrak g\right) \subset 
U\left(\mathfrak g\right)$ denote the center of $U\left(\mathfrak 
g\right)$.

If $E$ is a finite-dimensional real or complex 
vector space, and if $\rho^{E}: \mathfrak g\to \End\left(E\right)$ is 
a morphism of Lie algebras, the map $\rho^{E}$ extends to a morphism 
$U\left(\mathfrak g\right)\to \End\left(E\right)$.

Among the elements of $
Z \left(\mathfrak g\right)$, there is the Casimir element 
$C^{\mathfrak g}$. If $e_{1},\ldots,e_{m+n}$ is a basis of $\mathfrak 
g$, and if 
$e^{*}_{1},\ldots,e_{m+n}^{*}$ is the dual basis of $\mathfrak g$ 
with respect to $B$, then \footnote{In \cite[Section 8.3]{Knapp86}, the 
Casimir is defined with the opposite sign. We have adopted the sign 
conventions of \cite{Bismut08b}, which are closer to analysis.}
\begin{equation}\label{eq:inva8a1}
C^{\mathfrak g}=-\sum_{i=1}^{m+n}e_{i}^{*}e_{i}.
\end{equation}
If we consider instead the Lie algebra $\left(\mathfrak k, 
B\vert_{\mathfrak k}\right)$, 
$C^{\mathfrak k}\in Z\left(\mathfrak k\right)$ denotes the associated 
Casimir element.

If $e_{1},\ldots, e_{m}$ is a basis of $\mathfrak p$, and if 
	$e^{*}_{1},\ldots,e^{*}_{m}$ is the dual basis of $\mathfrak p$ 
	with respect to $B\vert_{\mathfrak p}$, set
	\begin{equation}\label{eq:inva14z1}
C^{\mathfrak p}=-\sum_{i=1}^{m}e^{*}_{i}e_{i}.
\end{equation}
Then $C^{\mathfrak p}\in U\left(\mathfrak g\right)$. 
Using (\ref{eq:inva8a1}), (\ref{eq:inva14z1}), we get
\begin{equation}\label{eq:inva14z2}
C^{\mathfrak g}=C^{\mathfrak p}+C^{\mathfrak k}.
\end{equation}
Also $C^{\mathfrak p}$ and $C^{\mathfrak k}$ commute.

If $\rho^{E}: \mathfrak g\to \End\left(E\right)$ is taken as above, 
put
\index{CgE@$C^{\mathfrak g,E}$}%
\begin{equation}\label{eq:inva8a1x1}
C^{\mathfrak g,E}=\rho^{E}\left(C^{\mathfrak g}\right). 
\end{equation}
Under the above conditions, we can define $C^{\mathfrak p, E}, 
C^{\mathfrak k,E}$.

Since $\mathfrak g$ is itself a representation of $\mathfrak g$, 
$C^{\mathfrak g, \mathfrak g}$ is  the action of $C^{\mathfrak 
g}$ on $\mathfrak g$. 
Since $\mathfrak k$ 
acts  on $\mathfrak p, \mathfrak k$, $C^{\mathfrak k, 
\mathfrak p}, C^{\mathfrak k, \mathfrak k}$ are also well-defined. 
\begin{proposition}\label{prop:Pidetr}
	The following identity holds:
	\begin{equation}\label{eq:inva14}
\Tr\left[C^{\mathfrak g, \mathfrak g}\right]=3\Tr\left[C^{\mathfrak 
k, \mathfrak p}\right]+\Tr\left[C^{\mathfrak k, \mathfrak k}\right].
\end{equation}
\end{proposition}
\begin{proof}
	By (\ref{eq:inva14z2}), we get
\begin{equation}\label{eq:inva11}
\Tr\left[C^{\mathfrak g,\mathfrak g}\right]=\Tr\left[C^{\mathfrak 
p,\mathfrak g}\right]+\Tr\left[C^{\mathfrak k,\mathfrak g}\right].
\end{equation}

Let $e_{1},\ldots,e_{m}$ be an orthonormal basis of $\mathfrak p$, 
and let $e_{m+1},\ldots,e_{n}$ be an orthonormal basis of $\mathfrak 
k$. Then
\begin{align}\label{eq:inva12ax1}
&\Tr\left[C^{\mathfrak p, \mathfrak g}\right]=-\sum_{\substack{1\le i\le m\\ 1\le 
j\le m+n}}^{}\left\vert  \left[e_{i},e_{j}\right]\right\vert^{2},\\
&\Tr\left[C^{\mathfrak k, \mathfrak p}\right]=
-\sum_{\substack{1\le 
i\le m\\m+1\le j\le m+n}}^{}\left\vert  
\left[e_{i},e_{j}\right]\right\vert^{2} =
-\sum_{1\le 
i,j\le m}^{}\left\vert  
\left[e_{i},e_{j}\right]\right\vert^{2}.
\nonumber 
\end{align}
By (\ref{eq:inva12ax1}), we deduce that
\begin{equation}\label{eq:inva13}
\Tr\left[C^{\mathfrak p, \mathfrak g}\right]=2\Tr\left[C^{\mathfrak 
k, \mathfrak p}\right].
\end{equation}

 Also 
\begin{equation}\label{eq:inva12}
\Tr\left[C^{\mathfrak k,\mathfrak g}\right]=\Tr\left[C^{\mathfrak k, 
\mathfrak p}\right]+\Tr\left[C^{\mathfrak k, \mathfrak k}\right].
\end{equation}
By (\ref{eq:inva11}), (\ref{eq:inva13}), and (\ref{eq:inva12}), we get (\ref{eq:inva14}). The proof of our proposition is completed. 
\end{proof}

Let $\mathfrak h \subset \mathfrak g$ be a $\theta$-stable Cartan subalgebra. 
Then we have the splitting $\mathfrak h= \mathfrak h_{\mathfrak p} 
\oplus \mathfrak h_{\mathfrak k}$. 
\footnote{More details will be given in Section \ref{sec:geofori} on Cartan subalgebras 
and root systems.} Let $R 
\subset \mathfrak h_{\C}^{*}$ be the corresponding root system. 
Let $R_{+}$ be a positive 
root system. If  $R_{-}=-R_{+}$,  $R$ is the disjoint union of 
$R_{+}$ and $R_{-}$. 
 Let 
 \index{rg@$\rho^{\mathfrak g}$}%
 $\rho^{\mathfrak g}\in \mathfrak h^{*}_{\C}$ be  the half sum of the positive roots. 
Here  $\rho^{\mathfrak g}\in \mathfrak h_{\mathfrak p}^{*} \oplus i 
\mathfrak h_{\mathfrak k}^{*}$. By Kostant's strange formula \cite{Kostant76}, we have the identity 
\begin{equation}\label{eq:inva10}
B^{*}\left(\rho^{\mathfrak g},\rho^{\mathfrak 
g}\right)=-\frac{1}{24}\Tr\left[C^{\mathfrak g, \mathfrak g}\right].
\end{equation}
By (\ref{eq:inva14}), (\ref{eq:inva10}), we get
\begin{equation}\label{eq:inva10a}
B^{*}\left(\rho^{\mathfrak g}, \rho^{\mathfrak 
g}\right)=-\frac{1}{8}\Tr\left[C^{\mathfrak k, \mathfrak 
p}\right]-\frac{1}{24}\Tr\left[C^{\mathfrak k, \mathfrak k}\right].
\end{equation}
Using the notation in \cite[eq. (2.6.11)]{Bismut08b} \footnote{The
definition of $\kappa^{\mathfrak g}$ is not needed. The formula is 
given  for later reference.}, by
(\ref{eq:inva10a}), we obtain
\begin{equation}\label{eq:fina-1}
B^{*}\left(\rho^{\mathfrak g},\rho^{\mathfrak 
g}\right)=-\frac{1}{4}B^{*}\left(\kappa^{\mathfrak g}, 
\kappa^{\mathfrak g}\right).
\end{equation}
\subsection{The elliptic operator $C^{\mathfrak g,X}$}%
\label{subsec:symsp}
Let $E$ be a finite-dimensional Hermitian vector space, and let 
$\rho^{E}:K\to U\left(E\right)$ denote a unitary representation of $K$. 
The Casimir $C^{\mathfrak k,E}$ is a self-adjoint nonpositive 
endomorphism of $E$. If $\rho^{E}$ is irreducible, then $C^{\mathfrak k,E}$ 
is a scalar.

Let $F$ be the vector bundle on $X$, 
\begin{equation}\label{eq:inva18}
F=G\times_{K}E.
\end{equation}
Then $F$ is a Hermitian vector bundle on $X$, which is equipped 
with a canonical connection $\n^{F}$. Also $C^{\mathfrak k,E}$ 
descend to a parallel section $C^{\mathfrak k,F}$ of 
$\End\left(F\right)$.  Moreover,  $G$ acts on $C^{ \infty 
}\left(X,F\right)$, so that if $g\in G, s\in C^{\infty 
}\left(X,F\right)$, if $g_{*}$ denotes the lift of the action of $g$ 
to $F$, 
\begin{equation}\label{eq:final2}
gs\left(x\right)=g_{*}s\left(g^{-1}x\right).
\end{equation}

The Casimir operator $C^{\mathfrak g}$  descends 
to a second order elliptic operator
\index{CgX@$C^{\mathfrak g,X}$}%
$C^{\mathfrak g,X}$ acting on $C^{ \infty 
}\left(X,F\right)$, which commutes with $G$. Let 
\index{DX@$\Delta^{X}$}%
$\Delta^{X}$ denote the classical Bochner Laplacian acting on 
$C^{\infty }\left(X,F\right)$. By \cite[eq. (2.13.2)]{Bismut08b}, 
the splitting (\ref{eq:inva14z2}) of $C^{\mathfrak g}$ descends to the 
splitting of $C^{\mathfrak g,X}$, 
\begin{equation}\label{eq:inva10b}
C^{\mathfrak g,X}=-\Delta^{X}+C^{\mathfrak k,F}.
\end{equation}
\subsection{Orbital integrals and the 
Casimir}%
\label{subsec:wake}
Let $\mu\in\mathcal{S}^{\mathrm{even}}\left(\R\right)$, and let $\widehat{\mu}
\in \mathcal{S}^{\mathrm{even}}\left(\R\right)$ be 
its Fourier transform, i.e., 
\begin{equation}
    \widehat{\mu}\left(y\right)=\int_{\R}^{}e^{-2i\pi 
    yx}\mu\left(x\right)dx.
    \label{eq:fourier1}
\end{equation}
We  assume that there exists $C>0$ such that for any $k\in \N$, 
there is $c_{k}>0$ such that
\begin{equation}
    \left\vert  \widehat{\mu}^{(k)}\left(y\right)\right\vert\le 
    c_{k}\exp\left(-Cy^{2}\right).
    \label{eq:fourier2}
\end{equation}
The above condition is verified if $\widehat{\mu}$ has compact 
support. For $t>0$, it is also verified by the Gaussian function 
$e^{-tx^{2}}$.

If $A\in \R$, the operator $\mu\left(\sqrt{C^{\mathfrak g,X}+A}\right)$ is  self-adjoint 
with a smooth kernel
$\mu\left(\sqrt{C^{\mathfrak 
g,X}+A}\right)\left(x,x'\right)$ with respect to the Riemannian 
volume $dx'$ on $X$. As explained in 
\cite[Section 6.2]{Bismut08b}, using finite propagation speed for the 
wave equation, condition (\ref{eq:fourier2}) implies 
that there exist $C>0,c>0$ such that if $x,x'\in X$, then
\begin{equation}\label{eq:ober1}
\left\vert  \mu\left(\sqrt{C^{\mathfrak g,X}+A}\right)\left(x,x'\right)\right\vert
\le Ce^{-cd^{2}\left(x,x'\right)}.
\end{equation}
If $\widehat{\mu}$ has 
compact support, then $\mu\left(\sqrt{C^{\mathfrak 
g,X}+A}\right)\left(x,x'\right)$ vanishes when $d\left(x,x'\right)$ 
is large enough.

As explained in \cite[Section 6.2]{Bismut08b}, the above condition 
guarantees that if $\gamma\in G$ is semisimple, the orbital integral 
$\Tr^{\left[\gamma\right]}\left[\mu\left(\sqrt{C^{\mathfrak 
g,X}+A}\right)\right]$ is well-defined. 
Let us give more details  
on our conventions. 

Let $\gamma\in G$ be taken as in (\ref{eq:dis4a}). Let 
$N_{X\left(\gamma\right)/X}$ be the orthogonal bundle to 
$TX\left(\gamma\right)$ in $TX$. By \cite[eq. (3.4.1)]{Bismut08b}, we 
have the identity
\begin{equation}\label{eq:fina3}
N_{X\left(\gamma\right)/X}=Z^{0}\left(\gamma\right)\times_{K^{0}\left(\gamma\right)} \mathfrak p^{\perp}\left(\gamma\right).
\end{equation}
Let $\mathcal{N}_{X\left(\gamma\right)/X}$ be the total space of 
$N_{X\left(\gamma\right)/X}$. By \cite[Theorem 3.4.1]{Bismut08b}, the 
normal geodesic coordinate system based at $X\left(\gamma\right)$ 
gives a smooth identification of 
$\mathcal{N}_{X\left(\gamma\right)/X}$ with $X$.  Let $dx,dy, df$ be the 
Riemannian volume forms on $X,X\left(\gamma\right), 
N_{X\left(\gamma\right)/X}$. Then $dydf$ is a volume form on 
$\mathcal{N}_{X\left(\gamma\right)/X}$. Let $r\left(f\right)$ denote 
the corresponding Jacobian, so that we have the identity of volume 
forms on $X$, 
\begin{equation}\label{eq:fina4}
dx=r\left(f\right)dydf.
\end{equation}
By \cite[eq. (3.4.36)]{Bismut08b}, there are constants $C>0,C'>0$ such that
\begin{equation}\label{eq:fina5}
r\left(f\right)\le Ce^{C'\left\vert  f\right\vert}.
\end{equation}

By \cite[Theorem 3.4.1]{Bismut08b}, there exists $C_{\gamma}>0$ such 
that for $f\in \mathfrak p^{\perp}\left(\gamma\right), \left\vert  
f\right\vert\ge 1$, 
\begin{equation}\label{eq:fina6}
d_{\gamma}\left(e^{f}x_{0}\right)\ge \left\vert  
a\right\vert+C_{\gamma}\left\vert  f\right\vert.
\end{equation}

As explained in \cite[eq. (4.2.6)]{Bismut08b}, by 
(\ref{eq:ober1}), (\ref{eq:fina6}),  there exist $C_{\gamma}>0,c_{\gamma}>0$ such that if $f\in \mathfrak 
p^{\perp}\left(\gamma\right)$, then
\begin{equation}\label{eq:fina7}
\left\vert  \mu\left(\sqrt{C^{\mathfrak 
g,X}+A}\right)\left(\gamma^{-1}e^{f}x_{0},e^{f}x_{0} 
\right)\right\vert\le C_{\gamma}\exp\left(-c_{\gamma}\left\vert  
f\right\vert^{2}\right).
\end{equation}

We denote by 
\index{gs@$\gamma_{*}$}%
$\gamma_{*}$ the action of $\gamma$ on $F$. More 
precisely, if $x\in X$, $\gamma_{*}$ maps $F_{x}$ into $F_{\gamma x}$. 

In \cite[Definition 4.2.2]{Bismut08b}, the orbital integral 
$$\Tr^{\left[\gamma\right]}
\left[\mu\left(\sqrt{C^{\mathfrak g,X}+A}\right)\right]$$ 
is defined 
by the formula
\begin{multline}\label{eq:fina6a1}
\Tr^{\left[\gamma\right]}
\left[\mu\left(\sqrt{C^{\mathfrak g,X}+A}\right)\right]\\
=
\int_{\mathfrak p^{\perp}\left(\gamma\right)}^{}\Tr\left[\gamma_{*}
\mu\left(\sqrt{C^{\mathfrak 
g,X}+A}\right)\left(\gamma^{-1}e^{f}x_{0}, 
e^{f}x_{0}\right) \right]r\left(f\right)df.
\end{multline}
Equations (\ref{eq:fina5}), (\ref{eq:fina7}) guarantee that the 
integral in (\ref{eq:fina6a1}) converges.

Let $dk$ be the Haar measure on $K$ such that $\Vol\left(K\right)=1$. 
Then $dg=dxdk$ is a Haar measure on $G$. Let $dy$ be the Riemannian 
volume form on $X\left(\gamma\right)$. Let $dk^{\prime 0}$ be the 
Haar measure on $K^{0}\left(\gamma\right)$ such that 
$\Vol\left(K^{0}\left(\gamma\right)\right)=1$. Then 
$dz^{0}=dydk^{\prime 0}$ is a Haar measure on 
$Z^{0}\left(\gamma\right)$. Let $dv^{0}$ be the volume on 
$Z^{0}\left(\gamma\right)\setminus G$ such that $dg=dz^{0}dv^{0}$. 

As explained in \cite[Section 4.2]{Bismut08b}, the smooth kernel 
$$\mu\left(\sqrt{C^{\mathfrak g,X}+A}\right)\left(x,x'\right)$$
lifts to a smooth 
function on $G$ with values in $\End\left(E\right)$, denoted 
$$\mu^{E}\left(\sqrt{C^{\mathfrak 
g,X}+A}\right)\left(g\right),$$ 
and by \cite[eq. (4.2.11)]{Bismut08b}, we have the identity 
\begin{multline}\label{eq:fina8}
\Tr^{\left[\gamma\right]}
\left[\mu\left(\sqrt{C^{\mathfrak g,X}+A}\right)\right]\\
=\int_{Z^{0}\left(\gamma\right)\setminus G}^{}\Tr^{E}
\left[\mu\left(\sqrt{C^{\mathfrak 
g,X}+A}\right)\left(\left(v^{0}\right)^{-1}\gamma 
v^{0}\right)\right]dv^{0}.
\end{multline}
This definition of orbital integrals coincides with the definition 
given by Selberg \cite[p. 66]{Selberg56}.
\subsection{The function $\mathcal{J}_{\gamma}$}%
\label{subsec:fuJg}
We use the assumptions in Subsection \ref{subsec:sesidis} and the 
corresponding notation. In particular $\gamma\in G$ is a semisimple 
element as in (\ref{eq:dis4a}). 

Then $\Ad\left(\gamma\right)$ preserves $\mathfrak z\left(a\right), 
\mathfrak z^{\perp}\left(a\right)$. Also $\Ad\left(k^{-1}\right)$ preserves 
$\mathfrak z_{a}^{\perp}\left(\gamma\right)$. If $\Yok\in 
\mathfrak k\left(\gamma\right)$,  $\mathrm{ad}\left(\Yok\right)$ 
preserves $ \mathfrak 
z_{a}^{\perp}\left(\gamma\right)$. The splitting 
(\ref{eq:inva19z2}) is   preserved by $\Ad\left(k^{-1}\right)$ and 
$\mathrm{ad}\left(\Yok\right)$.

If $x\in \R$, put
\begin{equation}\label{eq:aro1}
\widehat{A}\left(x\right)=\frac{x/2}{\sinh\left(x/2\right)}.
\end{equation}
If $\Yok\in \mathfrak k\left(\gamma\right)$, $\mathrm{ad}\left(\Yok\right)$ 
acts as an antisymmetric endomorphism of $\mathfrak 
p\left(\gamma\right), \mathfrak k\left(\gamma\right)$, so that its 
eigenvalues are either $0$, or they come by nonzero conjugate 
imaginary pairs. 
If $\Yok\in i \mathfrak k\left(\gamma\right)$, put\footnote{This fits 
with the classical notation in the theory of characteristic classes.}
\begin{align}\label{eq:aro2}
&\widehat{A}\left(\mathrm{ad}\left(\Yok\right)\vert
_{ \mathfrak 
p\left(\gamma\right)}\right)=\left[\det\left(\widehat{A}\left(\mathrm{ad}\left(\Yok\right)\vert_{\mathfrak p\left(\gamma\right)}\right) \right) \right]^{1/2},\\
&\widehat{A}\left(\mathrm{ad}\left(\Yok\right)\vert
_{ \mathfrak k\left(\gamma\right)}\right)=\left[\det\left(\widehat{A}\left(\mathrm{ad}\left(\Yok\right)\vert
_{ \mathfrak k\left(\gamma\right)}\right)\right)\right]^{1/2}. \nonumber 
\end{align}
The square root in (\ref{eq:aro2}) is  the positive square 
root of a positive real number.

We follow \cite[Theorem 5.5.1]{Bismut08b}, while slightly changing 
the notation.
\begin{definition}\label{DLg}
	If $\Yok\in i \mathfrak  k\left(\gamma\right)$, put
	\index{LgY@$\mathcal{L}_{\gamma}\left(\Yok\right)$}%
	\begin{equation}\label{eq:congo0}
\mathcal{L}_{\gamma}\left(\Yok\right)=\frac{\det\left(1-    
\Ad\left(k^{-1}e^{-\Yok}\right)\right)\vert_{\mathfrak 
	k^{\perp}_{a}\left(\gamma\right)}}{\det\left(1-
    \Ad\left(k^{-1}e^{-\Yok}\right)\right)\vert_{\mathfrak 
    p_{a}^{\perp}\left(\gamma\right)}}.
\end{equation}
Set
\index{MgY@$\mathcal{M}_{\gamma}\left(\Yok\right)$}%
\begin{equation}\label{eq:congo-1}
\mathcal{M}_{\gamma}\left(\Yok\right)
=\left[\frac{1}{\det\left(1-\Ad\left(k^{-1}\right)\right)\vert_{ 
\mathfrak z^{\perp}_{a}\left(\gamma\right)}
    }\mathcal{L}_{\gamma}\left(\Yok\right)\right]^{1/2}.
\end{equation}
	\end{definition}
	The fact that the square root in (\ref{eq:congo-1}) is unambiguously 
defined is established in \cite[Section 5.5]{Bismut08b}.  Let us 
explain the details. First we make $\Yok=0$. Then
\begin{multline}\label{eq:ret33s-1}
\frac{1}{\det\left(1-\Ad\left(k^{-1}\right)\right)\vert_{ 
\mathfrak z^{\perp}_{a}\left(\gamma\right)}
    }\frac{\det\left(1-\Ad\left(k^{-1}\right)\right)\vert_{\mathfrak 
	k^{\perp}_{a}\left(\gamma\right)}}{\det\left(1-    \Ad\left(k^{-1}\right)\right)\vert_{\mathfrak 
    p_{a}^{\perp}\left(\gamma\right)}}\\
	=\left[\frac{1}{\det\left(1-\Ad\left(k^{-1}\right)\right)\vert_{\mathfrak p_{a}^{\perp}\left(\gamma\right)}}
	\right]^{2}.
\end{multline}
The conventions in \cite{Bismut08b}  say that the square root of 
(\ref{eq:ret33s-1}) is  the obvious positive square root,  i.e., 
\begin{equation}\label{eq:ret33s-2}
\mathcal{M}_{\gamma}\left(0\right)
	=\frac{1}{\det\left(1-\Ad\left(k^{-1}\right)\right)\vert_{\mathfrak p_{a}^{\perp}\left(\gamma\right)}}.
\end{equation}
Using analyticity in the variable $\Yok\in i \mathfrak 
k\left(\gamma\right)$, the choice of the square root in 
(\ref{eq:congo-1}) determines a choice of the square root in 
(\ref{eq:ret33}). This point 
will be discussed at length in Section \ref{sec:rofu}. No  choice of 
a Cartan subalgebra or of a positive root system is 
needed at this stage.

\begin{definition}\label{def:Jg}
	\index{Jg@$\mathcal{J}_{\gamma}\left(\Yok\right)$}%
Let $\mathcal{J}_{\gamma}\left(\Yok\right)$ be the  smooth function of 
$\Yok\in i\mathfrak 
k\left(\gamma\right)$,
\begin{equation}\label{eq:ret33}
\mathcal{J}_{\gamma}\left(Y_{0}^{ \mathfrak 
k}\right)=\frac{1}{\left\vert  \det\left(1-\Ad\left(\gamma\right)\right)\vert_{ 
\mathfrak z^{\perp}\left(a\right)}\right\vert^{1/2}}
\frac{\widehat{A}\left(\mathrm{ad}\left(\Yok\right)\vert
_{ \mathfrak p\left(\gamma\right)}\right)}{\widehat{A}\left(\mathrm{ad}\left(\Yok\right)
\vert_{\mathfrak k\left(\gamma\right)}\right)}\mathcal{M}_{\gamma}\left(\Yok\right).
\end{equation}
\end{definition}

With the conventions in \cite[Chapter 5]{Bismut08b}, where instead a 
function $J_{\gamma}\left(\Yok\right)$ is defined on $\mathfrak 
k\left(\gamma\right)$, we have
\begin{equation}\label{eq:inva20a1}
J_{\gamma}\left(\Yok\right)=\mathcal{J}_{\gamma}\left(i\Yok\right).
\end{equation}

By \cite[eq. (5.5.11)]{Bismut08b} or by (\ref{eq:ret33}), if $\Yok\in 
i\mathfrak k$, then
\begin{equation}\label{eq:norm0}
\mathcal{J}_{1}\left(\Yok\right)=\frac{\widehat{A}\left(\mathrm{ad}\left(\Yok\right)\vert
_{ \mathfrak p}\right)}{\widehat{A}\left(\mathrm{ad}\left(\Yok\right)
\vert_{\mathfrak k}\right)}.
\end{equation}
\subsection{Some properties of the function $\mathcal{J}_{\gamma}$}%
\label{subsec:propjg}
Let $\rho^{E^{*}}: K\to U\left(E^{*}\right)$ denote the representation of $K$ which is dual to 
the representation $\rho^{E}$.

\begin{proposition}\label{prop:pcobis}
	If $\Yok\in i \mathfrak k\left(\gamma\right)$, then
	\begin{align}\label{eq:ida1}
		&\mathcal{J}_{\gamma^{-1}}\left(\Yok\right)=\mathcal{J}_{\gamma}\left(-\Yok\right),
		&\Tr^{E^{*}}\left[\rho^{E^{*}}\left(ke^{-\Yok}\right) \right]=
		\Tr^{E}\left[\rho^{E}\left(k^{-1}e^{\Yok}\right)\right],\\
&\overline{\mathcal{J}_{\gamma}\left(\Yok\right)}=\mathcal{J}_{\gamma}\left(-\Yok\right),
&\overline{\Tr^{E}\left[\rho^{E} \left( k^{-1}e^{-\Yok} \right) 
\right]}=\Tr^{E^{*}}\left[\rho^{E^{*}}\left(k^{-1}e^{\Yok}\right)\right]. \nonumber 
\end{align}
\end{proposition}
\begin{proof}
	If $f\in \End\left(\mathfrak g\right)$, let $\widetilde f\in 
	\End\left(\mathfrak g\right)$ denote 
	the adjoint of $f$ with respect to $B$. We have the identity
	\begin{equation}\label{eq:ida0}
\Ad\left(\gamma^{-1}\right)=\widetilde 
{\Ad\left(\gamma\right)}.
\end{equation}
By (\ref{eq:ida0}), we deduce that
\begin{equation}\label{eq:ida0a}
\det\left(1-\Ad\left(\gamma^{-1}\right)\right) \vert_{\mathfrak 
z^{\perp}\left(a\right)}=\det\left(1-\Ad\left(\gamma\right)\right)\vert_{\mathfrak z^{\perp}\left(a\right)}.
\end{equation}
A similar argument shows that
\begin{align}\label{eq:ida0b}
&\det\left(1-\Ad\left(ke^{-\Yok}\right)\right)\vert_{\mathfrak 
p_{a}^{\perp}\left(\gamma\right)}=\det\left(1-\Ad\left(k^{-1}e^{\Yok}\right)\right)\vert_{\mathfrak 
p_{a}^{\perp}\left(\gamma\right)},\\
&\det\left(1-\Ad\left(ke^{-\Yok}\right)\right)\vert_{\mathfrak 
k_{a}^{\perp}\left(\gamma\right)}=\det\left(1-\Ad\left(k^{-1}e^{\Yok}\right)\right)\vert_{\mathfrak 
k_{a}^{\perp}\left(\gamma\right)}. \nonumber 
\end{align}

By (\ref{eq:congo0}), (\ref{eq:congo-1}), (\ref{eq:ret33}), 
(\ref{eq:ida0a}), and 
(\ref{eq:ida0b}), we get the first identity in (\ref{eq:ida1}). The 
second identity in (\ref{eq:ida1}) is trivial.

	If $\Yok\in i\mathfrak 
	k\left(\gamma\right)$, 
	\begin{align}\label{eq:ida2}
&\overline{\det\left(1-\Ad\left(k^{-1}e^{-\Yok}\right)\right)\vert_{\mathfrak 
p_{a}^{\perp}\left(\gamma\right)}}=\det\left(1-\Ad\left(k^{-1}e^{\Yok}\right)\right)\vert_{\mathfrak 
p_{a}^{\perp}\left(\gamma\right)},\\
&\overline{\det\left(1-\Ad\left(k^{-1}e^{-\Yok}\right)\right)\vert_{\mathfrak 
k_{a}^{\perp}\left(\gamma\right)}}=\det\left(1-\Ad\left(k^{-1}e^{\Yok}\right)\right)\vert_{\mathfrak 
k_{a}^{\perp}\left(\gamma\right)}. \nonumber 
\end{align}
By (\ref{eq:ida2}), we get the third identity in (\ref{eq:ida1}).  
Since $\Yok\in i\mathfrak k\left(\gamma\right)$, the fourth identity is 
trivial.
The proof of our proposition is completed. 
\end{proof}

\subsection{A geometric formula for the orbital integrals associated 
with the Casimir}%
\label{subsec:geoorb}
Note that $i \mathfrak  
k\left(\gamma\right)$ is naturally an Euclidean 
vector space.  If $\Yok\in i\mathfrak k\left(\gamma\right)$, we denote by 
$\left\vert  \Yok\right\vert$ its Euclidean norm. More precisely, if 
$\Yok\in i\mathfrak k\left(\gamma\right)$, then
\begin{equation}\label{eq:norm1}
\left\vert  \Yok\right\vert^{2}=B\left(\Yok,\Yok\right).
\end{equation}

By
\cite[eq. (6.1.1)]{Bismut08b}, there exist 
$c>0,C>0$ such that if $\Yok \in i \mathfrak 
k\left(\gamma\right)$, 
\begin{equation}\label{eq:comm-1}
\left\vert  \mathcal{J}_{\gamma}\left(\Yok\right)\right\vert\le c
\exp\left(C\left\vert  \Yok\right\vert\right).
\end{equation}
 In the sequel, $\int_{i \mathfrak 
k\left(\gamma\right)}^{}$ denotes integration on the real vector space 
$i \mathfrak k\left(\gamma\right)$.

Let $d\Yok$ be the Euclidean volume on $i \mathfrak 
k\left(\gamma\right)$. Set $p=\dim \mathfrak p\left(\gamma\right), q=\dim 
\mathfrak k\left(\gamma\right)$. 
 Now we  state the result obtained in \cite[Theorem 
6.1.1]{Bismut08b}. Our reformulation takes  equation 
(\ref{eq:fina-1}) into account.
\begin{theorem}\label{thm:Tformlim2}
 For  $t>0$, the following identity holds:
\begin{multline}\label{eq:inva21a}
\Tr^{\left[\gamma\right]}\left[\exp\left(-tC^{\mathfrak 
g,X}/2\right) \right]=\exp\left(-tB^{*}\left(\rho^{\mathfrak 
g},\rho^{\mathfrak g}\right)/2\right)
\frac{\exp\left(-\left\vert  
a\right\vert^{2}/2t\right)}{\left(2\pi t\right)^{p/2}}\\
\int_{i\mathfrak k\left(\gamma\right)}^{}
\mathcal{J}_{\gamma}\left(Y_{0}^{ \mathfrak 
k}\right)\Tr^{E}\left[\rho^{E}\left(k^{-1}e^{-\Yok}\right)\right]
\exp\left(-\left\vert  Y_{0}^{ \mathfrak 
k}\right\vert^{2}/2t\right)\frac{dY_{0}^{ \mathfrak k}}{
\left(2\pi t\right)^{q/2}}.
\end{multline}
\end{theorem}

Let $B\vert_{\mathfrak z\left(\gamma\right)}$ be the restriction of 
$B$ to $\mathfrak z\left(\gamma\right)$, and let 
$B^{*}\vert_{\mathfrak z\left(\gamma\right)} $ be the corresponding 
quadratic form on $\mathfrak z^{*}\left(\gamma\right)$.  Let $ 
\Delta^{\mathfrak z\left(\gamma\right)}$ denote the associated 
generalized Laplacian on $\mathfrak z\left(\gamma\right)$. We can extend $\Delta^{\mathfrak z\left(\gamma\right)}$ to an 
operator acting via constant holomorphic vector fields on $\mathfrak 
z\left(\gamma\right)_{\C}$. 

Put
\index{zig@$\mathfrak z_{i}\left(\gamma\right)$}%
\begin{equation}\label{eq:fina-2}
\mathfrak z_{i}\left(\gamma\right)=\mathfrak p\left(\gamma\right) 
\oplus i \mathfrak k\left(\gamma\right).
\end{equation}
Then $B\vert_{\mathfrak z_{i}\left(\gamma\right)}$ is a scalar 
product on  $\mathfrak z_{i}\left(\gamma\right)$. 
The generalized Laplacian 
$\Delta^{\mathfrak z\left(\gamma\right)}$ restricts on $\mathfrak z_{i}\left(\gamma\right)$ to the 
standard Euclidean Laplacian of $\mathfrak z_{i}\left(\gamma\right)$. 

 We take $\mu\in \mathcal{S}^{\mathrm{even}}\left(\R\right)$ as in 
Subsection \ref{subsec:wake}. 
If $f\in \mathfrak z_{i}\left(\gamma\right)$,  let
$$\mu \left( \sqrt{-\Delta^{ \mathfrak 
z\left(\gamma\right)}+B^{*}\left(\rho^{\mathfrak g},\rho^{\mathfrak 
g}\right)+A}\right)\left(f\right)$$
be the smooth convolution kernel  for
$\mu \left( \sqrt{-\Delta^{ \mathfrak 
z\left(\gamma\right)}+B^{*}\left(\rho^{\mathfrak g},\rho^{\mathfrak 
g}\right)+A} \right) $ on $\mathfrak z_{i}\left(\gamma\right)$ with 
respect to the volume form associated with the  scalar product of $\mathfrak 
z_{i}\left(\gamma\right)$.  Using  (\ref{eq:fourier2}) and finite 
propagation speed for the wave equation, there exist $C>0,c>0$ such 
that if $f\in 
\mathfrak z_{i}\left(\gamma\right)$, then
\begin{equation}\label{eq:ober2}
\left\vert  \mu \left( \sqrt{-\Delta^{ \mathfrak 
z\left(\gamma\right)}+B^{*}\left(\rho^{\mathfrak g},\rho^{\mathfrak 
g}\right)+A} \right) \left(f\right)\right\vert\le Ce^{-c\left\vert  
f\right\vert^{2}}.
\end{equation}

Let $\delta_{a}$ be the Dirac mass at $a\in \mathfrak 
p\left(\gamma\right)$. Then
$$\mathcal{J}_{\gamma}\left(\Yok\right)\Tr^{E}\left[\rho^{E}\left(k^{-1}e^{-\Yok}
\right)\right]\delta_{a}$$
is a distribution on  $\mathfrak z_{i}\left(\gamma\right)$, to which 
the smooth convolution kernel 
$$\mu \left( \sqrt{-\Delta^{ \mathfrak 
z\left(\gamma\right)}+B^{*}\left(\rho^{\mathfrak g},\rho^{\mathfrak 
g}\right)+A} \right) $$
can be applied. By definition,
\begin{multline}\label{eq:fina-2a}
\mu\left(\sqrt{-\Delta^{ \mathfrak z\left(\gamma\right)}
    +B^{*}\left(\rho^{\mathfrak g},\rho^{\mathfrak g}\right)+A}
    \right)\\
	\left[\mathcal{J}_{\gamma}\left(\Yok\right)\Tr^{E}\left[\rho^{E}\left(k^{-1}e^{-\Yok}
\right)\right]\delta_{a}\right]\left(0\right)\\
=\int_{i \mathfrak k\left(\gamma\right)}^{}\mu\left(\sqrt{-\Delta^{ \mathfrak z\left(\gamma\right)}
   +B^{*}\left(\rho^{\mathfrak g},\rho^{\mathfrak g}\right)+A}
    \right)
\left(-\Yok,-a\right)\\
\mathcal{J}_{\gamma}\left(\Yok\right)\Tr^{E}\left[\rho^{E}\left(k^{-1}e^{-\Yok}
\right)\right]d\Yok.
\end{multline}
In the right hand-side of (\ref{eq:fina-2a}), $\left(-\Yok,-a\right)$ 
can as well be replaced by $\left(\Yok,a\right)$.

In \cite[Theorem 6.2.2]{Bismut08b}, the following extension  of 
Theorem \ref{thm:Tformlim2} was established.
\begin{theorem}\label{thm:Ttrfin}
The following identity holds:
\begin{multline}
    \Tr^{\left[\gamma\right]}\left[\mu\left(\sqrt{C^{\mathfrak g,X}+A}\right)\right] 
    =
  \mu\left(\sqrt{-\Delta^{ \mathfrak z\left(\gamma\right)}
   +B^{*}\left(\rho^{\mathfrak g},\rho^{\mathfrak g}\right)+A}
    \right)\\
	\left[\mathcal{J}_{\gamma}\left(\Yok\right)\Tr^{E}\left[\rho^{E}\left(k^{-1}e^{-\Yok}
\right)\right]\delta_{a}\right]\left(0\right).
    \label{eq:trib2}
\end{multline}
\end{theorem}

\section{Cartan subalgebras, Cartan subgroups,  and root systems}%
\label{sec:geofori}
 The purpose of this Section is to recall basic facts on   Cartan 
 subalgebras,  on Cartan subgroups, and on root systems.
 
This section is organized as follows. In Subsection 
\ref{subsec:lial}, we state some elementary facts of linear algebra.

In Subsection \ref{subsec:caal}, we recall the definition of Cartan 
subalgebras. 

In Subsection \ref{subsec:reglie}, we introduce the corresponding 
root system, and the associated algebraic Weyl group.

In Subsection \ref{subsec:reim}, we define the real and the imaginary 
roots.

In Subsection \ref{subsec:posro}, we construct  a positive root 
system.

In Subsection \ref{subsec:rofu}, when the Cartan subalgebra is 
fundamental, we compare the root system of $\mathfrak k$ with the 
root system of $\mathfrak g$. 

In Subsection \ref{subsec:casg}, we introduce the  Cartan 
subgroups, and of the corresponding regular elements.

In Subsection \ref{subsec:semisi}, we relate
semisimple elements in $G$ to Cartan subgroups.

In Subsection \ref{subsec:rosisi}, we describe the characters of 
Cartan subgroups associated with a root system.

In Subsection \ref{subsec:reimsesi}, we give some properties 
of the real and imaginary roots with respect to semisimple elements 
in $G$.

Finally, in Subsection \ref{subsec:care}, we give a well-known formula that 
relates the action of invariant differential operators on the Lie 
algebra $\mathfrak g$ and on a Cartan subalgebra $\mathfrak h$.

We make the same assumptions as in Section \ref{sec:geofo}, and we 
use the corresponding notation.
\subsection{Linear algebra}%
\label{subsec:lial}
Let $V$ be a finite-dimensional real vector space. The symmetric 
algebras 
\index{SV@$S\ac \left(V\right)$}%
$S\ac \left(V\right), S\ac \left(V^{*}\right)$ are the 
algebras of polynomials on $V^{*},V$.  If $v\in V$, $v$ acts as a 
derivation of $S\ac\left(V^{*}\right)$. More generally, 
$S\ac\left(V\right)$ acts on $S\ac\left(V^{*}\right)$, and this 
action identifies $S\ac\left(V\right)$ with the algebra
\index{DV@$D\ac\left(V\right)$}%
$D\ac\left(V\right)$ of  real partial differential operators on $V$ 
with constant coefficients.   In particular, if $B^{*}\in 
S^{2}\left(V\right)$ is a bilinear symmetric  form on $V^{*}$, the associated 
element in $D\ac\left(V\right)$ is the corresponding  
Laplacian $\Delta^{V}$. If $B^{*}$ is positive,  $\Delta^{V}$ is just a classical Laplacian. 
If  $B^{*}$ is  negative, then $\Delta^{V}$ is the 
negative of  a 
classical Laplacian on $V$. 

Let  $V_{\C}=V \otimes _{\R}\C$ be the 
complexification of $V$, a complex vector 
space. Its complex dual is given by $V^{*}_{\C}=V^{*} \otimes _{\R}\C$. 
The algebras $S\ac\left(V_{\C}\right),S\ac\left(V^{*}_{\C}\right)$ are the 
algebras of complex polynomials on $V^{*},V$. Note that
\begin{align}\label{eq:coc1}
&S\ac\left(V_{\C}\right)=S\ac\left(V\right) \otimes _{\R}\C,
&S\ac\left(V^{*}_{\C}\right)=S\ac\left(V^{*}\right) \otimes _{\R}\C.
\end{align}

Put
\begin{align}\label{eq:coc2}
&D\ac\left(V_{\C}\right)=D\left(V\right) \otimes _{\R}\C,
&D\ac\left(V^{*}_{\C}\right)=D\ac\left(V^{*}\right) \otimes _{\R}\C.
\end{align}
Then $D\ac\left(V_{\C}\right),D\ac\left(V^{*}_{\C}\right)$ are the 
complexifications of 
$D\ac\left(V\right),D\ac\left(V^{*}\right)$, and also the spaces 
of complex holomorphic differential operators with constant 
coefficients on $V_{\C},V^{*}_{\C}$. 

In particular, if $B^{*}\in S^{2}\left(V\right)$, $\Delta^{V}$ is now 
viewed as a holomorphic operator on $V_{\C}$, that coincides with the 
corresponding Laplacian $\Delta^{V}$ on $V$, and with $-\Delta^{V}$ 
on $iV \simeq V$. \footnote{On $\C \simeq \R^{2}$, 
when acting on holomorphic functions, the differential operators 
$\frac{\pa}{\pa z}, \frac{\pa}{\pa x}, -i\frac{\pa}{\pa y}$ coincide. 
In this sense, the differential operator $\frac{\pa}{\pa x}$ on $\R$ 
extends to the differential operator $\frac{\pa}{\pa z}$ on $\C$, and 
restricts to the operator  $-i\frac{\pa}{\pa  y}$ on the 
imaginary line $i\R$. The operator $\frac{\pa^{2}}{\pa x^{2}}$ on 
$\R$ restricts to the operator $-\frac{\pa^{2}}{\pa y^{2}}$ on $i\R$.}

Also  $S\ac\left(V^{*}\right) 
\subset C^{ \infty }\left(V,\R\right)$, and the action of  
$D\ac\left(V\right)$ extends to $C^{ \infty }\left(V,\R\right)$. 

Let 
\index{SV@$S\left[\left[V^{*}\right]\right]$}%
$S\left[\left[V^{*}\right]\right]$ be the algebra of 
formal power series $s=\sum_{i=0}^{+ \infty }s^{i}, s^{i}\in 
S^{i}\left(V^{*}\right)$.  Then $S\left[\left[V^{*}\right]\right]$ 
can be identified with the algebra $D\left[\left[V^{*}\right]\right]$ 
of differential operators of infinite order with constant 
coefficients on $V^{*}$. In particular, 
$S\left[\left[V^{*}\right]\right]$ acts on $S\ac\left(V\right)$.
\subsection{The Cartan subalgebras of $\mathfrak g$}%
\label{subsec:caal}
By  \cite[Section 0.2]{Wallach88}, a Lie subalgebra
\index{h@$\mathfrak h$}%
$\mathfrak h \subset \mathfrak g$ is 
said to be a Cartan subalgebra if $\mathfrak h$ is maximal among the 
abelian subalgebras of $\mathfrak g$ whose elements act as semisimple 
endomorphisms of $\mathfrak g$. Cartan subalgebras are known 
to exist and have the same dimension $r$, which is called the complex 
rank of $G$.  By \cite[Proposition 6.64]{Knapp02}, there is a finite family 
of nonconjugate  Cartan subalgebras in $\mathfrak g$. 
 By \cite[Lemma 2.3.3]{Wallach88}, in
every conjugacy class of Cartan subalgebras, there is a unique $\theta$-stable Cartan 
algebra, up to conjugation by $K$.  Therefore there is a finite family of nonconjugate 
$\theta$-stable Cartan subalgebras, up to conjugation by $K$.  By 
\cite[Theorem 2.15]{Knapp02},  the Cartan subalgebras of 
$\mathfrak g_{\C}$ are unique up to automorphisms induced by the 
adjoint group $\Ad\left(\mathfrak g_{\C}\right)$.

Let $\mathfrak h \subset \mathfrak g$ be a $\theta$-stable Cartan 
subalgebra. To the Cartan splitting of $\mathfrak g$ in 
(\ref{eq:cent1}) corresponds the splitting
\begin{equation}\label{eq:inva3}
\mathfrak h= \mathfrak h_{\mathfrak p} \oplus \mathfrak h_{\mathfrak 
k}.
\end{equation}
In particular the restriction $B\vert_{\mathfrak h}$ of $B$ to 
$\mathfrak h$ is nondegenerate. This is also the case if  $\mathfrak 
h$ is any Cartan 
subalgebra.

Up to conjugation by $K$, there is a unique $\theta$-stable Cartan 
subalgebra $\mathfrak h$, which is called fundamental, such that $\mathfrak h_{\mathfrak k}$ is a 
Cartan subalgebra of $\mathfrak k$.  Let us give more details on its 
construction \cite[pp. 129, 131]{Knapp86}.
Let 
\index{t@$\mathfrak t$}%
$\mathfrak t \subset \mathfrak k$ be a Cartan subalgebra of 
$\mathfrak k$. Let $\mathfrak z\left(\mathfrak t\right) \subset 
\mathfrak g$ be the centralizer of $\mathfrak t$, i.e., 
\begin{equation}\label{eq:gena1}
\mathfrak z\left(\mathfrak t\right)=\left\{f\in \mathfrak g, 
\left[\mathfrak t,f\right]=0\right\}.
\end{equation}
Then $\mathfrak h= \mathfrak z\left(\mathfrak t\right)$ is a 
$\theta$-stable
fundamental Cartan subalgebra of $\mathfrak g$, and $\mathfrak 
h_{\mathfrak k}= \mathfrak t$.

An 
element $f\in \mathfrak g$ is said to be semisimple if 
$\mathrm{ad}\left(f\right)\in \End\left(\mathfrak g\right)$ is semisimple. If 
$\mathfrak h$ is a Cartan subalgebra, elements of $\mathfrak h$ are 
semisimple. Any semisimple element of $\mathfrak g$ lies in a Cartan 
subalgebra. 

If $\mathfrak h \subset \mathfrak g$ is a Cartan subalgebra, let 
\index{h@$\mathfrak h^{\perp}$}%
$\mathfrak h^{\perp}$ be the orthogonal to $\mathfrak h$ in 
$\mathfrak g$. We have the $B$-orthogonal splitting, 
\begin{equation}\label{eq:cor0}
\mathfrak g= \mathfrak h \oplus \mathfrak h^{\perp},
\end{equation}
and $B$ is also nondegenerate on $\mathfrak h^{\perp}$. If $\mathfrak 
h$ is $\theta$-stable, then $\mathfrak h^{\perp}$ is also 
$\theta$-stable, and so it splits as
\index{hp@$\mathfrak h^{\perp}_{\mathfrak p}$}%
\index{hk@$\mathfrak h^{\perp}_{\mathfrak k}$}%
\begin{equation}\label{eq:spli1}
\mathfrak h^{\perp}=\mathfrak h^{\perp}_{\mathfrak p} \oplus 
\mathfrak h^{\perp}_{\mathfrak k}.
\end{equation}

Let
\index{u@$\mathfrak u$}%
$\mathfrak u= i \mathfrak p \oplus \mathfrak k$ be the 
compact form of  $\mathfrak g$. Then 
\index{hu@$\mathfrak h_{\mathfrak u}$}%
$\mathfrak h_{\mathfrak u}=i \mathfrak 
h_{\mathfrak p}\oplus \mathfrak h_{\mathfrak k}$ is a  Cartan 
subalgebra of $\mathfrak u$. If $\mathfrak h$ is $\theta$-stable, 
then $\mathfrak h_{\mathfrak u}$ is also $\theta$-stable.

An element $f\in \mathfrak g$ is said to be regular if $\mathfrak 
z\left(f\right)$ is a Cartan subalgebra. Regular elements in 
$\mathfrak g$ are semisimple. 

If $\mathfrak h$ is a Cartan subalgebra, if $f\in \mathfrak h$, 
$\mathrm{ad}\left(f\right)$ acts as an endomorphism of $\mathfrak g/\mathfrak 
h$. Then $f\in \mathfrak h$ is regular if and only if $\det 
\mathrm{ad}\left(f\right)\vert_{\mathfrak g/\mathfrak h}\neq 0$.

\subsection{A root system and the Weyl group}%
\label{subsec:reglie}
Let $\mathfrak h$ be a $\theta$-stable Cartan subalgebra.

Let 
\index{R@$R$}%
$R \subset  \mathfrak h^{*}_{\C}$ be the root system 
associated with $\mathfrak h, \mathfrak g$ \cite[Section II.4]{Knapp02}. If $\alpha\in R$, then $-\alpha\in R,\overline{\alpha}\in R$. If 
$\alpha\in R$, let 
\index{ga@$\mathfrak g_{\alpha}$}%
$\mathfrak g_{\alpha} \subset  \mathfrak g_{\C}$ 
be the weight space associated with $\alpha$, which is of dimension   
$1$.
Then we have the splitting
\begin{equation}\label{eq:coci1}
\mathfrak g_{\C}= \mathfrak h_{\C} \bigoplus \oplus_{\alpha\in 
R} \mathfrak g_{\alpha}.
\end{equation}
If $\alpha\in R$, then
\begin{equation}\label{eq:coci2}
\mathfrak g_{\overline{\alpha}}=\overline{\mathfrak g}_{\alpha}.
\end{equation}
If $f\in \mathfrak h$,  $\mathrm{ad}\left(f\right)\in \End\left(\mathfrak 
g\right)$ is antisymmetric with respect to  $B$, so that
the $\mathfrak g_{\alpha}\vert_{\alpha\in R}$ are $B$-orthogonal to 
$\mathfrak h_{\C }$. 
If $\alpha,\beta\in R$, then $\mathfrak g_{\alpha}, \mathfrak 
g_{\beta}$ are $B$-orthogonal except when $\beta=-\alpha$, and the 
pairing between $\mathfrak g_{\alpha},  \mathfrak g_{-\alpha}$ is 
nondegenerate, so that if $\alpha\in R$, the form $B$ induces the 
identification
\begin{equation}\label{eq:coc3}
\mathfrak g_{-\alpha} \simeq \mathfrak g_{\alpha}^{*}.
\end{equation}
Also
\begin{equation}\label{eq:comm-1a}
\mathfrak h^{\perp}_{\C}=\oplus_{\alpha\in R}\mathfrak g_{\alpha}.
\end{equation}

If
$\alpha\in R$,  $\alpha$ takes real values on $\mathfrak 
h_{\mathfrak p}$, and imaginary values on $\mathfrak h_{\mathfrak 
k}$, i.e., $\alpha\in \mathfrak h^{*}_{\mathfrak p} \oplus i 
\mathfrak h^{*}_{\mathfrak k}$. 
Also $\theta$ preserves the splitting (\ref{eq:cor0}) of $\mathfrak 
g$, and it maps $R$ into itself. More precisely, if $\alpha\in R$, 
\begin{align}\label{eq:cub1}
&\theta \alpha=-\overline{\alpha}, &\mathfrak g_{\theta\alpha}=\theta 
\mathfrak g_{\alpha}.
\end{align}

Let
 \index{Whg@$W\left(\mathfrak h_{\C}:\mathfrak g_{\C}\right)$}%
 $W\left(\mathfrak h_{\C}:\mathfrak g_{\C}\right) \subset  
 \Aut\left(\mathfrak h_{\C}\right)$ be the algebraic 
 Weyl group \cite[p. 131]{Knapp86}.  Then $R \subset i \mathfrak h^{*}_{\mathfrak u}$, and 
 $W\left(\mathfrak h_{\C}:\mathfrak g_{\C}\right) \subset 
 \Aut\left(\mathfrak h_{\mathfrak u}\right)$, i.e., $W\left(\mathfrak 
 h_{\C}:\mathfrak g_{\C}\right)$ preserves the real vector space 
 $\mathfrak h_{\mathfrak u}$.   Also $\theta$ acts as 
an automorphism of the Lie algebras $\mathfrak g,\mathfrak u$, and 
$W\left(\mathfrak h_{\C}: \mathfrak g_{\C}\right)$ is preserved by 
conjugation by $\theta$. 

In general, $W\left(\mathfrak h_{\C}: \mathfrak g_{\C}\right)$ does 
not preserve the real vector space $\mathfrak h$. Recall that if
$h\in \mathfrak h_{\C}$, we can define its complex conjugate 
$\overline{h}\in \mathfrak h_{\C}$. If $u\in 
\End\left(\mathfrak h_{\C}\right)$, its complex conjugate 
$\overline{u}\in \End\left(\mathfrak h_{\C}\right)$ is such that if $h\in \mathfrak 
h_{\C}$, then
\begin{equation}\label{eq:cla1}
\overline{u}\left(h\right)=\overline{u\left(\overline{h}\right)}.
\end{equation}

\begin{proposition}\label{prop:pstab}
	If $w\in W\left(\mathfrak h_{\C}: \mathfrak g_{\C}\right)$, then
	\begin{equation}\label{eq:sym0}
\overline{w}=\theta w\theta^{-1}.
\end{equation}
In particular, the  group $W\left(\mathfrak h_{\C}:\mathfrak g_{\C}\right)$ is 
	preserved by complex conjugation.
\end{proposition}
\begin{proof}
	The group $W\left(\mathfrak h_{\C}:\mathfrak g_{\C}\right)$ is 
	generated by the symmetries $s_{\alpha}, \alpha\in R$ with respect to the vanishing locus 
	of the $\alpha\in R$. By (\ref{eq:cub1}), we deduce that if 
	$\alpha\in R$, 
	\begin{equation}\label{eq:sym1}
\overline{s_{\alpha}}= \theta s_{\alpha}\theta^{-1}, 
\end{equation}
from which get (\ref{eq:sym0}). Since $W\left(\mathfrak h_{\C}: 
\mathfrak g_{\C}\right)$ is stable by conjugation by $\theta$, the 
group $W\left(\mathfrak h_{\C}:\mathfrak g_{\C}\right)$ is preserved 
by complex conjugation. 

Another proof is as follows. Observe that there is a canonical 
identification of complex vector spaces  $\varphi: \mathfrak h_{\C} \simeq \mathfrak h_{\mathfrak 
u,\C}$, but the complex conjugations on $\mathfrak h_{\C}$ and on $\mathfrak 
h_{\mathfrak u,\C}$ are not the same. More precisely, if $h\in 
\mathfrak h_{\C}$, 
\begin{equation}\label{eq:chif2bis1}
\overline{\varphi h}=\varphi\theta\overline{h}=\theta\varphi\overline{h}.
\end{equation}

If $w\in W\left(\mathfrak h_{\C}:\mathfrak g_{\C}\right)$, then 
\begin{equation}\label{eq:chif2}
w\vert_{\mathfrak h_{\C}}=\varphi^{-1}w\vert_{\mathfrak h_{\mathfrak 
u,\C}}\varphi.
\end{equation}
Recall that $w$ is a real automorphism of the real vector space   
$\mathfrak h_{\mathfrak u}$. By (\ref{eq:chif2bis1}), (\ref{eq:chif2}), we get
\begin{equation}\label{eq:chif3}
\overline{w}\vert_{\mathfrak 
h_{\C}}=\overline{\varphi}^{-1}w\vert_{\mathfrak h_{\mathfrak u,\C}}\overline{\varphi}.
\end{equation}
By (\ref{eq:chif2bis1})--(\ref{eq:chif3}), we get (\ref{eq:sym0}). The proof of our proposition is completed. 
\end{proof}
\subsection{Real roots and imaginary roots}%
\label{subsec:reim}
 Let 
 \index{Rre@$R^{\mathrm{re}}$}%
 $R^{\mathrm{re}} \subset R$ be the  roots $\alpha\in R$ such 
that $\theta\alpha=-\alpha$, let 
\index{Rim@$R^{\mathrm{im}}$}%
$R^{\mathrm{im}}$ be the roots $\alpha\in R$  such that $\theta\alpha=\alpha$. These are 
respectively the real roots and the imaginary roots. Imaginary roots 
vanish on $\mathfrak h_{\mathfrak p}$, real roots vanish on 
$\mathfrak h_{\mathfrak k}$. By  \cite[p. 349]{Knapp86}, the set of complex roots
\index{Rc@$R^{\mathrm{c}}$}%
$R^{\mathrm{c}} \subset R$ is defined to be
\begin{equation}\label{eq:clo-2}
R^{\mathrm{c}}=R\setminus\left(R^{\mathrm{re}}\cup 
R^{\mathrm{im}}\right).
\end{equation}

\begin{proposition}\label{prop:pco}
	If $\alpha\in R$, the  map  $f\in \mathfrak g_{\C}\to 
	\overline{f}\in \mathfrak g_{\C}$ induces an 
	antilinear isomorphism from $\mathfrak g_{\alpha}$ into $\mathfrak 
	g_{-\theta\alpha}$, and the map $f\in \mathfrak g_{\C} \to 
	\theta\overline{f}\in \mathfrak g_{\C}$ induces an antilinear 
	isomorphism from $\mathfrak g_{\alpha}$ into $\mathfrak 
	g_{-\alpha}$. If $\alpha\in R^{\mathrm{re}}$, $\mathfrak 
	g_{\alpha}$ is  the complexification of a real vector subspace of $\mathfrak 
	h^{\perp}$.
\end{proposition}
\begin{proof}
	If $b\in \mathfrak h,f\in \mathfrak g_{\alpha}$, then
	\begin{equation}\label{eq:jou0}
\left[b,f\right]=\left\langle  \alpha,b\right\rangle f.
\end{equation}
By taking the conjugate of (\ref{eq:jou0}), we obtain
\begin{equation}\label{eq:jou-1}
\left[b,\overline{f}\right]=\left\langle  -\theta 
\alpha,b\right\rangle\overline{f}.
\end{equation}
By (\ref{eq:jou-1}), we get the first part of our proposition. By 
composing this isomorphism with $\theta$, we obtain the second part. 
If $\alpha\in R^{\mathrm{re}}$, then $-\theta\alpha=\alpha$, so that 
$\mathfrak g_{\mathfrak \alpha}$ is real.
The proof of our proposition is completed. 
\end{proof}

\begin{definition}\label{def:Dm}
	Put 
	\index{i@$\mathfrak i$}%
	\index{r@$\mathfrak r$}%
	\begin{align}\label{eq:cong1}
&\mathfrak i=\ker \mathrm{ad}\left(\mathfrak h_{\mathfrak p}\right)\cap 
\mathfrak h^{\perp},
&\mathfrak r=\ker \mathrm{ad}\left(\mathfrak h_{\mathfrak k}\right)\cap \mathfrak 
h^{\perp}.
\end{align}
\end{definition} 
 Then  $\theta$ acts on $\mathfrak i, \mathfrak r$, so that we have 
 the
 splittings,
 \begin{align}\label{eq:cong2}
&\mathfrak i= \mathfrak i_{\mathfrak p} \oplus \mathfrak i_{\mathfrak 
k},
&\mathfrak r= \mathfrak r_{\mathfrak p} \oplus \mathfrak r_{\mathfrak 
k}.
\end{align}
\begin{proposition}\label{prop:porth}
	The vector spaces $\mathfrak i, \mathfrak r$ are orthogonal in 
	$\mathfrak h^{\perp}$. Moreover, 
	\begin{align}\label{eq:jou1}
&\mathfrak i_{\C}=\oplus_{\alpha\in R^{\mathrm{im}}}\mathfrak 
g_{\alpha},
&\mathfrak r_{\C}= \oplus _{\alpha\in R^{\mathrm{re}}}\mathfrak 
g_{\alpha}.
\end{align}
If $\alpha\in R^{\mathrm{im}}$, then either $\mathfrak g_{\alpha} \subset 
\mathfrak p_{\C}$, or $\mathfrak g_{\alpha} \subset  \mathfrak k_{\C}$.
\end{proposition}
\begin{proof}
	If $f\in \mathfrak h^{\perp}, f\in \mathfrak i \cap \mathfrak 
	r$,  then $f$ commutes with 
	$\mathfrak h$. Since $\mathfrak h$ is a Cartan subalgebra,  
	$f=0$, so that $\mathfrak i\cap \mathfrak r=0$. If $\alpha\in 
	R\setminus R^{\mathrm{im}}$, $\alpha$ does not vanish identically on 
	$\mathfrak h_{\mathfrak p}$, and its vanishing locus in 
	$\mathfrak h_{\mathfrak p}$ is a 
	hyperplane. So  one can find  $b_{\mathfrak p}\in \mathfrak h_{\mathfrak 
	p}\setminus 0$  such that for any $\alpha\in R\setminus 
	R^{\mathrm{im}}$, 
	$\left\langle  \alpha,b_{\mathfrak p}\right\rangle\neq 0$.  Then
	\begin{equation}\label{eq:jou3}
\mathfrak i=\ker \mathrm{ad}\left(b_{ \mathfrak p}\right)\cap \mathfrak 
h^{\perp}.
\end{equation}
Since $\mathfrak i\cap \mathfrak r=0$, $\mathrm{ad}\left(b_{\mathfrak 
p}\right)$ acts as an invertible morphism of $\mathfrak r$. Therefore 
any element of $\mathfrak r$ lies in the image of 
$\mathrm{ad}\left(b_{\mathfrak p}\right)$. Since $\mathrm{ad}\left(b_{\mathfrak p}\right)$  is symmetric in the classical 
sense,  $\mathfrak i$ and $\mathfrak r$ are orthogonal.  Equation (\ref{eq:jou1}) is elementary.
If $\alpha\in R^{\mathrm{im}}$, the action of $\mathfrak h$ on 
$\mathfrak g_{\alpha}$ factors through   $\mathfrak 
h_{\mathfrak k}$.  Also $\mathrm{ad}\left(\mathfrak h_{\mathfrak k}\right)$ 
preserves the splitting $\mathfrak g = \mathfrak p \oplus \mathfrak 
k$. Therefore if $\alpha\in R^{\mathrm{im}}$, either $\mathfrak 
g_{\alpha} \subset  \mathfrak p_{\C}$, or $\mathfrak g_{\alpha}\subset 
\mathfrak k_{\C}$. The proof of our proposition is completed. 
\end{proof}

\begin{definition}\label{def:Dimco}
	Put
	\index{Rimp@$R^{\mathrm{im}}_{\mathfrak p}$}%
	\index{Rimk@$R^{\mathrm{im}}_{\mathfrak k}$}%
	\begin{align}\label{eq:bax1}
&R^{\mathrm{im}}_{\mathfrak p}=\left\{\alpha\in R^{\mathrm{im}}, 
\mathfrak g_{\alpha}\subset \mathfrak p_{\C}\right\},
&R^{\mathrm{im}}_{\mathfrak k}=\left\{\alpha\in R^{\mathrm{im}},\mathfrak 
g_{\alpha}\subset \mathfrak k_{\C}\right\}.
\end{align}
\end{definition}
Then
\begin{equation}\label{eq:bax1z1}
R^{\mathrm{im}}=R^{\mathrm{im}}_{\mathfrak p}\cup 
R^{\mathrm{im}}_{\mathfrak k}.
\end{equation}

Let 
\index{c@$\mathfrak c$}%
$\mathfrak c$ denote the orthogonal to $\mathfrak i \oplus 
\mathfrak r$ in $\mathfrak h^{\perp}$. Again $\mathfrak c$ splits as 
\begin{equation}\label{eq:jou6}
\mathfrak c=\mathfrak c_{\mathfrak p} \oplus \mathfrak c_{\mathfrak 
k}.
\end{equation}
Moreover, we have the orthogonal splitting
\begin{equation}\label{eq:jou7}
\mathfrak h^{\perp}=\mathfrak i \oplus \mathfrak r \oplus \mathfrak c.
\end{equation}

\begin{proposition}\label{prop:psor}
	The following identity holds:
	\begin{equation}\label{eq:jou8}
\mathfrak c_{\C}=\oplus _{\alpha\in R^{\mathrm{c}}}\mathfrak g_{\alpha}.
\end{equation}
\end{proposition}
\begin{proof}
	This follows from equations (\ref{eq:comm-1a}),
	(\ref{eq:jou1}),  and (\ref{eq:jou7}). 
\end{proof}

Now we give a result taken from \cite[Lemma 2.3.5]{Wallach88}.
\begin{proposition}\label{prop:Preal}
	A $\theta$-stable Cartan subalgebra $\mathfrak h$ is fundamental 
	if and only if there are no 
	real roots.
\end{proposition}
\begin{proof}
	If $\alpha\in R$, then $\alpha\in R^{\mathrm{re}}$ if and only if
	when $f\in \mathfrak g_{\alpha}$, 
	\begin{equation}\label{eq:gena3}
\left[\mathfrak h_{\mathfrak k},f\right]=0.
\end{equation}
If $\mathfrak h$ is fundamental, by (\ref{eq:gena1}), then $f\in 
\mathfrak h_{\C}$, so that $f=0$, which proves there are no real 
roots. Put
\begin{equation}\label{eq:zhk}
\mathfrak z\left(\mathfrak h_{\mathfrak k}\right)=\left\{f\in \mathfrak g, 
\left[\mathfrak h_{\mathfrak k},f\right]=0\right\}.
\end{equation}
Then $\mathfrak z\left(\mathfrak h_{\mathfrak k}\right)$ is a Lie 
subalgebra of $\mathfrak g$ such that $\mathfrak h \subset \mathfrak 
z\left(\mathfrak h_{\mathfrak k}\right)$.
If $\mathfrak h$ is not fundamental, $\mathfrak z\left(\mathfrak 
h_{\mathfrak k}\right)$ is strictly larger than  $\mathfrak h$, and there is a real root.  The proof of our proposition is completed. 
\end{proof}

If $\mathfrak h$ is a  Cartan subalgebra, the associated Cartan subgroup 
\index{H@$H$}%
$H \subset G$ is the stabilizer of $\mathfrak h$. Then $H$ is a Lie 
subgroup of $G$, with Lie algebra $\mathfrak h$.
\begin{proposition}\label{prop:peven}
	The vector spaces $\mathfrak i_{\mathfrak p}, \mathfrak 
	i_{\mathfrak k}, \mathfrak c_{\mathfrak 
	p}, \mathfrak c_{\mathfrak k}$ have even dimension. Also 
	$H\cap K$ preserves these vector spaces, and 
	the corresponding determinants are  equal to $1$. Also 
	$\mathfrak r_{\mathfrak p}, \mathfrak r_{\mathfrak k}$ (resp. 
	$\mathfrak c_{\mathfrak p}, \mathfrak c_{\mathfrak k} $)  have 
	the same dimension, and the actions of $H\cap K$ on 
	these two vector spaces are equivalent. In particular, we have the 
	identity 
	\begin{equation}\label{eq:grup1}
\dim 
\mathfrak p-\dim \mathfrak h_{\mathfrak p}=
\dim \mathfrak i_{\mathfrak p}+\frac{1}{2}\dim \mathfrak 
r+\frac{1}{2}\dim \mathfrak c.
\end{equation}
\end{proposition}
\begin{proof}
	If $\alpha\in R\setminus R^{\mathrm{re}}$, the 
	 vanishing locus of $\alpha$ in $\mathfrak h_{\mathfrak k}$ is a 
	hyperplane. Therefore we can find $f_{\mathfrak k}\in \mathfrak 
	h_{\mathfrak k}\setminus 0$ such for any $\alpha\in R\setminus 
	R^{\mathrm{re}}$, $\left\langle  \alpha,f_{\mathfrak 
	k}\right\rangle\neq 0$, which just says  that 
	$\mathrm{ad}\left(f_{\mathfrak k}\right)$ acts as an invertible 
	endomorphism of $\mathfrak i, \mathfrak c$. This endomorphism preserves 
	their $\mathfrak p$ and $\mathfrak k$ components, and it is 
	classically antisymmetric. This is only possible if these vector 
	spaces are even-dimensional. If $k\in  H\cap K$,  $\Ad\left(k^{-1}\right)$ preserves 
	these vector spaces and commutes with $\mathrm{ad}\left(f_{\mathfrak 
	k}\right)$. Therefore the eigenspaces associated with the eigenvalue 
	$-1$ are preserved by $\mathrm{ad}\left(f_{\mathfrak k}\right)$, so they 
	are even-dimensional. This forces the determinant of 
	$\Ad\left(k^{-1}\right)$ to be equal to $1$ on each of these 
	vector spaces. 
	
	We choose $ b_{\mathfrak p}\in \mathfrak h_{\mathfrak 
	p}\setminus 0$ such that for any $\alpha\in R\setminus 
	R^{\mathrm{im}}$, $\left\langle  \alpha,b_{\mathfrak 
	p}\right\rangle\neq 0$. Therefore	$\mathrm{ad}\left(b_{\mathfrak p}\right)$ acts as an automorphism of 
	$\mathfrak r, \mathfrak c$ that 
	exchanges the corresponding $\mathfrak p$ and $\mathfrak k$ 
	parts, and commutes with $\Ad\left(k^{-1}\right)$. 
	
		By (\ref{eq:jou7}), we get
		\begin{equation}\label{eq:grup2}
\mathfrak h_{\mathfrak p}^{\perp}= \mathfrak i_{\mathfrak p} \oplus 
\mathfrak r_{\mathfrak p} \oplus \mathfrak c_{\mathfrak p}.
\end{equation}
Using the results we already established and (\ref{eq:grup2}), we get 
(\ref{eq:grup1}). The proof of our proposition is completed. 
\end{proof}
\subsection{A positive root system}%
\label{subsec:posro}
Let $\mathfrak h$ be a $\theta$-stable Cartan subalgebra, and let $R$ 
denote the corresponding root system. Let $R_{+} \subset R$ be a 
positive root system. Set
\index{Rre@$R_{+}^{\mathrm{re}}$}%
\index{Rim@$R_{+}^{\mathrm{im}}$}%
\index{Rcp@$R^{\mathrm{c}}_{+}$}%
\begin{align}\label{eq:cor1}
&R_{+}^{\mathrm{re}}=R_{+}\cap R^{\mathrm{re}}, 
&R_{+}^{\mathrm{im}}=R_{+}\cap R^{\mathrm{im}},
\qquad R^{\mathrm{c}}_{+}=R_{+}\cap R^{\mathrm{c}}, 
\end{align}
so that
\begin{equation}\label{eq:cor1ax1}
R_{+}=R^{\mathrm{re}}_{+}\cup R_{+}^{\mathrm{im}}\cup R_{+}^{\mathrm{c}}.
\end{equation}

In the whole paper, we  choose $R_{+}$ such that  $-\theta$ 
preserves 
$R_{+}\setminus R^{\mathrm{im}}_{+}$.  Equivalently, we assume that if $\alpha\in 
R_{+}\setminus R^{\mathrm{im}}_{+}$,  then 
$\overline{\alpha}\in R_{+}$. 

Let us  explain how to do this. If $\alpha\in R\setminus 
R^{\mathrm{im}}$, the vanishing locus of $\alpha$ in $\mathfrak 
h_{\mathfrak p}$ is a hyperplane, and so there is 
$b_{\mathfrak p}\in \mathfrak h_{\mathfrak p}, \left\vert  b_{\mathfrak 
p}\right\vert=1$ such that for any 
$\alpha\in R\setminus R^{\mathrm{im}}$, $\left\langle  
\alpha,b_{\mathfrak p}\right\rangle \neq 0$. The same argument shows 
that there is $b_{\mathfrak k}\in h_{\mathfrak k}, \left\vert  
b_{\mathfrak k}\right\vert=1$ such that for $\alpha\in 
R^{\mathrm{im}}$, $\left\langle  \alpha,b_{\mathfrak 
k}\right\rangle\neq 0$. For $\epsilon>0$, $b_{\pm}=\pm b_{\mathfrak p}+ 
i\epsilon b_{\mathfrak k}\in \mathfrak h_{\mathfrak p} \oplus i \mathfrak h_{\mathfrak k}$, and 
$\theta$ interchanges $b_{+}$ and $b_{-}$. Also for $\epsilon>0$ 
small enough, for $\alpha\in R$, the real numbers $\left\langle  
\alpha,b_{\pm}\right\rangle$ do not vanish, and if $\alpha\in 
R\setminus R^{\mathrm{im}}$, they   have opposite signs.  Put
\begin{equation}\label{eq:cup1}
R_{+}=\left\{\alpha\in R,\left\langle  
\alpha,b_{+}\right\rangle>0\right\}.
\end{equation}
Then $R_{+}$ is a positive root system such that $-\theta$ preserves 
$R_{+}\setminus R_{+}^{\mathrm{im}}$.

Note  that $-\theta$ acts without fixed points on 
$R^{\mathrm{c}}_{+}$, so that $\left\vert R^{\mathrm{c}}_{+} \right\vert$ is even.
\begin{definition}\label{def:hm0}
	Put
	\index{cp@$\mathfrak c_{+}$}%
	\index{cm@$\mathfrak c_{-}$}%
	\begin{align}\label{eq:ham0}
&\mathfrak c_{+,\C}=\oplus_{\alpha\in 
R^{\mathrm{c}}_{+}} \mathfrak g_{\alpha},
&\mathfrak c_{-,\C}=\oplus_{\alpha\in 
-R^{\mathrm{c}}_{+}} \mathfrak g_{\alpha}.
\end{align}
\end{definition}
\begin{proposition}\label{prop:Ps}
	The vector spaces $\mathfrak c_{+,\C}, \mathfrak c_{-,\C}$ are the 
	complexifications of real Lie subalgebras $\mathfrak c_{+}, 
	\mathfrak c_{-}$ of $\mathfrak g$,  which have the same even dimension, and are such 
	that
	\begin{align}\label{eq:ham-1}
&\mathfrak c= \mathfrak c_{+} \oplus \mathfrak c_{-},
&\mathfrak c_{-}=\theta \mathfrak c_{+}.
\end{align}
Also $B$ vanishes on $\mathfrak c_{+}, \mathfrak c_{-}$ and induces 
the identification,
\begin{equation}\label{eq:ham-2}
\mathfrak c_{-} \simeq  \mathfrak c_{+}^{*}.
\end{equation}
The projections on $\mathfrak p, \mathfrak k$ map $\mathfrak 
c_{\pm}$ into $\mathfrak c_{\mathfrak p}, \mathfrak c_{\mathfrak k}$ 
isomorphically. Finally, the actions of $H\cap K$ on $\mathfrak c_{+}, \mathfrak c_{-}, \mathfrak 
c_{\mathfrak p}, \mathfrak c_{\mathfrak k}$ are equivalent.
\end{proposition}
\begin{proof}
	By Proposition \ref{prop:pco}, $\mathfrak c_{+,\C}, \mathfrak 
	c_{-,\C}$ are  stable by conjugation, and so they are 
	complexifications of real vector spaces $\mathfrak c_{+}, 
	\mathfrak c_{-} $. The fact that these are Lie subalgebras is 
	obvious. Since $\left\vert  R_{+}^{\mathrm{c}}\right\vert$ is 
	even, these vector spaces are even-dimensional, and also they have the same dimension. By 
	Proposition \ref{prop:pco},  $\theta$ induces an isomorphism 
	of $\mathfrak c_{+}$ into $\mathfrak c_{-}$.  Using the 
	considerations that follow (\ref{eq:coc3}), we find that $B$ 
	vanishes on $\mathfrak c_{+}, \mathfrak c_{-}$, and we obtain 
	(\ref{eq:ham-2}). By Proposition \ref{prop:peven}, $\mathfrak 
	c_{+}, \mathfrak c_{-}, \mathfrak c_{\mathfrak p}, \mathfrak 
	c_{\mathfrak k}$ have the same even dimension. The projections on 
	$\mathfrak p, \mathfrak k$ are given respectively by 
	$\frac{1}{2}\left(1\mp\theta\right)$. Since $\theta$ exchanges 
	$\mathfrak c_{+}$ and $\mathfrak c_{-}$, they restrict to 
	isomorphisms on $\mathfrak c_{+}, \mathfrak c_{-}$. By Proposition 
	\ref{prop:peven}, we know that the actions $H\cap K$ on 
	$\mathfrak c_{\mathfrak p}, \mathfrak c_{\mathfrak k}$ are 
	equivalent. Since the adjoint action  of $H\cap K$ commutes with 
	$\theta$, the  corresponding representations of $H\cap K$ on 
	these vector spaces are equivalent. The proof of our proposition is completed. \end{proof}

If $f\in \mathfrak h$, then
\begin{equation}\label{eq:coc3a4}
\det\mathrm{ad}\left(f\right)_{\mathfrak h^{\perp}}=\prod_{\alpha\in 
R}^{}\left\langle  \alpha,f\right\rangle.
\end{equation}
\begin{definition}\label{def:dpi}
	Let 
\index{phg@$\pi^{\mathfrak h, \mathfrak g}$}%
	$\pi^{\mathfrak h, \mathfrak g}\in 
S\ac\left(\mathfrak h^{*}_{\C}\right)$ be such that if $h\in 
\mathfrak h_{\C}$, then
\begin{equation}\label{eq:inva4}
\pi^{\mathfrak h, \mathfrak g}\left(h\right)=\prod_{\alpha\in 
R_{+}}^{}\left\langle  \alpha,h\right\rangle.
\end{equation}
\end{definition}

 By (\ref{eq:coc3a4}), if $f\in \mathfrak h$,
\begin{equation}\label{eq:inva5}
\det \mathrm{ad}\left(f\right)_{\vert_{\mathfrak h^{\perp}}}=\pi^{\mathfrak h,\mathfrak g}\left(f\right)
\pi^{\mathfrak h, \mathfrak g}\left(-f\right).
\end{equation}
Also $f\in \mathfrak h$ is regular if and only if $\pi^{\mathfrak h, 
\mathfrak g}\left(f\right)\neq 0$.
\begin{proposition}\label{prop:Pvan}
	The function $\pi^{\mathfrak h,\mathfrak g}$ vanishes 
	identically on $\mathfrak h_{\mathfrak k}$ if and only if 
	$\mathfrak h$ is not fundamental.
\end{proposition}
\begin{proof}
	Assume that $\mathfrak h$ is not fundamental. By Proposition 
	\ref{prop:Preal}, there are real roots, and so there are real 
	positive roots. If $\alpha\in R^{\mathrm{re}}_{+}$, then 
	$\alpha$ vanishes on $\mathfrak h_{\mathfrak k}$, and so 
	$\pi^{\mathfrak h, \mathfrak g}$ vanishes on $\mathfrak 
	h_{\mathfrak k}$. Conversely, if $\pi^{\mathfrak h,\mathfrak g}$ 
	vanishes on $\mathfrak h_{\mathfrak k}$, one of the $\alpha\in 
	R_{+}$ has to vanish identically on $\mathfrak h_{\mathfrak k}$, 
	so that $\alpha\in R^{\mathrm{re}}_{+}$,  and $\mathfrak h$ is 
	not fundamental. The proof of our proposition is completed. 
\end{proof}
\subsection{The case when $\mathfrak h$ is fundamental and the root 
system of $\left(\mathfrak h_{\mathfrak k},\mathfrak k\right)$}%
\label{subsec:rofu}
In this Subsection, we assume that $\mathfrak h$ is a $\theta$-stable 
fundamental Cartan subalgebra of $\mathfrak g$. By (\ref{eq:clo-2}), 
we get
\begin{equation}\label{eq:clo-3}
R^{\mathrm{c}}=R\setminus R^{\mathrm{im}}.
\end{equation}
By Proposition \ref{prop:porth}, 
$\mathfrak r=0$.  By (\ref{eq:jou7}), we have the orthogonal 
splitting,
\begin{align}\label{eq:ham1}
&\mathfrak h^{\perp}_{\mathfrak p}=\mathfrak i_{\mathfrak p} \oplus 
\mathfrak c_{\mathfrak p},
&\mathfrak h^{\perp}_{\mathfrak k}=\mathfrak i_{\mathfrak k} \oplus \mathfrak 
c_{\mathfrak k}.
\end{align}
Also $\mathfrak i_{\mathfrak p}, \mathfrak i_{\mathfrak k}$ have even 
dimension,  $\mathfrak c_{\mathfrak p}, \mathfrak c_{\mathfrak k}$ 
have the same even dimension, and the action of $H\cap K$ on these 
last two vector spaces are conjugate. 

The roots in $R$ do not vanish identically on $\mathfrak h_{\mathfrak 
k}$. We will now  reinforce the choice of positive roots made in 
Subsection \ref{subsec:posro}.  We may and we will assume that 
$b_{\mathfrak k}\in \mathfrak h_{\mathfrak k}, \left\vert  
b_{\mathfrak k}\right\vert=1$ has been chosen so that if $\alpha\in 
R$, $\left\langle  \alpha,b_{\mathfrak k}\right\rangle\neq 0$. 

As we saw in Subsection \ref{subsec:posro}, $-\theta$ acts without 
fixed points on $R^{\mathrm{c}}_{+}$. Also if 
$\alpha\in R^{\mathrm{c}}_{+}$, $-\theta 
\alpha\vert_{\mathfrak h_{\mathfrak k}}=-\alpha\vert_{\mathfrak 
h_{\mathfrak k}}$, so that the nonzero real  numbers $\left\langle  \alpha, ib_{\mathfrak 
k}\right\rangle$ and $\left\langle  -\theta \alpha,ib_{\mathfrak 
k}\right\rangle$ have opposite signs.

Set
\index{Rimkp@$R^{\mathrm{im}}_{\mathfrak k,+}$}%
	\index{Rcpp@$R^{\mathrm{c}}_{++}$}%
	\begin{align}\label{eq:ham3}
		&R^{\mathrm{im}}_{\mathfrak k,+}=R^{\mathrm{im}}_{\mathfrak 
		k}\cap R_{+},
&R^{\mathrm{c}}_{++}=\left\{\alpha\in R^{\mathrm{c}}_{+}, 
 \left\langle  \alpha,ib_{\mathfrak k}\right\rangle>0 \right\}.
\end{align}

\begin{definition}\label{def:Dpp}
	Let 
	\index{Rhk@$R\left(\mathfrak h_{\mathfrak k}, \mathfrak k\right)$}%
	$R\left(\mathfrak h_{\mathfrak k}, \mathfrak k\right)$ be the 
	root system associated with the pair  $\left(\mathfrak h_{\mathfrak k}, 
	\mathfrak k\right)$.
	If 
	\index{Rhkk@$R_{+}\left(\mathfrak h_{\mathfrak k}, \mathfrak k\right)$}%
	$R_{+}\left(\mathfrak h_{\mathfrak k}, \mathfrak k\right)$ is a 
positive root system for $\left(\mathfrak h_{\mathfrak k}, \mathfrak 
k\right)$, if $h_{\mathfrak k}\in \mathfrak h_{\mathfrak k,\C}$, put
\index{phk@$\pi^{\mathfrak h_{\mathfrak k}, \mathfrak k}$}%
\begin{equation}\label{eq:ham-3}
\pi^{\mathfrak h_{\mathfrak k}, \mathfrak k}\left(h_{\mathfrak 
k}\right)=\prod_{\beta\in R_{+}\left(\mathfrak h_{\mathfrak k}, 
\mathfrak k\right)}^{}\left\langle  \beta, h_{\mathfrak 
k}\right\rangle.
\end{equation}
	\end{definition}
Then $\left[\pi^{\mathfrak h_{\mathfrak k}, \mathfrak 
k}\right]^{2}\left(h_{\mathfrak k}\right)$ 
does not depend on the choice of $R_{+}\left(\mathfrak h_{\mathfrak k}, 
\mathfrak k\right)$. The  arguments above (\ref{eq:ham3}) show that if $h\in \mathfrak h_{\mathfrak k}$, 
then 
\begin{equation}\label{eq:ham-4}
\prod_{\alpha\in R_{+}^{\mathrm{c}}}^{}\left\langle  \alpha,h_{\mathfrak 
k}\right\rangle\ge 0.
\end{equation}
\begin{proposition}\label{prop:proo}
	The map $\alpha\in R_{\mathfrak k}^{\mathrm{im}}\cup 
	R^{\mathrm{c}}_{+} \to \alpha\vert_{\mathfrak 
	h_{\mathfrak k}}$ is injective, and gives the identification
	\begin{equation}\label{eq:ham4}
R\left(\mathfrak h_{\mathfrak k}, \mathfrak 
k\right)=R^{\mathrm{im}}_{\mathfrak k} \cup R^{\mathrm{c}}_{+}.
\end{equation}
A positive root system $R_{+}\left(\mathfrak h_{\mathfrak k}, 
\mathfrak k\right)$ for $\left(\mathfrak h_{\mathfrak k}, 
\mathfrak k\right)$ is given by
\begin{equation}\label{eq:ham5}
R_{+}\left(\mathfrak h_{\mathfrak k}, \mathfrak 
k\right)=R_{\mathfrak k,+}^{\mathrm{im}} \cup R_{++}^{\mathrm{c}}.
\end{equation}
If $h_{\mathfrak k}\in  \mathfrak h_{\mathfrak k,\C}$, then
\begin{equation}\label{eq:ham6}
\left[\pi^{h_{\mathfrak k}, \mathfrak 
k}\left(h_{\mathfrak k}\right)\right]^{2}=\left(-1\right)^{\frac{1}{2}\left\vert  
R_{+}^{\mathrm{c}}\right\vert}\left[\prod_{\alpha\in 
R^{\mathrm{im}}_{\mathfrak k,+}}^{}\left\langle  \alpha,h_{\mathfrak k}\right\rangle\right]^{2}
\prod_{\alpha\in R_{+}^{\mathrm{c}}}^{}\left\langle  
\alpha,h_{\mathfrak k}\right\rangle.
\end{equation}
\end{proposition}
\begin{proof}
By (\ref{eq:jou7}), we get
\begin{equation}\label{eq:ham7}
\mathfrak h_{\mathfrak k}^{\perp}= \mathfrak i_{\mathfrak k} \oplus \mathfrak 
c_{\mathfrak k}, 
\end{equation}
and the above splitting is preserved by $ \mathfrak 
h_{\mathfrak k}$. 
The weights for this action on $\mathfrak i_{\mathfrak k}$ are given 
by $R^{\mathrm{im}}_{\mathfrak k}$. By Proposition \ref{prop:Ps}, 
$\mathfrak c_{\mathfrak k}$ and $\mathfrak c_{+}$ are equivalent 
under the action of $\mathfrak h_{\mathfrak k}$. By the first 
equation in (\ref{eq:ham0}), the weights for the action of $\mathfrak 
h_{\mathfrak k}$ on $\mathfrak c_{+}$ are given by the restriction of 
$R_{+}^{\mathrm{c}}$ to $\mathfrak h_{\mathfrak k}$. 
Since the weights for the action of $\mathfrak h_{\mathfrak k}$ on 
$\mathfrak c_{\mathfrak k}$ are  nonzero and of multiplicity  
$1$, 
the map $\alpha\in R_{\mathfrak k}^{\mathrm{im}}\cup R_{+}^{\mathrm{c}} \to \alpha\vert_{\mathfrak h_{\mathfrak 
k}}$ gives the identification in (\ref{eq:ham4}). By 
(\ref{eq:ham4}), we get (\ref{eq:ham5}). Using (\ref{eq:ham-3}) and 
the above results, we get (\ref{eq:ham6}). The proof of our proposition is completed. 
\end{proof}
\begin{remark}\label{rem:impo}
	The results contained in Proposition \ref{prop:proo} will play an 
	important role in the proof of the limit results of Subsection 
	\ref{subsec:jgnore}.
\end{remark}
\subsection{Cartan subgroups and regular elements}%
\label{subsec:casg}
Assume that $\mathfrak h$ is $\theta$-stable.
Then $\theta$ restricts to an involution of $H$, and (\ref{eq:inva3}) 
is the corresponding Cartan splitting of $\mathfrak h$. Also $B$ restricts to a  $H$ and $\theta$ invariant 
symmetric nondegenerate bilinear form 
\index{Bh@$B\vert_{\mathfrak h}$}%
$B\vert_{\mathfrak h}$ on $\mathfrak 
h$, so that $H$ is a
reductive subgroup of $G$.

We still assume $\mathfrak h$ to be $\theta$-stable. Let $Z_{G}\left(H\right) \subset G$ be the centralizer of $H$,  and 
let $N_{G}\left(H\right) \subset G$ be its normalizer. Then $Z_{G}\left(H\right)$ is included in $H$, it is just the 
center $Z\left(H\right)$ of  $H$. As in \cite[p. 131]{Knapp86}, the analytic Weyl group 
\index{WHG@$W\left(H:G\right)$}%
$W\left(H:G\right)$ is defined 
as the quotient
\begin{equation}\label{eq:inva3a1}
W\left(H:G\right)=N_{G}\left(H\right)/Z_{G}\left(H\right).
\end{equation}

Put
\begin{align}\label{eq:inva3a2}
&Z_{K}\left(H\right)=Z_{G}\left(H\right) \cap K, 
&N_{K}\left(H\right)=N_{G}\left(H\right)\cap K.
\end{align}
Then $N_{K}\left(H\right)/Z_{K}\left(H\right)$ embeds in 
$W\left(H:G\right)$. By \cite[p. 131]{Knapp86}, this embedding is an 
isomorphism, i.e., 
\begin{equation}\label{eq:inva3a2x1}
W\left(H:G\right)=N_{K}\left(H\right)/Z_{K}\left(H\right).
\end{equation}
By 
\cite[eq. (5.6)]{Knapp86}, $W\left(H:G\right) \subset 
W\left(\mathfrak h_{\C}:\mathfrak g_{\C}\right)$.

By \cite[p. 130]{Knapp86}, an element $\gamma\in G$ is said to be regular if 
$\mathfrak  z\left(\gamma\right)$ is a Cartan subalgebra.  If $H$ is the corresponding Cartan subgroup, then 
$\gamma\in H$.  By \cite[Theorem 
5.22]{Knapp86}, the set 
\index{Greg@$G^{\mathrm{reg}}$}%
$G^{\mathrm{reg}} \subset G$ of regular 
elements is open and 
conjugation invariant. More precisely, if $H_{1},\ldots,H_{\ell}$ 
denotes 
the finite family of nonconjugate Cartan subgroups, by \cite[Theorem 
5.22]{Knapp86}, $G^{\mathrm{reg}}$ splits as the disjoint union of open sets
\begin{equation}\label{eq:inva6a1bi}
G^{\mathrm{reg}}=\cup_{i=1}^{\ell}G^{\mathrm{reg}}_{H_{i}}, 
\end{equation}
where $G^{\mathrm{reg}}_{H_{i}}$ denote the open set of elements of 
$G^{\mathrm{reg}}$ that are conjugate to an element of $H_{i}$.

If  $\gamma\in H$, $\Ad\left(\gamma\right)$ acts on $\mathfrak g$ and fixes $\mathfrak 
h$. Since $\Ad\left(\gamma\right)$ preserves $B$, it also acts on 
$\mathfrak h^{\perp}$, so that $1-\Ad\left(\gamma\right)$ acts on 
$\mathfrak h^{\perp}$.  Then $\gamma$ is regular if and only this 
endomorphism is invertible, i.e.,  
$\det \left( 1-\Ad\left(\gamma\right)\right)\vert_{\mathfrak h^{\perp}}\neq 0 $. 
\subsection{Cartan subgroups and semisimple elements}%
\label{subsec:semisi}
The following result is established in \cite[Part I, Section 2.3, 
Theorem 4]{Vara77}.
\begin{proposition}\label{prop:semis}
	A group element $\gamma\in G$ is semisimple if and only if it 
	lies in a Cartan subgroup. 
\end{proposition}

Let us give a direct proof of part of our proposition. Le $\mathfrak h$ be a 
$\theta$-stable Cartan subalgebra, and let $H$ be the corresponding 
Cartan subgroup. If $\gamma\in H$, then $\mathfrak h \subset  \mathfrak 
z\left(\gamma\right)$. Moreover, $\gamma$ can be written uniquely 
in the form
\begin{align}\label{eq:bunb1}
&\gamma=e^{a}k^{-1}, &a\in \mathfrak h_{\mathfrak p}, \qquad k\in 
H\cap K.
\end{align}
Since $a\in \mathfrak h_{\mathfrak p}, k\in H$, then 
$\Ad\left(k\right)a=a$, which guarantees that $\gamma$ is semisimple in 
$G$.

Let $\mathfrak h \subset \mathfrak g$ be a Cartan subalgebra, and let 
$H \subset G$ be the associated Cartan subgroup.  If $\gamma\in H$, then $\mathfrak h \subset \mathfrak 
z\left(\gamma\right)$, so that $\mathfrak h$ is a Cartan subalgebra 
of $\mathfrak z\left(\gamma\right)$. In particular $G$ and 
$Z^{0}\left(\gamma\right)$ have the same complex rank.
\begin{proposition}\label{prop:psesicen}
	Any $\theta$-stable Cartan subalgebra $\mathfrak h_{0}$ of $\mathfrak 
	z\left(\gamma\right)$ is also a Cartan subalgebra of $\mathfrak 
	g$.
\end{proposition}
\begin{proof}
	As we saw in Subsection \ref{subsec:sesidis}, 
	$Z^{0}\left(\gamma\right)$ is a connected reductive group and 
	$\theta$ induces on $Z^{0}\left(\gamma\right)$ a corresponding 
	Cartan involution.  Since $\mathfrak h_{0}$  is commutative and 
	$\theta$-stable, and since its action on $\mathfrak g$ preserves 
	$B$,  it acts on $\mathfrak g$ by semisimple 
	endomorphisms of $\mathfrak g$. Since $G$ and 
	$Z^{0}\left(\gamma\right)$ have the same complex rank, $\mathfrak 
	h_{0}$ is a Cartan subalgebra of $\mathfrak g$. The proof of our proposition is completed. 
\end{proof}
\subsection{Root systems and their characters}%
\label{subsec:rosisi}
Let $\mathfrak h \subset \mathfrak g$ be a $\theta$-stable Cartan 
subalgebra. We use the notation of the previous subsections.

Take  $\gamma\in H$. As we saw after  Proposition 
\ref{prop:semis}, if $\gamma\in H$, we can write $\gamma$ uniquely in 
the form
\begin{align}\label{eq:guing1}
&\gamma=e^{a}k^{-1}, &a\in \mathfrak h_{\mathfrak p},\qquad  k\in H\cap K,
\end{align}
so that 
\begin{equation}\label{eq:guing2}
\Ad\left(k\right)a=a.
\end{equation}
Let $R\left(\gamma\right),R\left(a\right)$  be the root systems 
associated with $\left(\mathfrak h, \mathfrak 
z\left(\gamma\right)\right), \left(\mathfrak h, \mathfrak 
z\left(a\right)\right)$.  We will denote with  extra subscripts  the 
corresponding real, imaginary, and complex roots.
\begin{theorem}\label{thm:careg}
	If $\gamma\in H$, for any $\alpha\in R$, $\Ad\left(\gamma\right)$ 
	preserves the $1$-dimensional complex line $\mathfrak 
	g_{\alpha}$. For every $\alpha\in R$, there is a character 
	$\xi_{\alpha}: H\to \C^{*}$ such that $\Ad\left(\gamma\right)$ 
	acts on $\mathfrak g_{\alpha}$ by multiplication by 
	$\xi_{\alpha}\left(\gamma\right)$. If  $\alpha\in R$, 
	\begin{equation}\label{eq:comm2}
\xi_{\alpha}\xi_{-\alpha}=1.
\end{equation}
If $\alpha \in R$,  if $f\in \mathfrak h, k\in H\cap K$, then
\begin{align}\label{eq:comm2a0}
&\xi_{\alpha}\left(e^{f}\right)=e^{\left\langle  
\alpha,f\right\rangle},
&\left\vert  \xi_{\alpha}\left(k\right)\right\vert=1.
\end{align}
In particular, if $\gamma\in H$ is taken as in (\ref{eq:guing1}), then
\begin{align}\label{eq:comm2a-1}
&\xi_{\alpha}\left(\gamma\right)=e^{\left\langle  
\alpha,a\right\rangle}\xi_{\alpha}\left(k^{-1}\right),
&\xi_{-\theta\alpha}\left(\gamma\right)=\overline{\xi_{\alpha}\left(\gamma\right)}.
\end{align}
If $\alpha\in R^{\mathrm{re}}$, then 
$\xi_{\alpha}\left(\gamma\right)\in \R^{*}$, if $\alpha\in 
R^{\mathrm{im}}$, then $\left\vert  
\xi_{\alpha}\left(\gamma\right)\right\vert=1$. If
$\alpha\in R^{\mathrm{re}}$, the restriction of $\xi_{\alpha}$ to  
$H\cap K$ takes its values in $\left\{-1,+1\right\}$.

Also
\begin{equation}\label{eq:comm2a1}
\det\left(1-\Ad\left(\gamma\right)\right)\vert_{\mathfrak h^{\perp}}=\prod_{\alpha\in 
R}^{}\left(1-\xi_{\alpha}\left(\gamma\right)\right),
\end{equation}
and $\gamma$ is regular if and only if for any $\alpha\in R$, 
$\xi_{\alpha}\left(\gamma\right) \neq 1$. 

The following identities hold:
\index{Rreg@$R^{\mathrm{re}}\left(\gamma\right)$}%
\index{Rimg@$R^{\mathrm{im}}\left(\gamma\right)$}%
\begin{align}\label{eq:comm2a2}
&R\left(\gamma\right)=\left\{\alpha\in R, 
\xi_{\alpha}\left(\gamma\right)=1\right\},\qquad R\left(a\right)=\left\{\alpha\in R, 
\left\langle  
\alpha,a\right\rangle=0\right\},\nonumber \\
&R^{\mathrm{re}}\left(\gamma\right)=R\left(\gamma\right)\cap 
R^{\mathrm{re}}, \,\,
R^{\mathrm{im}}\left(\gamma\right)=R\left(\gamma\right)\cap 
R^{\mathrm{im}}, \,\,
R^{\mathrm{c}}\left(\gamma\right)=R\left(\gamma\right)\cap 
R^{\mathrm{c}},  \\
&R^{\mathrm{re}}\left(a\right)=R\left(a\right)\cap R^{\mathrm{re}},
\,\, R^{\mathrm{im}}\left(a\right) =R\left(a\right)\cap R^{\mathrm{im}}, 
\,\, R^{\mathrm{c}}\left(a\right)=R\left(a\right)\cap R^{\mathrm{c}}. \nonumber 
\end{align}
Also $R_{+}\left(\gamma\right)=R\left(\gamma\right) \cap R_{+}, 
R_{+}\left(a\right)=R\left(a\right)\cap R_{+}$ are  positive root 
systems 
for $\left(\mathfrak h, \mathfrak z\left(\gamma\right)\right), 
\left(\mathfrak h, \mathfrak z\left(a\right)\right)$. 
\end{theorem}
\begin{proof}
	If $\gamma\in H$, then $\Ad\left(\gamma\right)$ 
fixes $\mathfrak h$, and so if $h\in \mathfrak h$, we have the 
commutation relation in $\End\left(\mathfrak g\right)$, 
\begin{equation}\label{eq:comm1}
\left[\Ad\left(\gamma\right), \mathrm{ad}\left(h\right)\right]=0. 
\end{equation}
By (\ref{eq:coci1}),
(\ref{eq:comm1}), we deduce that for any $\alpha\in R$, 
$\Ad\left(\gamma\right)$ preserves $\mathfrak g_{\alpha}$.  Since 
$\mathfrak g_{\alpha}$ is a complex line, $H$ acts on $\mathfrak 
g_{\alpha}$ via a character $\xi_{\alpha}$.

Since $\Ad\left(\gamma\right)$ preserves $B$, if $f,f'\in \mathfrak 
g_{\C}$, we get
\begin{equation}\label{eq:gung1}
B\left(\Ad\left(\gamma\right)f,f'\right)=B\left(f,\Ad\left(\gamma\right)^{-1}f'\right). 
\end{equation}
 Take $\alpha\in R$. By (\ref{eq:gung1}), if   $f\in \mathfrak g_{\alpha}, f'\in 
\mathfrak g_{-\alpha}$,  then
\begin{equation}\label{eq:gung2}
\xi_{\alpha}\left(\gamma\right)B\left(f,f'\right)=\xi_{-\alpha}^{-1}\left(\gamma\right)B\left(f,f'\right).
\end{equation}
As we saw in Subsection \ref{subsec:reglie}, if $\alpha\in R$, the 
pairing between $\mathfrak g_{\alpha}$ and $\mathfrak g_{-\alpha}$ via 
$B$ is nondegenerate. By (\ref{eq:gung2}), we get (\ref{eq:comm2}). 

The first equation 
in (\ref{eq:comm2a0}) is trivial. Since $\xi_{\alpha}$ restricts to a character 
of the compact group $H\cap K$, we get the second equation in 
(\ref{eq:comm2a0}). The first equation in (\ref{eq:comm2a-1}) follows from the 
previous considerations. Since 
$\theta\left(\gamma\right)=e^{-a}k^{-1}$, and since $\theta$ maps 
$\mathfrak g_{\alpha}$ into $\mathfrak g_{\theta \alpha}$, we obtain the second 
equation in (\ref{eq:comm2a-1}). From this second equation, we deduce 
that if $\alpha\in R^{\mathrm{re}}$, then 
$\xi_{\alpha}\left(\gamma\right)$ is real, and if $\alpha\in 
R^{\mathrm{im}}$, then $\left\vert  
\xi_{\alpha}\left(\gamma\right)\right\vert=1$. If $\gamma\in H\cap K, \alpha\in 
R^{\mathrm{re}}$, we know that $\xi_{\alpha}\left(\gamma\right)\in 
\R^{*}, \left\vert  \xi_{\alpha}\left(\gamma\right)\right\vert=1$, so 
that $\xi_{\alpha}\left(\gamma\right)=\pm 1$.

Equations (\ref{eq:comm2a1}), (\ref{eq:comm2a2}) are trivial. By 
(\ref{eq:comm2a1}), $\gamma$ is regular if and only if for $\alpha\in 
R$, $\xi_{\alpha}\left(\gamma\right)\neq 1$.

Now we proceed as in 
\cite[Theorem 1.38]{BismutLabourie99}. If $\kappa \subset 
\mathfrak h$ is a positive Weyl chamber for $\left(\mathfrak h, 
\mathfrak g\right)$, the forms in $R$ do not vanish on $\kappa$, so 
that $\kappa$ is included in a $\mathfrak z\left(\gamma\right)$ Weyl 
chamber. It follows that $R\left(\gamma\right)\cap R_{+}$ is a 
positive root system on $\left(\mathfrak h, \mathfrak 
z\left(\gamma\right)\right)$.  The same argument is valid for 
$R\left(a\right)$.  
 The proof of our 
theorem is completed.  
\end{proof}
\subsection{Real roots, imaginary roots, and semisimple elements}%
\label{subsec:reimsesi}
We still take $\gamma$ as in Subsection \ref{subsec:rosisi}.
When taking the intersection of $\mathfrak i, \mathfrak r, \mathfrak c$ with 
$\mathfrak z\left(\gamma\right), \mathfrak z\left(a\right), \mathfrak 
z\left(k\right)$, this will be indicated with a  parenthesis 
containing the corresponding argument.  The intersection with 
$\mathfrak z^{\perp}\left(\gamma\right), \mathfrak 
z^{\perp}\left(a\right), \mathfrak z^{\perp}\left(k\right)$ will be 
denoted with an extra $\perp$.  These vector spaces also have a $\mathfrak p$ and a 
$\mathfrak k$ component.

By construction,
\begin{equation}\label{eq:rim1}
R^{\mathrm{im}}\left(\gamma\right)=R^{\mathrm{im}}\left(k\right).
\end{equation}
As in (\ref{eq:bax1}), (\ref{eq:bax1z1}), we get
\index{Rimpk@$R^{\mathrm{im}}_{\mathfrak p}\left(k\right)$}%
\index{Rimkk@$R^{\mathrm{im}}_{\mathfrak k}\left(k\right)$}%
\begin{align}\label{eq:rim1ax1}
&R^{\mathrm{im}}\left(\gamma\right)=R^{\mathrm{im}}_{\mathfrak p}\left(\gamma\right)\cup R^{\mathrm{im}}_{\mathfrak 
k}\left(\gamma\right),
&R^{\mathrm{im}}\left(k\right)=R^{\mathrm{im}}_{\mathfrak p}\left(k\right)\cup R^{\mathrm{im}}_{\mathfrak k}\left(k\right). 
\end{align}
To make the notation  simpler, in (\ref{eq:rim1ax1}), we did not use 
instead the notation $\mathfrak 
p\left(\gamma\right), \mathfrak k\left(\gamma\right), \mathfrak 
p\left(k\right), \mathfrak k\left(k\right)$.

\begin{proposition}\label{prop:ispa}
	The following identities holds:
	\begin{align}\label{eq:cong5a3}
&\mathfrak i \subset \mathfrak z\left(a\right), &\mathfrak i\left(\gamma\right)= \mathfrak 
i\left(k\right),\qquad\mathfrak i^{\perp}\left(k\right) \subset \mathfrak 
z_{a}^{\perp}\left(\gamma\right). 
\end{align}
Also
	\begin{align}\label{eq:cong5a6}
&\mathfrak i\left(k\right)_{\C}=\oplus_{\alpha\in 
R^{\mathrm{im}}\left(k\right)} \mathfrak g_{\alpha},
&\mathfrak i^{\perp}\left(k\right)_{\C}=\oplus _{\alpha\in 
R^{\mathrm{im}}\setminus R^{\mathrm{im}}\left(k\right)}\mathfrak 
g_{\alpha}.
\end{align}
Moreover,  
\begin{align}\label{eq:cong5a6a}
	&\mathfrak z\left(a\right)_{\C}=\mathfrak h_{\C} \bigoplus \oplus 
	_{\alpha\in R\left(a\right)}\mathfrak g_{\alpha}, \nonumber \\
&\mathfrak z^{\perp}\left(a\right) \subset \mathfrak r 
\oplus \mathfrak c, \\
&\mathfrak z^{\perp}\left(a\right)_{\C}=\oplus_{\alpha\in R\setminus 
R\left(a\right)} \mathfrak g_{\alpha}. \nonumber 
\end{align}
If $\gamma$ is regular, then
\begin{equation}\label{eq:cong5a5}
\mathfrak i\left(k\right)=0.
\end{equation}
\end{proposition}
\begin{proof}
	By the first identity in (\ref{eq:cong1}), since $a\in \mathfrak 
	h_{\mathfrak p}$, we get the first identity in (\ref{eq:cong5a3}).
	 Combining  the third identity in (\ref{eq:inva18a}) with this 
	 first identity,   we 
	get  the second identity in (\ref{eq:cong5a3}). The third 
	identity in (\ref{eq:cong5a3}) is a consequence of the first two. 
	By (\ref{eq:coci1}), we  get (\ref{eq:cong5a6}) and the first and the 
	third equations in 
	(\ref{eq:cong5a6a}), the second equation being a consequence of 
	the first equation in (\ref{eq:cong5a3}).  Also 
	$\gamma$ is regular if and only if $\mathfrak 
	z\left(\gamma\right)=\mathfrak h$. Since $\mathfrak h\cap 
	\mathfrak i=0$, by the second identity in (\ref{eq:cong5a3}), we get (\ref{eq:cong5a5}).
	The proof of our proposition is completed. 
\end{proof}

\subsection{Cartan subalgebras and differential operators}%
\label{subsec:care}
Let $\mathfrak h $ be a $\theta$-stable Cartan subalgebra. There is a natural 
projection $\mathfrak g^{*}\to \mathfrak h^{*}$.

By (\ref{eq:cor0}), there is a well-defined projection $\mathfrak g\to 
\mathfrak h$.  To the splitting (\ref{eq:cor0}) corresponds the 
dual splitting
\begin{equation}\label{eq:inva3b}
\mathfrak g^{*}=\mathfrak h^{*} \oplus \mathfrak h^{*\perp}.
\end{equation}
The projections $S\ac\left(\mathfrak g^{*}\right)\to S\ac\left(\mathfrak 
h^{*}\right),S\ac\left(\mathfrak g\right)\to S\ac\left(\mathfrak 
h\right)$ associated with (\ref{eq:cor0}), (\ref{eq:inva3b}) 
are just the restriction 
\index{r@$r$}%
$r$ of 
polynomials on $\mathfrak g$ to $\mathfrak h$, or of 
polynomials on $\mathfrak g^{*}$ to $\mathfrak h^{*}$.

The Lie algebra $\mathfrak g$ acts as an algebra of derivations on 
$S\ac\left(\mathfrak g^{*}\right)$. 
\begin{definition}\label{def:dinv}
	Let 
\index{Ig@$I\ac\left(\mathfrak g^{*}\right)$}%
$I\ac\left(\mathfrak g^{*}\right)  \subset S\ac\left(\mathfrak g^{*}\right)$  
be the algebra of invariant elements in $S\ac\left(\mathfrak 
g^{*}\right)$, i.e., the algebra of the elements of 
$S\ac\left(\mathfrak g^{*}\right)$ on which the derivations 
associated with $\mathfrak g$ vanish.  Let
\index{Ihg@$I\ac\left(\mathfrak h^{*}_{\C}, \mathfrak g^{*}_{\C}\right)$}%
$I\ac\left(\mathfrak h^{*}_{\C}, \mathfrak g^{*}_{\C}\right)$ be the 
algebra of $W\left(\mathfrak h_{\C}:\mathfrak 
g_{\C}\right)$-invariant elements in $S\ac\left(\mathfrak 
h^{*}_{\C}\right)$.
\end{definition}

Recall that 
\begin{equation}\label{eq:chif5}
S\ac\left(\mathfrak h_{\C}^{*}\right)=S\ac\left(\mathfrak 
h^{*}\right)_{\C}.
\end{equation}
In particular $S\ac\left(\mathfrak h_{\C}^{*}\right)$ is equipped 
with a natural conjugation. 
\begin{proposition}\label{prop:conja}
	The algebra $I\ac\left(\mathfrak h_{\C}^{*}, \mathfrak 
	g_{\C}^{*}\right)$ is preserved under complex conjugation. There is a real 
	algebra
	\index{Ihg@$I\ac\left(\mathfrak h^{*}, \mathfrak g^{*}\right)$}%
	$I\ac\left(\mathfrak h^{*}, \mathfrak g^{*}\right)\subset 
	S\ac\left(\mathfrak  h^{*}\right)$ such that
	\begin{equation}\label{eq:chif6}
I\ac\left(\mathfrak h^{*}_{\C}, \mathfrak 
g^{*}_{\C}\right)=I\ac\left(\mathfrak h^{*}, \mathfrak 
g^{*}\right)_{\C}.
\end{equation}
The map $r: S\ac\left(\mathfrak g^{*}\right)\to 
S\ac\left(\mathfrak h^{*}\right) $ induces the canonical isomorphism
\begin{equation}\label{eq:chif7}
r:I\ac\left(\mathfrak g^{*}\right) \simeq I\ac\left(\mathfrak h^{*},\mathfrak 
g^{*}\right).
\end{equation}
\end{proposition}
\begin{proof}
	By Proposition \ref{prop:pstab}, $W\left(\mathfrak h_{\C}: 
	\mathfrak g_{\C}\right)$ is preserved by conjugation. Therefore 
	$I\ac\left(\mathfrak h_{\C}^{*}, \mathfrak g_{\C}^{*}\right)$ is 
	 preserved by conjugation, which gives  (\ref{eq:chif6}). From the obvious isomorphism
	\begin{equation}\label{eq:chif8}
r:I\ac\left(\mathfrak g_{\C}^{*}\right) \to I\ac\left(\mathfrak 
h_{\C}^{*},\mathfrak g_{\C}^{*}\right), 
\end{equation}
we get (\ref{eq:chif7}). The proof of our proposition is completed. 
\end{proof}

What we did for $\mathfrak g^{*}$ can also be done for $\mathfrak g$. 
The same argument as in (\ref{eq:chif7}) leads to the identification
\index{Ihg@$I\ac\left(\mathfrak h,\mathfrak g\right)$}%
\begin{equation}\label{eq:inva6a}
r: I\ac\left(\mathfrak g\right) \simeq I\ac\left(\mathfrak 
h,\mathfrak g\right).
\end{equation}

As we saw in Subsection \ref{subsec:lial}, $S\ac\left(\mathfrak 
g\right)$ can be identified with the algebra 
\index{Dg@$D\ac\left(\mathfrak g\right)$}%
$D\ac\left(\mathfrak g\right)$ of real differential operators on 
$\mathfrak g$ with constant coefficients, so 
that $I\ac\left(\mathfrak g\right)$ is identified with the algebra 
\index{Dig@$D_{I}\ac\left(\mathfrak g\right)$}%
$D_{I}\ac\left(\mathfrak g\right)$ of  real
  differential operators with constant coefficients on $\mathfrak g$ 
  which commute with the above $\mathfrak g$-derivations.  
Similarly $I\ac\left(\mathfrak h,\mathfrak g\right)$ can be 
identified with the algebra 
\index{DIhg@$D_{I}\ac\left(\mathfrak h,\mathfrak g\right)$}%
$D\ac_{I}\left(\mathfrak h,\mathfrak g\right)$  of real  differential 
operators on $\mathfrak h$ with constant coefficients  that are $W\left(\mathfrak 
h_{\C}:\mathfrak g_{\C}\right)$-invariant.

Let $R_{+} \subset R$ be a positive root system as in Subsection 
\ref{subsec:posro}. Recall that the associated polynomial $\pi^{\mathfrak h, 
\mathfrak g}\in S\ac\left(\mathfrak h^{*}_{\C}\right)$ was defined in 
(\ref{eq:inva4}). 
If $A\in 
I\ac\left(\mathfrak g\right) =D_{I}\left(\mathfrak g\right) $, if 
$f\in I\ac\left(\mathfrak g^{*}\right)$, then $Af\in 
I\ac\left(\mathfrak g^{*}\right)$, so that $r \left( Af \right)  \in 
I\ac\left(\mathfrak h^{*},\mathfrak g^{*}\right)$. Also 
$r\left(A\right)\in I\ac\left(\mathfrak h, \mathfrak 
g\right)=D_{I}\left(\mathfrak h, \mathfrak g\right)$.  By 
\cite[Lemmas 6 and 8]{Harish57}, if $f\in I\ac\left(\mathfrak 
g^{*}\right)$, 
\begin{equation}\label{eq:inva6b}
r\left(Af\right)=\frac{1}{\pi^{\mathfrak h,\mathfrak 
g}}r\left(A\right)\pi^{\mathfrak h, \mathfrak g}rf.
\end{equation}
Let $C^{\infty,\mathfrak g}\left(\mathfrak g,\R\right)$ be the vector space of 
 smooth real functions on $\mathfrak g$ that vanish under the above 
 $\mathfrak g$-derivations. Then 
 (\ref{eq:inva6b}) extends to $f\in C^{ \infty ,\mathfrak g}\left(\mathfrak 
g,\R\right)$.
\section{Root systems and the function $\mathcal{J}_{\gamma}$}%
\label{sec:rofu}
The purpose of this Section is to give a drastically simplified 
version of the function $\mathcal{J}_{\gamma}\left(\Yok\right)$ introduced in 
Definition \ref{def:Jg}. This will be done by expressing this 
function in terms of a positive root system. Imaginary roots will 
play an essential role in this expression. In particular, the 
function $\mathcal{L}_{\gamma}$ introduced in Definition \ref{DLg} 
will turn out not to depend on $a$.

This section is organized as follows. In Subsection 
\ref{subsec:detad}, if $\mathfrak h$ is a $\theta$-stable Cartan subalgebra and $H$ 
is the corresponding Cartan subgroup, if $\gamma\in H$, we give 
explicit formulas for the 
determinant of $1-\Ad\left(\gamma\right)$ on various subspaces in terms of a 
positive root system.

In Subsection \ref{subsec:Kg}, we establish our formula for 
$\mathcal{J}_{\gamma}\left(\Yok\right)$ using the root system.

We use the assumptions and the notation of Section \ref{sec:geofori}.
\subsection{The determinant of $1-\Ad\left(\gamma\right)$}%
\label{subsec:detad}
Let $\mathfrak h  \subset \mathfrak g$ be a $\theta$-stable Cartan subalgebra, and let $H 
\subset G$ be the corresponding Cartan subgroup. 
Put
\index{h@$\mathfrak h^{\perp}_{\pm}$}%
\begin{align}\label{eq:crin1}
&\mathfrak h^{\perp}_{+}=\bigoplus_{\alpha\in R_{+}}\mathfrak 
g_{\alpha},&\mathfrak h^{\perp}_{-}=\bigoplus_{\alpha\in 
R_{+}}\mathfrak g_{-\alpha}.
\end{align}
By (\ref{eq:comm-1a}),  we get
\begin{equation}\label{eq:crin2}
\mathfrak h_{\C}^{\perp}=\mathfrak h^{\perp}_{+} \oplus \mathfrak 
h^{\perp}_{-}.
\end{equation}

Let $\gamma\in H$ be written as in (\ref{eq:guing1}). 
By Theorem \ref{thm:careg}, we obtain
\begin{equation}\label{eq:grug2}
\det\Ad\left(\gamma\right)\vert_{\mathfrak h^{\perp}_{+}}=\prod_{\alpha\in 
R_{+}}^{}\xi_{\alpha}\left(\gamma\right).
\end{equation}
We write (\ref{eq:grug2}) in the form
\begin{equation}\label{eq:grug2a1}
\det\Ad\left(\gamma\right)\vert_{\mathfrak h^{\perp}_{+}}=\prod_{\alpha\in 
R^{\mathrm{c}}_{+}}^{}\xi_{\alpha}\left(\gamma\right)
\prod_{\alpha\in R_{+}^{\mathrm{re}}}^{}\xi_{\alpha}\left(\gamma\right)
\prod_{\alpha\in 
R_{+}^{\mathrm{im}}}^{}\xi_{\alpha}\left(\gamma\right).
\end{equation}
By the considerations we made in Subsection \ref{subsec:posro}, 
$-\theta$ acts without fixed points on   $R^{\mathrm{c}}_{+}$. By  
 Theorem \ref{thm:careg}, in the right-hand side of 
(\ref{eq:grug2a1}), the first term is positive, the second is a product 
of nonzero real numbers, and the third term is a 
product of complex numbers of module $1$.

If $\alpha\in R_{+}$, we choose a square 
root $\xi_{\alpha}^{1/2}\left(k^{-1}\right)$ of 
$\xi_{\alpha}\left(k^{-1}\right)$.  In view of the 
second identity in (\ref{eq:comm2a-1}), if $\alpha\in R_{+}^{\mathrm{c}}$, we may and we will 
assume that
\begin{equation}\label{eq:gru2z1}
\xi_{-\theta\alpha}^{1/2}\left(k^{-1}\right)=\overline{\xi_{\alpha}^{1/2}\left(k^{-1}\right)}.
\end{equation}

For $\alpha\in R_{+}$, we choose the square root 
$\xi_{\alpha}^{1/2}\left(\gamma\right)$ so that
\begin{equation}\label{eq:gru2z2}
\xi_{\alpha}^{1/2}\left(\gamma\right)=e^{\left\langle  
\alpha,a\right\rangle/2}\xi^{1/2}_{\alpha}\left(k^{-1}\right).
\end{equation}
By (\ref{eq:gru2z1}), (\ref{eq:gru2z2}), if $\alpha\in 
R_{+}^{\mathrm{c}}$, then
\begin{equation}\label{eq:gru2z3}
\xi^{1/2}_{-\theta\alpha}\left(\gamma\right)=\overline{\xi_{\alpha}^{1/2}\left(\gamma\right)}.
\end{equation}

A square root of 
$\det\Ad\left(\gamma\right)\vert_{\mathfrak h^{\perp}_{+}}$ in (\ref{eq:grug2}) 
is given by
\begin{equation}\label{eq:gru2a1b}
\det\Ad\left(\gamma\right)\vert_{\mathfrak h^{\perp}_{+}}^{1/2}=
\prod_{\alpha\in R_{+}}^{}\xi_{\alpha}^{1/2}\left(\gamma\right).
\end{equation}
By proceeding as in (\ref{eq:grug2a1}), we can rewrite 
(\ref{eq:gru2a1b}) in the form
\begin{equation}\label{eq:grug2a1c}
\det\Ad\left(\gamma\right)\vert_{\mathfrak h^{\perp}_{+}}^{1/2}=\prod_{\alpha\in 
R_{+}^{\mathrm{c}}}^{}\xi_{\alpha}^{1/2}\left(\gamma\right)
\prod_{\alpha\in R_{+}^{\mathrm{re}}}^{}\xi^{1/2}_{\alpha}\left(\gamma\right)
\prod_{\alpha\in R_{+}^{\mathrm{im}}}^{}\xi^{1/2}_{\alpha}\left(\gamma\right).
\end{equation}
Using  (\ref{eq:gru2z2}),  
(\ref{eq:gru2z3}), we find that the 
first product in the right hand-side of (\ref{eq:grug2a1c}) is 
positive,  the
second product is either a nonzero  real number, or the
product of $\sqrt{-1}$ by a nonzero real number, and the third 
product is of  module  $1$.

\begin{definition}\label{def:sign}
	Put
	\begin{equation}\label{eq:grai1}
\epsilon_{D}\left(\gamma\right)=\mathrm{sgn} \prod_{\alpha\in 
R_{+}^{\mathrm{re}}\setminus 
R_{+}^{\mathrm{re}}\left(\gamma\right)}^{}\left(1-\xi_{\alpha}^{-1}\left(\gamma\right)\right).
\end{equation}
\end{definition}

Recall that if $\alpha\in R^{\mathrm{re}}$, then 
$\xi_{\alpha}\left(k^{-1}\right)=\pm 1$. 

\begin{proposition}\label{prop:si}
	The following identity holds:
	\begin{equation}\label{eq:grai2}
\epsilon_{D}\left(\gamma\right)=\mathrm{sgn} \prod_{\alpha\in 
R_{+}^{\mathrm{re}}\setminus 
R_{+}^{\mathrm{re}}\left(a\right)}^{}\left(1-\xi_{\alpha}^{-1}\left(\gamma\right)\right).
\end{equation}
\end{proposition}
\begin{proof}
	If $\alpha\in R\left(a\right)$, by (\ref{eq:comm2a-1}), 
	$\xi_{\alpha}\left(\gamma\right)=\xi_{\alpha}\left(k^{-1}\right)$. If $\alpha\notin R\left(\gamma\right)$, by (\ref{eq:comm2a2}), $\xi_{\alpha}\left(\gamma\right)\neq 1$. By Theorem \ref{thm:careg}, if $\alpha\in R^{\mathrm{re}}\left(a\right)\setminus R^{\mathrm{re}}\left(\gamma\right)$, we have $\xi_{\alpha}\left(\gamma\right)=-1$, so that $1-\xi_{\alpha}^{-1}\left(\gamma\right)=2$. This completes the proof of our proposition.
\end{proof}

\begin{theorem}\label{thm:extraide}
	The following identities hold:
	\begin{align}\label{eq:cong6} 
&\det\left(1-\Ad\left(\gamma\right)\right)\vert_{\mathfrak 
z^{\perp}\left(a\right)} \nonumber \\
& \qquad \qquad=\left(-1\right)^{\left\vert  R_{+}\setminus 
R_{+}\left(a\right)\right\vert}\prod_{\alpha\in R_{+}\setminus 
R_{+}\left(a\right)}^{}\left(\xi_{\alpha}^{1/2}\left(\gamma\right)-\xi_{\alpha}^{-1/2}\left(\gamma\right)\right)^{2}, \nonumber \\
&\left\vert \det\left(1-\Ad\left(\gamma\right)\right)\vert_{\mathfrak 
z^{\perp}\left(a\right)} 
\right\vert^{1/2}=\epsilon_{D}\left(\gamma\right)\prod_{\alpha\in 
R_{+}\setminus 
R_{+}\left(a\right)}^{}\left(\xi_{\alpha}^{1/2}\left(\gamma\right)-\xi_{\alpha}^{-1/2}\left(\gamma\right)\right)\\
&\qquad\qquad\qquad\prod_{\alpha\in R^{\mathrm{re}}_{+}\setminus 
R^{\mathrm{re}}_{+}\left(a\right)}^{}\xi_{\alpha}^{-1/2}\left(k^{-1}\right), \nonumber \\
&\det\left(1-\Ad\left(k^{-1}\right)\right)\vert_{\mathfrak 
z_{a}^{\perp}\left(\gamma\right)}=\left(-1\right)^{\left\vert  
R_{+}\left(a\right)\setminus 
R_{+}\left(\gamma\right)\right\vert}  \nonumber \\
&\qquad \qquad \prod_{\alpha\in 
R_{+}\left(a\right)\setminus 
R_{+}\left(\gamma\right)}^{}\left(\xi_{\alpha}^{1/2}\left(\gamma\right)-\xi_{\alpha}^{-1/2}\left(\gamma\right)\right)^{2}. \nonumber 
\end{align}
\end{theorem}
\begin{proof}
Using  Theorem \ref{thm:careg} and the third identity in 
(\ref{eq:cong5a6a}), we get
\begin{equation}\label{eq:pep1}
\det\left(1-\Ad\left(\gamma\right)\right)\vert_{\mathfrak 
z^{\perp}\left(a\right)}=\prod_{\alpha\in R_{+}\setminus 
R_{+}\left(a\right)}^{}\left(1-\xi_{\alpha}\left(\gamma\right)\right)
\left(1-\xi_{\alpha}^{-1}\left(\gamma\right)\right),
\end{equation}
from which the first equation in (\ref{eq:cong6}) follows. 

By proceeding as in  Subsection \ref{subsec:posro}, we find that 
$-\theta$ acts on    $R_{+}\setminus 
\left( R_{+}\left(a\right)\cup R_{+}^{\mathrm{re}} \right)$ without 
fixed points, so that $\left\vert R_{+}\setminus 
\left( R_{+}\left(a\right)\cup R_{+}^{\mathrm{re}} \right) 
\right\vert$ is even, and so
\begin{equation}\label{eq:pep2}
\left(-1\right)^{\left\vert  R_{+}\setminus 
R_{+}\left(a\right)\right\vert} 
=\left(-1\right)^{\left\vert  
R_{+}^{\mathrm{re}}\setminus 
R^{\mathrm{re}}_{+}\left(a\right)\right\vert}.
\end{equation}

The same arguments also show that
\begin{equation}\label{eq:pep1a0}
\prod_{\alpha\in R_{+}\setminus 
\left( R_{+}\left(a\right)\cup R_{+}^{\mathrm{re}} \right) }^{}\left(\xi_{\alpha}^{1/2}\left(\gamma\right)-\xi_{\alpha}^{-1/2}\left(\gamma\right)\right)
\end{equation}
is a positive number, and also that
\begin{multline}\label{eq:pep1a-1}
\prod_{\alpha\in R_{+}\setminus 
\left( R_{+}\left(a\right)\cup R_{+}^{\mathrm{re}} 
\right)}^{}\left(1-\xi_{\alpha}\left(\gamma\right)\right)
\left(1-\xi_{\alpha}^{-1}\left(\gamma\right)\right) \\
=\prod_{\alpha\in R_{+}\setminus 
\left( R_{+}\left(a\right)\cup R_{+}^{\mathrm{re}} \right) 
}^{}\left(\xi_{\alpha}^{1/2}\left(\gamma\right)-\xi_{\alpha}^{-1/2}\left(\gamma\right)\right)^{2}.
\end{multline}

Moreover, we have the identity of nonzero real numbers,
\begin{multline}\label{eq:pep3}
\prod_{\alpha\in R_{+}^{\mathrm{re}}\setminus 
R_{+}^{\mathrm{re}}\left(a\right)}^{}\left(1-\xi_{\alpha}\left(\gamma\right)\right)\left(1-\xi_{\alpha}^{-1}\left(\gamma\right)\right)\\
=\left(-1\right)^{\left\vert  R^{\mathrm{re}}_{+}
\setminus 
R^{\mathrm{re}}_{+}\left(a\right)\right\vert}\prod_{\alpha\in  R^{\mathrm{re}}_{+}
\setminus 
R^{\mathrm{re}}_{+}\left(a\right)}^{}\left(1-\xi_{\alpha}^{-1}\left(\gamma\right)\right)^{2}\xi_{\alpha}\left(\gamma\right).
\end{multline}
By Theorem \ref{thm:careg}, if $\alpha\in R_{+}^{\mathrm{re}}$, then
\begin{align}\label{eq:pep3x1}
&\xi_{\alpha}\left(\gamma\right)=e^{\left\langle  
\alpha,a\right\rangle}\xi_{\alpha}\left(k^{-1}\right),
&\xi_{\alpha}\left(k^{-1}\right)=\pm 1.
\end{align}
 Using Proposition 
\ref{prop:si} and (\ref{eq:pep3}), we get
\begin{multline}\label{eq:pep4}
\left\vert  \prod_{\alpha\in R_{+}^{\mathrm{re}}\setminus 
R_{+}^{\mathrm{re}}\left(a\right)}^{}\left(1-\xi_{\alpha}\left(\gamma\right)\right)\left(1-\xi_{\alpha}^{-1}\left(\gamma\right)\right)\right\vert^{1/2} \\
=\epsilon_{D}\left(\gamma\right)\prod_{\alpha\in R^{\mathrm{re}}_{+}\setminus R^{\mathrm{re}}_{+}\left(a\right)}^{}\left(1-\xi_{\alpha}^{-1}\left(\gamma\right)\right)e^{
\left\langle  \alpha,a\right\rangle/2}.
\end{multline}
Equation (\ref{eq:pep4}) can be rewritten in the form
\begin{multline}\label{eq:pep5}
\left\vert  \prod_{\alpha\in R_{+}^{\mathrm{re}}\setminus 
R_{+}^{\mathrm{re}}\left(a\right)}^{}\left(1-\xi_{\alpha}\left(\gamma\right)\right)\left(1-\xi_{\alpha}^{-1}\left(\gamma\right)\right)\right\vert^{1/2}\\
=\epsilon_{D}\left(\gamma\right)\prod_{\alpha\in R^{\mathrm{re}}_{+}\setminus R^{\mathrm{re}}_{+}\left(a\right)}^{}\left(\xi_{\alpha}^{1/2}\left(\gamma\right)-\xi_{\alpha}^{-1/2}\left(\gamma\right)\right)\prod_{\alpha\in R^{\mathrm{re}}_{+}\setminus 
R^{\mathrm{re}}_{+}\left(a\right)}^{}\xi_{\alpha}^{-1/2}\left(k^{-1}\right).
\end{multline}
By (\ref{eq:pep1})--(\ref{eq:pep5}), we get the second identity in 
(\ref{eq:cong6}). 

Since on $\mathfrak z_{a}^{\perp}\left(\gamma\right)$, $\Ad\left(\gamma\right)$ acts 
like $\Ad\left(k^{-1}\right)$, the proof of the third identity in (\ref{eq:cong6}) is 
the same as the proof of the first identity, which completes the 
proof of our theorem.
\end{proof}
\begin{remark}\label{rem:Rsign}
	By the first two equations in (\ref{eq:cong6}), we deduce that
	\begin{equation}\label{eq:sign0}
\mathrm{sgn} 
\det\left(1-\Ad\left(\gamma\right)\right)\vert_{\mathfrak 
z^{\perp}\left(a\right)}=\left(-1\right)^{\left\vert  R_{+}\setminus 
R_{+}\left(a\right)\right\vert}\prod_{\alpha\in 
R_{+}^{\mathrm{re}}\setminus 
R_{+}^{\mathrm{re}}\left(a\right)}^{}\xi_{\alpha}^{-1}\left(k^{-1}\right).
\end{equation}
Also
\begin{equation}\label{eq:sign1a}
\left\vert  R_{+}\setminus R_{+}\left(a\right)\right\vert=\dim 
\mathfrak p^{\perp}\left(a\right).
\end{equation}
Using (\ref{eq:sign1a}), we can rewrite (\ref{eq:sign0}) in the form
	\begin{equation}\label{eq:sign1}
\mathrm{sgn} 
\det\left(1-\Ad\left(\gamma\right)\right)\vert_{\mathfrak 
z^{\perp}\left(a\right)}=\left(-1\right)^{\dim \mathfrak 
p^{\perp}\left(a\right)}\det \Ad\left(k\right)\vert_{
\mathfrak r^{\perp}_{\mathfrak p}\left(a\right)}.
\end{equation}
Using Proposition \ref{prop:peven}, we get
\begin{equation}\label{eq:sign2}
\det \Ad\left(k\right)\vert_{
\mathfrak r^{\perp}_{\mathfrak p}\left(a\right)}=\det \Ad\left(k\right)\vert_{
\mathfrak p^{\perp}\left(a\right)}.
\end{equation}
By (\ref{eq:sign1}), (\ref{eq:sign2}), we get
\begin{equation}\label{eq:sign3}
\mathrm{sgn} 
\det\left(1-\Ad\left(\gamma\right)\right)\vert_{\mathfrak 
z^{\perp}\left(a\right)}=\left(-1\right)^{\dim \mathfrak 
p^{\perp}\left(a\right)}\det \Ad\left(k\right)\vert_{\mathfrak 
p^{\perp}\left(a\right)},
\end{equation}
a result already established in \cite[Proposition 5.4.1]{Bismut08b}.
\end{remark}

Let 
	\index{i@$\mathfrak i^{\perp}$}%
	$\mathfrak i^{\perp}$ be the orthogonal space to $\mathfrak 
	i$ in $\mathfrak h^{\perp}$.
By (\ref{eq:inva18a}), $\mathfrak z\left(\gamma\right) \subset 
\mathfrak z\left(a\right)$, and by (\ref{eq:cong5a3}), $\mathfrak i 
\subset \mathfrak z\left(a\right)$. Therefore,
\begin{equation}\label{eq:pep5a1}
\mathfrak z^{\perp}\left(a\right) \subset \mathfrak 
z^{\perp}\left(\gamma\right)\cap \mathfrak i^{\perp}.
\end{equation}
Similarly, we have the inclusion
\begin{equation}\label{eq:pep5a2}
R_{+}\left(\gamma\right)\cup R^{\mathrm{im}}_{+} \subset 
R_{+}\left(a\right).
\end{equation}
\begin{theorem}\label{thm:extraidebis}
	The following identities hold:
	\begin{align}\label{eq:cong6bis} 
&\det\left(1-\Ad\left(\gamma\right)\right)\vert_{\mathfrak 
z^{\perp}\left(\gamma\right)\cap \mathfrak i^{\perp}}
=\left(-1\right)^{\left\vert  R_{+}\setminus 
\left( R_{+}\left(\gamma\right)\cup R^{\mathrm{im}}_{+}\right) \right\vert} \nonumber \\
&\qquad \qquad \prod_{\alpha\in   R_{+}\setminus 
\left( R_{+}\left(\gamma\right)\cup R^{\mathrm{im}}_{+}\right)  
}^{}\left(\xi_{\alpha}^{1/2}\left(\gamma\right)-\xi_{\alpha}^{-1/2}\left(\gamma\right)\right)^{2},   \nonumber \\
&\left\vert \det\left(1-\Ad\left(\gamma\right)\right)\vert_{\mathfrak 
z^{\perp}\left(\gamma\right)\cap \mathfrak i^{\perp}} 
\right\vert^{1/2} =\epsilon_{D}\left(\gamma\right)\\
& \qquad \prod_{\alpha\in 
R_{+}\setminus 
\left( R_{+}\left(\gamma\right)\cup R^{\mathrm{im}}_{+} \right) 
}^{}\left(\xi_{\alpha}^{1/2}\left(\gamma\right)-\xi_{\alpha}^{-1/2}\left(\gamma\right)\right)
\prod_{\alpha\in R^{\mathrm{re}}_{+}\setminus 
R^{\mathrm{re}}_{+}\left(\gamma\right)}^{}\xi_{\alpha}^{-1/2}\left(k^{-1}\right). \nonumber 
\end{align}
\end{theorem}
\begin{proof}
	The proof of the first identity in (\ref{eq:cong6bis}) is the 
	same as the proof of the first identity in (\ref{eq:cong6}) that 
	was given in Theorem \ref{thm:extraide}. Instead of 
	(\ref{eq:pep2}), we get
	\begin{equation}\label{eq:cong6bisa}
\left(-1\right)^{\left\vert  
R_{+}\setminus\left(R_{+}\left(\gamma\right)\cup 
R^{\mathrm{im}}_{+}\right)\right\vert}=\left(-1\right)^{\left\vert  R^{\mathrm{re}}_{+}\setminus
R^{\mathrm{re}}_{+}\left(\gamma\right)\right\vert}.
\end{equation}
	If in (\ref{eq:pep1a0}), we 
	replace $R_{+}\setminus\left(R_{+}\left(a\right)\cup 
	R_{+}^{\mathrm{re}}\right)$ by 
	$R_{+}\setminus\left(R_{+}\left(\gamma\right)\cup 
	R_{+}^{\mathrm{re}}\cup R^{\mathrm{im}}_{+}\right)$, the 
	conclusions remain valid. Similarly, (\ref{eq:pep3}),
	(\ref{eq:pep4}) remain valid when replacing 
	$R^{\mathrm{re}}_{+}\setminus R^{\mathrm{re}}_{+}\left(a\right)$ 
	by $R^{\mathrm{re}}_{+}\setminus 
	R_{+}^{\mathrm{re}}\left(\gamma\right)$. This completes the proof of 
	our theorem.
\end{proof}
\subsection{Evaluation of the function $\mathcal{J}_{\gamma}$ on 
$i \mathfrak h_{\mathfrak k}$}%
\label{subsec:Kg}
Let $\mathfrak h \subset \mathfrak g$ be a $\theta$-stable Cartan 
subalgebra. Take $\gamma\in H$. Then $\gamma$ is semisimple, and 
$\mathfrak h \subset \mathfrak z\left(\gamma\right)$. In particular, 
$\mathfrak h_{\mathfrak k} \subset \mathfrak k\left(\gamma\right)$, 
so that functions defined on $i \mathfrak 
k\left(\gamma\right)$ restrict to functions on $i \mathfrak 
h_{\mathfrak k}$. 

Recall that the function $\mathcal{L}_{\gamma}\left(\Yok\right), 
\mathcal{M}_{\gamma}\left(\Yok\right)$ on 
$i \mathfrak k\left(\gamma\right)$ were  defined in Definition 
\ref{DLg}. 
\begin{definition}\label{def:Kg}
	If $h_{\mathfrak k}\in i \mathfrak h_{\mathfrak k}$, put
	\index{Lkh@$\mathscr{L}_{k^{-1}}\left(h_{\mathfrak k}\right)$}%
\begin{equation}\label{eq:congo0a1}
\mathscr{L}_{k^{-1}}\left(h_{\mathfrak 
k}\right)=\frac{\det\left(1-\Ad\left(k^{-1}e^{-h_{\mathfrak k}}\right)\right)\vert_{\mathfrak i_{\mathfrak 
	k}^{\perp}\left(k\right)}}{\det\left(1-    
	\Ad\left(k^{-1}e^{-h_{\mathfrak k}}\right)\right)\vert_{\mathfrak  i_{\mathfrak 
	p}^{\perp}\left(k\right)}}.
\end{equation}
\end{definition}
Like the function 
$\mathcal{J}_{\gamma}\left(\Yok\right)$ in (\ref{eq:ret33}),  the function 
$\mathscr{L}_{k^{-1}}\left(h_{\mathfrak k}\right)$ is a smooth 
function of $h_{\mathfrak k}$, which verifies estimates similar to 
(\ref{eq:comm-1}).
Exactly the same arguments as in \cite[Section 5.5]{Bismut08b} and after 
(\ref{eq:congo-1}) show that there is 
an unambiguously defined square root
\index{Mkh@$\mathscr{M}_{k^{-1}}\left(h_{\mathfrak k}\right)$}%
\begin{equation}\label{eq:bint1}
\mathscr{M}_{k^{-1}}\left(h_{\mathfrak k}\right)=\left[\frac{1}{\det\left(1-\Ad\left(k^{-1}\right)\right)\vert 
_{\mathfrak 
i^{\perp}\left(k\right)}}\mathscr{L}_{k^{-1}}\left(h_{\mathfrak 
k}\right)\right]^{1/2}.
\end{equation}
This square root is positive for $h_{\mathfrak k}=0$.
\begin{theorem}\label{thm:TJy}
	If $h_{\mathfrak k}\in i\mathfrak h_{\mathfrak k}$, then 
	\begin{align}\label{eq:congo1}
		&\frac{\widehat{A}\left(\mathrm{ad}\left(h_{\mathfrak 
k}\right)\vert_{\mathfrak 
p\left(\gamma\right)}\right)}{\widehat{A}\left(\mathrm{ad}\left(h_{\mathfrak 
k}\right)\vert_{\mathfrak k\left(\gamma\right)}\right)}=
\frac{\widehat{A}\left(\mathrm{ad}\left(h_{\mathfrak 
k}\right)\vert_{\mathfrak i_{\mathfrak p}\left(k\right)}\right)}{\widehat{A}\left(\mathrm{ad}\left(h_{\mathfrak 
k}\right)\vert_{\mathfrak i_{\mathfrak k}\left(k\right)}\right)},\\
&\mathcal{L}_{\gamma}\left(h_{\mathfrak k}\right)	
=\mathscr{L}_{k^{-1}}\left(h_{\mathfrak k}\right). \nonumber 
\end{align}
In particular $\mathcal{L}_{\gamma}\left(h_{\mathfrak 
k}\right)$ does 
not depend on $a$. 

If $h_{\mathfrak k}\in i \mathfrak  h_{\mathfrak k}$, we have the 
identity,
\begin{equation}\label{eq:idfan1}
\mathcal{J}_{\gamma}\left(h_{\mathfrak k}\right)=\frac{1}{\left\vert  
\det\left(1-\Ad\left(\gamma\right)\right)\vert_{\mathfrak 
z^{\perp}\left(\gamma\right)\cap \mathfrak 
i^{\perp}}\right\vert^{1/2}}\frac{\widehat{A}\left(\mathrm{ad}\left(h_{\mathfrak k}\right)\vert_{ 
\mathfrak i_{\mathfrak 
p}\left(k\right)}\right)}{\widehat{A}\left(\mathrm{ad}\left(h_{\mathfrak 
k}\right)\vert_{\mathfrak i_{\mathfrak k}\left(k\right)}\right)}
\mathscr{M}_{k^{-1}}\left(h_{\mathfrak k}\right).
\end{equation}
This identity can be written in the form,
\begin{multline}\label{eq:idfan2}
\mathcal{J}_{\gamma}\left(h_{\mathfrak 
k}\right) \\
=\frac{\left(-1\right)^{\left\vert  R^{\mathrm{im}}_{ \mathfrak 
p,+}\setminus 
R^{\mathrm{im}}_{\mathfrak 
p,+}\left(k\right)\right\vert}\epsilon_{D}\left(\gamma\right)
\prod_{\alpha\in 
R_{+}^{\mathrm{re}}\setminus 
R_{+}^{\mathrm{re}}\left(\gamma\right)}^{}\xi_{\alpha}^{1/2}\left(k^{-1}\right)}{
\prod_{\alpha\in R_{+}\setminus 
R_{+}\left(\gamma\right)}^{}\left(\xi^{1/2}_{\alpha}\left(\gamma\right)-\xi^{-1/2}_{\alpha}\left(\gamma\right)\right)}
\frac{\prod_{\alpha\in R_{ \mathfrak 
p,+}^{\mathrm{im}}\left(k\right)}^{}\widehat{A}\left(\left\langle  
\alpha,h_{\mathfrak k}\right\rangle\right)}{\prod_{\alpha\in  
R_{\mathfrak k,+}^{\mathrm{im}}\left(k\right)}^{}\widehat{A}\left(\left\langle  
\alpha,h_{\mathfrak k}\right\rangle\right)}\\
\frac{\prod_{\alpha\in R_{\mathfrak k,+}^{\mathrm{im}}\setminus 
	R_{\mathfrak k,+}^{\mathrm{im}}\left(k\right)}^{}\left( 
	\xi_{\alpha}^{1/2}\left(k^{-1}e^{-h_{\mathfrak k}}\right)- \xi^{-1/2}_{\alpha}\left(k^{-1}e^{-h_{\mathfrak k}}\right)\right)}{\prod_{\alpha\in R_{\mathfrak p,+}^{\mathrm{im}}\setminus 
	R_{\mathfrak p,+}^{\mathrm{im}}\left(k\right)}^{}\left( 
	\xi_{\alpha}^{1/2}\left(k^{-1}e^{-h_{\mathfrak k}}\right)- 
	\xi^{-1/2}_{\alpha}\left(k^{-1}e^{-h_{\mathfrak k}}\right)\right)}.
\end{multline}
\end{theorem}
\begin{proof}
	 By (\ref{eq:jou7}),  we get
	\begin{equation}\label{eq:guni1}
\mathfrak i^{\perp}= \mathfrak r \oplus \mathfrak c.
\end{equation}
Also $\mathfrak i^{\perp}$ splits as 
\begin{equation}\label{eq:guni2}
\mathfrak i^{\perp}= \mathfrak i^{\perp}_{\mathfrak p} \oplus \mathfrak 
i^{\perp}_{\mathfrak k}.
\end{equation}
By Proposition \ref{prop:peven}, as 
	representations of $H\cap K$,  $\mathfrak i^{\perp}_{\mathfrak 
	p}$ and $\mathfrak i^{\perp}_{\mathfrak 
	k}$ are equivalent, so that 
(\ref{eq:congo1}) holds. 

Observe that 
$\det\left(1-\Ad\left(k^{-1}\right)\right)\vert _{\mathfrak 
z_{a}^{\perp}\left(\gamma\right)\cap \mathfrak i^{\perp}}>0$, and so 
this number has a positive square root.  Moreover,
\begin{multline}\label{eq:guni3}
\det\left(1-\Ad\left(\gamma\right)\right)\vert_{\mathfrak 
z^{\perp}\left(\gamma\right)\cap \mathfrak i^{\perp}}=
\det\left(1-\Ad\left(\gamma\right)\right)\vert_{\mathfrak 
z^{\perp}\left(a\right)}\\
\det\left(1-\Ad\left(k^{-1}\right)\right)\vert_{\mathfrak 
z_{a}^{\perp}\left(\gamma\right)\cap \mathfrak i^{\perp}}.
\end{multline}
By (\ref{eq:congo0}), (\ref{eq:congo-1}), (\ref{eq:ret33}),  
(\ref{eq:congo1}), and (\ref{eq:guni3}), we get (\ref{eq:idfan1}).

Clearly,
\begin{equation}\label{eq:congo2}
\frac{\widehat{A}\left(\mathrm{ad}\left(h_{\mathfrak k}\right)\vert_{ 
\mathfrak i_{\mathfrak 
p}\left(k\right)}\right)}{\widehat{A}\left(\mathrm{ad}\left(h_{\mathfrak 
k}\right)\vert_{\mathfrak i_{\mathfrak k}\left(k\right)}\right)}=
\frac{\prod_{\alpha\in 
R^{\mathrm{im}}_{\mathfrak p,+}\left(k\right)}^{}\widehat{A}\left(\left\langle  
\alpha,h_{\mathfrak k}\right\rangle\right)}{\prod_{\alpha\in 
R^{\mathrm{im}}_{\mathfrak k,+}\left(k\right)}^{}\widehat{A}\left(\left\langle  
\alpha,h_{\mathfrak k}\right\rangle\right)}.
\end{equation}

By proceeding as in the proof of the third identity in (\ref{eq:cong6}), 
we get
\begin{multline}\label{eq:congo3}
\det\left(1-\Ad\left(k^{-1}\right) \right) \vert_{\mathfrak 
i^{\perp}\left(k\right)} \\
=\left(-1\right)^{\left\vert  R^{\mathrm{im}}_{+}\setminus
R_{+}^{\mathrm{im}}\left(k\right)\right\vert} 
\prod_{\alpha\in 
R^{\mathrm{im}}_{+}\setminus 
R^{\mathrm{im}}_{+}\left(k\right)}^{}\left(\xi_{\alpha}^{1/2}\left(k^{-1}\right)-\xi_{\alpha}^{-1/2}\left(k^{-1}\right)\right)^{2}.
\end{multline}
The same argument shows that
\begin{multline}\label{eq:congo4}
\mathscr{L}_{k^{-1}}\left(h_{\mathfrak 
k}\right)=\left(-1\right)^{\left\vert  R^{\mathrm{im}}_{+}\setminus
R_{+}^{\mathrm{im}}\left(k\right)\right\vert}\\
\frac{\prod_{\alpha\in R_{\mathfrak k,+}^{\mathrm{im}}\setminus 
	R_{ \mathfrak 
	k,+}^{\mathrm{im}}\left(k\right)}^{}\left(\xi_{\alpha}^{1/2}\left(k^{-1}e^{-h_{\mathfrak k}}\right)- \xi^{-1/2}_{\alpha}\left(k^{-1}e^{-h_{\mathfrak k}}\right)\right)^{2}}{
	\prod_{\alpha\in R_{\mathfrak p,+}^{\mathrm{im}}\setminus 
	R_{\mathfrak p,+}^{\mathrm{im}}\left(k\right)}^{}\left( 
	\xi_{\alpha}^{1/2}\left(k^{-1}e^{-h_{\mathfrak 
	k}}\right)-\xi^{-1/2}_{\alpha}\left(k^{-1}e^{-h_{\mathfrak k}}\right)\right)^{2}}.
\end{multline}

By (\ref{eq:congo3}), (\ref{eq:congo4}), and keeping in mind the fact 
that we take the properly positive square root in (\ref{eq:bint1}), 
we get
\begin{multline}\label{eq:congo5}
\mathscr{M}_{k^{-1}}\left(h_{\mathfrak 
k}\right)=
\frac{\left(-1\right)^{\left\vert  R^{\mathrm{im}}_{ \mathfrak 
p,+}/
R^{\mathrm{im}}_{\mathfrak p,+}\left(k\right)\right\vert}}{\prod_{\alpha\in 
R^{\mathrm{im}}_{+}\setminus 
R^{\mathrm{im}}_{+}\left(k\right)}^{}\left(\xi_{\alpha}^{1/2}\left(k^{-1}\right)-\xi_{\alpha}^{-1/2}\left(k^{-1}\right)\right)}\\
\frac{\prod_{\alpha\in R_{\mathfrak k,+}^{\mathrm{im}}\setminus 
	R_{\mathfrak k,+}^{\mathrm{im}}\left(k\right)}^{}\left( 
	\xi_{\alpha}^{1/2}\left(k^{-1}e^{-h_{\mathfrak k}}\right)- 
	\xi^{-1/2}_{\alpha}\left(k^{-1}e^{-h_{\mathfrak k}}\right)\right)}{
	\prod_{\alpha\in R_{\mathfrak p,+}^{\mathrm{im}}\setminus 
	R_{\mathfrak p,+}^{\mathrm{im}}\left(k\right)}^{}\left( 
	\xi_{\alpha}^{1/2}\left(k^{-1}e^{-h_{\mathfrak k}}\right)- 
	\xi^{-1/2}_{\alpha}\left(k^{-1}e^{-h_{\mathfrak 
	k}}\right)\right)}.
\end{multline}
In the first product in the right-hand side of (\ref{eq:congo5}), we 
may as well replace $k^{-1}$ by $\gamma$. 

By combining the second identity in (\ref{eq:cong6bis}), 
(\ref{eq:idfan1}), (\ref{eq:congo2}), and (\ref{eq:congo5}), we get
(\ref{eq:idfan2}).
The proof of our theorem is completed. 
\end{proof}
\section{The function $\mathcal{J}_{\gamma}$ when $\gamma$ is regular}%
\label{sec:reg}
The purpose of this Section is to study extra properties of the 
function $\mathcal{J}_{\gamma}$ when $\gamma$ is regular.

This section is organized as follows. In Subsection 
\ref{subsec:nereg}, if $\mathfrak h$ is a $\theta$-stable Cartan 
subalgebra and if $H$ is the corresponding  Cartan subgroup, if 
$\gamma\in H$, we describe a neighborhood of $\gamma$ in $H$.

In Subsection \ref{subsec:greg}, if $\gamma\in H$, we define the 
$\gamma$-regular elements in $\mathfrak h$, which are such that a 
small perturbation of $\gamma$ by a $\gamma$-regular element  is regular.

In Subsection \ref{subsec:DHg},  following Harish-Chandra 
\cite{Harish65}, we introduce the function 
$D_{H}$ on $H$. This function is  an analogue of the 
denominator in the Lefschetz formulas. 

Finally, in Subsection \ref{subsec:Jgreg}, we specialize the formula obtained 
in Theorem \ref{thm:TJy} for $\mathcal{J}_{\gamma}\left(h_{\mathfrak k}\right)$ to the case where $\gamma\in 
H^{\mathrm{reg}}$. As a consequence,  we  prove  the unexpected 
result that the function 
$\left(\gamma,h_{\mathfrak k}\right)\in H^{\mathrm{reg}}\times i 
\mathfrak h_{\mathfrak k}\to \mathcal{J}_{\gamma}\left(h_{\mathfrak 
k}\right)\in \C$ is smooth. 

We make the same assumptions and we use the  same notation as in  Section  
\ref{sec:rofu}.
\subsection{A neighborhood of $\gamma$ in $H$}%
\label{subsec:nereg}
If $b\in \mathfrak h$, $b$ splits as 
\begin{align}\label{eq:bon1}
&b= b_{\mathfrak p} + b_{\mathfrak k}, &b_{\mathfrak p}\in 
\mathfrak h_{\mathfrak p}, \qquad b_{\mathfrak 
k}\in \mathfrak h_{\mathfrak k}.
\end{align}
Put
\begin{equation}\label{eq:bon2}
\gamma'=\gamma e^{b}.
\end{equation}
Then $\gamma'\in  H\cap Z\left(\gamma\right)$. 

Set
\begin{align}\label{eq:bon3}
&a'=a+ b_{\mathfrak p}, &k'=ke^{-b_{\mathfrak k}}.
\end{align}
Then
\begin{align}\label{eq:bon3a1}
&\gamma'=e^{a'}k^{\prime -1}, &a'\in \mathfrak h_{\mathfrak p}, 
\qquad k'\in H\cap K\left(\gamma\right),\qquad \Ad\left(k'\right)a'=a'.
\end{align}
Also $\Ad\left(\gamma'\right)$ 
preserves the splitting $\mathfrak g=\mathfrak z\left(\gamma\right) 
\oplus \mathfrak z^{\perp}\left(\gamma\right)$. Since 
$1-\Ad\left(\gamma\right)$ is invertible on $\mathfrak 
z^{\perp}\left(\gamma\right)$, 
 we conclude that  for $\epsilon>0$ small 
enough, if $\left\vert  b\right\vert\le \epsilon$, 
\begin{equation}\label{eq:bon3a7}
\mathfrak h \subset \mathfrak z\left(\gamma'\right) \subset \mathfrak z\left(\gamma\right).
\end{equation}

Let 
\index{Hreg@$H^{\reg}$}%
$H^{\reg}$ be the set of regular elements in $H$. 
Assume temporarily that $\gamma\in H^{\mathrm{reg}}$,  i.e., $\mathfrak 
z\left(\gamma\right)= \mathfrak h$.  By (\ref{eq:bon3a7}),  for 
$\epsilon>0$ small enough if $\left\vert  b\right\vert\le \epsilon$, 
then
\begin{equation}\label{eq:bon3a9}
\mathfrak z\left(\gamma'\right)= \mathfrak h,
\end{equation}
i.e., $\gamma'\in H^{\mathrm{reg}}$, which is a 
trivial conclusion. Since $b\in \mathfrak h$, we conclude that $\gamma'\in 
Z\left(\gamma\right),\gamma\in Z\left(\gamma'\right)$. A priori, 
$Z\left(\gamma\right)$ and $Z\left(\gamma'\right)$ may be distinct.
Still we have the obvious identity 
\begin{equation}\label{eq:coc0}
Z^{0}\left(\gamma\right)=Z^{0}\left(\gamma'\right)=H^{0}.
\end{equation}
\subsection{The $\gamma$-regular elements in $\mathfrak h$}%
\label{subsec:greg}
We no longer assume  $\gamma$ to be regular. By 
(\ref{eq:cong5a6a}), we get 
\begin{equation}\label{eq:bon3a12}
\mathfrak 
z_{a}^{\perp}\left(\gamma\right)_{\C}= \oplus_{\alpha\in 
R\left(a\right)\setminus R\left(\gamma\right) }\mathfrak g_{\alpha}.
\end{equation}
Let 
\index{ha@$\mathfrak h^{\perp}_{a}$}%
$\mathfrak h^{\perp}_{a}$ denote the orthogonal space 
to $\mathfrak h$ in $\mathfrak z\left(a\right)$. Then we have the 
splitting
\begin{equation}\label{eq:bon3a12x1}
\mathfrak h_{a}^{\perp}=\mathfrak h_{a, \mathfrak p}^{\perp} \oplus 
\mathfrak h_{a, \mathfrak k}^{\perp}.
\end{equation}
By 
(\ref{eq:cong5a6a}), we get 
\begin{equation}\label{eq:bon3a12a1}
\mathfrak h_{a,\C}^{\perp}=\oplus_{\alpha\in R\left(a\right)}\mathfrak g_{\alpha}.
\end{equation}
\begin{definition}\label{def:elreg}
An element $h\in \mathfrak h$ is said to be $\gamma$-regular if for 
$\alpha\in R\left(\gamma\right)$, $\left\langle  \alpha,h\right\rangle\neq 0$. 	
\end{definition}

 The $\gamma$-regular elements in $\mathfrak h$ are exactly the 
 regular elements in $\mathfrak h$  viewed as a Cartan 
 subalgebra of $\mathfrak z\left(\gamma\right)$. The $\gamma$-regular elements lie in the complement of a finite family of 
hyperplanes in $\mathfrak h$.

Since $\mathfrak h$ is a Cartan subalgebra of $\mathfrak 
z\left(\gamma\right)$, we define the function 
\index{phz@$\pi^{\mathfrak h, \mathfrak z\left(\gamma\right)}$}%
$\pi^{\mathfrak h, \mathfrak z\left(\gamma\right)}$ 
on $\mathfrak h_{\C}$ as in (\ref{eq:inva4}), i.e.,
\begin{equation}\label{eq:inva4bis}
\pi^{\mathfrak h,\mathfrak 
z\left(\gamma\right)}\left(h\right)=\prod_{\alpha\in 
R_{+}\left(\gamma\right)}^{}\left\langle  \alpha,h\right\rangle.
\end{equation}
Then $h\in \mathfrak h$ is $\gamma$-regular if and only if 
$\pi^{\mathfrak h,\mathfrak z\left(\gamma\right)}\left(h\right)\neq 
0$.

Now we use the notation of Subsection \ref{subsec:nereg}.
\begin{proposition}\label{prop:greg}
	There
exists $\epsilon>0$ such that  if $b\in 
\mathfrak h$ is $\gamma$-regular, and $\left\vert  b\right\vert\le \epsilon$,
if $\gamma'=\gamma e^{b}$, then $\gamma'\in H^{\mathrm{reg}}$. 
\end{proposition}
\begin{proof}
	For $\epsilon>0$ small enough, (\ref{eq:bon3a7}) holds, so that
	\begin{equation}\label{eq:gurf1}
\mathfrak z\left(\gamma'\right)= \mathfrak z\left(\gamma\right)\cap 
\mathfrak z\left(e^{b}\right).
\end{equation}
By (\ref{eq:coci1}), we get
\begin{equation}\label{eq:gurf2}
\mathfrak z\left(\gamma\right)_{\C}= \mathfrak h_{\C} \bigoplus \oplus 
_{\alpha\in R\left(\gamma\right)}\mathfrak g_{\alpha}.
\end{equation}
For $\alpha\in R\left(\gamma\right)$, $e^{b}$ acts on $\mathfrak 
g_{\alpha}$ by multiplication by $e^{\left\langle  
\alpha,b\right\rangle}$. For $\epsilon>0$ small enough, if $b$ is 
$\gamma$-regular and $\left\vert  b\right\vert\le \epsilon$, for 
$\alpha\in R\left(\gamma\right)$, $e^{\left\langle  
\alpha,b\right\rangle}\neq 1$. By (\ref{eq:gurf1}), (\ref{eq:gurf2}), 
we conclude that under the given conditions on $b$, $\mathfrak 
z\left(\gamma'\right)= \mathfrak h$, i.e., $\gamma'$ is regular. The proof of our proposition is completed. 
\end{proof}
\subsection{The function $D_{H}\left(\gamma\right)$}%
\label{subsec:DHg}
Here, we follow Harish-Chandra \cite[Section 
19]{Harish65}.
\begin{definition}\label{def:DH}
	If $\gamma\in H^{\mathrm{reg}}$, put\footnote{In \cite[Section 
	18]{Harish65},  Harish-Chandra 
	assumes  $G$ to be acceptable,  i.e., $\rho^{\mathfrak g} $ is 
	assumed to be a weight, so that 
$D_{H}\left(\gamma\right)$ can be globally defined. Here, we only 
need a local definition of $D_{H}\left(\gamma\right)$, and we do not 
need this assumption.}
	\index{DHg@$D_{H}\left(\gamma\right)$}%
	\begin{equation}\label{eq:dHg}
D_{H}\left(\gamma\right)=\prod_{\alpha\in 
R_{+}}^{}\left(\xi_{\alpha}^{1/2}\left(\gamma\right)-\xi_{\alpha}^{-1/2}\left(\gamma\right)\right).
\end{equation}
\end{definition}

Using (\ref{eq:comm2a1}) and proceeding as in the proof of Theorem \ref{thm:extraide}, we get
\begin{equation}\label{eq:congo13}
\det\left(1-\Ad\left(\gamma\right)\right)\vert_{\mathfrak h^{\perp}}=
\left(-1\right)^{\left\vert  
R_{+}\right\vert}D_{H}^{2}\left(\gamma\right).
\end{equation}
By (\ref{eq:congo13}), we deduce that if $\gamma\in H^{\mathrm{reg}}$, 
then
\begin{equation}\label{eq:congo14}
\left\vert  \det\left(1-\Ad\left(\gamma\right)\right)\vert_{\mathfrak 
h^{\perp}}\right\vert=\left\vert  
D_{H}\left(\gamma\right)\right\vert^{2},
\end{equation}
so that $D_{H}\left(\gamma\right)\neq 0$.
\subsection{The function $\mathcal{J}_{\gamma}$ when $\gamma$ is 
regular}%
\label{subsec:Jgreg}
In this subsection,  we assume that $\gamma\in H^{\mathrm{reg}}$, i.e., 
$D_{H}\left(\gamma\right)\neq 0$.  

By (\ref{eq:congo0a1}), (\ref{eq:bint1}), we get
\begin{align}\label{eq:congo0a1beb}
&\mathscr{L}_{k^{-1}}\left(h_{\mathfrak k}\right)=\frac{\det\left(1-
    \Ad\left(k^{-1}e^{-h_{\mathfrak k}}\right)\right)\vert_{\mathfrak i_{\mathfrak 
	k}}}{\det\left(1-
    \Ad\left(k^{-1}e^{-h_{\mathfrak k}}\right)\right)\vert_{\mathfrak  i_{\mathfrak 
	p}}},\\
	&\mathscr{M}_{k^{-1}}\left(h_{\mathfrak 
	k}\right)=\left[\frac{1}{\det\left(1-\Ad\left(k^{-1}\right)\right)\vert _{\mathfrak i}}\mathscr{L}_{k^{-1}}
	\left(h_{\mathfrak k}\right)\right]^{1/2}. \nonumber 
\end{align}

By (\ref{eq:grai1}), $\epsilon_{D}\left(\gamma\right)$  is given 
by
\begin{equation}\label{eq:congo12}
\epsilon_{D}\left(\gamma\right)=\mathrm{sgn}\prod_{\alpha\in 
R_{+}^{\mathrm{re}}}^{}\left(1-\xi_{\alpha}^{-1}\left(\gamma\right)\right).
\end{equation}
The function $\epsilon_{D}\left(\gamma\right)$ is locally constant on 
$H^{\mathrm{reg}}$.
\begin{theorem}\label{thm:caregbi}
If $h_{\mathfrak k}\in i \mathfrak  h_{\mathfrak k}$, we have the 
identity,
\begin{equation}\label{eq:idfan1beb}
\mathcal{J}_{\gamma}\left(h_{\mathfrak k}\right)=\frac{1}{\left\vert  
\det\left(1-\Ad\left(\gamma\right)\right)\vert_{\mathfrak 
i^{\perp}}\right\vert^{1/2}}
\mathscr{M}_{k^{-1}}\left(h_{\mathfrak k}\right).
\end{equation}
This identity can be written in the form,
\begin{multline}\label{eq:idfan2beb}
\mathcal{J}_{\gamma}\left(h_{\mathfrak 
k}\right)=\frac{\left(-1\right)^{\left\vert  
R^{\mathrm{im}}_{\mathfrak p,+} 
\right\vert}\epsilon_{D}\left(\gamma\right)\prod_{\alpha\in 
R_{+}^{\mathrm{re}}}^{}\xi_{\alpha}^{1/2}\left(k^{-1}\right)}{
D_{H}\left(\gamma\right)}
\\
\frac{\prod_{\alpha\in R_{\mathfrak k,+}^{\mathrm{im}}
	}^{}\left( \xi_{\alpha}^{1/2}\left(k^{-1}e^{-h_{\mathfrak 
	k}}\right)- \xi^{-1/2}_{\alpha}\left(k^{-1}e^{-h_{\mathfrak 
	k}}\right)\right)}{\prod_{\alpha\in R_{\mathfrak 
	p,+}^{\mathrm{im}}}^{}\left( 
	\xi_{\alpha}^{1/2}\left(k^{-1}e^{-h_{\mathfrak k}}\right)- 
	\xi^{-1/2}_{\alpha}\left(k^{-1}e^{-h_{\mathfrak k}}\right)\right)}.
\end{multline}
The function $\left(\gamma,h_{\mathfrak k}\right)\in 
H^{\mathrm{reg}}\times i \mathfrak h_{\mathfrak k}\to 
\mathcal{J}_{\gamma}\left(h_{\mathfrak k}\right)\in \C$ is smooth.
\end{theorem}
\begin{proof}
	The first part of our theorem is a trivial consequence of Theorem 
	\ref{thm:TJy}. For $b_{\mathfrak k}\in \mathfrak h_{\mathfrak 
	k}$, for $\left\vert  b_{\mathfrak k}\right\vert$ small enough, 
	we take 
	\begin{equation}\label{eq:chif-1}
\xi_{\alpha}^{1/2}\left(k^{\prime -1}\right)=e^{\left\langle  
	\alpha,b_{\mathfrak 
	k}/2\right\rangle}\xi^{1/2}_{\alpha}\left(k^{-1}\right).
\end{equation}
By 
	(\ref{eq:gru2z1}), (\ref{eq:chif-1}),  we deduce that if $\alpha\in 
	R_{+}^{\mathrm{c}}$, then
	\begin{equation}\label{eq:chif-2}
\xi_{-\theta\alpha}^{1/2}\left(k^{\prime 
-1}\right)=\overline{\xi_{\alpha}^{1/2}\left(k^{\prime -1}\right)}.
\end{equation}
The stated smoothness is an obvious consequence of 
	the above formulas. The proof of our theorem is completed. 
\end{proof}
\section{The Harish-Chandra isomorphism}%
\label{sec:haricha}
In this section, if $\mathfrak h \subset \mathfrak g$ is a Cartan 
subalgebra, we describe the Harish-Chandra isomorphism of algebras
\index{fHC@$\phi_{\mathrm{HC }}$}%
$\phi_{\mathrm{HC}}:Z\left(\mathfrak g\right) \simeq I\ac\left(\mathfrak h,\mathfrak 
g\right)$. Also we explain the action of $Z\left(\mathfrak g\right)$ 
on $C^{\infty }\left(X,F\right)$,  and we introduce certain  semisimple orbital integrals in which 
$Z\left(\mathfrak g\right)$  appears.

This section is organized as follows. In Subsection 
\ref{subsec:cenenv}, we introduce the center of the enveloping 
algebra $Z\left(\mathfrak g\right)$. 

In Subsection \ref{subsec:cofo}, we recall some properties of the 
complex
Harish-Chandra isomorphism $\phi_{\mathrm{HC}}: Z\left(\mathfrak 
g_{\C}\right) \simeq  I\ac\left(\mathfrak h_{\C}, \mathfrak 
g_{\C}\right)$, including some aspects of its construction.

In Subsection \ref{subsec:refo}, we show that there is a real form 
of the Harish-Chandra isomorphism 
$\phi_{\mathrm{HC}}:Z\left(\mathfrak h\right) \simeq 
I\ac\left(\mathfrak h,\mathfrak g\right)$.

In Subsection \ref{subsec:duflo}, we recall the relation of the 
Harish-Chandra isomorphism to the Duflo isomorphism that was 
established in \cite{Duflo70}.

In Subsection \ref{subsec:cacasi}, we consider the case of the 
Casimir.

In Subsection \ref{subsec:actF}, we describe the action of 
$Z\left(\mathfrak g\right)$ on $C^{ \infty }\left(X,F\right)$.

Finally, in Subsection \ref{subsec:sesiorb}, we consider the orbital 
integrals in which $Z\left(\mathfrak g\right)$ appears.
\subsection{The center of the enveloping algebra}%
\label{subsec:cenenv}
Recall that the enveloping algebra 
\index{Ug@$U\left(\mathfrak g\right)$}%
$U\left(\mathfrak g\right)$ was 
introduced in Subsection \ref{subsec:cas}.
Then $U\left(\mathfrak g\right)$ is a 
filtered algebra, and the corresponding $\mathrm{Gr}$ is just the 
algebra of polynomials $S\ac\left(\mathfrak g\right)$ on $\mathfrak 
g^{*}$.  

Note that $\mathfrak g$ acts by derivations on  
$U\left(\mathfrak g\right)$. 
Recall that  
\index{Zg@$Z\left(\mathfrak g\right)$}%
$Z\left(\mathfrak g\right)$ is the center of 
$U\left(\mathfrak g\right)$, i.e., it is the kernel of the above 
derivations.

Observe that $G$ acts both on the left and on the right on $C^{\infty 
}\left(G,\R\right)$ by the formula
\begin{align}\label{eq:us0}
&\gamma_{L} s\left(g\right)=s\left(\gamma^{-1}g\right),
&\gamma_{R}s\left(g\right)=s\left(g\gamma\right),
\end{align}
 and these two actions commute. They are intertwined by 
 the involution induced by the involution 
 \index{s@$\sigma$}%
 $g\to \sigma g=g^{-1}$. 
 Let 
 \index{DLG@$D_{L}\left(G\right)$}%
 $D_{L}\left(G\right)$ be 
the Lie algebras of  left-invariant real differential operators on 
$G$. As we saw in Subsection \ref{subsec:cas}, $U\left(\mathfrak 
g\right)$ can be identified with $D_{L}\left(G\right)$. The algebra 
$D_{L}\left(G\right)$ commutes with the left action of $G$.

If 
\index{g@$\mathfrak g_{-}$}%
$\mathfrak g_{-}$ is the Lie algebra $\mathfrak g$ with the negative of the 
original Lie bracket, the  isomorphism of $\mathfrak g$ $f\to 
-f$ identifies $\mathfrak g$ and $\mathfrak g_{-}$. This isomorphism 
is induced by the involution $\sigma$.

Let 
\index{Ug@$U\left(\mathfrak g_{-}\right)$}%
$U\left(\mathfrak g_{-}\right)$ be the enveloping algebra 
associated with $\mathfrak g_{-}$. Then $U\left(\mathfrak 
g_{-}\right)$ can be identified with the algebra of right-invariant 
real
differential operators 
\index{DRG@$D_{R}\left(G\right)$}%
$D_{R}\left(G\right)$. This algebra commutes 
with the right action of $G$. Also the isomorphism $f\to -f$ induces an 
identification of $U\left(\mathfrak g\right)$ and $U\left(\mathfrak 
g_{-}\right)$. This identification is still induced by $\sigma$.

We  equip $U\left(\mathfrak g\right), U\left(\mathfrak 
g_{-}\right)$ with the antiautomorphism 
\index{st@$*$}%
$*$ which is just the adjoint 
in the classical $L_{2}$ sense when identifying $U\left(\mathfrak 
g\right), U\left(\mathfrak g_{-}\right)$ with $D_{L}\left(G\right), 
D_{R}\left(G\right)$. This involution extends to a $\C$-linear 
involution of $U\left(\mathfrak g_{\C}\right), U\left(\mathfrak 
g_{-,\C}\right)$.

By definition,   $Z\left(\mathfrak g\right) \subset U\left(\mathfrak 
g\right)$ is the subalgebra of $D_{L}\left(G\right)$ which commutes with 
right multiplication. Equivalently
\begin{equation}\label{eq:us1}
Z\left(\mathfrak g\right)= D_{L}\left(G\right) \cap 
D_{R}\left(G\right).
\end{equation}
Note that $*$ induces an automorphism of $Z\left(\mathfrak g\right)$, 
which is an involution, and which we still denote $*$.

The isomorphism of $U\left(\mathfrak g\right)$ with $U\left(\mathfrak 
g_{-}\right)$   which was described before is the one induced by 
$\sigma$. It  induces the obvious 
isomorphism of $D_{L}\left(G\right)$ with  $D_{R}\left(G\right)$. 
This way, we obtain an automorphism  $\sigma$ of $
Z\left(\mathfrak g\right)$, which is also an involution.

Clearly, 
\begin{equation}\label{eq:clo-7}
Z\left(\mathfrak g_{\C}\right)=Z\left(\mathfrak g\right)_{\C}. 
\end{equation}
Equivalently, $Z\left(\mathfrak g_{\C}\right)$ is equipped with a 
complex conjugation, and $Z\left(\mathfrak g\right)$ is the algebra 
of complex conjugation invariants in $Z\left(\mathfrak 
g_{\C}\right)$. Also $*$ and $\sigma$ extend to complex automorphisms 
of $Z\left(\mathfrak g_{\C}\right)$. 
\subsection{The complex form of the Harish-Chandra isomorphism}%
\label{subsec:cofo}
Let $\mathfrak h \subset \mathfrak g$ be a $\theta$-stable Cartan  
subalgebra. By \cite[Theorem 8.18]{Knapp86},  there is  a canonical
Harish-Chandra isomorphism  of filtered algebras,
\begin{equation}\label{eq:inva7}
\phi_{\mathrm{HC}}: Z\left(\mathfrak g_{\C}\right)  \simeq  I\ac\left(\mathfrak 
h_{\C}, \mathfrak g_{\C}\right).
\end{equation}

We need to  describe the 
Harish-Chandra isomorphism in more detail.  We fix a positive root 
	$R_{+}$ as in Subsection \ref{subsec:posro}. Put
	\begin{equation}\label{eq:clo-9}
\mathscr P=\sum_{\alpha\in R_{+}}^{}U\left(\mathfrak 
g_{\C}\right)\mathfrak g_{\alpha}.
\end{equation}
Observe that $S\ac\left(\mathfrak h_{\C}\right) =
U\left(\mathfrak h_{\C}\right)$, and also that $U\left(\mathfrak 
h_{\C}\right) \subset U\left(\mathfrak g_{\C}\right)$, so that 
$S\ac\left(\mathfrak h_{\C}\right) \subset U\left(\mathfrak 
g_{\C}\right)$. By \cite[Lemma 8.17]{Knapp86}, we get
\begin{align}\label{eq:clo-10}
&S\ac\left(\mathfrak h_{\C} \right) \cap \mathscr P =0,
&Z\left(\mathfrak g_{\C}\right) \subset  S\ac\left(\mathfrak 
h_{\C}\right) \oplus  \mathscr P.
\end{align}
Let $\phi_{1,R_{+}}$ be the projection from $Z\left(\mathfrak 
g_{\C}\right)$ on $S\ac\left(\mathfrak h_{\C}\right)$. 

Recall that $S\ac\left(\mathfrak h_{\C}\right)$ is the algebra of 
polynomials on $\mathfrak h^{*}_{\C}$, and that 
\index{rg@$\rho^{\mathfrak g}$}%
$\rho^{\mathfrak 
g}\in \mathfrak h_{\C}^{*}$ is the half sum of the roots in $R_{+}$.  
Let $\phi_{2,R_{+}}$ be the filtered
automorphism of $S\ac\left(\mathfrak h_{\C}\right)$ that is such that 
if $f\in S\ac\left(\mathfrak h_{\C}\right)$, if $h^{*}\in 
\mathfrak h_{\C}^{*}$,  then
\begin{equation}\label{eq:clo-11}
\phi_{2,R_{+}}f\left(h^{*}\right)=f\left(h^{*}-\rho^{\mathfrak 
g}\right).
\end{equation}

The fundamental result of Harish-Chandra \cite[Lemmas 
18-20]{Harish56},  \cite[Theorem 
8.18]{Knapp86} is that 
$\phi_{2,R_{+}}\phi_{1,R_{+}}$ maps $Z\left(\mathfrak g_{\C}\right)$  onto $I\ac\left(\mathfrak h_{\C}, 
\mathfrak g_{\C}\right)$, that it induces an isomorphism of filtered algebras 
 that does not depend on the choice of 
$R_{+}$. This is exactly the Harish-Chandra isomorphism 
$\phi_{\mathrm{HC}}: Z\left(\mathfrak g_{\C}\right) \simeq 
I\ac\left(\mathfrak h_{\C}, \mathfrak g_{\C}\right)$.

Now we  proceed as in  Subsection \ref{subsec:care}, i.e.,  we identify $I\ac\left(\mathfrak 
h_{\C}, \mathfrak g_{\C}\right)$ with the algebra 
\index{DIhg@$D_{I}\ac\left(\mathfrak h_{\C}, \mathfrak g_{\C}\right)$}%
$D_{I}\ac\left(\mathfrak h_{\C}, \mathfrak g_{\C}\right)$ of   holomorphic differential 
operators on $\mathfrak h_{\C}$  with constant complex coefficients which 
are $W\left(\mathfrak h_{\C}:\mathfrak g_{\C}\right)$-invariant. The 
same arguments as in Subsection \ref{subsec:care} show that there is 
an algebra 
\index{DIhg@$D_{I}\ac\left(\mathfrak h,\mathfrak g\right)$}%
$D_{I}\ac\left(\mathfrak h, \mathfrak g\right)$ of real 
differential operators with constant coefficients on $\mathfrak h$ 
such that
\begin{equation}\label{eq:clo-5}
D\ac_{I}\left(\mathfrak h_{\C}, \mathfrak 
g_{\C}\right)=D\ac_{I}\left(\mathfrak h, \mathfrak g\right)_{\C},
\end{equation}
and that $I\ac\left(\mathfrak h, \mathfrak g\right)$ can be 
identified with $D\ac_{I}\left(\mathfrak  h, \mathfrak g\right)$.

We will now use the assumptions and notation of Subsection 
\ref{subsec:casg}. 
Let $C^{\infty, G}\left(G^{\mathrm{reg}},\C\right)$ denote the $\Ad$-invariant 
smooth complex functions on the open set $G^{\mathrm{reg}}$. 

Let $C^{\infty ,
W\left(H:G\right)}\left(H^{\mathrm{reg}},\C\right)$ be the smooth 
$W\left(H:G\right)$-invariant functions on $H^{\mathrm{reg}}$. There 
is a restriction map 
$$r:C^{\infty, 
G}\left(G^{\mathrm{reg}},\C\right)\to C^{\infty, 
W\left(H:G\right)}\left(H^{\mathrm{reg}},\C\right).$$ 

Observe that $ Z\left(\mathfrak g_{\C}\right)$ acts on 
$C^{\infty, G}\left(G^{\mathrm{reg}},\C\right)$, and 
$I\ac\left(\mathfrak h_{\C}, \mathfrak g_{\C}\right)$ acts on 
$C^{\infty, 
W\left(H:G\right)}\left(H^{\mathrm{reg}},\C\right)$.

Let $L\in Z\left(\mathfrak g_{\C} \right)$. By \cite[Lemma 
13]{Harish65}, \cite[Theorem 10.33]{Knapp86},  if $f\in C^{\infty, 
G}\left(G^{\mathrm{reg}},\C\right)$, on $H^{\mathrm{reg}}$, we have 
the identity
\begin{equation}\label{eq:inva17}
rLf=\frac{1}{D_{H}}\left(\phi_{\mathrm{HC}} L\right)D_{H}rf.
\end{equation}
\subsection{The real form of the Harish-Chandra isomorphism}%
\label{subsec:refo}
The involution $h\to -h$ induces an involution of 
$I\ac\left(\mathfrak h_{\C},\mathfrak g_{\C}\right) \simeq 
D_{I}\ac\left(\mathfrak h_{\C}, 
\mathfrak g_{\C}\right)$.  If $N$ counts the degree in 
$I\ac\left(\mathfrak h_{\C}, \mathfrak 
g_{\C}\right)$, this involution is just 
$\left(-1\right)^{N}$. We 
still denote this involution by $*$.

In Proposition \ref{prop:conja}, we proved that $I\ac\left(\mathfrak 
h_{\C}, \mathfrak g_{\C}\right)$ is preserved by complex conjugation. 
At the end of Subsection \ref{subsec:cenenv}, 
we proved that $Z\left(\mathfrak g_{\C}\right)$ is also
preserved by complex conjugation. Observe that $\theta$ acts on 
$Z\left(\mathfrak g_{\C}\right), I\ac\left(\mathfrak h_{\C}, 
\mathfrak g_{\C}\right)$ and preserves $Z\left(\mathfrak g\right), 
I\ac\left(\mathfrak h,\mathfrak g\right)$.
\begin{theorem}\label{thm:pinv}
	If $L\in Z\left(\mathfrak g_{\C }\right)$, then 
	\begin{align}\label{eq:star1}
&\phi_{\mathrm{HC}} \left( L^{*} \right) =\left(\phi_{\mathrm{HC}} 
L\right)^{*},
&\phi_{\mathrm{HC}} \left( \overline{L} \right) 
=\overline{\phi_{\mathrm{HC}}\left( L \right)},\qquad 
\phi_{\mathrm{HC}}\theta L=\theta\phi_{\mathrm{HC}}L.
\end{align}
On $Z\left(\mathfrak g_{\C}\right)$, the involutions 
	$\sigma$ and $*$ coincide. Finally, $\phi_{\mathrm{HC}}$ induces an isomorphism 
	of real filtered algebras: 
	\begin{equation}\label{eq:clo-8}
Z\left(\mathfrak g\right) \simeq I\ac\left(\mathfrak 
	h, \mathfrak g\right).
\end{equation} 
\end{theorem}
\begin{proof}
	The first equation in (\ref{eq:star1}) was established by Harish-Chandra  \cite[Lemma 
	20]{Harish56}. For the proof of the next two equations, we will 
	 follow  Harish-Chandra, and use the notation in 
	Subsection \ref{subsec:cofo}.

Observe that $\overline{R}_{+}$ is also a positive root system. More 
precisely, by (\ref{eq:cor1ax1}), we get
\begin{equation}\label{eq:clo-10a}
\overline{R}_{+}=\overline{R}^{\mathrm{im}}_{+}\cup 
\overline{R}^{\mathrm{re}}_{+}\cup \overline{R}^{\mathrm{c}}_{+}.
\end{equation}
From the properties of $R_{+}$,  (\ref{eq:clo-10a}) can be rewritten in the form
\begin{equation}\label{eq:clo-11a}
\overline{R}_{+}=\left( -R^{\mathrm{im}}_{+}\right)  \cup R^{\mathrm{re}}_{+}\cup 
R_{+}^{\mathrm{c}}.
\end{equation}

Let $\overline{\mathscr P}$ denote the conjugate of $\mathscr P$ in 
$U\left(\mathfrak g_{\C}\right)$. By (\ref{eq:coci2}), 
$\overline{\mathscr P}$ is just the object defined in 
(\ref{eq:clo-9}) associated with $\overline{R}_{+}$.  
We deduce that if $L\in Z\left(\mathfrak g_{\C}\right)$, we have the 
identity in $S\ac\left(\mathfrak h_{\C}\right)$, 
\begin{equation}\label{eq:clo-12}
\phi_{1,\overline{R}_{+}}\overline{L}=\overline{\phi_{1,R_{+}}L}.
\end{equation}
Also $\overline{\rho^{\mathfrak g}}\in \mathfrak h_{\C}^{*}$ is the 
half sum of the roots in $\overline{R}_{+}$. By (\ref{eq:clo-11a}), if 
$f\in S\ac\left(\mathfrak h_{\C}\right)$, then
\begin{equation}\label{eq:clo-13}
\phi_{2,\overline{R}_{+}}\overline{f}=\overline{\phi_{2,R_{+}}f}.
\end{equation}

By (\ref{eq:clo-12}), (\ref{eq:clo-13}), we get the identity in 
$S\ac\left(\mathfrak h_{\C}\right)$,
\begin{equation}\label{eq:clo-14}
\phi_{2,\overline{R}_{+}}\phi_{1,\overline{R}_{+}}\overline{L}=\overline{\phi_{2,R_{+}}
\phi_{1,R_{+}}L}.
\end{equation}

Let us now use Harish-Chandra's result described after (\ref{eq:clo-11}). For $L\in 
Z\left(\mathfrak g_{\C}\right)$, we can rewrite (\ref{eq:clo-14}) as 
an identity in $S\ac\left(\mathfrak h_{\C}\right)$, 
\begin{equation}\label{eq:clo-15}
\phi_{\mathrm{HC}}\overline{L}=\overline{\phi_{\mathrm{HC}}L}.
\end{equation}
But Harish-Chandra gives more, namely that the image of 
$\phi_{\mathrm{HC}}$ is exactly $I\ac\left(\mathfrak h_{\C}, 
\mathfrak g_{\C}\right)$. By (\ref{eq:clo-15}),  $I\ac\left(\mathfrak h_{\C}, \mathfrak g_{\C}\right)$ is 
preserved by complex conjugation, which we already knew by 
Proposition \ref{prop:conja}, and we also obtain the second equation 
in (\ref{eq:star1}), and (\ref{eq:clo-8}). 

Also $\theta R_{+}=-\overline{R}_{+}$
is a positive root system, and the corresponding half-sum of roots is 
given by $\theta\rho^{\mathfrak g}$. As in (\ref{eq:clo-12}), if 
$L\in Z\left(\mathfrak g_{\C}\right)$, then
\begin{equation}\label{eq:clo-16}
\phi_{1,\theta R_{+}}\theta L=\theta\phi_{1,R_{+}}L.
\end{equation}
If $f\in  S\ac\left(\mathfrak h_{\C}\right)$, then
\begin{equation}\label{eq:clo-17}
\phi_{2,\theta R_{+}}\theta f=\theta \phi_{2,R_{+}}f.
\end{equation}
By (\ref{eq:clo-16}), (\ref{eq:clo-17}), we conclude that
\begin{equation}\label{eq:clo-18}
\phi_{2,\theta R_{+}}\phi_{1,\theta 
R_{+}}\theta L=\theta\phi_{2,R_{+}}\phi_{1,R_{+}}L.
\end{equation}
Using again the result of Harish-Chandra, from (\ref{eq:clo-18}), we 
obtain the third equation in (\ref{eq:star1}).

If $f,h\in C^{\infty,c}\left(G,\R\right)$, the convolution 
	$f*h\in C^{\infty ,c}\left(G,\R\right)$ is defined by the formula,
	\begin{equation}\label{eq:sipa1}
		f*h\left(g\right)=\int_{G}^{}f\left(g^{-1}g'\right)h\left(g'\right)dg'.
\end{equation}
 If $A\in 
D_{R}\left(G\right),B\in D_{L}\left(G\right)$, we get easily
\begin{align}\label{eq:sipa2}
&f*Ah=A\left(f*h\right),&f*Bh=\left( B^{*}f \right) *h,\,\,B\left(f*h\right)=\left( \left(\sigma B\right)f
\right) *h.
\end{align}
By (\ref{eq:us1}), (\ref{eq:sipa2}), we conclude that if $L\in Z\left(\mathfrak 
g_{\C}\right)$, then
\begin{equation}\label{eq:sipa4-a1}
\left( L^{*}f \right) *h=\left( \left(\sigma L\right)f \right) *h.
\end{equation}
from which we get
\begin{equation}\label{eq:sipa5}
L^{*}f=\sigma Lf.
\end{equation}
This completes the proof of our theorem.
\end{proof}
\subsection{The Duflo and the Harish-Chandra isomorphisms}%
\label{subsec:duflo}
Here,  
\index{Sg@$S\ac\left[\left[\mathfrak g^{*}\right]\right]$}%
$S\ac\left[\left[\mathfrak g^{*}\right]\right]$ denotes the algebra of 
formal power series $\alpha=\sum_{i=0}^{+ \infty 
}\alpha_{i},\alpha_{i}\in S^{i}\left[\mathfrak g^{*}\right]$. Then 
$\mathfrak g$ still acts on $S\ac\left[\left[\mathfrak 
g^{*}\right]\right]$ as an algebra of derivations. Let
\index{Ig@$I\ac\left[\left[\mathfrak g^{*}\right]\right]$}%
$I\ac\left[\left[\mathfrak g^{*}\right]\right]$ be the subalgebra of 
invariant elements in $S\ac\left[\left[\mathfrak g^{*}\right]\right]$. 

As in Subsection \ref{subsec:lial}, $S\ac\left[\left[\mathfrak 
g^{*}\right]\right]$ can be identified with the algebra 
\index{Dg@$D\ac\left[\left[\mathfrak g^{*}\right]\right]$}%
$D\ac\left[\left[\mathfrak g^{*}\right]\right]$ of formal real partial
differential operators with constant coefficients on $\mathfrak 
g^{*}$, and $I\ac\left[\left[\mathfrak g^{*}\right]\right]$ with the 
algebra of formal real invariant differential operators with constant 
coefficients
\index{DIg@$D_{I}\ac\left[\left[\mathfrak g^{*}\right]\right]$}%
$D_{I}\ac\left[\left[\mathfrak g^{*}\right]\right]$, which acts 
on $S\ac\left(\mathfrak g\right)$.

Then $\widehat{A}^{-1}\left(\ad\left(\cdot\right)\right)\in 
I\ac\left[\left[\mathfrak g^{*}\right]\right]$.
 In the sequel, we view 
 $\widehat{A}^{-1}\left(\ad\left(\cdot\right)\right)$ as an element 
 of $D\ac_{I}\left[\left[\mathfrak g^{*}\right]\right]$.

Let 
\index{tP@$\tau_{\mathrm{PBW}}$}%
$\tau_{\mathrm{PBW}}$ be the Poincaré-Birkhoff-Witt isomorphism of 
filtered vector spaces $S\ac\left(\mathfrak g\right) \simeq 
U\left(\mathfrak g\right)$. Then $\tau_{\mathrm{PBW}}$ induces an identification of filtered 
vector spaces $I\ac\left(\mathfrak g\right) \simeq Z\left(\mathfrak 
g\right)$.

\begin{definition}\label{def:Ddufl}
	Put
	\index{tD@$\tau_{\mathrm{D}}$}%
\begin{equation}\label{eq:cliff1}
 \tau_{\mathrm{D}}=\tau_{\mathrm{PBW}} 
 \widehat{A}^{-1}\left(\mathrm{ad}\left(\cdot\right)\right):	S\ac(\mathfrak{g})\to U(\mathfrak{g}).
\end{equation}
\end{definition}
Then $\tau_{\mathrm{D}}$ is an isomorphism of filtered vector spaces, 
which  commutes with $\theta$.

A  result by Duflo \cite[Théorème V.2]{Duflo70} asserts that 
when restricted to $I\ac\left(\mathfrak g\right)$, 
$\tau_{\mathrm{D}}$ induces an isomorphism of filtered algebras,
\begin{equation}\label{eq:cliff2}
 	I\ac(\mathfrak{g})\simeq  Z(\mathfrak{g}).
\end{equation}

By  \cite[Lemme V.1]{Duflo70}, we have the commutative diagram
\begin{equation}\label{eq:DHC}
	\xymatrix{
    I\ac (\mathfrak{g}_{\C}) \ar[rr]^{\tau_{\mathrm{D}}}\ar[dr]_{r} & & 
	Z(\mathfrak{g}_{\C}) \ar@{->}[dl]^{\phi_{\mathrm{HC}}} \\
     & I(\mathfrak{h}_{\C},\mathfrak{g}_{\C}) &
    }.
	\end{equation}
By Theorem \ref{thm:pinv} and by (\ref{eq:DHC}), we get the commutative diagram
\begin{equation}\label{eq:DHC1}
    \xymatrix{
    I(\mathfrak{g}) \ar[rr]^{\tau_{\mathrm{D}}}\ar[dr]_{r} & & 
	Z(\mathfrak{g}) \ar@{->}[dl]^{\phi_{\mathrm{HC}}} \\
     & I(\mathfrak{h},\mathfrak{g}) &
    },
\end{equation}
and the morphisms in (\ref{eq:DHC1}) commute with $\theta$. 
\subsection{The case of the Casimir}%
\label{subsec:cacasi}
Note that $B^{*}\vert_{\mathfrak h}\in I^{2}\left(\mathfrak h,
\mathfrak g\right)$ corresponds to the Laplacian 
$\Delta^{\mathfrak h}$ on $\mathfrak h$  associated with 
$B\vert_{ \mathfrak h}$.  
The following result of Harish-Chandra is established in  \cite[Example 
5.64]{Knapp02} as a consequence of the constructions in Subsection 
\ref{subsec:cofo}. 
\begin{proposition}\label{prop:cas}
	We have the identity:
\begin{equation}\label{eq:inva9}
\phi_{\mathrm{HC}} C^{\mathfrak g}=-\Delta^{\mathfrak h}+B^{*}\left(\rho^{\mathfrak 
g},\rho^{\mathfrak g}\right).
\end{equation}
\end{proposition}

\begin{proposition}\label{prop:tdu}
	The following identity holds:
	\begin{equation}\label{eq:clif0}
\tau_{\mathrm{D}}^{-1}C^{\mathfrak g}=-B^{*}+B^{*}\left(\rho^{\mathfrak g}, 
\rho^{\mathfrak g}\right).
\end{equation}
\end{proposition}
\begin{proof}
	Clearly,
	\begin{equation}\label{eq:clif-1}
\widehat{A}^{-1}\left(x\right)=1+\frac{1}{24}x^{2}+\ldots
\end{equation}
Also $B^{*}\in I^{2}\left(\mathfrak g\right)$. Let 
$e_{1}\ldots,e_{m+n}$ be a basis of $\mathfrak g$, and let 
$e^{*}_{1},\ldots,e^{*}_{m+n}$ be the basis of $\mathfrak g$ which is 
dual with respect to $B$. By (\ref{eq:clif-1}), 
we get
\begin{equation}\label{eq:clif-2}
\widehat{A}^{-1}\left(\ad\left(\cdot\right)\right)B^{*}=B^{*}+\frac{1}{24}\Tr^{\mathfrak g }\left[\ad\left(e_{i}\right)
\ad\left(e_{j}\right)\right]B^{*}\left(e_{i}^{*},e_{j}^{*}\right).
\end{equation}
Equation (\ref{eq:clif-2}) can be written in the form
\begin{equation}\label{eq:clif-3}
\widehat{A}^{-1}\left(\ad\left(\cdot\right)\right)B^{*}=B^{*}-\frac{1}{24}\Tr^{\mathfrak g}\left[C^{\mathfrak g, \mathfrak g}\right].
\end{equation}
By (\ref{eq:inva10}), we can rewrite (\ref{eq:clif-3}) in the form
\begin{equation}\label{eq:clif-4}
\widehat{A}^{-1}\left(\ad\left(\cdot\right)\right)B^{*}=B^{*}+B^{*}\left(\rho^{\mathfrak g},\rho^{\mathfrak g}\right).
\end{equation}
By (\ref{eq:cliff1}), (\ref{eq:clif-4}), we get
\begin{equation}\label{eq:clif-5}
\tau_{\mathrm{D}}B^{*}=-C^{\mathfrak g}+B^{*}\left(\rho^{\mathfrak 
g},\rho^{\mathfrak g}\right),
\end{equation}
which is equivalent to (\ref{eq:clif0}). The proof of our proposition is completed. 
\end{proof}
\begin{remark}\label{rem:casi}
	Using (\ref{eq:DHC1}), Propositions \ref{prop:cas} and 
	\ref{prop:tdu} can be derived from each other.
\end{remark}
\subsection{The action of $Z\left(\mathfrak g\right)$ on $C^{\infty 
}\left(X,F\right)$}%
\label{subsec:actF}
Note that $G$ acts on the left on $C^{\infty }\left(G,E\right)$ as in 
(\ref{eq:us0}), and there is a corresponding action of 
$D_{R}\left(G\right)$.  Also $K$ acts on $C^{\infty }\left(G,E\right)$ by the formula
\begin{equation}\label{eq:bena2}
k_{R}s\left(g\right)=\rho^{E}\left(k\right)s\left(gk\right),
\end{equation}
and this action of $K$ commutes with the left action of $G$. 
Moreover, we have the identity 
\begin{equation}\label{eq:bena3}
C^{\infty }\left(X,F\right)=\left[C^{\infty }\left(G,E\right)
\right]^{K}.
\end{equation}

Then $D_{R}\left(G\right)$ commutes with the action of $K$ on 
$C^{\infty }\left(G,E\right)$. As a subalgebra of 
$D_{R}\left(G\right)$, $Z\left(\mathfrak g\right)$ also acts on 
$C^{\infty }\left(G,E\right)$, and its action 
commutes with the left action of $G$ and with the right action of 
$K$. The action of $D_{R}\left(G\right)$ descends to $C^{\infty}
\left(X,F\right)$, so that the action of $Z\left(\mathfrak g\right)$ descends to 
$C^{\infty }\left(X,F\right)$ and commutes with the left action of $G$. 
\subsection{The semisimple orbital integrals involving 
$Z\left(\mathfrak g\right)$}%
\label{subsec:sesiorb}
Let 
\index{S@$\mathscr S$}%
$\mathscr S$ be the algebra of differential operators acting on 
$C^{\infty }\left(X,F\right)$ with uniformly bounded coefficients 
together with their derivatives of any order.\footnote{These are 
linear combinations of operators $\n^{F}_{U_{1}}\ldots 
\n^{F}_{U_{k}}$, where $U_{1},\ldots,U_{k}$ are smooth bounded vector fields 
with uniformly bounded covariant derivatives of any order.}
\begin{definition}\label{def:Dker}
	Let 
	\index{Cb@$C^{\infty,b}\left(X,F\right)$}%
	$C^{\infty,b}\left(X,F\right)$ be the vector space of smooth 
	sections of $F$ on $X$ which are bounded together with their  
	covariant
	derivatives of any order.
	
	Let $\mathcal{Q}$ be the  space of smooth kernels 
	$Q\left(x,x'\right)\vert_{x,x'\in X}$ acting on $C^{\infty,b}\left(X,F\right)$ and commuting with the 
left action of $G$ such that there exists $C>0$, and for any 
$S,S'\in \mathscr S$, there exists $C_{S,S'}>0$ for which
\begin{equation}\label{eq:siga2}
\left\vert  SQS'\left(x,x'\right)\right\vert\le 
C_{S,S'}\exp\left(-Cd^{2}\left(x,x'\right)\right).
\end{equation}
\end{definition}
The same arguments as in \cite[Proposition 4.1.2]{Bismut08b} shows 
that the vector space $\mathcal{Q}$ is an algebra with respect to the 
composition of operators.

In particular $D_{R}\left(G\right)$ commutes with $Q$, and so 
$Z\left(\mathfrak g\right)$ commutes with $Q$. 
\begin{proposition}\label{prop:kerce}
	If $L\in Z\left(\mathfrak g\right), Q\in \mathcal{Q}$, then 
	$LQ\in \mathcal{Q}$.
\end{proposition}
\begin{proof}
	Since $L\in Z\left(\mathfrak g\right)$, $L$ commutes with the 
	left action of $G$, and so $LQ$ commutes with this action of $G$. 
	We fix $x_{0}=p1\in G$. Since $LQ$ commutes with $G$, and $G$ 
	acts isometrically on $X$, to 
	establish (\ref{eq:siga2}) for $LQ$, we may as well take 
	$x=x_{0}$. If $U\in \mathfrak g$, and if $U^{X}$ is the 
	corresponding vector field on $X$, since $U^{X}$ is a Jacobi 
	field along the geodesics in $X$, there exist $C>0, c>0$ such 
	that
	\begin{equation}\label{eq:sipa4}
\left\vert  U^{X}\left(x\right)\right\vert\le 
C\exp\left(cd\left(x_{0},x\right)\right).
\end{equation}
The above estimate is also valid for the corresponding covariant 
derivatives. From (\ref{eq:siga2}), we get the estimate (\ref{eq:siga2}) 
for $LQ$ when $x=x_{0}$. The proof of our proposition is completed. 
\end{proof}

Let $\gamma\in G$ be semisimple. By (\ref{eq:fina6}), 
(\ref{eq:siga2}), we have the analogue of (\ref{eq:fina7}), i.e., if 
$f\in \mathfrak p^{\perp}\left(\gamma\right)$, then
\begin{equation}\label{eq:fina7a1}
\left\vert Q\left(\gamma^{-1}e^{f}x_{0},e^{f}x_{0} 
\right)\right\vert\le C_{\gamma}\exp\left(-c_{\gamma}\left\vert  
f\right\vert^{2}\right).
\end{equation}

In \cite[Definition 4.2.2]{Bismut08b}, if $\gamma\in G$ is semisimple, 
if $Q\in \mathcal{Q}$, the 
orbital integral $\Tr^{\left[\gamma\right]}\left[Q\right]$ is 
defined by a formula  similar to (\ref{eq:fina6a1}), i.e., 
\begin{equation}\label{eq:sira0}
\Tr^{\left[\gamma\right]}\left[Q\right]=
\int_{\mathfrak p^{\perp}\left(\gamma\right)}^{}\Tr\left[\gamma_{*}
Q\left(\gamma^{-1}e^{f}x_{0}, 
e^{f}x_{0}\right) \right]r\left(f\right)df.
\end{equation}
The estimates (\ref{eq:fina5}), (\ref{eq:fina7a1}) guarantees that the 
integral in (\ref{eq:sira0}) is well-defined.

By Proposition 
\ref{prop:kerce}, if $L\in Z\left(\mathfrak g\right)$, $LQ\in 
\mathcal{Q}$, and so $\Tr^{\left[\gamma\right]}\left[LQ\right]$ is 
also well-defined.
\section{The center of $U\left(\mathfrak g\right)$ and the regular 
orbital integrals}%
\label{sec:cenreg}
The purpose of this Section is to evaluate the orbital integrals 
for kernels of the 
form $L\mu\left(\sqrt{C^{\mathfrak g,X}+A}\right)$ associated with 
regular elements in $G$, when $L\in Z\left(\mathfrak g\right)$. To 
establish our formula, we will use the main result of \cite{Bismut08b} 
described in Theorem \ref{thm:Ttrfin}, the smoothness properties of the function 
$\mathcal{J}_{\gamma}\left(h_{\mathfrak k}\right)$ that were obtained 
in Section \ref{sec:reg}, and  the Harish-Chandra 
isomorphism in the form given in equation (\ref{eq:inva17}).

This section is organized as follows. In Subsection 
\ref{subsec:cenreg}, we recall the classical result of Harish-Chandra \cite[Theorem 
3]{Harish57a}, \cite[Section 18]{Harish66} that expresses 
certain orbital integrals on $H^{\mathrm{reg}}$ via the action of 
$Z\left(\mathfrak g\right)$ on 
the orbital integral as a function of $\gamma$. 

In Subsection \ref{subsec:georeg}, using Theorem \ref{thm:Ttrfin}, we 
obtain our formula.
\subsection{The algebra $Z\left(\mathfrak g\right)$ and the regular orbital 
integrals}%
\label{subsec:cenreg}
Let $L\in Z\left(\mathfrak g\right), Q\in \mathcal{Q}$, by 
Proposition \ref{prop:kerce}, $LQ\in \mathcal{Q}$. If $f\in C^{\infty 
}\left(G^{\mathrm{reg}}, \C\right)$, $Lf$ is a smooth function on 
$G^{\mathrm{reg}}$. For greater clarity, this function will be 
denoted instead $L_{\gamma}f$.

Now we give another proof of a result of Harish-Chandra  
\cite[Theorem 3]{Harish57a}, \cite[Section 18]{Harish66}, \cite[Proposition 
11.9]{Knapp86}. 
\begin{proposition}\label{prop:pdifre}
	If $Q\in \mathcal{Q}$, the map $\gamma\in 
	G^{\mathrm{reg}}\to\Tr^{\left[\gamma\right]}\left[Q\right]$ is 
	smooth. If $L\in Z\left(\mathfrak g\right)$, we have the identity of smooth functions on 
	$G^{\mathrm{reg}}$: 
	\begin{equation}\label{eq:diff1}
\Tr^{\left[\gamma\right]}\left[LQ\right]=\left(\sigma 
L\right)_{\gamma}\Tr^{\left[\gamma\right]}\left[Q\right].
\end{equation}
\end{proposition}
\begin{proof}
	The map $\left(\gamma,g\right)\in H^{\mathrm{reg}}\times G/H\to 
	g^{-1}\gamma g\in G^{\mathrm{reg}}$ is locally a 
	diffeomorphism. Since $\Tr^{\left[\gamma\right]}\left[Q\right]$ 
	is invariant by conjugation, to obtain  the required smoothness, 
	it is enough to prove that $\gamma\in 
	H^{\mathrm{reg}}\to\Tr^{\left[\gamma\right]}\left[Q\right]\in \C$ 
	is smooth.
	
	Put
	\begin{equation}\label{eq:sira1}
K^{0}\left(H\right)=H^{0}\cap K.
\end{equation}
As we saw in Subsection \ref{subsec:sesidis}, $K^{0}\left(H\right)$ 
is a maximal compact subgroup of $H^{0}$. 

	Using the notation in 
Subsection \ref{subsec:nereg},  if $\gamma\in 
H^{\mathrm{reg}}$, and  if $\gamma'\in H$ is close  to $\gamma$, 
 then $\gamma'\in 
	H^{\reg}$ and (\ref{eq:coc0}) holds. 
	Using equation 
	(\ref{eq:dist13}) in Theorem \ref{Ttotal}, we deduce that
	\begin{equation}\label{eq:bobe2}
X\left(\gamma'\right)=H^{0}/K^{0}\left(H\right).
\end{equation}
In particular $X\left(\gamma'\right)$ does not depend on $\gamma'$, 
and $\mathfrak p^{\perp}\left(\gamma'\right)= \mathfrak h_{\mathfrak 
p}^{\perp}$. 
 By 
(\ref{eq:sira0}), we get
\begin{equation}\label{eq:siga4}
\Tr^{\left[\gamma'\right]}\left[Q\right]=\int_{\mathfrak h_{\mathfrak 
p}^{\perp}}^{}\Tr\left[\gamma'_{*}Q\left(\gamma^{\prime-1}e^{f}x_{0},e^{f}x_{0} \right) \right]r\left(f\right)df.
\end{equation}

For $\gamma'\in H$ close enough to $\gamma\in H$, it is elementary to make 
the estimate in \cite[Theorem 3.4.1 ]{Bismut08b}, which was explained 
in (\ref{eq:fina6}), uniform, so that there exists $C>0$ such that if 
$\gamma'\in H^{\mathrm{reg}}$ is close enough to $\gamma$, if $f\in \mathfrak h_{\mathfrak 
p}^{\perp},\left\vert  f\right\vert\ge 1$, 
\begin{equation}\label{eq:sira2}
d_{\gamma'}\left(e^{f}x_{0}\right)\ge C\left\vert  
f\right\vert.
\end{equation}
By combining (\ref{eq:siga2}) with $S=1,S'=1$ and (\ref{eq:sira2}), 
for $\gamma'\in H$ close enough to $\gamma$, there exist $C>0,c>0$ 
such that if $f\in \mathfrak 
h^{\perp}_{\mathfrak p},\left\vert  f\right\vert\ge 1$, we get
\begin{equation}\label{eq:sira4}
\left\vert  
\gamma_{*}Q\left(\gamma^{-1}e^{f}x_{0},e^{f}x_{0}\right)\right\vert\le C\exp\left(-c\left\vert  f\right\vert^{2}\right).
\end{equation}
Using dominated convergence, by (\ref{eq:fina5}), (\ref{eq:sira4}), 
we deduce that $\Tr^{\left[\gamma\right]}\left[Q\right]$ is a 
continuous function of $\gamma\in H^{\mathrm{reg}}$. 

Let us now prove the above function is smooth on $H^{\mathrm{reg}}$. 
The argument is essentially the same as before, by combining the 
 estimates in (\ref{eq:siga2}) with $S$ arbitrary, together  with 
 uniform estimates given in (\ref{eq:sipa4}).

Equation (\ref{eq:diff1}) just reflects the fact that $\sigma L$ 
	is the image of $L$ by the map $g\to g^{-1}$.
\end{proof}
\subsection{A geometric formula for the regular orbital integrals}%
\label{subsec:georeg}
In the sequel we take the function $\mu\in 
\mathcal{S}^{\mathrm{even}}\left(\R\right)$ as in Subsection \ref{subsec:wake}. Let 
$A\in \R,L\in Z\left(\mathfrak g\right)$.

Put
	\index{hi@$\mathfrak h_{i}$}%
	\begin{equation}\label{eq:last1}
\mathfrak h_{i}= \mathfrak h_{\mathfrak p} \oplus i \mathfrak 
h_{\mathfrak k}.
\end{equation}

Recall that $\phi_{\mathrm{HC}}L\in D_{I}^{\ac}\left(\mathfrak 
h\right)$. This differential operator acts on smooth functions on 
$\mathfrak h$, but as explained in Subsection  \ref{subsec:lial}, it 
also acts on smooth functions on $\mathfrak h_{i}$.

In the next statement, the smooth kernel $\left( 
\phi_{\mathrm{HC}} L \right) \mu\left(\sqrt{\phi_{\mathrm{HC}} 
C^{\mathfrak g}+A}\right)$ on $\mathfrak h_{i}$ acts on the distribution 
$$\mathcal{J}_{\gamma}\left(h_{\mathfrak 
k}\right)\Tr^{E}\left[\rho^{E}\left(k^{-1}e^{-h_{\mathfrak 
k}}\right)\right]\delta_{a}.$$
\begin{theorem}\label{thm:Tgenorb}
	The following identity holds on $H^{\mathrm{reg}}$:
	\begin{multline}\label{eq:diff2}
\Tr^{\left[\gamma\right]}\left[L\mu\left(\sqrt{C^{\mathfrak g,X}+A}\right)\right]\\
=\left( \phi_{\mathrm{HC}} L \right) \mu\left(\sqrt{\phi_{\mathrm{HC}} C^{\mathfrak g}+A}\right)\left[\mathcal{J}_{\gamma}\left(h_{\mathfrak 
k}\right)\Tr^{E}\left[\rho^{E}\left(k^{-1}e^{-h_{\mathfrak k}}\right)\right]\delta_{a}\right]\left(0\right).
\end{multline}
\end{theorem}
\begin{proof}
	If $\gamma\in H^{\mathrm{reg}}$, then $\mathfrak 
	z\left(\gamma\right)= \mathfrak h$, and $\mathfrak 
	k\left(\gamma\right)=\mathfrak h_{\mathfrak k}$. When $L=1$, our theorem is 
	just Theorem \ref{thm:Ttrfin} combined with Proposition 
	\ref{prop:cas}. When $L\in Z\left(\mathfrak g\right)$ is arbitrary, we use equation  (\ref{eq:inva17}) and Proposition 
	\ref{prop:pdifre}. Here $\phi_{\mathrm{HC}} \sigma L$ is a differential 
	operator on $\mathfrak h$.  We find that on $H^{\mathrm{reg}}$, 
\begin{multline}\label{eq:bon9}
\Tr^{\left[\gamma\right]}\left[L\mu\left(\sqrt{C^{\mathfrak 
g,X}+A}\right)\right]\\
=\frac{1}{D_{H}\left(\gamma\right)}\left(\phi_{\mathrm{HC}} \sigma
L\right)\left[D_{H}\left(\gamma\right)r\Tr^{\left[\gamma\right]}\left[
\mu\left(\sqrt{C^{\mathfrak g,X}+A}\right)\right]\right].
\end{multline}
By Theorem \ref{thm:pinv}, we get
\begin{equation}\label{eq:bon9a1}
\phi_{\mathrm{HC}}\sigma L=\left( \phi_{\mathrm{HC}} L \right) ^{*}.
\end{equation}
We combine Theorem \ref{thm:Ttrfin} and equations (\ref{eq:inva9}), 
(\ref{eq:bon9}),  and
(\ref{eq:bon9a1}). We get
\begin{multline}\label{eq:bon9b}
\Tr^{\left[\gamma\right]}\left[L\mu\left(\sqrt{C^{\mathfrak 
g,X}+A}\right)\right]=\frac{1}{D_{H}\left(\gamma\right)}\left(\phi_{\mathrm{HC}} 
L\right)^{*}\\
\left[D_{H}\left(\gamma\right)r\mu
\left(\sqrt{\phi_{\mathrm{HC}} C^{\mathfrak 
g}+A}\right)\left[\mathcal{J}_{\gamma}\left(h_{\mathfrak 
k}\right)\Tr^{E}\left[\rho^{E} \left( k^{-1}e^{-h_{\mathfrak 
k}}\right)\right]\delta_{a}\right]\left(0\right) \right].
\end{multline}

For 
greater clarity, we  fix $\gamma\in H^{\mathrm{reg}}$, and for 
$b\in \mathfrak h$ with $\left\vert  b\right\vert$ small enough, we 
take $\gamma'=\gamma e^{b}$ as in (\ref{eq:bon2}),  so that the 
differential operator $ \left(  \phi_{\mathrm{HC}} L \right)^{*}$ 
acts on the variable $b\in \mathfrak h$. This action will be 
now be denoted $\left(\phi_{\mathrm{HC}} L\right)^{*}_{b}$. Also we use the 
notation of Subsection \ref{subsec:nereg}. 

By equation (\ref{eq:bon3}) and  equation (\ref{eq:idfan2beb}) in Theorem \ref{thm:caregbi}, 
and we get
\begin{multline}\label{eq:idfan2bebe}
D_{H}\left(\gamma'\right)\mathcal{J}_{\gamma'}\left(h_{\mathfrak 
k}\right)=\left(-1\right)^{\left\vert  R^{\mathrm{im}}_{ \mathfrak 
p,+}\right\vert}\epsilon_{D}\left(\gamma'\right)\prod_{\alpha\in 
R_{+}^{\mathrm{re}}}^{} \xi_{\alpha}^{1/2}\left(k^{-1}e^{b_{\mathfrak 
k}}\right)\\
\frac{\prod_{\alpha\in R_{ \mathfrak 
k,+}^{\mathrm{im}}}^{}\left( \xi_{\alpha}^{1/2}\left(k^{-1}e^{b_{\mathfrak 
k}-h_{\mathfrak 
k}}\right)-\xi_{\alpha}^{-1/2}\left(k^{-1}e^{b_{\mathfrak 
k}-h_{\mathfrak k}}\right) \right) }{\prod_{\alpha\in R_{\mathfrak p,+}^{\mathrm{im}}}^{}\left(\xi_{\alpha}^{1/2}\left(k^{-1}e^{b_{\mathfrak 
	k}- h_{\mathfrak k}}\right)- 
	\xi^{-1/2}_{\alpha}\left(k^{-1}e^{b_{\mathfrak k}-h_{\mathfrak 
	k}}\right) \right)}.
\end{multline}

By the same argument as in (\ref{eq:chif-1}), (\ref{eq:chif-2}),    
if $\alpha\in R^{\mathrm{re}}_{+}$, we can choose 
$\xi_{\alpha}^{1/2}\left(k^{-1}e^{b_{\mathfrak k}}\right)$ so that 
for $\left\vert  b_{\mathfrak k}\right\vert$ small enough,
\begin{equation}\label{eq:congo20a1}
\xi_{\alpha}^{1/2}\left(k^{
-1}e^{b_{\mathfrak k}}\right)=\xi_{\alpha}^{1/2}\left(k^{-1}\right),
\end{equation}
i.e., (\ref{eq:congo20a1}) does not depend on $b$. Similarly, 
$\epsilon_{D}\left(\gamma'\right)$ is locally constant on 
$H^{\mathrm{reg}}$. Therefore, the product of these terms in the 
right hand-side of (\ref{eq:idfan2bebe}) is unaffected by  the 
action of 
$\left(\phi_{\mathrm{HC}} L\right)^{*}_{b}$ in the right-hand side of 
(\ref{eq:bon9b}).

Also we have the identity,
\begin{equation}\label{eq:congo22}
\Tr^{E}\left[\rho^{E}\left(k^{ \prime -1}e^{-h_{\mathfrak 
k}}\right)\right]=\Tr^{E}\left[\rho^{E}\left(k^{-1}e^{b_{\mathfrak k}-h_{\mathfrak 
k}}\right)\right].
\end{equation}

The right-hand sides of (\ref{eq:idfan2bebe}) and (\ref{eq:congo22}) 
depend on $b_{\mathfrak k}-h_{\mathfrak k}$. When only considering 
the action of $\left( \phi_{\mathrm{HC}} L \right)^{*}_{b}$ in the variable 
$b_{\mathfrak k}$, this action 
 can instead be transferred to the variable $ h_{\mathfrak 
k}$ with a correcting sign.  This argument still does not take into account the fact that 
$\left(\phi_{\mathrm{HC}} L\right)^{*}_{b}$ also acts in the variable $b_{\mathfrak p}$. However, 
differentiating a smooth kernel at the terminal point is equivalent to 
compose the smooth kernel with the same change of signs as before. By 
combining (\ref{eq:bon9b})--(\ref{eq:congo22}),  this 
ultimately  explains the  disappearance of $*$, and to equation 
(\ref{eq:diff2}). The proof of our theorem is completed. 
\end{proof}
\section{The function  $\mathcal{J}_{\gamma}$ and the  limit of  regular orbital integrals}%
\label{sec:rojg}
In this section, we verify the compatibility of our formula for 
 regular orbital integrals of Theorem 
\ref{thm:Tgenorb} with the limit theorems obtained by Harish-Chandra 
for such orbital integrals. Key properties of the function 
$\mathcal{J}_{\gamma}$ play a key role in the proofs. 

This section is organized as follows. In Subsection 
\ref{subsec:jgnore}, given $\gamma\in H$ not necessarily regular, we 
study the function $\mathcal{J}_{\gamma'}\left(h_{\mathfrak 
k}\right)$ for $\gamma'\in H^{\mathrm{reg}}$  close to $\gamma$. 

In Subsection \ref{subsec:moha}, if $L\in Z\left(\mathfrak g\right)$, 
we define the proper image $L^{\mathfrak 
z\left(\gamma\right)}\in D_{I}\ac\left(\mathfrak 
z\left(\gamma\right)\right)$.

In Subsection \ref{subsec:ros}, using a formula by Rossmann, we 
express a smooth kernel involving $\Delta^{\mathfrak h}$ as the 
restriction of another kernel for $\mathfrak z\left(\gamma\right)$.

Finally, in  Subsection \ref{subsec:aadif}, we compute the limit of orbit 
integrals as $\gamma'\in H^{\mathrm{reg}}$ converges to $\gamma\in H$. 
\subsection{The function $\mathcal{J}_{\gamma}$ when $\gamma$ is not 
regular}%
\label{subsec:jgnore}
Let $\gamma\in G$ be a semisimple element  as in 
(\ref{eq:dis4a}). Then $Z^{0}\left(\gamma\right) \subset G$ is a 
reductive Lie group. Let $\mathfrak h \subset \mathfrak 
z\left(\gamma\right)$ be a $\theta$-stable Cartan subalgebra of $\mathfrak 
z\left(\gamma\right)$. As we saw in Proposition \ref{prop:psesicen}, 
$\mathfrak h$ is also a Cartan subalgebra of $\mathfrak g$.  Let $H 
\subset G$ be the corresponding Cartan 
subgroup. Recall that the function 
\index{phz@$\pi^{\mathfrak h, \mathfrak z\left(\gamma\right)}$}%
$\pi^{\mathfrak h, \mathfrak z\left(\gamma\right)}$ on $\mathfrak 
h_{\C}$ was introduced in (\ref{eq:inva4bis}).
\begin{definition}\label{def:gim}
	An element $h_{\mathfrak k}\in i \mathfrak h_{\mathfrak k}$ is said to be 
	$\gamma\, \mathrm{im}$-regular if for  $\alpha\in 
	R^{\mathrm{im}}\left(k\right)$, $\left\langle  
	\alpha,h_{\mathfrak k}\right\rangle\neq 0$. 
\end{definition}
The vanishing locus of an imaginary root being a hyperplane in $i 
\mathfrak h_{\mathfrak k}$, the set of 
$\gamma\, \mathrm{im}$-regular 
elements has full Lebesgue measure.

We use the conventions of Section \ref{sec:rofu}, where we explained 
in particular how to choose the 
$\xi_{\alpha}^{1/2}\left(\gamma\right)\vert_{\alpha\in R_{+}}$. We extend the definition of $D_{H}\left(\gamma\right)$ for $\gamma\in 
H^{\mathrm{reg}}$ in Definition \ref{def:DH} to general elements  $\gamma\in 
H$ by the formula
\index{DHg@$D_{H}\left(\gamma\right)$}%
\begin{equation}\label{eq:gena-1}
D_{H}\left(\gamma\right)=\prod_{\alpha\in R_{+}\setminus 
R_{+}\left(\gamma\right)}^{}\left(\xi_{\alpha}^{1/2}\left(\gamma\right)-\xi_{\alpha}^{-1/2}\left(\gamma\right)\right).
\end{equation}
By Theorem \ref{thm:careg}, if $\alpha\in 
R^{\mathrm{re}}_{+}\left(\gamma\right)\cup R_{+}^{\mathrm{im}}\left(k\right)$,  then 
$\xi_{\alpha}\left(k^{-1}\right)=1$, and 
$\xi_{\alpha}^{1/2}\left(k^{-1}\right)=\xi_{\alpha}^{-1/2}\left(k^{-1}\right)=\pm 1$,  
so that $\prod_{\alpha\in 
R^{\mathrm{re}}_{+}\left(\gamma\right)\cup 
R^{\mathrm{im}}_{+}\left(k\right)}^{}\xi^{1/2}_{\alpha}\left(k^{-1}\right)$ is equal to $\pm 1$.

If  $\mathfrak h$ is  a $\theta$-stable fundamental Cartan 
subalgebra of $\mathfrak  z\left(\gamma\right)$, we will use the 
notation introduced in Subsection \ref{subsec:rofu}, except that 
$\mathfrak g$ is now replaced by $\mathfrak z\left(\gamma\right)$, 
and the pair $\left(\mathfrak h_{\mathfrak k}, \mathfrak k\right)  $ is 
replaced by the pair $\left( \mathfrak h_{\mathfrak k},\mathfrak 
k\left(\gamma\right) \right) $.  Let $R\left(\mathfrak h_{\mathfrak k}, 
\mathfrak k\left(\gamma\right)\right)$ denote the associated root 
system. Let $R_{+}\left(\mathfrak h_{\mathfrak k}, \mathfrak 
k\left(\gamma\right)\right)$ denote a positive root system. As in 
(\ref{eq:ham-3}), if $h_{\mathfrak k}\in \mathfrak h_{\mathfrak k}$, set
\index{phk@$\pi^{\mathfrak h_{\mathfrak k}, \mathfrak 
k\left(\gamma\right)}$}%
\begin{equation}\label{eq:ham8}
\pi^{\mathfrak h_{\mathfrak k}, \mathfrak 
k\left(\gamma\right)}\left(h_{\mathfrak k}\right)=\prod_{\beta\in 
R_{+}\left(\mathfrak h_{\mathfrak k}, \mathfrak 
k\left(\gamma\right)\right)}^{}\left\langle  \beta,h_{\mathfrak 
k}\right\rangle.
\end{equation}
By (\ref{eq:ham6}), we get
\begin{equation}\label{eq:ham9}
\left[\pi^{h_{\mathfrak k}, \mathfrak 
k\left(\gamma\right)}\left(h_{\mathfrak k}\right)\right]^{2} 
=\left(-1\right)^{\frac{1}{2}\left\vert  
R^{\mathrm{c}}_{+}\left(\gamma\right)\right\vert}\left[\prod_{\alpha\in 
R^{\mathrm{im}}_{\mathfrak 
k,+}\left(k\right)}^{}\left\langle  \alpha,h_{\mathfrak 
k}\right\rangle\right]^{2}
\prod_{\alpha\in R^{\mathrm{c}}_{+}\left(\gamma\right)}^{}\left\langle  
\alpha,h_{\mathfrak k}\right\rangle.
\end{equation}

We no longer assume $\mathfrak h$ to be fundamental in $\mathfrak 
z\left(\gamma\right)$. 

In the sequel, we use the notation of Subsection \ref{subsec:nereg}. 
In particular $\gamma\in H$ is fixed, and for $b\in \mathfrak h$, 
$\gamma'=\gamma e^{b}$. In particular, equations (\ref{eq:bon3}), 
(\ref{eq:bon3a1}) hold.

We establish the following important result.
\begin{theorem}\label{thm:conv}
For $\epsilon>0$ small enough,  there exist $c>0,C>0$ such that for $b\in \mathfrak h$ 
$\gamma$-regular with $\left\vert  
b\right\vert\le \epsilon, h_{\mathfrak k}\in i \mathfrak h_{\mathfrak 
k}$, then
\begin{equation}\label{eq:capr14}
\left\vert  \pi^{\mathfrak h, \mathfrak 
z\left(\gamma\right)}\left(b_{\mathfrak p}+h_{\mathfrak 
k}\right)D_{H}\left(\gamma'\right)\mathcal{J}_{\gamma'}\left( h_{\mathfrak k}\right)\right\vert
\le C\exp\left(c\left\vert  h_{\mathfrak k}\right\vert\right).
\end{equation}
If $h_{\mathfrak k}\in i \mathfrak h_{\mathfrak k}$ is not 
$\gamma\, \mathrm{im}$-regular, the left-hand side of (\ref{eq:capr14})  
vanishes.

If $\mathfrak h$ is not the fundamental Cartan subalgebra of $\mathfrak 
z\left(\gamma\right)$, for $h_{\mathfrak k}\in i \mathfrak 
h_{\mathfrak k}$, as $b\in \mathfrak h$ $\gamma$-regular tends to $0$, 
\begin{equation}\label{eq:gena5}
\pi^{\mathfrak h,\mathfrak z\left(\gamma\right)}\left(b_{\mathfrak 
p}+ h_{\mathfrak 
k}\right)D_{H}\left(\gamma'\right)\mathcal{J}_{\gamma'}\left(h_{\mathfrak k}\right)\to 0.
\end{equation}
	
If $\mathfrak h$ is the fundamental Cartan subalgebra of $\mathfrak 
z\left(\gamma\right)$,  if $ h_{\mathfrak k}\in i \mathfrak 
h_{\mathfrak k}$ is $\gamma\,\mathrm{im}$-regular, as 
	$b\in \mathfrak h$ $\gamma$-regular tends to $0$, 	
	\begin{multline}\label{eq:gena19}
\pi^{\mathfrak h,\mathfrak z\left(\gamma\right)}\left(b_{\mathfrak 
p}+ h_{\mathfrak 
k}\right)D_{H}\left(\gamma'\right)\mathcal{J}_{\gamma'}\left(h_{\mathfrak k}\right)\\
\to \left(-1\right)^{\frac{1}{2}\left\vert  
R^{\mathrm{c}}_{+}\left(\gamma\right)\right\vert+\left\vert  R_{\mathfrak k,+}^{\mathrm{im}}\left(k\right)\right\vert}\left[\pi^{\mathfrak 
h_{\mathfrak k}, \mathfrak k\left(\gamma\right)}\left(h_{\mathfrak 
k}\right)\right]^{2}\\
\prod_{\alpha\in 
R^{\mathrm{im}}_{
+}\left(k\right)}^{}\xi_{\alpha}^{1/2}\left(k^{-1}\right)
D_{H}\left(\gamma\right)
\mathcal{J}_{\gamma}\left(h_{\mathfrak 
k}\right).
\end{multline}
\end{theorem}
\begin{proof}
By (\ref{eq:inva4bis}), we get
\begin{equation}\label{eq:gena10}
\pi^{\mathfrak h, \mathfrak z\left(\gamma\right)}\left(b_{\mathfrak 
p}+h_{\mathfrak k}\right)=\prod_{\alpha\in R_{+}\left(\gamma\right)}^{}
\left\langle  \alpha,b_{\mathfrak p}+h_{\mathfrak k}\right\rangle.
\end{equation}
We can rewrite (\ref{eq:gena10}), in the form
\begin{equation}\label{eq:gena11}
\pi^{\mathfrak h, \mathfrak z\left(\gamma\right)}\left(b_{\mathfrak 
p}+h_{\mathfrak k}\right)=\prod_{\alpha\in R^{\mathrm{re}}_{+}\left(\gamma\right)}^{}\left\langle  
\alpha,b_{\mathfrak p}\right\rangle\prod_{\alpha\in 
R_{+}^{\mathrm{im}}\left(k\right)}^{}\left\langle  
\alpha,h_{\mathfrak k}\right\rangle
\prod_{\alpha\in 
R_{+}^{\mathrm{c}}\left(\gamma\right)}^{}\left\langle  
\alpha,b_{\mathfrak p}+h_{\mathfrak k}\right\rangle.
\end{equation}
By (\ref{eq:gena11}), we deduce that if $h_{\mathfrak k}\in i 
\mathfrak h_{\mathfrak k}$ is not $\gamma\,\mathrm{im}$-regular, 
(\ref{eq:gena11}) vanishes.

If $\alpha\in R^{\mathrm{im}}\left(k\right)$, then 
$\xi_{\alpha}\left(k^{-1}\right)=1$, so that 
$\xi_{\alpha}^{1/2}\left(k^{-1}\right)=\pm 1$. Since $\alpha\in 
R\left(\gamma\right)$,
 if $b\in \mathfrak h$ is 
$\gamma$-regular, then $\left\langle\alpha,b_{\mathfrak k}  
\right\rangle\neq 0$. If $\epsilon>0$ is small enough, if $b\in 
\mathfrak h$ is $\gamma$-regular, and $\left\vert  b_{\mathfrak k}\right\vert\le 
\epsilon$, if $ h_{\mathfrak k}\in i \mathfrak h_{\mathfrak k}$, then $1-e^{-\left\langle  \alpha,b_{\mathfrak 
k}-h_{\mathfrak k}\right\rangle}\neq 0$, 
so that the following expression is well-defined,
\begin{multline}\label{eq:gena11a}
\prod_{\alpha\in R^{\mathrm{im}}_{\mathfrak p,+}\left(k\right)}^{}\frac{
\left\langle  \alpha,h_{\mathfrak 
k}\right\rangle}{\xi_{\alpha}^{1/2}\left(k^{-1}e^{b_{\mathfrak k}-h_{\mathfrak 
k}}\right)-\xi_{\alpha}^{-1/2}\left(k^{-1}e^{b_{\mathfrak k}-h_{\mathfrak 
k}}\right)}\\
=\prod_{\alpha\in R^{\mathrm{im}}_{\mathfrak p,+}\left(k\right)}^{}\frac{
\left\langle  \alpha,h_{\mathfrak 
k}\right\rangle} {1-e^{-\left\langle  \alpha,b_{\mathfrak k}-h_{\mathfrak 
k}\right\rangle}}\xi_{\alpha}^{-1/2}\left(k^{-1}e^{b_{\mathfrak k}-h_{\mathfrak 
k}}\right).
\end{multline}
Equation (\ref{eq:gena11a}) can be rewritten in the form
\begin{multline}\label{eq:gena11aax}
\prod_{\alpha\in R^{\mathrm{im}}_{\mathfrak p,+}\left(k\right)}^{}\frac{
\left\langle  \alpha,h_{\mathfrak 
k}\right\rangle}{\xi_{\alpha}^{1/2}\left(k^{-1}e^{b_{\mathfrak k}-h_{\mathfrak 
k}}\right)-\xi_{\alpha}^{-1/2}\left(k^{-1}e^{b_{\mathfrak k}-h_{\mathfrak 
k}}\right)}\\
=\prod_{\alpha\in R^{\mathrm{im}}_{\mathfrak p,+}\left(k\right)}^{}
\frac{\left\langle  \alpha,h_{\mathfrak k}\right\rangle /2}{\sinh\left(
\left\langle  \alpha,b_{\mathfrak k}-h_{\mathfrak 
k}\right\rangle/2\right)}\xi_{\alpha}^{-1/2}\left(k^{-1}\right).
\end{multline}

When $h_{\mathfrak k}$ is $\gamma$ $\mathrm{im}$-regular, we have an 
identity similar to (\ref{eq:gena11aax}),
\begin{multline}\label{eq:gena11aay}
\prod_{\alpha\in R^{\mathrm{im}}_{\mathfrak 
k,+}\left(k\right)}^{}\frac{\xi_{\alpha}^{1/2}\left(k^{-1}e^{b_{\mathfrak 
k}-h_{\mathfrak 
k}}\right)-\xi_{\alpha}^{-1/2}\left(k^{-1}e^{b_{\mathfrak 
k}-h_{\mathfrak k}}\right)}{\left\langle  \alpha,h_{\mathfrak 
k}\right\rangle}\\
=\prod_{\alpha\in R^{\mathrm{im}}_{\mathfrak 
k,+}\left(k\right)}^{}\frac{\sinh\left(\left\langle  
\alpha,b_{\mathfrak k}-h_{\mathfrak k}\right\rangle/2\right)}{\left\langle  
\alpha,h_{\mathfrak k}\right\rangle/2}\xi_{\alpha}^{1/2}\left(k^{-1}\right).
\end{multline}

If $x\in \R, y\in \R$, 
\begin{equation}\label{eq:capr22a1}
\left\vert  1-e^{x+iy}\right\vert\ge \left\vert  1-e^{x}\right\vert. 
\end{equation}
By (\ref{eq:capr22a1}), we deduce that
\begin{equation}\label{eq:capr22a1x}
\left\vert  
\sinh\left(\left(x+iy\right)/2\right)\right\vert\ge\left\vert  
\sinh\left(x/2\right)\right\vert.
\end{equation}

By (\ref{eq:capr22a1x}), there exists $C>0$ such that  if $e^{x+iy}\neq 1$, 
\begin{equation}\label{eq:capr22a2}
\left\vert  \frac{x}{\sinh\left(\left(x+iy\right)/2\right)}\right\vert\le C. 
\end{equation}
By (\ref{eq:gena11aax}), (\ref{eq:capr22a2}),  if $h_{\mathfrak k}\in i 
\mathfrak h_{\mathfrak k}$ is $\gamma$ $\mathrm{im}$-regular, we get
\begin{equation}\label{eq:gena12}
\left\vert  \prod_{\alpha\in R^{\mathrm{im}}_{\mathfrak p,+}\left(k\right)}^{}\frac{
\left\langle  \alpha,h_{\mathfrak 
k}\right\rangle}{\xi_{\alpha}^{1/2}\left(k^{-1}e^{b_{\mathfrak k}-h_{\mathfrak 
k}}\right)-\xi_{\alpha}^{-1/2}\left(k^{-1}e^{b_{\mathfrak k}-h_{\mathfrak 
k}}\right)}\right\vert\le C.
\end{equation}
The $\gamma$-regularity of $b$ does not play any role in the above 
estimate.

If $\alpha\in R_{\mathfrak p,+}^{\mathrm{im}}\setminus 
R_{\mathfrak p,+}^{\mathrm{im}}\left(k\right)$, then 
$\xi_{\alpha}\left(k^{-1}\right)\neq 1$, so that for $\epsilon>0$ 
small enough, if $\left\vert  b_{\mathfrak k}\right\vert\le 
\epsilon$, the complex numbers 
$\xi_{\alpha}\left(k^{-1}e^{b_{\mathfrak k}}\right)$ which have 
module $1$, stay away from $1$. Given $\eta\in \left]0,\pi\right[$, 
there is $C_{\eta}>0$ such that
if $x\in \R_{+}, y\in \left[\eta,2\pi-\eta\right]$, then
\begin{equation}\label{eq:gena14}
\left\vert  xe^{iy}-x^{-1}e^{-iy}\right\vert\ge C_{\eta}.
\end{equation}
By  (\ref{eq:gena14}), we deduce that for 
$\epsilon>0$ small enough,  if 
$\left\vert  b_{\mathfrak k}\right\vert\le \epsilon$, and $h_{\mathfrak k}\in i 
\mathfrak h_{\mathfrak k}$, then
\begin{multline}\label{eq:gena13}
\prod_{\alpha\in R^{\mathrm{im}}_{\mathfrak p,+}\setminus 
R^{\mathrm{im}}_{\mathfrak p,+}\left(k\right)}^{}\frac{
1}{\xi_{\alpha}^{1/2}\left(k^{-1}e^{b_{\mathfrak k}-h_{\mathfrak 
k}}\right)-\xi_{\alpha}^{-1/2}\left(k^{-1}e^{b_{\mathfrak k}-h_{\mathfrak 
k}}\right)}\\
=\prod_{\alpha\in R^{\mathrm{im}}_{\mathfrak p,+}\setminus 
R^{\mathrm{im}}_{\mathfrak p,+}\left(k\right)}^{}\frac{\xi_{\alpha}^{-1/2}\left(k^{-1}e^{b_{\mathfrak k}-h_{\mathfrak 
k}}\right)} {1-e^{-\left\langle  \alpha,b_{\mathfrak k}-h_{\mathfrak 
k}\right\rangle}\xi_{\alpha}^{-1}\left(k^{-1}\right)}
\end{multline}
is well-defined, and moreover,
\begin{equation}\label{eq:gena15}
\left\vert  \prod_{\alpha\in R_{\mathfrak p,+}^{\mathrm{im}}\setminus 
R^{\mathrm{im}}_{\mathfrak p,+}\left(k\right)}^{}\frac{
1}{\xi_{\alpha}^{1/2}\left(k^{-1}e^{b_{\mathfrak k}-h_{\mathfrak 
k}}\right)-\xi_{\alpha}^{-1/2}\left(k^{-1}e^{b_{\mathfrak k}-h_{\mathfrak 
k}}\right)}\right\vert\le C.
\end{equation}

By combining (\ref{eq:idfan2beb}), (\ref{eq:gena11}), (\ref{eq:gena12}), 
and (\ref{eq:gena15}), we get (\ref{eq:capr14}), and also the fact 
that if $h_{\mathfrak k}\in i \mathfrak h_{\mathfrak k}$ is not 
$\gamma$ $\mathrm{im}$-regular, the 
left-hand side of (\ref{eq:capr14}) vanishes.

By Proposition \ref{prop:Preal}, $\mathfrak h$ is the fundamental 
Cartan subalgebra of $\mathfrak z\left(\gamma\right)$  if and only if 
$R^{\mathrm{re}}\left(\gamma\right)$ is empty.

If $\mathfrak h$ is not the fundamental Cartan 
subalgebra of $\mathfrak z\left(\gamma\right)$, 
$R^{\mathrm{re}}_{+}\left(\gamma\right)$ is nonempty. Using 
(\ref{eq:idfan2beb}), 
(\ref{eq:gena11}) and the previous bounds, we get (\ref{eq:gena5}). 

From now on, we assume that  $\mathfrak h$ is the fundamental Cartan subalgebra of $\mathfrak 
z\left(\gamma\right)$. By (\ref{eq:grai1}), since 
$R^{\mathrm{re}}_{+}\left(\gamma\right)$ is empty, we get
\begin{equation}\label{eq:gena20}
\epsilon_{D}\left(\gamma\right)=\mathrm{sgn}\prod_{\alpha\in 
R^{\mathrm{re}}_{+}}^{}\left(1-\xi_{\alpha}^{-1}\left(\gamma\right)\right).
\end{equation}
For $\alpha\in 
R^{\mathrm{re}}_{+}$, then
$\xi_{\alpha}\left(\gamma\right)\neq 1$ so that  for 
$\epsilon>0$ small enough, if $\left\vert  b_{\mathfrak 
p}\right\vert\le \epsilon$, 
\begin{equation}\label{eq:gena21}
\epsilon_{D}\left(\gamma'\right)=\epsilon_{D}\left(\gamma\right).
\end{equation}
Also if $\alpha\in R^{\mathrm{re}}_{+}$, 
$\xi_{\alpha}\left(k'\right)=\xi_{\alpha}\left(k\right)$. By 
the above, it follows that for $b\in \mathfrak h$ $\gamma$-regular close enough to $0$, 
\begin{equation}\label{eq:clo1}
\epsilon_{D}\left(\gamma'\right)\prod_{\alpha\in 
R^{\mathrm{re}}_{+}}^{}\xi_{\alpha}^{1/2}\left(k^{ \prime -1}\right)=\epsilon_{D}
\left(\gamma\right)\prod_{\alpha\in 
R^{\mathrm{re}}_{+}}^{}\xi_{\alpha}^{1/2}\left(k^{-1}\right).
\end{equation}

By (\ref{eq:gena11aax}), (\ref{eq:gena11aay}), and (\ref{eq:capr22a2}),   
  if  $h_{\mathfrak k}\in i \mathfrak 
h_{\mathfrak k}$ is $\gamma$ im-regular, if $b\in \mathfrak h$ 
$\gamma$-regular tends to $0$, 
\begin{align}\label{eq:gena17}
&\prod_{\alpha\in R^{\mathrm{im}}_{\mathfrak p,+}\left(k\right)}^{}\frac{
\left\langle  \alpha,h_{\mathfrak 
k}\right\rangle}{\xi_{\alpha}^{1/2}\left(k^{-1}e^{b_{\mathfrak k}-h_{\mathfrak 
k}}\right)-\xi_{\alpha}^{-1/2}\left(k^{-1}e^{b_{\mathfrak k}-h_{\mathfrak 
k}}\right)} \nonumber \\
&\qquad \to \left(-1\right)^{\left\vert  R_{\mathfrak 
p,+}^{\mathrm{im}}\left(k\right)\right\vert}
\prod_{\alpha\in R^{\mathrm{im}}_{\mathfrak p,+}\left(k\right)}^{}
\widehat{A}\left(\left\langle  \alpha,h_{\mathfrak 
k}\right\rangle\right)\xi_{\alpha}^{-1/2}\left(k^{-1}\right),\\
&\prod_{\alpha\in R^{\mathrm{im}}_{\mathfrak 
k,+}\left(k\right)}^{}\frac{\xi_{\alpha}^{1/2}\left(k^{-1}e^{b_{\mathfrak 
k}-h_{\mathfrak 
k}}\right)-\xi_{\alpha}^{-1/2}\left(k^{-1}e^{b_{\mathfrak 
k}-h_{\mathfrak k}}\right)}{\left\langle  \alpha,h_{\mathfrak 
k}\right\rangle} \nonumber \\
&\qquad\to \left(-1\right)^{\left\vert  R^{\mathrm{im}}_{\mathfrak 
k,+}\left(k\right)\right\vert}\prod_{\alpha\in 
R^{\mathrm{im}}_{\mathfrak 
k,+}\left(k\right)}^{}\frac{\xi_{\alpha}^{1/2}\left(k^{-1}\right)}{\widehat{A}\left(\left\langle  
\alpha,h_{\mathfrak 
k}\right\rangle\right)}. \nonumber 
\end{align}

By (\ref{eq:idfan2beb}), by the considerations after 
(\ref{eq:gena-1}), by (\ref{eq:ham9}), (\ref{eq:gena11}),   (\ref{eq:clo1}),  
and (\ref{eq:gena17}),    we deduce that if 
$h_{\mathfrak k}\in i \mathfrak h_{\mathfrak k}$ is 
$\gamma$ im-regular, then 
\begin{multline}\label{eq:gena18}
\pi^{\mathfrak h,\mathfrak z\left(\gamma\right)}\left(b_{\mathfrak 
p}+ h_{\mathfrak 
k}\right)D_{H}\left(\gamma'\right)\mathcal{J}_{\gamma'}\left(h_{\mathfrak 
k}\right)\\
\to 
\epsilon_{D}\left(\gamma\right)\left(-1\right)^{\left\vert  
R_{\mathfrak p,+}^{\mathrm{im}}\right\vert+\left\vert  
R^{\mathrm{im}}_{+}\left(k\right)\right\vert+\frac{1}{2}\left\vert  
R_{+}^{\mathrm{c}}\left(\gamma\right)\right\vert}\\
\prod_{\alpha\in  R^{\mathrm{re}}_{+}\cup R^{\mathrm{im}}_{+}\left(k\right)}^{}\xi_{\alpha}^{1/2}\left(k^{-1}\right)
\left[\pi^{\mathfrak h_{\mathfrak k},\mathfrak 
k\left(\gamma\right)}\left(h_{\mathfrak k}\right)\right]^{2}
\frac{\prod_{\alpha\in R^{\mathrm{im}}_{\mathfrak p,+}\left(k\right)}^{}
\widehat{A}\left(\left\langle  \alpha,h_{\mathfrak 
k}\right\rangle\right)}{\prod_{\alpha\in R^{\mathrm{im}}_{\mathfrak k,+}\left(k\right)}^{}
\widehat{A}\left(\left\langle  \alpha,h_{\mathfrak 
k}\right\rangle\right)}\\
\frac{\prod_{\alpha\in R_{\mathfrak k,+}^{\mathrm{im}}\setminus 
	R_{\mathfrak k,+}^{\mathrm{im}}\left(k\right)}^{}\left( 
	\xi_{\alpha}^{1/2}\left(k^{-1}e^{-h_{\mathfrak k}}\right)- \xi^{-1/2}_{\alpha}\left(k^{-1}e^{-h_{\mathfrak k}}\right)\right)}{\prod_{\alpha\in R_{\mathfrak p,+}^{\mathrm{im}}\setminus 
	R_{\mathfrak p,+}^{\mathrm{im}}\left(k\right)}^{}\left( 
	\xi_{\alpha}^{1/2}\left(k^{-1}e^{-h_{\mathfrak k}}\right)- 
	\xi^{-1/2}_{\alpha}\left(k^{-1}e^{-h_{\mathfrak k}}\right)\right)}.
\end{multline}
By comparing (\ref{eq:idfan2}) and the right-hand side of 
(\ref{eq:gena18}), we get (\ref{eq:gena19}). The proof of our theorem is completed. 
\end{proof}
\subsection{The Lie algebra $\mathfrak z\left(\gamma\right)$ and the 
isomorphisms of Harish-Chandra and Duflo}%
\label{subsec:moha}
Here, we will use  results contained in Subsections \ref{subsec:care}, 
\ref{subsec:cenenv}, and \ref{subsec:duflo}. 

Let $\mathfrak h$ be a 
$\theta$-stable Cartan subalgebra of $\mathfrak z\left(\gamma\right)$. 
We have the Harish-Chandra and Duflo isomorphisms of filtered 
algebras:
\begin{align}\label{eq:mos1}
&\phi_{\mathrm{HC}}: Z\left(\mathfrak g\right) \simeq 
I\ac\left(\mathfrak h,\mathfrak g\right),
&\tau_{\mathrm{D}}: I\ac\left(\mathfrak g\right) \simeq Z\left(\mathfrak 
g\right).
\end{align}
As explained in Subsection \ref{subsec:duflo}, the above isomorphisms 
are compatible.

By   (\ref{eq:inva6a}), we have the isomorphisms, 
\begin{align}\label{eq:com4}
&r: I\ac\left(\mathfrak g\right) \simeq I\ac\left(\mathfrak 
h,\mathfrak g\right),&r: I\ac\left(\mathfrak 
z\left(\gamma\right)\right) \simeq I\ac\left(\mathfrak h, \mathfrak 
z\left(\gamma\right)\right).
\end{align}

Recall that we have the splitting
\begin{equation}\label{eq:comm3}
\mathfrak g= \mathfrak z\left(\gamma\right) \oplus \mathfrak 
z^{\perp}\left(\gamma\right).
\end{equation}
Let 
\index{rzg@ $r_{\mathfrak z\left(\gamma\right)}$}%
$r_{\mathfrak z\left(\gamma\right)}$ denote the projection 
$\mathfrak g\to \mathfrak z\left(\gamma\right)$. This map induces a 
corresponding morphism of $\Z$-graded algebras $r_{\mathfrak z\left(\gamma\right)}:I\ac\left(\mathfrak g\right)\to 
I\ac\left(\mathfrak z\left(\gamma\right)\right)$. 
   Let
\index{i@$i$}%
$i$ be the obvious 
morphism $I\ac\left(\mathfrak h, \mathfrak g\right)\to 
I\ac\left(\mathfrak h, \mathfrak z\left(\gamma\right)\right)$.

We have the 
commutative diagram
\begin{equation}\label{eq:com2}
  \xymatrixcolsep{3pc}\xymatrix{
    I\ac(\mathfrak{g}) \ar^{r}[r]^{}\ar^{r_{\mathfrak 
	z\left(\gamma\right)}}[d]& I\ac\left(\mathfrak h, \mathfrak 
	g\right) \ar[d]^{i}\\
	I\ac\left(\mathfrak z\left(\gamma\right)\right)\ar^{r}[r] 
	& I\ac\left(\mathfrak h, \mathfrak z\left(\gamma\right)\right) 
    }.
\end{equation}
\begin{definition}\label{def:duzg}
	If $L\in Z\left(\mathfrak g\right)$, let 
	\index{Lzg@$L^{\mathfrak z\left(\gamma\right)}$}%
	$L^{\mathfrak z\left(\gamma\right)}\in I\ac\left(\mathfrak 
	z\left(\gamma\right)\right)$ be given by
	\begin{equation}\label{eq:com3}
L^{\mathfrak z\left(\gamma\right)}=r_{\mathfrak 
z\left(\gamma\right)}\tau^{-1}_{\mathrm D}\left(L\right).
\end{equation}
\end{definition}

The map $L\in Z\left(\mathfrak g\right)\to L^{\mathfrak 
z\left(\gamma\right)}\in I\ac\left(\mathfrak 
z\left(\gamma\right)\right)$ is a morphism of filtered algebras.
\begin{proposition}\label{prop:idz}
	If $L\in Z\left(\mathfrak g\right)$, the following identity holds:
	\begin{equation}\label{eq:com5}
L^{\mathfrak z\left(\gamma\right)}=r^{-1}i\phi_{\mathrm{HC}}L.
\end{equation}
\end{proposition}
\begin{proof}
	This follows from (\ref{eq:DHC1}), (\ref{eq:com2}), and 
	(\ref{eq:com3}).
\end{proof}

As we saw in Subsections \ref{subsec:lial} and \ref{subsec:care}, we have the identification
\begin{equation}\label{eq:com6}
I\ac\left(\mathfrak z\left(\gamma\right)\right) \simeq 
D_{I}\ac\left(\mathfrak z\left(\gamma\right)\right).
\end{equation}
In the sequel, when there is no ambiguity, if $L\in 
Z\ac\left(\mathfrak g\right)$, $L^{\mathfrak z\left(\gamma\right)}$ 
will be considered as an element of $D_{I}\ac\left(\mathfrak 
z\left(\gamma\right)\right)$.
\begin{proposition}\label{prop:czg}
	The following identity holds:
	\index{Cgz@$\left(C^{\mathfrak g}\right)^{\mathfrak 
z\left(\gamma\right)}$}%
	\begin{equation}\label{eq:clif0a}
\left(C^{\mathfrak g}\right)^{\mathfrak 
z\left(\gamma\right)}=-\Delta^{\mathfrak 
z\left(\gamma\right)}+B^{*}\left(\rho^{\mathfrak g},\rho^{\mathfrak 
g}\right).
\end{equation}
\end{proposition}
\begin{proof}
	By Proposition \ref{prop:cas} and by (\ref{eq:com5}), we get 
	(\ref{eq:clif0a}).
\end{proof}

\subsection{An application of Rossmann's formula}%
\label{subsec:ros}
We know that $a\in \mathfrak h_{\mathfrak p}$. Since $a$ is in the 
center of $\mathfrak z\left(\gamma\right)$, if $\alpha\in 
R\left(\gamma\right)$, then $\left\langle  \alpha,a\right\rangle=0$. 
By (\ref{eq:inva4bis}), we deduce that if $h\in \mathfrak h$, 
\begin{equation}\label{eq:auj1a2}
\pi^{\mathfrak h,\mathfrak z\left(\gamma\right)}\left(h+a\right)=
\pi^{\mathfrak h, \mathfrak z\left(\gamma\right)}\left(h\right).
\end{equation}

We identify $\mathfrak h$ and $\mathfrak h^{*}$ via the form $B$, so 
that 
\begin{equation}\label{eq:inva24a1}
S\ac\left(\mathfrak h^{*}\right) \simeq S\ac\left(\mathfrak h\right).
\end{equation}
As we saw in Subsection \ref{subsec:lial}, 
$S\ac\left(\mathfrak h\right)$ can be identified with the 
algebra $D\ac\left(\mathfrak h\right)$ of differential operators with constant coefficients on 
$\mathfrak h$. Let $\overline{\pi}^{\mathfrak h, \mathfrak 
z\left(\gamma\right)}\in D\ac\left(\mathfrak h_{\C}\right) $ 
denote  the  differential operator on $\mathfrak h_{\C}$ associated with $\pi^{\mathfrak h, \mathfrak 
z\left(\gamma\right)}\in S\ac\left(\mathfrak h_{\C}\right)$. 

Recall that 
\index{hi@$\mathfrak h_{i}$}%
$\mathfrak h_{i}$ was defined in (\ref{eq:last1}). Let $\mu\in \mathcal{S}^{\mathrm{even}}\left(\R\right)$ be taken as 
in Subsection \ref{subsec:wake}. If $A\in \R$, let $\mu\left(\sqrt{-\Delta^{\mathfrak 
	h}+A}\right)\left(h\right)$ be the smooth convolution kernel on $\mathfrak 
	h_{i}$ associated 
	with the operator $\mu\left(\sqrt{-\Delta^{\mathfrak 
	h}+A}\right)$. If $f\in C^{\infty ,c}\left(\mathfrak h_{\mathfrak 
	p}\oplus i\mathfrak h_{\mathfrak k},\R\right)$, then
	\begin{equation}\label{eq:tann1}
\mu\left(\sqrt{-\Delta^{\mathfrak 
	h}+A}\right)f\left(h\right)=\int_{\mathfrak h_{\mathfrak  p}
	\oplus i \mathfrak h_{\mathfrak k}}^{}\mu\left(\sqrt{-\Delta^{\mathfrak 
	h}+A}\right)\left(h-h'\right)f\left(h'\right)dh'.
\end{equation}
We use the corresponding notation on 
\index{zig@$\mathfrak z_{i}\left(\gamma\right)$}%
$\mathfrak 
z_{i}\left(\gamma\right)= \mathfrak p\left(\gamma\right) \oplus i 
\mathfrak k\left(\gamma\right)$.

Here, we use the notation of Subsection \ref{subsec:moha}. In the 
	sequel, $\phi_{\mathrm{HC}} L,L^{\mathfrak z\left(\gamma\right)}$ 
	are viewed as differential operators in
	 $D\ac_{I}\left(\mathfrak h\right), 
	 D\ac_{I}\left(\mathfrak z\left(\gamma\right)\right)$. These 
	 differential operators can be composed with the above 
	 convolution kernels.
\begin{proposition}\label{prop:keha}
We have the identity of smooth functions on $\mathfrak h_{\mathfrak 
p} \oplus i \mathfrak h_{\mathfrak k}$, 
	\begin{multline}\label{eq:inva24a2}
\overline{\pi}^{\mathfrak h, \mathfrak 
	z\left(\gamma\right)}\left(\phi_{\mathrm{HC}} L \right)\mu\left(\sqrt{-\Delta^{\mathfrak 
	h}+A}\right)\left(h\right) \\
	= \pi^{\mathfrak h,\mathfrak 
	z\left(\gamma\right)}\left(-2\pi h\right)L^{\mathfrak z\left(\gamma\right)}\mu\left(\sqrt{-\Delta^{\mathfrak 
	z\left(\gamma\right)}+A}\right)\left(h\right).
\end{multline}
\end{proposition}
\begin{proof}
	In the proof, we will identify $\mathfrak z_{i}\left(\gamma\right)=\mathfrak p\left(\gamma\right) \oplus i \mathfrak k\left(\gamma\right)$ 
	as a real vector space to $\mathfrak u\left(\gamma\right)= i \mathfrak p\left(\gamma\right) \oplus 
	\mathfrak k\left(\gamma\right)$, which is the compact form of 
	$\mathfrak z\left(\gamma\right)$.

We identify the real Euclidean vector space $\mathfrak h_{i}$ to its 
dual by its scalar product. Let $\mathcal{F}^{\mathfrak h_{i}}$ denote the classical Fourier transform on 
	$\mathfrak h_{i}$. If $f\in \mathcal{S}\left(\mathfrak 
	h_{i} \right) $, if $h^{*}\in \mathfrak h_{i}$, then
	\begin{equation}\label{eq:gena22}
\mathcal{F}^{\mathfrak h_{i}}f\left(h^{*}\right)=\int_{\mathfrak 
h_{i}}^{}\exp\left(-2i\pi 
B\left(h,h^{*}\right)\right)f\left(h\right)dh.
\end{equation}
Put
\begin{equation}\label{eq:gena23}
\breve{\mathcal{F}}^{\mathfrak 
h_{i}}f\left(h^{*}\right)=\mathcal{F}^{\mathfrak h_{i}}\left(-h^{*}\right).
\end{equation}

Then
	\begin{multline}\label{eq:inva24a3}
\mathcal{F}^{\mathfrak h_{i}}\left[\overline{\pi}^{\mathfrak h, \mathfrak 
	z\left(\gamma\right)}\left( \phi_{\mathrm{HC}} L \right) \mu\left(\sqrt{-\Delta^{\mathfrak 
	h}+A}\right)\right]\left(h^{*}\right)\\
	=\pi^{\mathfrak h, \mathfrak z\left(\gamma\right)}\left(2i\pi h^{*}\right)
	\left(\phi_{\mathrm{HC}} L\right)\left(2i\pi h^{*}\right)\mu\left(\sqrt{4\pi^{2}B\left(h^{*},h^{*}\right)+A}\right).
\end{multline}
By (\ref{eq:inva24a3}), we get
\begin{multline}\label{eq:inva24a4}
\overline{\pi}^{\mathfrak h, \mathfrak 
	z\left(\gamma\right)}\left(\phi_{\mathrm{HC}} L\right)\mu\left(\sqrt{-\Delta^{\mathfrak 
	h}+A}\right)\left(h\right) \\
	=  
	\breve{\mathcal{F}}^{\mathfrak h_{i}}
	\left[\pi^{\mathfrak h, \mathfrak z\left(\gamma\right)}
	\left(2i\pi 
	h^{*}\right)\left(\phi_{\mathrm{HC}} L\right)\left(2i\pi h^{*}\right)\mu\left(\sqrt{4\pi^{2}B\left(h^{*},h^{*}\right)+A} \right) \right]\left(h\right).
\end{multline}

We can define the Fourier transform $\mathcal{F}^{
\mathfrak z_{i}\left(\gamma\right)}$ on the Euclidean vector space 
\index{zig@$\mathfrak z_{i}\left(\gamma\right)$}%
$\mathfrak z_{i}\left(\gamma\right)$, which is canonically 
identified to the Lie algebra 
\index{ug@$\mathfrak u\left(\gamma\right)$}%
$\mathfrak u\left(\gamma\right)$. 
The function $B\left(h^{*},h^{*}\right)$ extends to a 
$\Ad \left( \mathfrak u\left(\gamma\right) \right) $-invariant function on $\mathfrak 
u\left(\gamma\right)$. By Proposition \ref{prop:idz}, the 
$\Ad \left( \mathfrak u\left(\gamma\right) \right) $-invariant function 
$L^{\mathfrak z\left(\gamma\right)}$ on $\mathfrak z_{i}\left(\gamma\right)$ restricts to the 
function $\phi_{\mathrm{HC}}L$ on $\mathfrak h_{i}$. By Rossmann's 
formula \cite[Theorem p. 209]{Rossmann78}, 
\cite[Theorem p. 13]{Vergne79}, we get
\begin{multline}\label{eq:ibva24a5}
\breve{\mathcal{F}}^{\mathfrak h_{i}}
	\left[\pi^{\mathfrak h, \mathfrak z\left(\gamma\right)}
	\left(2i\pi 
	h^{*}\right)\left(\phi_{\mathrm{HC}} L\right)\left(2i\pi h^{*}\right)\mu\left(\sqrt{4\pi^{2}B\left(h^{*},h^{*}\right)+A}\right)\right]\left(h\right)\\
	=\pi^{\mathfrak h,\mathfrak 
	z\left(\gamma\right)}\left(-2\pi h\right)\breve{\mathcal{F}}^{\mathfrak 
	z_{i}\left(\gamma\right)}\left[
	L^{\mathfrak z\left(\gamma\right)}\left(2i\pi 
	h^{*}\right)\mu\left(\sqrt{4\pi^{2}B\left(h^{*},h^{*}\right)+A}\right)\right]
	\left(h\right).
\end{multline}
By (\ref{eq:inva24a4}), (\ref{eq:ibva24a5}), we get 
(\ref{eq:inva24a2}). The proof of our proposition is completed. 
\end{proof}
\subsection{The limit of certain orbital integrals}%
\label{subsec:aadif}
We use the notation of Subsection \ref{subsec:jgnore}. As we saw in 
Subsection \ref{subsec:cenreg}, 
$\Tr^{\left[\gamma'\right]}\left[\mu\left(\sqrt{C^{\mathfrak 
g,X}+A}\right)\right]$ is a smooth function of $\gamma'\in 
H^{\mathrm{reg}}$. Recall that $\overline{\pi}^{\mathfrak h,\mathfrak 
z\left(\gamma\right)}$ is a differential operator on $\mathfrak 
h_{\C}$. To make our formulas clearer, we will denote its action in 
the variables $b$ or $h$ as $\overline{\pi}^{\mathfrak h,\mathfrak 
z\left(\gamma\right)}_{b}$ or $\overline{\pi}^{\mathfrak h,\mathfrak 
z\left(\gamma\right)}_{h}$. 
\begin{proposition}\label{prop:tcohash}
	The following identity holds on $H^{\mathrm{reg}}$:
	\begin{multline}\label{eq:inva25a8}
		\overline{\pi}^{\mathfrak h,\mathfrak 
		z\left(\gamma\right)}\left[
D_{H}\left(\gamma'\right)
\Tr^{\left[\gamma'\right]}\left[L\mu\left(\sqrt{C^{\mathfrak g,X}+A}\right) \right]\right]\\
=\left(-1\right)^{\left\vert  R_{+}\left(\gamma\right)\right\vert}\overline{\pi}^{\mathfrak h,\mathfrak 
z\left(\gamma\right)}_{h}\left(\phi_{\mathrm{HC}} L\right)\mu\left(\sqrt{\phi_{\mathrm{HC}} C^{\mathfrak g}+A}\right) \\
\left[D_{H}\left(\gamma'\right)\mathcal{J}_{\gamma'}\left(h_{\mathfrak k}\right)
\Tr^{E}\left[\rho^{E}\left(k^{-1}e^{b_{\mathfrak k}-h_{\mathfrak k}} 
\right) \right]\delta_{a'}\right]\left(0\right)\\
\\
=\left(-2\pi\right)^{\left\vert  
R_{+}\left(\gamma\right)\right\vert}\int_{i \mathfrak  h_{\mathfrak 
k}}^{}L^{\mathfrak z\left(\gamma\right)}
\mu\left(\sqrt{\left(C^{\mathfrak g}\right)^{\mathfrak 
z\left(\gamma\right)}+A}\right)\left(-a- b_{\mathfrak 
p}-h_{\mathfrak k}\right)\\
\pi^{\mathfrak h, \mathfrak  z\left(\gamma\right)}
\left(b_{\mathfrak p}+h_{\mathfrak 
k}\right)D_{H}\left(\gamma'\right)\mathcal{J}_{\gamma'}\left(h_{\mathfrak k}\right)
\Tr^{E}\left[\rho^{E}\left(k^{ -1}e^{b_{\mathfrak k}-h_{\mathfrak k}} \right) \right] dh_{\mathfrak k}.
\end{multline}
\end{proposition}
\begin{proof}
	As we saw in Theorem \ref{thm:caregbi}, the function 
$\left(\gamma',h_{\mathfrak k}\right)\in H^{\mathrm{reg}}\times i\mathfrak 
h_{\mathfrak k}\to\mathcal{J}_{\gamma'}\left(h_{\mathfrak 
k}\right)\in\C$ is smooth. If $\gamma'\in H^{\mathrm{reg}}$, then 
$\mathfrak z\left(\gamma'\right)=\mathfrak h$, so that by equation 
(\ref{eq:diff2}) in Theorem 
\ref{thm:Tgenorb}, we get
\begin{multline}\label{eq:inva24a6}
\overline{\pi}^{\mathfrak h,\mathfrak z\left(\gamma\right)}\left[
D_{H}\left(\gamma'\right)\Tr^{\left[\gamma'\right]}\left[L\mu\left(\sqrt{C^{\mathfrak g,X}+A}\right)\right]\right]\\
=
\left(\phi_{\mathrm{HC}} L\right)\mu\left(\sqrt{\phi_{\mathrm{HC}} C^{\mathfrak g}+A}\right)
\\
\overline{\pi}^{\mathfrak h,\mathfrak 
z\left(\gamma\right)}_{b}\left[D_{H}\left(\gamma'\right)\mathcal{J}_{\gamma'}\left(h_{\mathfrak 
k}\right)
\Tr^{E}\left[\rho^{E}\left(k^{-1}e^{b_{\mathfrak k}-h_{\mathfrak k}} 
\right) \right]\delta_{a'}\right]\left(0\right).
\end{multline}
In the right hand-side of (\ref{eq:inva24a6}), the differential 
operator 
$\overline{\pi}^{\mathfrak h, \mathfrak z\left(\gamma\right)}_{b}$ acts 
on the distribution on the right in the variables $b=\left(b_{\mathfrak 
p}, b_{\mathfrak k}\right)\in \mathfrak h$, and 
not on the variable $h_{\mathfrak k}$. The situation is actually 
strictly similar to what we already met in the proofs of  Proposition 
\ref{prop:pdifre} and Theorem 
\ref{thm:Tgenorb}. We obtain this way the first identity in 
(\ref{eq:inva25a8}). 

By Proposition \ref{prop:keha},  using the conventions in  
(\ref{eq:tann1}), we get
\begin{multline}\label{eq:capr24a8b0}
\left(-1\right)^{\left\vert  R_{+}\left(\gamma\right)\right\vert}\overline{\pi}^{\mathfrak h,\mathfrak 
z\left(\gamma\right)}\left(\phi_{\mathrm{HC}} L\right)\mu\left(\sqrt{\phi_{\mathrm{HC}} C^{\mathfrak g}+A}\right) \\
\left[D_{H}\left(\gamma'\right)\mathcal{J}_{\gamma'}\left(h_{\mathfrak k}\right)
\Tr^{E}\left[\rho^{E}\left(k^{-1}e^{b_{\mathfrak k}-h_{\mathfrak k}} 
\right) \right]\delta_{a'}\right]\left(0\right)\\
=\left(-2\pi\right)^{\left\vert  
R_{+}\left(\gamma\right)\right\vert}\int_{i \mathfrak  h_{\mathfrak 
k}}^{} L^{\mathfrak z\left(\gamma\right)}
\mu\left(\sqrt{\left( C^{\mathfrak g}\right)^{\mathfrak 
z\left(\gamma\right)}+A}\right)\left(-a- b_{\mathfrak 
p}-h_{\mathfrak k}\right)\\
\pi^{\mathfrak h, \mathfrak  z\left(\gamma\right)}
\left(a+b_{\mathfrak p}+h_{\mathfrak 
k}\right)D_{H}\left(\gamma'\right)\mathcal{J}_{\gamma'}\left(h_{\mathfrak k}\right)
\Tr^{E}\left[\rho^{E}\left(k^{ -1}e^{b_{\mathfrak k}-h_{\mathfrak k}} \right) \right] dh_{\mathfrak k}.
\end{multline}
By (\ref{eq:auj1a2}),  we get
\begin{equation}\label{eq:24a6b2}
\pi^{\mathfrak h,\mathfrak z\left(\gamma\right)}\left(a+b_{\mathfrak p}+h_{\mathfrak k}\right)=\pi^{\mathfrak h, \mathfrak 
z\left(\gamma\right)}\left(b_{\mathfrak p}+h_{\mathfrak k}\right).
\end{equation}
   By the first identity in (\ref{eq:inva25a8}) and  by 
   (\ref{eq:capr24a8b0}), (\ref{eq:24a6b2}), we get the second 
   identity in (\ref{eq:inva25a8}). The proof of our proposition is completed. 
\end{proof}

Let 
\index{Tg@$T\left(\gamma\right)$}%
$T\left(\gamma\right)$ be a maximal torus in 
$K^{0}\left(\gamma\right)$,  let 
\index{tg@$\mathfrak t\left(\gamma\right)$}%
$\mathfrak t\left(\gamma\right) 
\subset \mathfrak k\left(\gamma\right)$ be the corresponding Lie 
algebra, 
and let 
\index{Wg@$W\left(\mathfrak t\left(\gamma\right):\mathfrak 
k\left(\gamma\right)\right)$}%
$W\left(\mathfrak t\left(\gamma\right):\mathfrak 
k\left(\gamma\right)\right)$\footnote{We use this 
notation instead of $W\left(\mathfrak t\left(\gamma\right)_{\C}:\mathfrak k\left(\gamma\right)_{\C}\right)$ 
because this Weyl group is real.} be the corresponding Weyl group.
\begin{theorem}\label{thm:limhar}
	If $\mathfrak h$ is not the fundamental Cartan subalgebra of 
	$\mathfrak z\left(\gamma\right)$, as $b\in \mathfrak h$ 
	$\gamma$-regular tends to $0$, then
	\begin{equation}\label{eq:gogo1}
\overline{\pi}^{\mathfrak h,\mathfrak z\left(\gamma\right)}
\Biggl[D_{H}\left(\gamma'\right)\\
\Tr^{\left[\gamma'\right]}\left[L\mu\left(\sqrt{C^{\mathfrak 
g,X}+A}\right)\right]\Biggr]\to 0.
\end{equation}

	If $\mathfrak h$ is the fundamental Cartan subalgebra of 
	$\mathfrak z\left(\gamma\right)$, as $b\in \mathfrak h$  
	$\gamma$-regular tends to $0$, then
	\begin{multline}\label{eq:gong1}
\overline{\pi}^{\mathfrak h,\mathfrak z\left(\gamma\right)}
\Biggl[D_{H}\left(\gamma'\right)
\Tr^{\left[\gamma'\right]}\left[L\mu\left(\sqrt{C^{\mathfrak g,X}+A}\right)\right]\Biggr]\\
\to  \left(-1\right)^{\frac{1}{2}\left(\dim \mathfrak 
p\left(\gamma\right)-\dim \mathfrak h_{\mathfrak p}\right)}\frac{\left\vert  
W\left(\mathfrak t\left(\gamma\right):\mathfrak k\left(\gamma\right)\right)\right\vert}{\Vol\left(K^{0}\left(\gamma\right)/T\left(\gamma\right)\right)}\prod_{\alpha\in 
 R_{\mathfrak 
+}^{\mathrm{im}}\left(k\right)}^{}\xi_{\alpha}^{1/2}\left(k^{-1}\right)\\
D_{H}\left(\gamma\right) 
\left(2\pi\right)^{\left\vert  R_{+}\left(\gamma\right)\right\vert}
\int_{i \mathfrak 
k\left(\gamma\right)}^{}L^{\mathfrak z\left(\gamma\right)}\mu\left(\sqrt{\left( C^{\mathfrak 
g}\right)^{\mathfrak 
z\left(\gamma\right)}+A}\right)\left(-a-\Yok\right)\\
\mathcal{J}_{\gamma}\left(\Yok\right)
\Tr^{E}\left[\rho^{E}\left(k^{-1}e^{-\Yok} \right)\right]d\Yok.
\end{multline}
\end{theorem}
\begin{proof}
	First, we consider the case where $\mathfrak h$ 
is not the fundamental Cartan subalgebra of $\mathfrak 
z\left(\gamma\right)$. Using equations (\ref{eq:ober2}), 
	(\ref{eq:capr14}),     
(\ref{eq:gena5}),  
(\ref{eq:inva25a8}), and dominated convergence, 
we get (\ref{eq:gogo1}).

	Assume now that $\mathfrak h$ is the fundamental Cartan 
	subalgebra of $\mathfrak z\left(\gamma\right)$. Using equations  
	(\ref{eq:ober2}), 
	(\ref{eq:capr14}),   (\ref{eq:gena19}),  (\ref{eq:inva25a8}),  since the 
	convergence in (\ref{eq:gena19}) takes place except on a Lebesgue 
	negligible set of $i \mathfrak h_{\mathfrak k}$, we can use 
	dominated convergence, so that as 
	$b\in \mathfrak h$ $\gamma$-regular tends to $0$, 
	\begin{multline}\label{eq:gong2}
\overline{\pi}^{\mathfrak h,\mathfrak z\left(\gamma\right)}\left[D_{H}\left(\gamma'\right)
\Tr^{\left[\gamma'\right]}\left[L\mu\left(\sqrt{C^{\mathfrak g,X}+A}\right)\right]\right]\\
\to \left(-1\right)^{\left\vert  R_{+}\left(\gamma\right)\right\vert+\frac{1}{2}\left\vert  
R^{\mathrm{c}}_{+}\left(\gamma\right)\right\vert+\left\vert  R^{\mathrm{im}}_{\mathfrak 
k,+}\left(k\right)\right\vert}\prod_{\alpha\in 
R^{\mathrm{im}}_{+}\left(k\right)}^{}\xi_{\alpha}^{1/2}\left(k^{-1}\right)
D_{H}\left(\gamma\right) \\
\left(2\pi\right)^{\left\vert  
R_{+}\left(\gamma\right)\right\vert}\int_{i \mathfrak h_{\mathfrak k}}^{}L^{\mathfrak z\left(\gamma\right)}  \mu
\left(\sqrt{\left( C^{\mathfrak g} \right) ^{\mathfrak 
z\left(\gamma\right)}+A}\right)\left(-a-h_{\mathfrak k}\right)\\
\mathcal{J}_{\gamma}\left(h_{\mathfrak k}\right)
\Tr^{E}\left[\rho^{E}\left(k^{-1}e^{-h_{\mathfrak k}}\right) \right]\left[\pi^{\mathfrak 
h_{\mathfrak k}, \mathfrak k\left(\gamma\right)}\left(h_{\mathfrak 
k}\right)\right]^{2}
dh_{\mathfrak k}.
\end{multline}

The  operators $L^{\mathfrak 
z\left(\gamma\right)}$ and $\left( C^{\mathfrak 
g}\right)^{\mathfrak z\left(\gamma\right)}$ on $\mathfrak z\left(\gamma\right)$ are both 
$K\left(\gamma\right)$-invariant, and so the smooth function
$$ L^{\mathfrak z\left(\gamma\right)}  \mu
\left(\sqrt{\left( C^{\mathfrak g} \right) ^{\mathfrak 
z\left(\gamma\right)}+A}\right)\left(-a-\Yok\right)$$
on $i \mathfrak k\left(\gamma\right)$ is also $K\left(\gamma\right)$-invariant. This is also the case for the function
$$\mathcal{J}_{\gamma}\left(\Yok\right)\Tr^{E}\left[\rho^{E}\left(k^{-1}e^{-\Yok}\right)\right].$$
Since $\mathfrak h$ is the fundamental Cartan subalgebra of 
$\mathfrak z\left(\gamma\right)$, $\mathfrak h_{\mathfrak k}$ is a 
Cartan subalgebra of $\mathfrak k\left(\gamma\right)$. By Weyl's 
integration formula,  and taking into account the fact that on 
$ i \mathfrak h_{\mathfrak k}$, $\left[\pi^{\mathfrak h_{\mathfrak 
k}, \mathfrak 
k\left(\gamma\right)}\left(h_{\mathfrak k}\right)\right]^{2}$ is 
nonnegative, we obtain
\begin{multline}\label{eq:gong3}
	\int_{i \mathfrak 
k\left(\gamma\right)}^{} L^{\mathfrak z\left(\gamma\right)}\mu\left(\sqrt{\left( C^{\mathfrak 
g}\right)^{\mathfrak z\left(\gamma\right)}+A}\right)\left(-a-\Yok\right)\\
\mathcal{J}_{\gamma}\left(\Yok\right)
\Tr^{E}\left[\rho^{E}\left(k^{-1}e^{-\Yok} \right)\right]d\Yok\\
=\frac{\Vol\left(K^{0}\left(\gamma\right)/T\left(\gamma\right)\right)}{\left\vert  W\left(\mathfrak t\left(\gamma\right):\mathfrak k\left(\gamma\right)\right)\right\vert}
\int_{i \mathfrak 
h_{\mathfrak k}}^{} L^{\mathfrak z\left(\gamma\right)}
\mu\left(\sqrt{\left( C^{\mathfrak g}\right)^{\mathfrak 
z\left(\gamma\right)}+A}\right)\left(-a-h_{\mathfrak k}\right)\\
\mathcal{J}_{\gamma}\left(h_{\mathfrak k}\right)
\Tr^{E}\left[\rho^{E}\left(k^{-1}e^{-h_{\mathfrak k}}\right)\right]
\left[\pi^{\mathfrak h_{\mathfrak k},\mathfrak 
k\left(\gamma\right)}\left(h_{\mathfrak k}\right)\right]^{2}dh_{\mathfrak k}.
\end{multline}

Using in particular Proposition \ref{prop:Ps} applied to $\mathfrak z\left(\gamma\right)$, we have the identities,
\begin{align}\label{eq:cong3az1}
&\left\vert  R_{+}\left(\gamma\right)\right\vert=\dim \mathfrak 
c_{+}\left(\gamma\right)+\frac{1}{2}\dim \mathfrak i\left(k\right), 
\nonumber \\
&\left\vert  R_{\mathfrak 
k,+}^{\mathrm{im}}\left(k\right)\right\vert=\frac{1}{2}\dim \mathfrak i_{\mathfrak 
k}\left(k\right),\\
&\left\vert  R^{\mathrm{c}}_{+}\left(\gamma\right)\right\vert=
\dim \mathfrak c_{+}\left(\gamma\right). \nonumber 
\end{align}
By (\ref{eq:cong3az1}),  we deduce that
	\begin{multline}\label{eq:grup2a1}
\left\vert  R_{+}\left(\gamma\right)\right\vert+\left\vert  
R_{\mathfrak k,+}^{\mathrm{im}}\left(k\right)\right\vert+\frac{1}{2}\left\vert  
R_{+}^{\mathrm{c}}\left(\gamma\right)\right\vert \\
=
\dim \mathfrak c_{+}\left(\gamma\right)+\frac{1}{2}\dim \mathfrak 
i_{\mathfrak k}\left( k\right)+\frac{1}{2}\dim \mathfrak i\left(k\right)+\frac{1}{2}\dim \mathfrak c_{+}\left(\gamma\right).
\end{multline}
By Proposition \ref{prop:Ps}, $\mathfrak c_{+}\left(\gamma\right)$ 
has even dimension. This is also the case for $\mathfrak i_{\mathfrak 
k}\left(k\right)$. By (\ref{eq:grup2a1}), we get the equality 
$\mathrm{mod}\, 2$, 
\begin{equation}\label{eq:grup3}
\left\vert  R_{+}\left(\gamma\right)\right\vert+\left\vert  
R_{\mathfrak k,+}^{\mathrm{im}}\left(k\right)\right\vert+\frac{1}{2}\left\vert  
R_{+}^{\mathrm{c}}\left(\gamma\right)\right\vert=\frac{1}{2}\dim 
\mathfrak i_{\mathfrak p}\left(k\right)+\frac{1}{4}\dim \mathfrak 
c\left(\gamma\right).
\end{equation}
By equation (\ref{eq:grup1}) in Proposition \ref{prop:peven} 
	applied to $\mathfrak z\left(\gamma\right)$, since $\mathfrak h$ is 
	fundamental in $\mathfrak z\left(\gamma\right)$, we get
	\begin{equation}\label{eq:grup3ax1}
\dim \mathfrak p\left(\gamma\right)-\dim \mathfrak h_{\mathfrak 
p}=\dim \mathfrak i_{\mathfrak p}\left(k\right)+\frac{1}{2}\dim 
\mathfrak c\left(\gamma\right).
\end{equation}
By (\ref{eq:grup3}), (\ref{eq:grup3ax1}), we get
	\begin{equation}\label{eq:grup4}
\left(-1\right)^{\left\vert  R_{+}\left(\gamma\right)\right\vert+\left\vert  
R_{\mathfrak k,+}^{\mathrm{im}}\left(k\right)\right\vert+\frac{1}{2}\left\vert  
R_{+}^{\mathrm{c}}\left(\gamma\right)\right\vert}=\left(-1\right)^{\frac{1}{2}\left(\dim \mathfrak 
p\left(\gamma\right)-\dim \mathfrak h_{\mathfrak p}\right)}.
\end{equation}
When $\gamma=1$, the above identity had been established by 
\cite[Lemma 18]{Harish64c}. 

By (\ref{eq:gong2}), (\ref{eq:gong3}), and (\ref{eq:grup4}),  we get 
(\ref{eq:gong1}). The proof of our theorem is completed. 
\end{proof}
\section{The final formula}%
\label{sec:senUg}
In this section, we establish our final formula in the case of a 
non necessarily regular semisimple element $\gamma\in G$. Our formula 
extends both  the formula in Theorem \ref{thm:Ttrfin} valid for 
$\gamma$ semisimple 
and $L=1$, and the formula in Theorem \ref{thm:Tgenorb} valid for $\gamma$ regular. 
To establish our main result, we combine a fundamental result of 
Harish-Chandra with the results we obtained in Section 
\ref{sec:rojg}. Also, along the lines of \cite[Chapter 6]{Bismut08b}, we give a wave kernel formulation of our main 
result. 

This section is organized as follows. In Subsection 
\ref{subsec:geca}, we establish our main result. 

In Subsection \ref{subsec:micro}, as in \cite{Bismut08b}, we reformulate our main  result in terms 
of wave kernels. 

Finally, in Subsection \ref{subsec:natu}, we verify our main formula 
is compatible to natural operations on orbital integrals.
\subsection{The general case}%
\label{subsec:geca}
Let $L\in Z\left(\mathfrak g\right)$. Here, we take  $\gamma\in G$ semisimple as in 
(\ref{eq:dis4a}).
We will extend Theorem \ref{thm:Tgenorb} to non regular
$\gamma$. 
\begin{theorem}\label{thm:Tgenorbis}
The following identity holds:
	\begin{multline}\label{eq:diff2bis}
\Tr^{\left[\gamma\right]}\left[L\mu\left(\sqrt{C^{\mathfrak g,X}+A}\right)\right]\\
=
L^{\mathfrak z\left(\gamma\right)}
\mu\left(\sqrt{\left( C^{\mathfrak g}\right)^{\mathfrak 
z\left(\gamma\right)}+A}\right)\left[
\mathcal{J}_{\gamma}\left(\Yok\right)\Tr^{E}\left[\rho^{E}\left(k^{-1}e^{-\Yok}\right)\right]\delta_{a}\right]\left(0\right).
\end{multline}
\end{theorem}
\begin{proof}
		By a  result of Harish-Chandra \cite[Lemma 28]{Harish66}, 
		\cite[Part II, Section 12.5, Theorem 13]{Vara77},  we know that if 
		$\mathfrak h $ is the fundamental Cartan subalgebra in 
		$\mathfrak z\left(\gamma\right)$,  there is a 
	universal constant $c_{\gamma}$ depending only on $\gamma$ 
	such that with the notation in 
	Theorem \ref{thm:limhar}, as $b\in \mathfrak h$ $\gamma$-regular 
	tends to $0$, 	
	\begin{multline}\label{eq:gong7}
\overline{\pi}^{\mathfrak h,\mathfrak z\left(\gamma\right)}
\left[D_{H}\left(\gamma'\right)\Tr^{\left[\gamma'\right]}\left[L\mu\left(\sqrt{C^{\mathfrak g,X}+A}\right)\right]\right]\\
\to 
c_{\gamma}\Tr^{\left[\gamma\right]}\left[L
\mu\left(\sqrt{C^{\mathfrak g,X}+A}\right)\right].
\end{multline}
In Theorem \ref{thm:limhar}, we gave another proof of 
the existence of the limit in (\ref{eq:gong7}). Also by the 
fundamental result of \cite{Bismut08b} stated as Theorem 
\ref{thm:Ttrfin}, when $L=1$, the integral in the right-hand side of 
(\ref{eq:gong1}) coincides with the orbital integral 
$\Tr^{\left[\gamma\right]}\left[\mu\left(\sqrt{C^{\mathfrak 
g,X}+A}\right)\right]$. 
To identify the constant $c_{\gamma}$, we only need to prove that one 
of these last orbital integrals does not vanish. It is enough to 
take $E$ to be the trivial representation, and 
$\mu\left(x\right)=\exp\left(-x^{2}\right)$.  Since the scalar heat 
kernel on $X$ is positive, the corresponding orbital integrals do not 
vanish. So we find that
\begin{multline}\label{eq:gong8}
c_{\gamma}=\left(-1\right)^{\frac{1}{2}\left(\dim \mathfrak 
p\left(\gamma\right)-\dim \mathfrak h_{\mathfrak 
p}\right)}\frac{\left\vert  W\left(\mathfrak t\left(\gamma\right):\mathfrak k\left(\gamma\right)\right)\right\vert}{\Vol\left(K^{0}\left(\gamma\right)/T\left(\gamma\right)\right)}\\
\left(2\pi\right)^{\left\vert  R_{+}\left(\gamma\right)\right\vert}\prod_{\alpha\in R_{\mathfrak 
+}^{\mathrm{im}}\left(k\right)}^{}\xi_{\alpha}^{1/2}\left(k^{-1}\right)
D_{H}\left(\gamma\right).
\end{multline}
When $\gamma=1$, this computation has already been done by 
Harish-Chandra in  \cite[Section 37, 
Theorem 1]{Harish75}.  \footnote{More precisely, if  $G=KAN$ is the 
Iwasawa decomposition,  when $\gamma=1$, the constant obtained in 
[HC75] is $2^{\frac{\dim N}{2}}c_\gamma$. In [HC75, Section 7], 
Harish-Chandra uses another normalization for the Haar measure on $G$, 
which is adapted to the Iwasawa decomposition. By [HC75, p.202], the 
ratio of these two normalizations is given by  $2^{\frac{\dim 
N}{2}}$, which explains the discrepancy.} For the case of a general $\gamma$, this 
formula can also be derived from \cite[p. 34]{Harish66} and from 
the reference given before.

By combining (\ref{eq:gong1}), (\ref{eq:gong7}), and 
(\ref{eq:gong8}), we get (\ref{eq:diff2bis}). The proof of our 
theorem is completed. 
\end{proof}
\subsection{A microlocal version}%
\label{subsec:micro}
We still take  $\gamma\in G$ semisimple as in 
(\ref{eq:dis4a}).

We will proceed as in \cite[Section 6.3]{Bismut08b}, to which we 
refer for more details. In the sequel we identify $TX$ and $T^{*}X$ 
by the metric. 

Let
$\Tr^{\left[\gamma\right]}\left[L\cos\left(s\sqrt{C^{\mathfrak g,X}+A}\right)\right]$ be the even 
distribution on $\R$ such that 
for any  $\mu\in \mathcal{S}^{\mathrm{even}}\left(\R\right)$ with
$\widehat{\mu}$ having compact support,
\begin{equation}
    \Tr^{\left[\gamma\right]}\left[L\mu\left(\sqrt{C^{\mathfrak g,X}+A}\right)\right]=
    \int_{\R}^{}\widehat{\mu}\left(s\right)
    \Tr^{\left[\gamma\right]}\left[L\cos\left(2\pi s\sqrt{C^{\mathfrak g,X}+A}\right)\right]ds.
    \label{eq:trib4}
\end{equation}

The
 operator $L\cos\left(s\sqrt{C^{\mathfrak g,X}+A}\right)$ defines a 
distribution on $\R\times X\times X$. By finite propagation speed  \cite[section 
7.8]{ChazarainPiriou}, \cite[section 4.4]{Taylor81}, its 
support is included in $\left(s,x,x'\right),\left\vert  
s\right\vert\ge d\left(x,x'\right)$.  Let $\mathcal{X}$ be the total 
space of $TX$. Let   $s\in\R\to\varphi_{s}$ be the geodesic 
flow on 
$\mathcal{X}$.  Let 
$\tau$ be the variable dual to $s$. By
\cite[Theorem 23.1.4 and remark]{Hormander85a}, the
wave front set 
$\mathrm{WF}\left(L\cos\left(s\sqrt{C^{\mathfrak g,X}+A}\right)\right)$ of the distribution 
$L\cos\left(s\sqrt{C^{\mathfrak g,X}+A}\right)$ is the conic set in 
$\R^{2}\times T^{*}X\times T^{*}X$ generated by 
$\left(x',-Y'\right)=\varphi_{\pm s}\left(x,Y\right),\left\vert  
Y\right\vert=1,\tau=\pm 1$. Conic  means that the dilations by 
$\lambda>0$ are applied simultaneously to the variables $Y,Y',\tau$.

As explained in \cite[Section 3.4]{Bismut08b}, in the geodesic coordinate system centered at $x_{0}=p1$,   $\mathfrak 
p^{\perp}\left(\gamma\right)$ can be identified 
with a smooth 
submanifold 
$P^{\perp}\left(\gamma\right) $ of $X$. Let
\index{NPgX@$N_{P^{\perp}\left(\gamma\right)/X}$}%
$N_{P^{\perp}\left(\gamma\right)/X}$ be the orthogonal bundle to 
$TP^{\perp}\left(\gamma\right)$ in $TX$.

Set 
\index{DgX@$\Delta^{\gamma}_{X}$}%
\begin{equation}
    \Delta^{\gamma}_{X}=\left\{\left(x,\gamma x\right),x\in P^{\perp}\left(\gamma\right)\right\}.
    \label{eq:flimsdorf1}
\end{equation}
Then $\Delta^{\gamma}_{X}$ is a smooth submanifold of $X\times X$.
  The
conormal bundle to   $\R\times \Delta^{\gamma}_{X} \subset \R\times 
X\times X$ is  the set 
$\left(\left(s,\tau\right),\left(x,Y\right),\left(x',Y'\right)\right)
\in \R^{2}\times \mathcal{X}\times \mathcal{X}$ such  
that $\tau=0,x\in P^{\perp}\left(\gamma\right),x'=\gamma x, 
\gamma^{*}Y'+Y\in N_{P^{\perp}\left(\gamma\right)/X}$. 

By \cite[Theorem 8.2.10]{Hormander83a},    
$$L\cos\left(s\sqrt{C^{\mathfrak g,X}+A}\right)\Delta^{\gamma}_{X}$$
is 
a well-defined  distribution on $\R\times X\times X$, and its wave 
front set is the formal sum of the wave front sets  of the two above 
distributions. 
In particular
the pushforward of the distribution
$\Tr^{F}\left[\gamma L\cos\left( s\sqrt{C^{\mathfrak g,X}+A}\right)\right]$ by the projection $\R\times 
X\times X\to \R$ is well-defined. It will be denoted 
\begin{equation}
    \int_{\Delta^{\gamma}_{X}}\Tr^{F}\left[\gamma L\cos\left( s\sqrt{C^{\mathfrak g,X}+A}\right)\right].
    \label{eq:trib5}
\end{equation}
This is an even distribution on $\R$.

Tautologically, we have the identity of even distributions on $\R$, 
\begin{equation}
    \Tr^{\left[\gamma\right]}\left[L\cos\left(s\sqrt{C^{\mathfrak g,X}+A}\right)\right]=
    \int_{\Delta^{\gamma}_{X}}\Tr^{F}\left[\gamma L\cos \left( 
  s\sqrt{C^{\mathfrak g,X}+A}\right)\right].
    \label{eq:trib6}
\end{equation}

We have the result stablished in \cite[Proposition 6.3.1]{Bismut08b}.
\begin{proposition}\label{Ptribax}
The singular support of 
$  \Tr^{\left[\gamma\right]}\left[L\cos\left( s\sqrt{C^{\mathfrak g,X}+A}\right)\right]$
is included in  $s=\pm\left\vert  a\right\vert$, and the ordinary support is 
included in 
$$\left\{s\in\R,\left\vert  s\right\vert\ge\left\vert  
a\right\vert\right\}.$$ 
If $a=0$, if $\mathfrak 
p\left(\gamma\right)=0$,  the singular support of 
$$  
\Tr^{\left[\gamma\right]}\left[L\cos\left( s\sqrt{C^{\mathfrak 
g,X}+A}\right)\right]$$
is empty.
\end{proposition}

We define the even distribution on $\R$,
\begin{equation}
  L^{\mathfrak 
   z\left(\gamma\right)}\cos\left(s\sqrt{\left( C^{\mathfrak 
   g} \right) ^{\mathfrak z\left(\gamma\right)}+A}\right)\left[\mathcal{J}_{\gamma}\left(\Yok\right)\Tr^{E}\left[\rho^{E}\left(k^{-1}\
    e^{-\Yok} \right)\right] \delta_{a}\right] \left(0\right)
    \label{eq:bronx1}
\end{equation}
by the formula
\begin{multline}
    L^{\mathfrak z\left(\gamma\right)}\mu\left(\sqrt{\left( C^{\mathfrak g}\right)^{\mathfrak z\left(\gamma\right)}+A}
       \right)\left[\mathcal{J}_{\gamma}\left(\Yok\right)\Tr^{E}\left[\rho^{E}\left(k^{-1}e^{-\Yok}
   \right)\right]\delta_{a}\right]\left(0\right)\\
   =
   \int_{\R}^{}\widehat{\mu}\left(s\right)
L^{\mathfrak z\left(\gamma\right)}\cos\left(2\pi s\sqrt{\left(  C^{\mathfrak g}\right)^{\mathfrak 
 z\left(\gamma\right)}+A}\right)\\
 \left[\mathcal{J}_{\gamma}\left(\Yok\right)
 \Tr^{E}\left[\rho^{E}\left(k^{-1}
   e^{-\Yok}\right)\right]\delta_{a}\right]\left(0\right).
    \label{eq:brahms1}
\end{multline}

Let $z=\left(y,\Yok\right)$ be the generic element of 
$\mathfrak z_{i}\left(\gamma\right) =\mathfrak p\left(\gamma\right) 
\oplus i\mathfrak  k\left(\gamma\right)$. 
Using   finite propagation speed for the wave equation,  
$$L^{\mathfrak z\left(\gamma\right)}\cos\left( 
s\sqrt{\left( C^{\mathfrak g}\right)^{\mathfrak 
z\left(\gamma\right)}+A}\right)$$
is a distribution on $\R\times 
\mathfrak z_{i}\left(\gamma\right)\times \mathfrak z_{i}\left(\gamma\right)$ 
whose support is  included in $\left(s,z,z'\right),\left\vert  
s\right\vert\ge \left\vert  z'-z\right\vert$. Moreover, by
\cite[Theorem 23.1.4 and remark]{Hormander85a}, its 
wave front set is the conic 
set associated with 
$\left(y',-Y'\right)=\left(y\pm sY,Y\right),\left\vert  
Y\right\vert=1,\tau=\pm 1$. Conic set means again that the dilations 
by $\lambda>0$ are applied to the variables $Y,Y',\tau$.

Set
\index{Hg@$H^{\gamma}$}%
\begin{equation}
    H^{\gamma}=\left\{0\right\}\times \left(a,i\mathfrak 
    k\left(\gamma\right) \right) \subset \mathfrak z_{i}\left(\gamma\right)\times 
    \mathfrak z_{i}\left(\gamma\right).
    \label{eq:senignew0}
\end{equation}
  The wave front set associated with 
$\R\times H^{\gamma} \subset \R\times \mathfrak z_{i}\left(\gamma\right)
\times \mathfrak z_{i}\left(\gamma\right)$ is  such that $Y^{ \prime \mathfrak 
k\left(\gamma\right)}=0, \tau=0$, 
so that  the product 
$$L^{\mathfrak 
z\left(\gamma\right)}\cos\left(s\sqrt{\left( C^{\mathfrak 
g} \right)^{\mathfrak z\left(\gamma\right)}+A}\right)H^{\gamma}$$
is well-defined. 

The function $\mathcal{J}_{\gamma}\left(\Yok\right)\Tr^{E}\left[\rho^{E}\left(k^{-1}
e^{-\Yok}\right)\right]$ can be viewed as a 
smooth function on the second copy of $\mathfrak 
z_{i}\left(\gamma\right)$ in $\mathfrak z_{i}\left(\gamma\right)\times 
\mathfrak z_{i}\left(\gamma\right)$. It lifts to a smooth 
function on $\mathfrak z_{i}\left(\gamma\right)\times \mathfrak 
z_{i}\left(\gamma\right)$. 

Therefore, 
\begin{equation}
L^{\mathfrak z\left(\gamma\right)}\cos\left(s\sqrt{\left(C^{\mathfrak g}\right)^{\mathfrak 
   z\left(\gamma\right)}+A}\right)H^{\gamma}\mathcal{J}_{\gamma}\left(\Yok\right)\Tr^{E}\left[\rho^{E}\left(k^{-1}
   e^{-\Yok}\right)\right]
    \label{eq:tromb1}
\end{equation}
is a well-defined distribution on $\R\times \mathfrak 
z_{i}\left(\gamma\right)\times \mathfrak z_{i}\left(\gamma\right)$. The pushforward 
of this distribution by the projection $\R\times \mathfrak z_{i}\left(\gamma\right)
\times \mathfrak z_{i}\left(\gamma\right)\to \R$ is denoted
\begin{equation}
    \int_{H^{\gamma}}^{}L^{\mathfrak 
	z\left(\gamma\right)}\cos\left(s\sqrt{\left( C^{\mathfrak 
	g}\right)^{\mathfrak 
	z\left(\gamma\right)}+A}\right)\mathcal{J}_{\gamma}\left(\Yok\right)\Tr^{E}\left[\rho^{E}\left(k^{-1}e^{-\Yok} \right) \right].
    \label{eq:trib7}
\end{equation}
This is an even distribution  supported in $\left\vert  
s\right\vert\ge  \left\vert  a\right\vert$, with singular support  
included in   $s=\pm  
\left\vert  a\right\vert$. Note that if $a=0$ and if $\mathfrak 
p\left(\gamma\right)=0$, the singular support of this distribution is 
empty.
\begin{theorem}\label{TFourier}
We have the identity  of even distributions on $\R$ supported on 
$\left\vert  s\right\vert\ge \left\vert  a\right\vert$ with singular 
support included in $\pm \left\vert  a\right\vert$,
\begin{multline}
    \int_{\Delta^{\gamma}_{X}}\Tr^{F}\left[\gamma L\cos\left(s\sqrt{C^{\mathfrak g,X}+A}\right)\right]
    \\
    =
   \int_{H^{\gamma}}^{}L^{\mathfrak 
   z\left(\gamma\right)}\cos\left(s\sqrt{\left( C^{\mathfrak 
   g}\right)^{\mathfrak 
   z\left(\gamma\right)}+A}\right)\mathcal{J}_{\gamma}\left(\Yok\right)\Tr^{E}\left[\rho^{E}\left(k^{-1}
  e^{-\Yok}\right)\right].
    \label{eq:trib8}
\end{multline}
\end{theorem}
\begin{proof}
	We use Theorem \ref{thm:Tgenorbis}, and we proceed as in the 
	proof of \cite[Theorem 6.3.2]{Bismut08b}.
\end{proof}

\subsection{Compatibility properties of the formula}%
\label{subsec:natu}
Let us give a direct proof that the right-hand side of 
(\ref{eq:diff2bis}) is invariant by conjugation 
of $\gamma$ in $G$. 
Indeed let $\gamma,\gamma'$ be two conjugate elements  in $G$ as in 
Theorem \ref{thm:Tincl}. By this theorem,
 they are also 
conjugate by an element $k''$ of $K$, and equation 
(\ref{eq:dis4b}) holds. Since the character 
of the representation $\rho^{E}$ is invariant by conjugation by 
elements of $K$, the right hand-sides of (\ref{eq:diff2bis}) associated with 
$\gamma,\gamma'$ coincide.

We will denote  the dependence of our orbital integrals on 
$E$
with an extra superscript $E$. If $L\in Z\left(\mathfrak g\right)$, 
by Theorem \ref{thm:pinv}, the $L_{2}$ transpose of $L$ is 
just $\sigma\left(L\right)$. Observe that $C^{\mathfrak g,X}$ is 
symmetric, i.e., it is equal to its transpose. Then one has the easy 
formula
\begin{align}\label{eq:gong9}
&\Tr^{\left[\gamma^{-1}\right],E^{*}}\left[\sigma\left(L\right)\mu
\left(\sqrt{C^{\mathfrak g,X}+A}\right)\right]
=\Tr^{\left[\gamma\right],E}\left[L\mu\left(\sqrt{C^{\mathfrak 
g,X}+A}\right)\right],\\
&\overline{\Tr^{\left[\gamma\right],E}\left[L\mu\left(\sqrt{C^{\mathfrak g,X}+A}\right)\right]}=\Tr^{\left[\gamma\right],E^{*}}\left[L
\mu\left(\sqrt{C^{\mathfrak g,X}+A}\right)\right]. \nonumber 
\end{align}
Using the  identities in  (\ref{eq:ida1}), we can recover 
(\ref{eq:gong9}) from (\ref{eq:diff2bis}). 

Finally, it is easy to verify that, as it should be,  our formula is unchanged when 
replacing $\gamma,L$ by $\theta\gamma,\theta L$. 
\section{Orbital integrals and the index theorem}%
\label{sec:index}
The purpose of this section is to verify the compatibility of our 
formula for orbital integrals with the index theorem of Atiyah-Singer 
\cite{AtiyahSinger68,AtiyahSinger68b},  to the Lefschetz formulas 
of \cite{AtiyahBott67,AtiyahBott68} for Dirac operators, to the index 
formula of Kawasaki \cite{Kawasaki79}. More precisely we extend to 
the case of an arbitrary $L$ what was done in \cite[Chapter 
7]{Bismut08b} in the case $L=1$. Also we verify the compatibility of 
our results with results of Huang-Pand\v{z}i\'{c} \cite{HuangPan02}  
who established the 
Vogan conjecture on Dirac cohomology.

This section is organized as follows. In Subsection 
\ref{subsec:dirX}, we construct the Dirac operator $D^{X}$ on the 
symmetric space $X$. 

In Subsection \ref{subsec:dimb}, we introduce the relevant notation 
when  $G$ and 
$K$ have the same complex rank.

In Subsection \ref{subsec:orbi}, we evaluate the orbital integrals 
associated with the index theorem for Dirac operators when $\gamma$ 
semisimple is nonelliptic, and also when $\gamma=1$. 

In Subsection \ref{subsec:dimbb}, when $\gamma$ is elliptic, we 
consider again the case  where the difference of complex ranks is 
still equal to $0$. 

In Subsection \ref{subsec:ellel}, we evaluate the orbital integrals  
associated with the index theorem for the Dirac operator. 

Finally, in Subsection \ref{subsec:resHP}, we verify the 
compatibility of our results with the results of Huang-Pand\v{z}i\'{c}  \cite{HuangPan02}.
\subsection{The Dirac operator on $X$}%
\label{subsec:dirX}
Here, we use the notation of Section \ref{sec:geofo}. We assume that 
$K$ is  simply connected, and also that $\mathfrak p$ is even-dimensional and oriented. Let $c\left(\mathfrak 
p\right)$ be the Clifford algebra associated with 
$\left(\mathfrak p,B\vert_{\mathfrak p}\right)$. 

As explained in \cite[Section 
7.2]{Bismut08b}, the representation $\rho^{\mathfrak p}: K\to 
\mathrm{SO}\left(\mathfrak p\right)$ lifts to a representation $K\to 
\Aut^{\mathrm{even}}\left(S^{\mathfrak p}\right)$, where 
$S^{\mathfrak p}=S^{\mathfrak p}_{+} \oplus S^{\mathfrak p}_{-}$ is 
the $\Z_{2}$-graded Hermitian vector space of $\mathfrak p$-spinors. 
We have the identification of $\Z_{2}$-graded algebras,
\begin{equation}\label{eq:clif1}
c\left(\mathfrak p\right) \otimes _{\R}\C=\End\left(S^{\mathfrak 
p}\right).
\end{equation}

Set
 \begin{equation}\label{eq:clif3}
	 S^{TX}=G\times_{K} S^{\mathfrak p}.
\end{equation}
The $\Z_{2}$-graded vector bundle   $S^{TX}$ inherits a 
unitary connection $\n^{S^{TX}}$. 

Let $\n^{S^{TX}\otimes F}$ be 
the  connection on $S^{TX} \otimes F$ associated with $\n^{S^{TX}}, 
\n^{F}$.  

Recall that $C^{\mathfrak k,E}$ descends to a parallel section 
$C^{\mathfrak k,F }$ of $\End\left(F\right)$. Here, $C^{\mathfrak 
g,X}$ denotes the action of $C^{\mathfrak g}$ on $C^{\infty 
}\left(X,S^{TX} \otimes F\right)$.
 
Here, $D^{X}$ denotes 
the Dirac operator acting on $C^{\infty}(X,S^{TX}\otimes F)$. If 
$e_{1},\ldots,e_{m}$ is an orthonormal basis of $TX$ then
\begin{equation}\label{eq:clif2a1}
D^{X}=\sum_{1}^{n}c\left(e_{i}\right)\n^{S^{TX} \otimes F}_{e_{i}}.
\end{equation}

Let 
$\mathfrak h  \subset \mathfrak g$ be a 
$\theta$-stable fundamental Cartan subalgebra of $\mathfrak g$. We will use the notation
\index{b@$\mathfrak b$}%
\begin{align}\label{eq:clif2}
&\mathfrak b=\mathfrak h_{\mathfrak p}, &\mathfrak t=\mathfrak 
h_{\mathfrak k}.
\end{align}
Then $\mathfrak t \subset \mathfrak k$ is the Lie algebra of a maximal torus $T \subset 
K$. Also $\dim \mathfrak t,\dim \mathfrak h$ are the complex ranks of 
$K,G$, and $\dim \mathfrak b$ is the difference of these complex 
ranks. Since $m$ is even, $\dim \mathfrak b$ is also even.  Let 
$\phi_{\mathrm{HC}}: Z\left(\mathfrak g\right) \simeq  I\ac\left(\mathfrak 
h, \mathfrak g\right)$ be the corresponding isomorphism of 
Harish-Chandra.

We fix a system of positive roots in $i \mathfrak  t^{*}$ associated 
with the pair $\left(\mathfrak t,\mathfrak k\right)$. In particular
\index{rk@$\rho^{\mathfrak k}$}%
$\rho^{\mathfrak k}\in i \mathfrak t^{*}$ is calculated with respect 
to this system.

By \cite[eqs. (7.2.8) and (7.2.9)]{Bismut08b} and by 
(\ref{eq:fina-1}), we get
\begin{equation}\label{eq:clif4}
D^{X,2}=C^{\mathfrak g,X}-B^{*}\left(\rho^{\mathfrak 
g},\rho^{\mathfrak g}\right)+B^{*}\left(\rho^{\mathfrak 
k},\rho^{\mathfrak k}\right)-C^{\mathfrak k,F}.
\end{equation}

We may and we will assume that $\rho^{E}$ is an  irreducible representation of $K$ 
with  dominant weight $\lambda\in i \mathfrak t^{*}$. Then
\begin{equation}\label{eq:clif6}
C^{\mathfrak k,E}=-B^{*}\left(\rho^{\mathfrak 
k}+\lambda,\rho^{\mathfrak 
k}+\lambda\right)+B^{*}\left(\rho^{\mathfrak k},\rho^{\mathfrak 
k}\right).
\end{equation}

By (\ref{eq:clif4}), (\ref{eq:clif6}), we get
\begin{equation}\label{eq:clif7}
D^{X,2}=C^{\mathfrak g,X}-B^{*}\left(\rho^{\mathfrak 
g},\rho^{\mathfrak g}\right)+B^{*}\left(\rho^{\mathfrak 
k}+\lambda,\rho^{\mathfrak k}+\lambda\right).
\end{equation}
By (\ref{eq:inva9}), we can rewrite (\ref{eq:clif7}) in the form
\begin{equation}\label{eq:clif7a1}
D^{X,2}=C^{\mathfrak g,X}-\phi_{\mathrm{HC}}C^{\mathfrak 
g}\left(  \rho^{\mathfrak k}+\lambda \right).
\end{equation}
\subsection{The case where $\dim \mathfrak b=0$}%
\label{subsec:dimb}
In this Subsection, we assume that $\dim \mathfrak b=0$. Then $\mathfrak h= \mathfrak t$
is a fundamental Cartan subalgebra of $\mathfrak g$, and 
\begin{equation}\label{eq:clif13a1}
R=R^{\mathrm{im}}.
\end{equation}
Also $R^{\mathrm{im}}_{\mathfrak k}$ is just the root system 
associated with the pair $\left(\mathfrak t, \mathfrak k\right)$. We 
fix a positive root system 
\index{Rimkp@$R^{\mathrm{im}}_{\mathfrak k,+}$}%
$R^{\mathrm{im}}_{\mathfrak k,+} \subset 
R^{\mathrm{im}}_{\mathfrak k}$, and 
a positive weight system 
\index{Rpp@$R^{\mathrm{im}}_{\mathfrak p,+}$}%
$R^{\mathrm{im}}_{\mathfrak p,+} \subset R_{\mathfrak 
p}^{\mathrm{im}}$ which is compatible with the orientation of 
$\mathfrak p$, so that
\index{Rim@$R_{+}^{\mathrm{im}}$}%
$R_{+}^{\mathrm{im}}=R^{\mathrm{im}}_{\mathfrak p,+}\cup 
R^{\mathrm{im}}_{\mathfrak k,+}$ is a positive root system for the 
pair $\left(\mathfrak t,\mathfrak g\right)$. 

The functions 
\index{ptg@$\pi^{\mathfrak t, \mathfrak g}$}%
\index{ptk@$\pi^{\mathfrak t, \mathfrak k}$}%
$\pi^{\mathfrak t, \mathfrak g},\pi^{\mathfrak t, 
\mathfrak k}$ on $\mathfrak t$ are given by
\begin{align}\label{eq:clif13a3}
&\pi^{\mathfrak t, \mathfrak g}\left(h\right)=\prod_{\alpha\in 
R^{\mathrm{im}}_{+}}^{}\left\langle  \alpha,h\right\rangle,
&\pi^{\mathfrak t, \mathfrak k}\left(h\right)=\prod_{\alpha\in 
R^{\mathrm{im}}_{\mathfrak k,+}}^{}\left\langle  \alpha,h\right\rangle.
\end{align}

Here,  $\rho^{\mathfrak k}, \lambda\in i \mathfrak  t^{*}$ are calculated with this
choice of $R^{\mathrm{im}}_{\mathfrak k,+}$. We identify $\mathfrak t$ and $\mathfrak t^{*}$ by 
the quadratic form $B\vert_{\mathfrak t}$.  In particular, 
$\pi^{\mathfrak t, \mathfrak k}\left(\frac{\rho^{\mathfrak 
k}}{2\pi}\right)$ and $\pi^{\mathfrak t, \mathfrak 
k}\left(\frac{\rho^{\mathfrak k}+\lambda}{2\pi}\right)$ are 
well-defined, and $\frac{\pi^{\mathfrak t,\mathfrak 
g}\left(\frac{\rho^{\mathfrak 
k}+\lambda}{2\pi}\right)}{\pi^{\mathfrak t, \mathfrak 
k}\left(\frac{\rho^{\mathfrak k}}{2\pi}\right)}$ only depends on the 
orientation of $\mathfrak p$.   Also
$\phi_{\mathrm{HC}}L$ is a polynomial on $\mathfrak t^{*}$, and so $\phi_{\mathrm{HC}}L\left(-\rho^{\mathfrak 
k}-\lambda\right)$ is well-defined.  
\subsection{Orbital integrals and the index theorem: the case of the 
 identity}%
\label{subsec:orbi}
 Take $L\in Z(\mathfrak g)$. For $t>0$, $L\exp\left(-tD^{X,2}\right)$ 
acts on $C^{\infty }\left(X,S^{TX} \otimes F\right)$. 

In the sequel, 
\index{Trs@$\Trs$}%
$\Trs$ is our notation for the supertrace.\footnote{If $V=V_{+} 
\oplus V_{-}$ is a $\Z_{2}$-graded vector space, if $\tau=\pm 1$ is 
the involution defining the grading,  if $A\in \End\left(V\right)$, 
the supertrace of $A$ is defined to be 
$\Trs\left[A\right]=\Tr\left[\tau A\right]$.}

We will extend \cite[Theorem 7.4.1]{Bismut08b}. 
As in \cite[Section 7.1]{Bismut08b}, 
$\widehat{A}\left(TX,\n^{TX}\right),\ch\left(F,\n^{F}\right)$ denote 
the obvious
 characteristic forms on $X$. Let $\eta\in 
\Lambda^{m}\left(T^{*}X\right)$ be the canonical volume form on $X$ 
that defines its orientation.  If $\alpha\in 
\Lambda\ac\left(T^{*}X\right)$, let $\alpha^{(p)}$ denote its 
component in $\Lambda^{p}\left(T^{*}X\right)$. Let $\alpha^{\max}\in 
\R$ be such that
\begin{equation}\label{eq:clif8}
a^{(m)}=\alpha^{\max}\eta.
\end{equation}

Let $\gamma\in G$ be 
semisimple.
\begin{theorem}\label{thm:indt}
	If $\gamma$ is nonelliptic, for any $t>0$, 
	\begin{equation}\label{eq:D0}
\Trs^{\left[\gamma\right]}\left[L\exp\left(-tD^{X,2}\right)\right]=0.
\end{equation}

	If $\dim \mathfrak b>0$, for any $t>0$, 
	\begin{align}\label{eq:clif9}
&\Trs^{\left[1\right]}\left[L\exp(-tD^{X,2})\right]=0,\\
&\left[\widehat{A}\left(TX,\n^{TX}\right)\ch\left(F,\n^{F}\right)\right]^{\max}=0. \nonumber 
\end{align}

If $\dim \mathfrak b=0$, then
\begin{align}\label{eq:D1a1}
&\Trs^{\left[1\right]}\left[L\exp\left(-tD^{X,2}\right)\right]=\phi_{\mathrm{HC}}L\left(-\rho^{\mathfrak k}-\lambda \right) \left(-1\right)^{m/2}\frac{\pi^{\mathfrak t, \mathfrak g}\left(\frac{\rho^{\mathfrak 
k}+\lambda}{2\pi}\right)}{\pi^{\mathfrak t,\mathfrak 
k}\left(\frac{\rho^{\mathfrak k}}{2\pi}\right)},\\
&\left[\widehat{A}\left(TX,\n^{TX}\right)\ch\left(F,\n^{F}\right)\right]^{\max}=\left(-1\right)^{m/2}
\frac{\pi^{\mathfrak t, \mathfrak g}\left(\frac{\rho^{\mathfrak 
k}+\lambda  }{2\pi} \right)} {\pi^{\mathfrak t, \mathfrak 
k}\left(\frac{\rho^{\mathfrak k}}{2\pi}\right)}. \nonumber 
\end{align}
\end{theorem}
\begin{proof}
	First we prove (\ref{eq:D0}). We proceed as in \cite{Bismut08b}. By Proposition \ref{prop:czg}, by Theorem \ref{thm:Tgenorbis} and by 
(\ref{eq:clif7}), we get
\begin{multline}\label{eq:clif13}
	\Trs^{\left[\gamma\right]}\left[L\exp\left(-tD^{X,2}\right)\right]=\exp\left(-tB^{*}\left(\rho^{\mathfrak k}+\lambda,\rho^{\mathfrak k}+\lambda\right)\right)\\
	L^{ \mathfrak 
	z\left(\gamma\right)}\exp\left(t\Delta^{\mathfrak z\left(\gamma\right)}\right)\left[ \mathcal{J}_{\gamma}\left(Y_{0}^{ \mathfrak k}\right)
 \Trs^{S^{ \mathfrak p}\otimes 
 E}\left[\rho^{S^{\mathfrak p} \otimes 
 E}\left(k^{-1}e^{-\Yok}\right)\right]\delta_{a}\right]\left(0\right).
\end{multline}
Also
\begin{multline}\label{eq:clif13x1}
\Trs^{S^{ \mathfrak p}\otimes 
 E}\left[\rho^{S^{\mathfrak p} \otimes 
 E}\left(k^{-1}e^{-\Yok}\right)\right]=
 \Trs^{S^{ \mathfrak p}}\left[\rho^{S^{\mathfrak 
 p}}\left(k^{-1}e^{-\Yok}\right)\right]\\
 \Tr^{ E}\left[\rho^{ E}\left(k^{-1}e^{-\Yok}\right)\right].
\end{multline}
It is well-known that 
$$\Trs^{S^{ \mathfrak p}}\left[\rho^{S^{\mathfrak p}}\left(k^{-1}e^{-\Yok}\right)\right]$$
is a square root of 
$\det\left(1-\Ad\left(k^{-1}e^{-\Yok}\right)\vert_{\mathfrak p}\right)$.

If $\gamma$ is nonelliptic, $a\neq 0$ lies in the kernel of 
$1-\Ad\left(k^{-1}e^{-\Yok}\right)\vert_{\mathfrak p}$, and so 
(\ref{eq:clif13x1}) vanishes. By (\ref{eq:clif13}), we get 
(\ref{eq:D0}).

By \cite[Theorem 7.4.1]{Bismut08b}, we get
\begin{equation}\label{eq:grsen1}
\Trs^{\left[1\right]}\left[\exp\left(-tD^{X,2}\right)\right]=\left[\widehat{A}\left(TX,\n^{TX}\right)\ch\left(F,\n^{F}\right)\right]^{\max}.
\end{equation}

By (\ref{eq:clif13}),     we get
	 \begin{multline}\label{EQ:TURLU5GEN}
    \Trs^{\left[1\right]}\left[L\exp\left(-tD^{X,2}\right)\right]=\exp\left(-tB^{*}\left(\rho^{\mathfrak k}+\lambda,\rho^{\mathfrak k}+\lambda\right)\right)\\
L^{\mathfrak g}\exp\left(t\Delta^{\mathfrak g}\right)\left[	
 \mathcal{J}_{1}\left(Y_{0}^{ \mathfrak k}\right)
 \Trs^{S^{ \mathfrak p}\otimes 
 E}\left[\rho^{S^{\mathfrak p} \otimes 
 E}\left(e^{-\Yok}\right)\right]\delta_{0}\right]\left(0\right).
 \end{multline}
 We will use (\ref{eq:clif13x1}) with $k=1$. 
We have the well-known identity in \cite[eq. (7.5.11)]{Bismut08b}, 
\begin{equation}\label{eq:clif11}
\Trs^{S^{ \mathfrak p}}\left[\rho^{S^{\mathfrak 
p}}\left(e^{-\Yok}\right)\right]=\left(-i\right)^{m/2}\Pf\left[\ad\left(\Yok\right)\vert_{\mathfrak p}\right]
\widehat{A}^{-1}\left(\ad\left(\Yok\vert_{\mathfrak p}\right)\right).
\end{equation}

In (\ref{eq:clif11}), we may and we will assume that $\Yok\in i 
\mathfrak t$. Since $\mathfrak b \subset \ker 
\ad\left(\Yok\right)\vert_{\mathfrak p}$, if $\dim \mathfrak b>0$, 
then
\begin{equation}\label{eq:clif11a1}
\Pf\left[\ad\left(\Yok\right)\vert_{\mathfrak p}\right]=0, 
\end{equation}
and so (\ref{eq:clif11}) vanishes, which implies the vanishing of 
(\ref{EQ:TURLU5GEN}), 
i.e.,  we have established the first identity in (\ref{eq:clif9}).
Combining this equation for $L=1$ and (\ref{eq:grsen1}), we obtain 
the second equation in (\ref{eq:clif9}).

In the sequel, we assume that $\dim \mathfrak b=0$.
We will use the notation and results of Subsection \ref{subsec:dimb}.

For $\gamma=1$, we use the notation  
\index{gi@$\mathfrak g_{i}$}%
$\mathfrak g_{i}= \mathfrak z_{i}\left(\gamma\right)$, so that
$\mathfrak g_{i}= \mathfrak p \oplus i 
\mathfrak k$.
Let $L^{\mathfrak 
g}\exp\left(t\Delta^{\mathfrak g}\right)\left(f\right), f\in 
\mathfrak g_{i}$ be the 
convolution kernel for $L^{\mathfrak g}\exp\left(t\Delta^{\mathfrak 
g}\right)$. Then
\begin{multline}\label{eq:clif14}
L^{\mathfrak g}\exp\left(t\Delta^{\mathfrak g}\right)\left[ \mathcal{J}_{1}\left(Y_{0}^{ \mathfrak k}\right)
 \Trs^{S^{ \mathfrak p}\otimes 
 E}\left[\rho^{S^{\mathfrak p} \otimes 
 E}\left(e^{-\Yok}\right)\right]\delta_{0}\right]\left(0\right)\\
 =\int_{i \mathfrak k}^{}L^{\mathfrak g}\exp\left(t\Delta^{\mathfrak 
 g}\right)\left(-\Yok\right)\mathcal{J}_{1}\left(Y_{0}^{ \mathfrak k}\right)
 \Trs^{S^{ \mathfrak p}\otimes 
 E}\left[\rho^{S^{\mathfrak p} \otimes 
 E}\left(e^{-\Yok}\right)\right]d\Yok.
\end{multline}

Let 
\index{Wtk@$W\left(\mathfrak t:\mathfrak k\right)$}%
$W\left(\mathfrak t:\mathfrak k\right)$ \footnote{As before, we use this 
notation instead of $W\left(\mathfrak t_{\C}:\mathfrak k_{\C}\right)$ 
because this Weyl group is real.} denote the Weyl group associated 
with the pair $\left(\mathfrak t, \mathfrak k\right)$. Since the integrated function on $i \mathfrak k$ is $K$-invariant, by 
Weyl's integration formula, as in (\ref{eq:gong3}), from 
(\ref{eq:clif14}), we get
\begin{multline}\label{eq:clif15}
L^{\mathfrak g}\exp\left(t\Delta^{\mathfrak g}/2\right)\left[ \mathcal{J}_{1}\left(Y_{0}^{ \mathfrak k}\right)
 \Trs^{S^{ \mathfrak p}\otimes 
 E}\left[\rho^{S^{\mathfrak p} \otimes 
 E}\left(e^{-\Yok}\right)\right]\delta_{0}\right]\left(0\right)\\
 =\frac{\Vol\left(K/T\right)}{\left\vert  W\left(\mathfrak t:\mathfrak k\right)\right\vert}\int_{i 
 \mathfrak t}^{}L^{\mathfrak g}\exp\left(t\Delta^{\mathfrak 
 g}\right)\left(-h_{\mathfrak 
 k}\right)\\
 \mathcal{J}_{1}\left(h_{\mathfrak k}\right)
 \Trs^{S^{ \mathfrak p}\otimes 
 E}\left[\rho^{S^{\mathfrak p} \otimes 
 E}\left(e^{-h_{\mathfrak k}}\right)\right]
 \left[\pi^{\mathfrak t, 
 \mathfrak k}\left(h_{\mathfrak k}\right)\right]^{2}dh_{\mathfrak k}.
\end{multline}

By (\ref{eq:norm0}), (\ref{eq:clif11}), we get
\begin{equation}\label{eq:clif16}
\mathcal{J}_{1}\left(h_{\mathfrak 
k}\right)\Trs^{S^{\mathfrak p}}\left[e^{-h_{\mathfrak k}}\right]=\left(-i\right)^{m/2}\Pf\left[\ad\left(h_{\mathfrak k}\right)\vert_{\mathfrak p}\right]\widehat{A}^{-1}\left(\ad\left(h_{\mathfrak k}\right)\vert_{\mathfrak k}\right).
\end{equation}
Moreover, given our choice of $R^{\mathrm{im}}_{\mathfrak p,+}$, we have
\begin{equation}\label{eq:clif17}
\Pf\left[\ad\left(h_{\mathfrak k}\right)\vert_{\mathfrak 
p}\right]=i^{m/2}\prod_{\alpha\in R^{\mathrm{im}}_{\mathfrak 
p,+}}^{}\left\langle  \alpha,h_{\mathfrak k}\right\rangle.
\end{equation}

If $w\in W\left(\mathfrak t:\mathfrak k\right)$, let $\epsilon_{w}=\pm 1$ be the 
determinant of $w$ on $\mathfrak t$.
Using the Weyl character formula, we have the identity
\begin{multline}\label{eq:clif18}
\left[\pi^{\mathfrak t,\mathfrak k}\left(h_{\mathfrak k} \right) \right]^{2}\widehat{A}^{-1}\left(\ad\left(h_{\mathfrak k}\right)\vert_{\mathfrak 
k}\right)\Tr^{E}\left[\rho^{E}\left(e^{-h_{\mathfrak 
k}}\right)\right]\\
=\left(-1\right)^{\left\vert  R^{\mathrm{im}}_{\mathfrak 
k,+}\right\vert}\pi^{\mathfrak t,\mathfrak k}\left(h_{\mathfrak k}\right)
\sum_{w\in W\left(\mathfrak t:\mathfrak k\right)}^{}\epsilon_{w}e^{-\left\langle  
w\left(\rho^{\mathfrak k}+\lambda\right),h_{\mathfrak k}\right\rangle}.
\end{multline}

By (\ref{eq:clif13a3}), and (\ref{eq:clif16})--(\ref{eq:clif18}), we conclude that
\begin{multline}\label{eq:clif19}
\mathcal{J}_{1}\left(h_{\mathfrak k}\right)\Trs^{S^{\mathfrak p} 
\otimes E}\left[\rho^{S^{\mathfrak p} \otimes 
E}\left(e^{-h_{\mathfrak k}}\right)\right]\left[\pi^{\mathfrak 
t,\mathfrak k}\left(h_{\mathfrak k} \right) \right]^{2}\\
=\left(-1\right)^{\left\vert  R^{\mathrm{im}}_{\mathfrak k,+}\right\vert}\pi^{\mathfrak t,\mathfrak g}\left(h_{\mathfrak k}\right)
\sum_{w\in W\left(\mathfrak t:\mathfrak k\right)}^{}\epsilon_{w}e^{-\left\langle  
w\left(\rho^{\mathfrak k}+\lambda\right),h_{\mathfrak k}\right\rangle}.
\end{multline}

By  (\ref{eq:clif19}), we obtain
\begin{multline}\label{eq:clif20}
\int_{i 
 \mathfrak t}^{}L^{\mathfrak g}\exp\left(t\Delta^{\mathfrak 
 g}\right)\left(-h_{\mathfrak 
 k}\right)\mathcal{J}_{1}\left(h_{\mathfrak k}\right)
 \Trs^{S^{ \mathfrak p}\otimes 
 E}\left[\rho^{S^{\mathfrak p} \otimes 
 E}\left(e^{-h_{\mathfrak k}}\right)\right] \left[\pi^{\mathfrak t, 
 \mathfrak k}\left(h_{\mathfrak k}\right)\right]^{2}dh_{\mathfrak k}\\
 =\left(-1\right)^{\left\vert  R^{\mathrm{im}}_{\mathfrak k,+}\right\vert}\int_{i \mathfrak t}^{}
 L^{\mathfrak g}\exp\left(t\Delta^{\mathfrak 
 g}\right)\left(-h_{\mathfrak 
 k}\right) \\
 \left(\pi^{\mathfrak t,\mathfrak g}\left(h_{\mathfrak k}\right)
\sum_{w\in W\left(\mathfrak t:\mathfrak k\right)}^{}\epsilon_{w}e^{-\left\langle  
w \left( \rho^{\mathfrak 
k} +\lambda\right) ,h_{\mathfrak k}\right\rangle}\right)dh_{\mathfrak k}.
\end{multline}

 As in (\ref{eq:ibva24a5}), we will use  Rossmann formula in (\ref{eq:clif20}) with respect 
to the par $\left(\mathfrak t, \mathfrak u\right)$.  If $e\in \mathfrak t^{*}_{\C}$, we get
\begin{multline}\label{eq:clif21}
\int_{i\mathfrak t}^{}L^{\mathfrak g}\exp\left(t\Delta^{\mathfrak 
g}\right)\left(-h_{\mathfrak k}\right)\pi^{\mathfrak t, \mathfrak 
g}\left(h_{\mathfrak k}\right)e^{-\left\langle e,h_{\mathfrak 
k}\right\rangle}dh_{\mathfrak k}\\
=\left(-1\right)^{\left\vert  
R^{\mathrm{im}}_{+}\right\vert}
\pi^{\mathfrak t, \mathfrak 
g}\left(\frac{e}{2\pi}\right)L^{\mathfrak 
g}\left(-e\right)\exp\left(tB^{*}\left(e,e\right)\right).
\end{multline}
Also $W\left(\mathfrak t:\mathfrak k\right) \subset W\left(\mathfrak 
t_{\C}:\mathfrak g_{\C}\right)$.  
Moreover, if $w\in W\left(\mathfrak t_{\C}:\mathfrak g_{\C}\right)$, if $\epsilon_{w}$ 
still denotes the determinant of $w$ on $\mathfrak t$, by 
\cite[Corollary V.4.6 and Lemma V.4.10]{BrockerDieck95}, 
\begin{equation}\label{eq:clif22}
\pi^{\mathfrak t, \mathfrak 
g}\left(we\right)=\epsilon_{w}\pi^{\mathfrak t, \mathfrak 
g}\left(e\right).
\end{equation}
Finally,  $L^{\mathfrak g}\vert_{\mathfrak t}$ is $W\left(\mathfrak 
t_{\C}:\mathfrak 
g_{\C}\right)$-invariant. 

By (\ref{eq:clif13}), (\ref{eq:clif15}), (\ref{eq:clif20}), and 
(\ref{eq:clif21}), we obtain
\begin{multline}\label{eq:clif22a1}
\Tr^{\left[1\right]}\left[L\exp\left(-tD^{X,2}\right)\right] \\
=\Vol\left(K/T\right)\left(-1\right)^{\left\vert  R^{\mathrm{im}}_{\mathfrak p,+}\right\vert}
\pi^{\mathfrak t, \mathfrak g}\left(\frac{\rho^{\mathfrak 
k}+\lambda}{2\pi}\right)
L^{\mathfrak g}\left(-\rho^{\mathfrak k}-\lambda\right).
\end{multline}
By construction
\begin{equation}\label{eq:clif23}
L^{\mathfrak g}\left(-\rho^{\mathfrak 
k}-\lambda\right)=\phi_{\mathrm{HC}}L\left(- \rho^{\mathfrak 
k}-\lambda\right).
\end{equation}
By \cite[Corollary 7.27]{BerlineGetzlerVergne04}, we get
\begin{equation}\label{eq:clif23a1}
\Vol\left(K/T\right)=\frac{1}{\pi^{\mathfrak t, \mathfrak 
k}\left(\frac{\rho^{\mathfrak k}}{2\pi}\right)}.
\end{equation}
Also
\begin{equation}\label{eq:clif24}
\left\vert  R^{\mathrm{im}}_{\mathfrak p,+}\right\vert=m/2.
\end{equation}
By (\ref{eq:clif22a1})--(\ref{eq:clif24}), we get the first equation 
in (\ref{eq:D1a1}). 
 When $L=1$, we can compare (\ref{eq:grsen1}) and 
this first equation, and we obtain the second equation in 
(\ref{eq:D1a1}). 
The proof of our theorem is completed. 
\end{proof}
\begin{remark}\label{rem:atsc}
	Equation (\ref{eq:D1a1}) was  obtained by 
Atiyah-Schmid \cite[eq. 
	(3.10)]{AtiyahSchmid77},
using Hirzebruch proportionality 
principle \cite{Hirzebruch58}, together with formulas like 
(\ref{eq:clif23a1}).
\end{remark}
\subsection{The case where $\gamma=k^{-1},\dim \mathfrak b=0$}%
\label{subsec:dimbb}
In this Subsection, we assume that $\gamma$ is elliptic, i.e.,  
$\gamma=k^{-1}, k\in K$. Recall that the orientation of $\mathfrak p$ 
is fixed. 

Let $T \subset K$ be a maximal torus, and 
let $\mathfrak t \subset \mathfrak k$ be the corresponding Lie 
algebra. We may and we will assume that $k\in T$, so that $\mathfrak 
t \subset \mathfrak k\left(k\right)$.  Then $T$ is a maximal torus in 
$K^{0}\left(k\right)$, and $\mathfrak t \subset \mathfrak 
k\left(k\right)$. 
Let $\kappa\in 
\mathfrak t$ be such that
\begin{equation}\label{eq:gu1}
k=e^{\kappa}. 
\end{equation}
Then $\kappa$ is well-defined up to the lattice of integral elements 
in $\mathfrak t$ associated with $K$. Since $k$ is in the center of
$K^{0}\left(k\right)$ , if $w\in W\left(\mathfrak t:\mathfrak 
k\left(k \right)\right)$, $w\kappa-\kappa$ is integral in 
$\mathfrak t$. 

Let $\mathfrak h= \mathfrak b \oplus \mathfrak t$ be the 
associated fundamental $\theta$-stable Cartan subalgebra of 
$\mathfrak g$. Then $\mathfrak h$ is a $\theta$-stable fundamental 
Cartan subalgebra of $\mathfrak z\left(k\right)$. 

In this Subsection,  we assume that $\dim \mathfrak b=0$. Then $\mathfrak t$ is a 
$\theta$-stable fundamental Cartan subalgebra of $\mathfrak g$ and of 
$\mathfrak z\left(k\right)$. As in (\ref{eq:clif13a1}), we have
\begin{align}\label{eq:gusen3}
&R=R^{\mathrm{im}}, &R\left(k\right)=R^{\mathrm{im}}\left(k\right).
\end{align}

We make the same choice of $R^{\mathrm{im}}_{\mathfrak 
k,+},R^{\mathrm{im}}_{\mathfrak p,+}$ as in Subsection 
\ref{subsec:dimb}. Set
\begin{equation}\label{eq:gusen4}
R^{\mathrm{im}}_{+}\left(k\right)=R^{\mathrm{im}}_{+}\cap 
R^{\mathrm{im}}\left(k\right).
\end{equation}
Then $R_{+}^{\mathrm{im}}\left(k\right)$ is a positive root system associated with 
the pair $\left(\mathfrak t,\mathfrak z\left(k\right)\right)$, and we 
still have
\begin{equation}\label{eq:gusen5}
R_{+}^{\mathrm{im}}\left(k\right)=R^{\mathrm{im}}_{\mathfrak p,+}\left(k\right)\cup 
R^{\mathrm{im}}_{\mathfrak k,+}\left(k\right).
\end{equation}
The choice of $R_{\mathfrak p,+}^{\mathrm{im}}\left(k\right)$ defines 
an orientation on $\mathfrak p\left(k\right)$.

The functions $\pi^{\mathfrak t, \mathfrak z\left(k\right)}, 
\pi^{\mathfrak t, \mathfrak k\left(k\right)}$ on $\mathfrak t$ are given by
\begin{align}\label{eq:gusen6}
&\pi^{\mathfrak t, \mathfrak z\left(k\right)}=\prod_{\alpha\in 
R^{\mathrm{im}}_{+}\left(k\right)}^{}\left\langle  
\alpha,h\right\rangle,
&\pi^{\mathfrak t, \mathfrak k\left(k\right)}=\prod_{\alpha\in 
R^{\mathrm{im}}_{\mathfrak k,+}\left(k\right)}^{}\left\langle  
\alpha,h\right\rangle.
\end{align}

Again $\rho^{\mathfrak k},\lambda\in i \mathfrak t^{*}$ are calculated with respect to 
$R^{\mathrm{im}}_{\mathfrak k,+}$, and  
\index{rkk@$\rho^{\mathfrak k\left(k\right)}$}%
$\rho^{\mathfrak 
k\left(k\right)}\in i \mathfrak t^{*}$ is obtained via 
$R^{\mathrm{im}}_{\mathfrak k,+}\left(k\right)$. We identify $\mathfrak t$ and 
$\mathfrak t^{*}$ by the quadratic form $B\vert _{\mathfrak t}$. In 
particular $\pi^{\mathfrak t, \mathfrak 
k\left(k\right)}\left(\frac{\rho^{\mathfrak 
k\left(k\right)}}{2\pi}\right),\pi^{\mathfrak t,\mathfrak 
z\left(k\right)}\left(\frac{\rho^{\mathfrak k}+\lambda}{2\pi}\right)$ 
are well-defined as well as 
$\phi_{\mathrm{HC}}L\left(-\rho^{\mathfrak k}-\lambda\right)$. 

Note that $W\left(\mathfrak t: 
\mathfrak k\left(k\right)\right) \subset W\left(\mathfrak t:\mathfrak 
k\right)$. 
If $w\in W\left(\mathfrak t:\mathfrak k\right)$, 
$e^{-\left\langle  w\left(\rho^{\mathfrak 
k}+\lambda\right),\kappa\right\rangle}$ depends only on the image of 
$w$ in $W\left(\mathfrak t:\mathfrak 
k\left(k\right)\right)\setminus W\left(\mathfrak t:\mathfrak 
k\right)$.
The same is true for  
$\epsilon_{w}\pi^{\mathfrak t,\mathfrak 
z\left(k\right)}\left(\frac{w\left(\rho^{\mathfrak 
k}+\lambda\right)}{2\pi}\right)$.
\subsection{Orbital integrals and index theory: the case of elliptic elements}%
\label{subsec:ellel}
We use the same notation as in Subsection \ref{subsec:dimbb}. In 
particular, $\gamma=k^{-1}, k\in K$.

Let $X\left(\gamma\right)$ be the fixed point set of $\gamma$ in $X$. 
Let 
$$\widehat{A}^{\gamma}\left(TX\vert_{X\left(\gamma\right)},\n^{TX\vert_{X\left(\gamma\right)}}\right),
\ch^{\gamma}\left(F,\n^{F}\right)$$
denote the corresponding 
Atiyah-Bott characteristic forms on $X\left(\gamma\right)$, that are 
defined as in \cite[eqs. (7.7.2) and (7.7.4)]{Bismut08b}.
\begin{theorem}\label{thm:Tfino}
	If $\dim \mathfrak b>0$, then
\begin{align}\label{eq:clif27}
&\Trs^{\left[\gamma\right]}\left[L\exp\left(-tD^{X,2}\right)\right]=0,\\ 
&\left[\widehat{A}^{\gamma}\left(TX\vert_{X\left(\gamma\right)},\n^{TX\vert_{X\left(\gamma\right)}}\right)
\ch^{\gamma}\left(F,\n^{F}\right)\right]^{\max}=0. \nonumber 
\end{align}

If $\dim \mathfrak b=0$, then
\begin{multline}\label{eq:clif28}
	\Trs^{\left[\gamma\right]}\left[L\exp\left(-tD^{X,2}\right)\right]=\phi_{\mathrm{HC}}L\left(-\rho^{\mathfrak k}-\lambda \right)\\ 
\left(-1\right)^{\dim \mathfrak p\left(k\right)/2}
\frac{1}{\pi^{\mathfrak t, \mathfrak k\left(k\right)}\left(\frac{\rho^{\mathfrak k\left(k\right)}}{2\pi}\right)}\frac{1}{\prod_{\alpha\in R^{\mathrm{im}}_{+}\setminus 
R^{\mathrm{im}}_{+}\left(k\right)}^{}2\sinh\left(-\left\langle  
\alpha,\kappa\right\rangle/2\right)}\\
\sum_{w\in W\left(\mathfrak t:\mathfrak 
k\left(k\right)\right)\setminus W\left(\mathfrak t:\mathfrak k\right)}^{}
\epsilon_{w}\pi^{\mathfrak t, 
\mathfrak z\left(k\right)}\left(\frac{w\left(\rho^{\mathfrak 
k}+\lambda\right)}{2\pi}\right)e^{-\left\langle  w\left(\rho^{\mathfrak 
k}+\lambda\right),\kappa\right\rangle},
\end{multline}
and
\begin{multline}\label{eq:clif28a1}
\left[\widehat{A}^{\gamma}\left(TX\vert_{X\left(\gamma\right)},\n^{TX\vert_{X\left(\gamma\right)}}\right)
\ch^{\gamma}\left(F,\n^{F}\right)\right]^{\max}\\=\left(-1\right)^{\dim \mathfrak p\left(k\right)/2}
\frac{1}{\pi^{\mathfrak t, \mathfrak k\left(k\right)}\left(\frac{\rho^{\mathfrak k\left(k\right)}}{2\pi}\right)}\frac{1}{\prod_{\alpha\in R^{\mathrm{im}}_{+}\setminus 
R^{\mathrm{im}}_{+}\left(k\right)}^{}2\sinh\left(-\left\langle  
\alpha,\kappa\right\rangle/2\right)}\\
\sum_{w\in W\left(\mathfrak t:\mathfrak 
k\left(k\right)\right)\setminus W\left(\mathfrak t:\mathfrak k\right)}^{}
\epsilon_{w}\pi^{\mathfrak t, 
\mathfrak z\left(k\right)}\left(\frac{w\left(\rho^{\mathfrak 
k}+\lambda\right)}{2\pi}\right)e^{-\left\langle  w\left(\rho^{\mathfrak 
k}+\lambda\right),\kappa\right\rangle}.
\end{multline}
\end{theorem}
\begin{proof}
	By
\cite[Theorem 7.7.1]{Bismut08b}, for $t>0$, we get
\begin{multline}\label{eq:crcr1}
\Trs^{\left[\gamma\right]}\left[\exp\left(-tD^{X,2}\right)\right]\\
=\left[\widehat{A}^{\gamma}\left(TX\vert_{X\left(\gamma\right)},\n^{TX\vert_{X\left(\gamma\right)}}\right),
\ch^{\gamma}\left(F,\n^{F}\right)\right]^{\max}.
\end{multline}

Equation (\ref{eq:clif13}) still 
holds. We claim that if $\dim \mathfrak b>0$, 
\begin{equation}\label{eq:clif29}
\Trs^{S^{\mathfrak p}}\left[\rho^{S^{\mathfrak 
p}}\left(k^{-1}e^{-\Yok}\right)\right]=0.
\end{equation}
 If $\Yok\in i\mathfrak k\left(k\right)$, after conjugation by an element of
$K^{0}\left(k\right)$, we may assume that $\Yok\in i \mathfrak 
t$. If $\dim \mathfrak b>0$, then $1-\Ad\left(k^{-1}e^{-\Yok}\right)$ vanishes 
on $\mathfrak b$. The argument we gave after (\ref{eq:clif13x1}) shows 
that (\ref{eq:clif29}) vanishes. This proves the first equation  in 
(\ref{eq:clif27}). Combining this  equation 
for $L=1$ and (\ref{eq:crcr1}), we get the second equation in (\ref{eq:clif27}).

In the sequel, we assume that $\dim \mathfrak b=0$. We will use the notation and results of 
Subsection \ref{subsec:dimbb}, and also equation (\ref{eq:clif13}).   As in (\ref{eq:clif14}), and with a similar 
notation, we get
\begin{multline}\label{eq:clif31}
L^{\mathfrak z\left(k\right)}\exp\left(t\Delta^{\mathfrak 
z\left(k\right)}\right)\left[ \mathcal{J}_{k^{-1}}\left(Y_{0}^{ \mathfrak k}\right)
 \Trs^{S^{ \mathfrak p}\otimes 
 E}\left[\rho^{S^{\mathfrak p} \otimes 
 E}\left(k^{-1}e^{-\Yok}\right)\right]\delta_{0}\right]\left(0\right)\\
 =\int_{i \mathfrak k\left(k\right)}^{}L^{\mathfrak z\left(k\right)}
 \exp\left(t\Delta^{\mathfrak 
 z\left(k\right)}\right)\left(-\Yok\right)\\
 \mathcal{J}_{k^{-1}}\left(\Yok\right)
 \Trs^{S^{\mathfrak p} \otimes E}\left[\rho^{S^{\mathfrak 
p} \otimes E}\left(k^{-1}e^{-\Yok}\right)\right]d\Yok.
\end{multline}
Using Weyl integration as in (\ref{eq:clif15}), we deduce from 
(\ref{eq:clif31}) that
\begin{multline}\label{eq:clif32}
L^{\mathfrak z\left(k\right)}\exp\left(t\Delta^{\mathfrak 
z\left(k\right)}\right)\left[ \mathcal{J}_{k^{-1}}\left(Y_{0}^{ \mathfrak k}\right)
 \Trs^{S^{ \mathfrak p}\otimes 
 E}\left[\rho^{S^{\mathfrak p} \otimes 
 E}\left(k^{-1}e^{-\Yok}\right)\right]\delta_{0}\right]\left(0\right)\\
 =\frac{\Vol\left(K^{0}\left(k\right)/T\right)}{\left\vert  
 W\left(\mathfrak t:\mathfrak k\left(k\right) \right)\right\vert}
 \int_{i \mathfrak t}^{}L^{\mathfrak z\left(k\right)}\exp\left(t 
 \Delta^{\mathfrak z\left(k\right)}\right)\left(-h_{\mathfrak k}\right)
 \\
\mathcal{J}_{k^{-1}}\left(h_{\mathfrak k}\right) \Trs^{S^{\mathfrak p} \otimes E}\left[\rho^{S^{\mathfrak 
p} \otimes E}\left(k^{-1}e^{-h_{\mathfrak k}}\right)\right]
\left[\pi^{\mathfrak t, \mathfrak 
k\left(k\right)}\left(h_{\mathfrak 
k}\right)\right]^{2}dh_{\mathfrak k}.
\end{multline}

By \cite[eq. (7.7.7)]{Bismut08b} and using the corresponding notation, if $h_{\mathfrak k}\in i\mathfrak  
t$, we have the 
identity
\begin{multline}\label{eq:clif31a}
\Trs^{S^{\mathfrak p}}\left[\rho^{S^{\mathfrak 
p}}\left(k^{-1}e^{-h_{\mathfrak k}}\right)\right]
=\left(-i\right)^{\dim \mathfrak p\left(k\right)/2}\Pf\left[\ad\left(h_{\mathfrak 
k}\right)\vert_{\mathfrak 
p\left(k\right)}\right] \\
\widehat{A}^{-1}\left(\ad\left(h_{\mathfrak k}\right)\vert_{\mathfrak p\left(k\right)}\right) \left(\widehat{A}^{ke^{h_{\mathfrak k}}\vert_{\mathfrak p^{\perp}\left(k\right)}}\left(0\right) \right)^{-1}.
\end{multline}
By (\ref{eq:ret33}), (\ref{eq:clif31a}), and by proceeding as in 
\cite[eq. (7.7.8)]{Bismut08b}, we get
\begin{multline}\label{eq:clif32a}
\mathcal{J}_{k^{-1}}\left(h_{\mathfrak k}\right)\Trs^{S^{\mathfrak p}}\left[\rho^{S^{\mathfrak 
p}}\left(k^{-1}e^{-h_{\mathfrak k}}\right)\right]
=\left(-i\right)^{\dim \mathfrak 
p\left(k\right)/2}\\
\Pf\left[\ad\left(h_{\mathfrak 
k}\right)\vert_{\mathfrak 
p\left(k\right)}\right] 
\widehat{A}^{-1}\left(\ad\left(h_{\mathfrak k}\right)
\vert_{\mathfrak 
k\left(k\right)}\right)
\widehat{A}^{k^{-1}\vert_{\mathfrak 
p^{\perp}\left(k\right)}}\left(0\right)\\
 \left[\frac{\det\left(1-\Ad\left(k^{-1}e^{-h_{\mathfrak k}}\right)\right)\vert_{ \mathfrak k^{\perp}
   \left(k\right)}}{\det\left(1-\Ad\left(k^{-1}\right)\right)\vert _{ \mathfrak 
   k^{\perp}\left(k\right)}}\right]^{1/2}.
\end{multline}
Also we have the identity,
\begin{equation}\label{eq:clif32aa}
 \left[\frac{\det\left(1-\Ad\left(k^{-1}e^{-h_{\mathfrak k}}\right)\right)\vert_{ \mathfrak k^{\perp}
   \left(k\right)}}{\det\left(1-\Ad\left(k^{-1}\right)\right)\vert _{ \mathfrak 
   k^{\perp}\left(k\right)}}\right]^{1/2}=\frac{\widehat{A}^{k^{-1}\vert_{\mathfrak k^{\perp}\left(k\right)}}\left(0\right)}
   {\widehat{A}^{k^{-1}e^{-h_{\mathfrak k}}\vert_{\mathfrak 
   k^{\perp}\left(k\right)}}\left(0\right)}.
\end{equation}
Using (\ref{eq:clif32aa}), we can rewrite (\ref{eq:clif32a}) in the form,
\begin{multline}\label{eq:clif33}
\mathcal{J}_{k^{-1}}\left(h_{\mathfrak k}\right)\Trs^{S^{\mathfrak p}}\left[\rho^{S^{\mathfrak 
p}}\left(k^{-1}e^{-h_{\mathfrak k}}\right)\right]\\
=\left(-i\right)^{\dim \mathfrak 
p\left(k\right)/2}
\Pf\left[\ad\left(h_{\mathfrak 
k}\right)\vert_{\mathfrak 
p\left(k\right)}\right] 
\widehat{A}^{k^{-1}}\left(0\right)\left( 
\widehat{A}^{k^{-1}} \right) ^{-1}\left(-\ad\left(h_{\mathfrak k}\right)\vert_{\mathfrak k}\right) .
\end{multline}
As in (\ref{eq:clif17}), we get
\begin{equation}\label{eq:clif34}
\Pf\left[\ad\left(h_{\mathfrak k} \right) \vert_{\mathfrak p\left(k\right)}\right]
=i^{\dim \mathfrak 
p\left(k\right)/2}\prod_{\alpha\in R_{\mathfrak 
p,+}^{\mathrm{im}}\left(k\right)}^{}\left\langle  
\alpha,h_{\mathfrak k}\right\rangle.
\end{equation}

By Weyl's character formula, we obtain
\begin{multline}\label{eq:clif36}
\left[\pi^{\mathfrak t,\mathfrak 
k\left(k\right)}\left(h_{\mathfrak 
k}\right)\right]^{2}\widehat{A}^{k^{-1}}\left(-\ad\left(h_{\mathfrak 
k}\right)\vert_{\mathfrak 
k}\right)\Tr^{E}\left[\rho^{E}\left(k^{-1}e^{-h_{\mathfrak 
k}}\right)\right]\\
=	\left(-1\right)^{\left\vert  R^{\mathrm{im}}_{\mathfrak 
k,+}\left(k\right)\right\vert}\pi^{\mathfrak 
t,\mathfrak k\left(k\right)}\left(h_{\mathfrak 
k}\right)\sum_{w\in W\left(\mathfrak t:\mathfrak 
k\right)}^{}\epsilon_{w}e^{-\left\langle w \left( \rho^{\mathfrak 
k}+\lambda\right), \kappa+h_{\mathfrak k}  \right\rangle}.
\end{multline}

By (\ref{eq:clif33})--(\ref{eq:clif36}), we find that
\begin{multline}\label{eq:clif37}
\mathcal{J}_{k^{-1}}\left(h_{\mathfrak k}\right)\Trs^{S^{\mathfrak p} 
\otimes E}\left[\rho^{S^{\mathfrak 
p} \otimes E}\left(k^{-1}e^{-h_{\mathfrak k}}\right)\right]\left[\pi^{\mathfrak t, \mathfrak 
k\left(k\right)}\left(h_{\mathfrak k}\right)\right]^{2}\\
=\left(-1\right)^{\left\vert  R^{\mathrm{im}}_{\mathfrak 
k,+}\left(k\right)\right\vert}\widehat{A}^{k^{-1}}\left(0\right)\pi^{\mathfrak t, \mathfrak 
z\left(k\right)}\left(h_{\mathfrak k}\right)\sum_{w\in 
W\left(\mathfrak t:\mathfrak k\right)}^{}
\epsilon_{w}e^{-\left\langle w \left(   \rho^{\mathfrak 
k}+\lambda \right)   , \kappa+h_{\mathfrak k}    \right\rangle}.
\end{multline}

By (\ref{eq:clif37}), we obtain
\begin{multline}\label{eq:clif38}
\int_{i \mathfrak t}^{}L^{\mathfrak z\left(k\right)}\exp\left(t 
 \Delta^{\mathfrak z\left(k\right)}\right)\left(-h_{\mathfrak k}\right)
\mathcal{J}_{k^{-1}}\left(h_{\mathfrak k}\right) \Trs^{S^{\mathfrak p} \otimes E}\left[\rho^{S^{\mathfrak 
p} \otimes E}\left(k^{-1}e^{-h_{\mathfrak k}}\right)\right]\\
\left[\pi^{\mathfrak t, \mathfrak 
k\left(k\right)}\left(h_{\mathfrak 
k}\right)\right]^{2}dh_{\mathfrak k}
=\left(-1\right)^{\left\vert  R^{\mathrm{im}}_{\mathfrak 
k,+}\left(k\right)\right\vert}\widehat{A}^{k^{-1}}\left(0\right) \\
\int_{i \mathfrak  t}^{}L^{\mathfrak z\left(k\right)}\exp\left(t 
 \Delta^{\mathfrak z\left(k\right)}\right)\left(-h_{\mathfrak k}\right)
 \pi^{\mathfrak t, \mathfrak 
z\left(k\right)}\left(h_{\mathfrak k}\right)\sum_{w\in 
W\left(\mathfrak t:\mathfrak k\right)}^{}
\epsilon_{w}e^{-\left\langle w \left(  \rho^{\mathfrak 
k}+\lambda \right) , \kappa+h_{\mathfrak k} \right\rangle}dh_{\mathfrak k}.
 \end{multline}
 Using Rossmann's formula as in (\ref{eq:clif21}), if $e\in \mathfrak 
 t^{*}_{\C}$, we find that
\begin{multline}\label{eq:clif39}
\int_{i \mathfrak  t}^{}L^{\mathfrak z\left(k\right)}\exp\left(t 
 \Delta^{\mathfrak z\left(k\right)}\right)\left(-h_{\mathfrak k}\right)
 \pi^{\mathfrak t, \mathfrak 
z\left(k\right)}\left(h_{\mathfrak k}\right)e^{-\left\langle  
e,h_{\mathfrak k}\right\rangle}dh_{\mathfrak k}\\
=\left(-1\right)^{\left\vert R^{\mathrm{im}}_{+}\left( k \right) 
\right\vert}\pi^{\mathfrak t,\mathfrak 
z\left(k\right)}\left(\frac{e}{2\pi}\right)L^{\mathfrak 
g}\left(-e\right)\exp\left(tB^{*}\left(e,e\right)\right).
\end{multline}
By (\ref{eq:clif39}), we get
\begin{multline}\label{eq:clif40}
\int_{i \mathfrak  t}^{}L^{\mathfrak z\left(k\right)}\exp\left(t 
 \Delta^{\mathfrak z\left(k\right)}\right)\left(-h_{\mathfrak k}\right)
 \pi^{\mathfrak t, \mathfrak 
z\left(k\right)}\left(h_{\mathfrak k}\right)\sum_{w\in 
W\left(\mathfrak t:\mathfrak k\right)}^{}
\epsilon_{w}e^{-\left\langle  w\left(\rho^{\mathfrak 
k}+\lambda\right), \kappa+h_{\mathfrak 
k} \right\rangle}dh_{\mathfrak k}\\
=\left(-1\right)^{\left\vert  R^{\mathrm{im}}_{+}\left(k\right)\right\vert}
L^{\mathfrak 
g}\left(-\rho^{\mathfrak 
k}-\lambda\right)\exp\left(tB^{*}\left(\rho^{\mathfrak 
k}+\lambda,\rho^{\mathfrak k}+\lambda\right)\right)\\
\sum_{w\in W\left(\mathfrak t:\mathfrak k\right)}^{}\epsilon_{w}\pi^{\mathfrak t, 
\mathfrak z\left(k\right)}\left(\frac{w\left(\rho^{\mathfrak 
k}+\lambda\right)}{2\pi}\right)e^{-\left\langle  w\left(\rho^{\mathfrak 
k}+\lambda\right),\kappa\right\rangle}.
\end{multline}
By the considerations that follow (\ref{eq:gu1}) and by (\ref{eq:clif40}), we obtain
\begin{multline}\label{eq:clif41}
\int_{i \mathfrak  t}^{}L^{\mathfrak z\left(k\right)}\exp\left(t 
 \Delta^{\mathfrak z\left(k\right)}\right)\left(-h_{\mathfrak k}\right)
 \pi^{\mathfrak t, \mathfrak 
z\left(k\right)}\left(h_{\mathfrak k}\right)\sum_{w\in 
W\left(\mathfrak t:\mathfrak k\right)}^{}
\epsilon_{w}e^{-\left\langle  w\left(\rho^{\mathfrak 
k}+\lambda\right), \kappa+h_{\mathfrak 
k} \right\rangle}dh_{\mathfrak k}\\
=\left(-1\right)^{\left\vert  R^{\mathrm{im}}_{+}\left(k\right)\right\vert}
\left\vert  W\left(\mathfrak t:\mathfrak k\left(k\right)\right)\right\vert L^{\mathfrak 
g}\left(-\rho^{\mathfrak 
k}-\lambda\right)\exp\left(tB^{*}\left(\rho^{\mathfrak 
k}+\lambda,\rho^{\mathfrak k}+\lambda\right)\right)\\
\sum_{w\in W\left(\mathfrak t:\mathfrak 
k\left(k\right)\right)\setminus W\left(\mathfrak t:\mathfrak k\right)}^{}\epsilon_{w}\pi^{\mathfrak t, 
\mathfrak z\left(k\right)}\left(\frac{w\left(\rho^{\mathfrak 
k}+\lambda\right)}{2\pi}\right)e^{-\left\langle  w\left(\rho^{\mathfrak 
k}+\lambda\right),\kappa\right\rangle}.
\end{multline}

By (\ref{eq:clif13}), (\ref{eq:clif32}), (\ref{eq:clif38}), and 
(\ref{eq:clif41}), we get
\begin{multline}\label{eq:clif42}
\Trs^{\left[\gamma\right]}\left[L\exp\left(-tD^{X,2}\right)\right]=\Vol\left(K^{0}\left(k\right)/T\right)\left(-1\right)^{\left\vert  R^{\mathrm{im}}_{\mathfrak p,+}\left(k\right)\right\vert}\widehat{A}^{k^{-1}}\left(0\right)\\
\phi_{\mathrm{HC}}L\left(-\rho^{\mathfrak k}-\lambda \right) 
\sum_{w\in W\left(\mathfrak t:\mathfrak 
k\left(k\right)\right)\setminus W\left(\mathfrak t:\mathfrak k\right)}^{}
\epsilon_{w}\pi^{\mathfrak t, 
\mathfrak z\left(k\right)}\left(\frac{w\left(\rho^{\mathfrak 
k}+\lambda\right)}{2\pi}\right)e^{-\left\langle  w\left(\rho^{\mathfrak 
k}+\lambda\right),\kappa\right\rangle}.
\end{multline}
As in (\ref{eq:clif23a1}), we obtain
\begin{equation}\label{eq:clif45}
\Vol\left(K^{0}\left(k\right)/T\right)=\frac{1}{\pi^{\mathfrak 
t, \mathfrak k\left(k\right)}\left(\frac{\rho^{\mathfrak 
k\left(k\right)}}{2\pi}\right)}.
\end{equation}
As in (\ref{eq:clif24}), we get
\begin{equation}\label{eq:clif46}
\left\vert  R^{\mathrm{im}}_{\mathfrak 
p,+}\left(k\right)\right\vert=\dim \mathfrak 
p\left(k\right)/2.
\end{equation}
Moreover, 
\begin{equation}\label{eq:clif47}
\widehat{A}^{k^{-1}}\left(0\right)=\frac{1}{\prod_{\alpha\in 
R^{\mathrm{im}}_{+}\setminus 
R^{\mathrm{im}}_{+}\left(k\right)}^{}2\sinh\left(-\left\langle  \alpha,\kappa\right\rangle/2\right)}.
\end{equation}

By (\ref{eq:clif42})--(\ref{eq:clif47}), we 
get (\ref{eq:clif28}), which combined with (\ref{eq:crcr1}) gives (\ref{eq:clif28a1}). The proof of our theorem is completed. 
\end{proof}
\begin{remark}\label{rem:asc}
	Equation (\ref{eq:clif28a1}) can also be obtained by a suitable 
	application of Hirzebruch proportionality principle 
	\cite{Hirzebruch58} similar to what was done by Atiyah-Schmid 
	\cite{AtiyahSchmid77}.
\end{remark}
\subsection{Orbital integrals and a conjecture by Vogan}%
\label{subsec:resHP}
Let $\pi$ be an irreducible unitary representation of $G$ acting on a 
Hilbert space $V_{\pi}$. By \cite[p. 205]{Knapp86},   a vector 
$v\in V_{\pi}$ is called $K$-finite if the vector subspace 
generated by   the vectors $k v\vert_{k\in K}$ has  finite 
dimension.   Let $V_{\pi,K}\subset V_{\pi}$ be the vector subspace  
of the $K$-finite 
vectors in $V_{\pi}$. By   \cite[Proposition 8.5]{Knapp86}, $U(\mathfrak g_{\C})$ acts on 
$V_{\pi,K}$, so that $V_{\pi,K}$ is a $(\mathfrak g_{\C},K)$-module. 

 Let $e_{1},\ldots, e_{m}$ be 
an orthonormal basis of $\mathfrak p$. Let $D\in U(\mathfrak g)\otimes 
c(\mathfrak p)$ be the Dirac operator,
\begin{align}
	D=\sum_{i=1}^{m}c(e_{i})e_{i}.\end{align}
We denote by $D\vert_{V_{\pi,K}\otimes S^{\mathfrak p}}$  the restriction 
of $D$ to 
$V_{\pi,K}\otimes S^{\mathfrak p}$. By \cite[p. 189]{HuangPan02}, the Dirac 
cohomology  of $V_{\pi,K}$ is the $K$-module defined by 
\begin{align}
	H_{D}(V_{\pi,K})=\ker D|_{V_{\pi,K}\otimes S^{\mathfrak p}}. 
\end{align}
By \cite[Theorem 8.1]{Knapp86}, each $K$-type in $H_{D}(V_{\pi,K})$  has finite multiplicity.  

The Vogan conjecture, solved by Huang-Pand\v{z}i\'{c} 
\cite[Corollary 2.4]{HuangPan02} states the following.
\begin{theorem}\label{thm:voganconj}
If the Dirac cohomology $H_{D}(V_{\pi,K})$ contains a $K$-type of  
highest weight $\lambda\in i\mathfrak t^{*}$,  the infinitesimal 
character of $V_{\pi,K}$ is $\rho^{\mathfrak k}+\lambda$. 
\end{theorem}

An equivalent formulation of Theorem \ref{thm:voganconj} says that  if $E$ is an irreducible $K$ representation of highest 
weight $\lambda\in i\mathfrak t^{*}$,  if $D^{S^{\mathfrak p} \otimes 
E}$ denotes the restriction of $D$ to  $(V_{\pi,K}\otimes 
S^{\mathfrak p}\otimes E)^{K}$,  if $\ker D^{S^{\mathfrak p} \otimes 
E}\neq 0$,  then $L\in Z(\mathfrak g)$ acts on $V_{\pi,K}$ as the 
scalar  $\phi_{\rm HC}L\left(-\rho^{\mathfrak k}-\lambda\right)$.

We will show that  (\ref{eq:D1a1}) and 
(\ref{eq:clif28})  are compatible with  Theorem 
\ref{thm:voganconj}.  Let $\Gamma$ be a discrete cocompact  subgroup 
of $G$. By  \cite[Theorem, p.23]{GelfandGraevPia90},  we have
\begin{align}\label{eq:GMP}
	L^{2}(\Gamma\backslash G)=\bigoplus_{\pi\in 
	\widehat{G}_{u}}^{\rm Hil}n_{\Gamma}(\pi)V_{\pi},
\end{align}
with $n_{\Gamma}(\pi)\in \N$. 

We use the notation of Subsection \ref{subsec:dirX}. In particular, we assume that 
$\rho^{E}$ is an irreducible representation of $K$ with highest weight 
$\lambda\in i\mathfrak t^{*}$. Let $Z$ be the compact orbifold 
$Z=\Gamma\setminus X$. The vector bundle $F$ on $X$ descends to an 
orbifold
vector bundle on $Z$, which we still denote $F$. Also $D^{X}$ 
descends to the orbifold Dirac operator  $D^{Z}$. By \eqref{eq:GMP}, we have 
\begin{equation}\label{eq:pan1}
	\ker D^{Z}=\bigoplus_{\pi\in \widehat{G}_{u}} 
	n_{\Gamma}(\pi)(H_{D}(V_{\pi,K})\otimes E)^{K}. 
\end{equation}
Since $\ker D^{Z}$ is finite-dimensional, the  sum in the right-hand 
side only contains finitely many nonzero terms.  By Theorem \ref{thm:voganconj},  $L\in Z(\mathfrak g)$ acts on $\ker 
D^{Z}$ as $\phi_{\rm HC}L\left(-\rho^{\mathfrak k}-\lambda\right)$.

Using the McKean-Singer formula \cite{McKeanSinger} and 
the above, for $t>0$, we get 
\begin{equation}\label{eq:pan2}
\Trs\left[L\exp\left(-tD^{Z,2}\right)\right]=\phi_{\mathrm{HC}}L\left(-\rho^{\mathfrak k}-\lambda\right)\Trs\left[\exp\left(-tD^{Z,2}\right)\right].
\end{equation}
Also $\Trs\left[L\exp\left(-tD^{Z,2}\right)\right]$ can be evaluated 
in terms of corresponding orbital integrals using Selberg's trace 
formula. 

Assume first that $\Gamma$ is torsion free. 
By   equation (\ref{eq:D0}) in 
Theorem \ref{thm:indt}, only the identity element contributes to the 
above supertrace. Then equation  (\ref{eq:D1a1}) can  be viewed as a 
consequence of (\ref{eq:grsen1}), (\ref{eq:pan2}).

When $\Gamma$ is not torsion free, only the finite number of conjugacy 
classes of elliptic elements in $\Gamma$ contribute to 
(\ref{eq:pan2}). Then equations (\ref{eq:clif28}), (\ref{eq:crcr1}) 
 are compatible to (\ref{eq:pan2}). 

\printindex

\bibliography{Bismut,Others}

\def\cprime{$'$}
\begin{thebibliography}{BeGV04}

\bibitem[ABo67]{AtiyahBott67}
M.~F. Atiyah and R.~Bott.
\newblock A {L}efschetz fixed point formula for elliptic complexes. {I}.
\newblock {\em Ann. of Math. (2)}, 86:374--407, 1967.

\bibitem[ABo68]{AtiyahBott68}
M.~F. Atiyah and R.~Bott.
\newblock A {L}efschetz fixed point formula for elliptic complexes. {I}{I}.
  {A}pplications.
\newblock {\em Ann. of Math. (2)}, 88:451--491, 1968.

\bibitem[ASc77]{AtiyahSchmid77}
M.~F. Atiyah and W.~Schmid.
\newblock A geometric construction of the discrete series for semisimple {L}ie
  groups.
\newblock {\em Invent. Math.}, 42:1--62, 1977.

\bibitem[AS68a]{AtiyahSinger68}
M.~F. Atiyah and I.~M. Singer.
\newblock The index of elliptic operators. {I}.
\newblock {\em Ann. of Math. (2)}, 87:484--530, 1968.

\bibitem[AS68b]{AtiyahSinger68b}
M.~F. Atiyah and I.~M. Singer.
\newblock The index of elliptic operators. {I}{I}{I}.
\newblock {\em Ann. of Math. (2)}, 87:546--604, 1968.

\bibitem[BaGS85]{BalGroSchro}
W.~Ballmann, M.~Gromov, and V.~Schroeder.
\newblock {\em Manifolds of nonpositive curvature}, volume~61 of {\em Progress
  in Mathematics}.
\newblock Birkh\"auser Boston Inc., Boston, MA, 1985.

\bibitem[BeGV04]{BerlineGetzlerVergne04}
N.~Berline, E.~Getzler, and M.~Vergne.
\newblock {\em Heat kernels and {D}irac operators}.
\newblock Grundlehren Text Editions. Springer-Verlag, Berlin, 2004.
\newblock Corrected reprint of the 1992 original.

\bibitem[B11]{Bismut08b}
J.-M. Bismut.
\newblock {\em Hypoelliptic {L}aplacian and orbital integrals}, volume 177 of
  {\em Annals of Mathematics Studies}.
\newblock Princeton University Press, Princeton, NJ, 2011.

\bibitem[BL99]{BismutLabourie99}
J.-M. Bismut and F.~Labourie.
\newblock Symplectic geometry and the {V}erlinde formulas.
\newblock In {\em Surveys in differential geometry: differential geometry
  inspired by string theory}, volume~5 of {\em Surv. Differ. Geom.}, pages
  97--311. Int. Press, Boston, MA, 1999.

\bibitem[BrD95]{BrockerDieck95}
T.~Br{\"o}cker and T.~tom Dieck.
\newblock {\em Representations of compact {L}ie groups}, volume~98 of {\em
  Graduate Texts in Mathematics}.
\newblock Springer-Verlag, New York, 1995.
\newblock Translated from the German manuscript, Corrected reprint of the 1985
  translation.

\bibitem[CP81]{ChazarainPiriou}
J.~Chazarain and A.~Piriou.
\newblock {\em Introduction \`a la th\'eorie des \'equations aux d\'eriv\'ees
  partielles lin\'eaires}.
\newblock Gauthier-Villars, Paris, 1981.

\bibitem[Du70]{Duflo70}
M.~Duflo.
\newblock Caract\`eres des groupes et des alg\`ebres de {L}ie r\'{e}solubles.
\newblock {\em Ann. Sci. \'{E}cole Norm. Sup. (4)}, 3:23--74, 1970.

\bibitem[E96]{Eberlein96}
P.~B. Eberlein.
\newblock {\em Geometry of nonpositively curved manifolds}.
\newblock Chicago Lectures in Mathematics. University of Chicago Press,
  Chicago, IL, 1996.
  
  \bibitem[GGP90]{GelfandGraevPia90}
I.~M. Gelfand, M.~I. Graev, and I.~I. Pyatetskii-Shapiro.
\newblock {\em Representation theory and automorphic functions}, volume~6 of
  {\em Generalized Functions}.
\newblock Academic Press, Inc., Boston, MA, 1990.
\newblock Translated from the Russian by K. A. Hirsch, Reprint of the 1969
  edition.

  \bibitem[HC56]{Harish56}
Harish-Chandra.
\newblock The characters of semisimple {L}ie groups.
\newblock {\em Trans. Amer. Math. Soc.}, 83:98--163, 1956.

\bibitem[HC57a]{Harish57}
Harish-Chandra.
\newblock Differential operators on a semisimple {L}ie algebra.
\newblock {\em Amer. J. Math.}, 79:87--120, 1957.

\bibitem[HC57b]{Harish57a}
Harish-Chandra.
\newblock A formula for semisimple {L}ie groups.
\newblock {\em Amer. J. Math.}, 79:733--760, 1957.

\bibitem[HC64]{Harish64c}
Harish-Chandra.
\newblock Some results on an invariant integral on a semisimple {L}ie algebra.
\newblock {\em Ann. of Math. (2)}, 80:551--593, 1964.

\bibitem[HC65]{Harish65}
Harish-Chandra.
\newblock Invariant eigendistributions on a semisimple {L}ie group.
\newblock {\em Trans. Amer. Math. Soc.}, 119:457--508, 1965.

\bibitem[HC66]{Harish66}
Harish-Chandra.
\newblock Discrete series for semisimple {L}ie groups. {II}. {E}xplicit
  determination of the characters.
\newblock {\em Acta Math.}, 116:1--111, 1966.

\bibitem[HC75]{Harish75}
Harish-Chandra.
\newblock Harmonic analysis on real reductive groups. {I}. {T}he theory of the
  constant term.
\newblock {\em J. Functional Analysis}, 19:104--204, 1975.

\bibitem[Hi58]{Hirzebruch58}
F.~Hirzebruch.
\newblock Automorphe {F}ormen und der {S}atz von {R}iemann-{R}och.
\newblock In {\em Symposium internacional de topolog\'\i a algebraica
  International symposium on algebraic topology}, pages 129--144. Universidad
  Nacional Aut\'onoma de M\'exico and UNESCO, Mexico City, 1958.
 

\bibitem[H{\"o}83]{Hormander83a}
L.~H{\"o}rmander.
\newblock {\em The analysis of linear partial differential operators. {I}},
  volume 256 of {\em Grundlehren der Mathematischen Wissenschaften [Fundamental
  Principles of Mathematical Sciences]}.
\newblock Springer-Verlag, Berlin, 1983.
\newblock Distribution theory and Fourier analysis.

\bibitem[H{\"o}85]{Hormander85a}
L.~H{\"o}rmander.
\newblock {\em The analysis of linear partial differential operators.
  {I}{I}{I}}.
\newblock Grundl. Math. Wiss. Band 274. Springer-Verlag, Berlin, 1985.
\newblock Pseudodifferential operators.

\bibitem[HuP02]{HuangPan02}
J.-S. Huang and P.~Pand\v{z}i\'{c}.
\newblock Dirac cohomology, unitary representations and a proof of a conjecture
  of {V}ogan.
\newblock {\em J. Amer. Math. Soc.}, 15(1):185--202, 2002.

\bibitem[K79]{Kawasaki79}
T.~Kawasaki.
\newblock The {R}iemann-{R}och theorem for complex {$V$}-manifolds.
\newblock {\em Osaka J. Math.}, 16(1):151--159, 1979.

\bibitem[Kn86]{Knapp86}
A.~W. Knapp.
\newblock {\em Representation theory of semisimple groups}, volume~36 of {\em
  Princeton Mathematical Series}.
\newblock Princeton University Press, Princeton, NJ, 1986.
\newblock An overview based on examples.

\bibitem[Kn02]{Knapp02}
A.~W. Knapp.
\newblock {\em Lie groups beyond an introduction}, volume 140 of {\em Progress
  in Mathematics}.
\newblock Birkh\"{a}user Boston, Inc., Boston, MA, second edition, 2002.

\bibitem[Ko73]{Kostant73}
B.~Kostant.
\newblock On convexity, the {W}eyl group and the {I}wasawa decomposition.
\newblock {\em Ann. Sci. \'{E}cole Norm. Sup. (4)}, 6:413--455 (1974), 1973.

\bibitem[Ko76]{Kostant76}
B.~Kostant.
\newblock On {M}acdonald's {$\eta $}-function formula, the {L}aplacian and
  generalized exponents.
\newblock {\em Advances in Math.}, 20(2):179--212, 1976.

\bibitem[MS67]{McKeanSinger}
H.~P. McKean, Jr. and I.~M. Singer.
\newblock Curvature and the eigenvalues of the {L}aplacian.
\newblock {\em J. Differential Geometry}, 1(1):43--69, 1967.

\bibitem[R78]{Rossmann78}
W.~Rossmann.
\newblock Kirillov's character formula for reductive {L}ie groups.
\newblock {\em Invent. Math.}, 48(3):207--220, 1978.

\bibitem[Se56]{Selberg56}
A.~Selberg.
\newblock Harmonic analysis and discontinuous groups in weakly symmetric
  {R}iemannian spaces with applications to {D}irichlet series.
\newblock {\em J. Indian Math. Soc. (N.S.)}, 20:47--87, 1956.

\bibitem[T81]{Taylor81}
M.~E. Taylor.
\newblock {\em Pseudodifferential operators}, volume~34 of {\em Princeton
  Mathematical Series}.
\newblock Princeton University Press, Princeton, N.J., 1981.

\bibitem[Va77]{Vara77}
V.~S. Varadarajan.
\newblock {\em Harmonic analysis on real reductive groups}.
\newblock Lecture Notes in Mathematics, Vol. 576. Springer-Verlag, Berlin-New
  York, 1977.


\bibitem[V79]{Vergne79}
M.~Vergne.
\newblock On {R}ossmann's character formula for discrete series.
\newblock {\em Invent. Math.}, 54(1):11--14, 1979.

\bibitem[W88]{Wallach88}
N.~R. Wallach.
\newblock {\em Real reductive groups. {I}}, volume 132 of {\em Pure and Applied
  Mathematics}.
\newblock Academic Press Inc., Boston, MA, 1988.

\end{thebibliography}
\end{document}